\documentclass[10pt, hidelinks]{amsart} 
 \usepackage{enumitem}
\usepackage[all, cmtip]{xy}
\usepackage{amsmath}
\newtheoremstyle{exampstyle}
  {1pt} 
  {1pt} 
  {} 
  {} 
  {\bfseries} 
  {.} 
  {.5em} 
  {} 

\usepackage[
  bookmarksnumbered=true,
  psdextra,
]{hyperref}
\usepackage{bookmark}
\hypersetup{
  bookmarksopen=true,
  unicode=true             
}
\usepackage{textgreek} 
\usepackage{hypcap}

\usepackage{amscd} 
\usepackage{amssymb}
\usepackage{mathtools}
\usepackage[small,nohug,heads=vee]{diagrams} 

\newcommand{\myhocolim}[1]{\mathbin{\operatorname*{hocolim}_{#1}^{}}}
\newcommand{\mywedge}[1]{\mathbin{\operatorname*{\wedge}_{#1}^{}}}

\newcommand{\myotimes}[1]{\mathbin{\operatorname*{\otimes}_{#1}^{}}}

\DeclareMathOperator{\Alg}{Alg}
\DeclareMathOperator{\Ob}{Ob}
\DeclareMathOperator{\sgn}{sgn} 

\DeclareMathOperator{\Mor}{Mor}
\DeclareMathOperator{\Wh}{Wh}
\DeclareMathOperator{\Barr}{Bar}
\DeclareMathOperator{\Ho}{Ho}

\DeclareMathOperator{\id}{id}
\newcommand{\mywedgetwo}[2]{\mathbin{\operatorname*{\wedge}_{#1}^{#2}}}
\newcommand{\mytimestwo}[2]{\mathbin{\operatorname*{\times}_{#1}^{#2}}}

\usepackage{hyperref}

\makeatletter 
\newcommand*{\ovA}[1]{%
  \m@th\overline{#1\raisebox{2.5mm}{}}%
}
\newcommand*{\ovB}[1]{%
  \m@th\overline{#1\raisebox{1.5mm}{}}%
}
\makeatother
\usepackage[mathscr]{euscript}
\let\euscr\mathscr \let\mathscr\relax 
\usepackage[scr]{rsfso}
\usepackage{graphicx}
\graphicspath{ {images/} }
\usepackage{wrapfig}
\DeclareMathOperator{\coeq}{coeq}
\DeclareMathOperator{\gr}{gr} 

\DeclareMathOperator*{\Orb}{Orb} 

\DeclareMathOperator{\Mod}{Mod}
\DeclareMathOperator{\Ind}{Ind}
\DeclareMathOperator{\End}{End}
\DeclareMathOperator{\SSeq}{SSeq}
\DeclareMathOperator{\Po}{Po}

\DeclareMathOperator{\Sing}{Sing}

\DeclareMathOperator{\chara}{char}

\DeclareMathOperator{\bfCAlg}{\mathbf{CAlg}}
\DeclareMathOperator{\bfCMon}{\mathbf{CMon}}

\DeclareMathOperator{\hSCR}{hSCR}
\DeclareMathOperator{\hsMod}{hsMod}
\DeclareMathOperator{\Comm}{Comm}
\DeclareMathOperator{\sSet}{sSet}
\DeclareMathOperator{\sk}{sk}

\newcommand{\NN}{\mathbb{N}}

\usepackage{upgreek}
\usepackage{caption}
\newcommand{\ZZ}{\mathbb{Z}}
\newcommand{\QQ}{\mathbb{Q}}
\newcommand{\RR}{\mathbb{R}}
\newcommand{\FF}{\mathbb{F}}

\newcommand{\DD}{\mathbb{D}}

\DeclareMathOperator{\Fun}{Fun}

\DeclareMathOperator{\AQ}{AQ}

\DeclareMathOperator{\sCpl}{sCpl}
\DeclareMathOperator{\symsset}{SymsSet}
\DeclareMathOperator{\Fin}{Fin}
\DeclareMathOperator{\Set}{Set}

\DeclareMathOperator{\im}{im}

\DeclareMathOperator{\Stab}{Stab}

\DeclareMathOperator{\coker}{coker}
\DeclareMathOperator{\cofib}{cofib}

\DeclareMathOperator{\St}{St}

\DeclareMathOperator{\BT}{BT}

\numberwithin{equation}{section}

\DeclareMathOperator{\CW}{CW}  
\newtheorem*{question}{Question}
\theoremstyle{exampstyle}\theoremstyle{plain}  
\theoremstyle{exampstyle}
\theoremstyle{exampstyle}
\theoremstyle{exampstyle}
\theoremstyle{exampstyle}

\theoremstyle{exampstyle}\newtheorem{warning}[equation]{Warning}
\theoremstyle{exampstyle}\newtheorem{theorem}[equation]{Theorem}
\theoremstyle{exampstyle}\newtheorem{lemma}[equation]{Lemma}
\theoremstyle{exampstyle}\newtheorem{proposition}[equation]{Proposition}
\theoremstyle{exampstyle}\newtheorem*{nonquestion}{Question}{}
\theoremstyle{exampstyle}\newtheorem{corollary}[equation]{Corollary}
\theoremstyle{exampstyle}\newtheorem*{claim}{Claim}
\theoremstyle{exampstyle}

\theoremstyle{exampstyle}\theoremstyle{definition}
\theoremstyle{exampstyle}
\theoremstyle{exampstyle}\newtheorem{definition}[equation]{Definition}

\newtheorem{notation}[equation]{Notation}

\theoremstyle{exampstyle}\newtheorem{remark}[equation]{Remark}
\theoremstyle{exampstyle}\newtheorem*{nonremark}{Remark}
\theoremstyle{exampstyle}\newtheorem*{nonexample}{Example}

\newtheorem*{remarks*}{Remarks}
\theoremstyle{exampstyle}\newtheorem{examples}[equation]{Examples}
\theoremstyle{exampstyle}\newtheorem{example}[equation]{Example}

\theoremstyle{exampstyle}\newtheorem*{VariableNoNum}{{\VariableText}}
\theoremstyle{exampstyle}\newtheorem{Variable}[equation]{{\VariableText}}

\theoremstyle{exampstyle}\theoremstyle{definition}
\theoremstyle{exampstyle}\newtheorem*{VariableNoNumBold}{{\VariableText}}
\theoremstyle{exampstyle}\newtheorem{VariableBold}[equation]{{\VariableText}}

\theoremstyle{exampstyle}\theoremstyle{definition}

 \theoremstyle{exampstyle}

\newlength{\asidelength}

\def\Changed/{\ifvmode\else\vadjust{%
\vbox to 0pt{\vskip -\baselineskip%
\hbox to 0pt{\hss\vrule height 0pt depth 1.2\baselineskip\hskip 1em}\vss}}\fi}
\def\CHanged{\ifvmode\else\vadjust{%
\vbox to 0pt{\vskip -\baselineskip%
\hbox to 0pt{\hss\vrule height 0pt depth 1.2\baselineskip\hskip 1em}\vss}}\fi}

\def\Math#1{\def\MathString{#1}\futurelet\MathDelim\MathChoose}
\def\MathChoose{\ifmmode\let\MathDo\MathString%
              \else\let\MathDo\MathSkip\fi%
              \MathDo}
\def\MathSkip{\ifx\MathDelim/\def\MathDo{$\MathString$\EatOne}%
              \else\def\MathDo{$\MathString$}\fi%
              \MathDo}
 
\def\Text#1{\def\TextString{#1}\futurelet\TextDelim\TextSkip}
\def\TextSkip{\ifx\TextDelim/\def\TextDo{\TextString\EatOne}%
              \else\let\TextDo\TextString\fi%
              \TextDo}
\def\EatOne#1{}

    \newtheoremstyle{TheoremNum}
        {}{}              
        {\itshape}                      
        {}                              
        {\bfseries}                     
        {.}                             
        { }                             
        {\thmname{#1}\thmnote{ \bfseries #3}}
    \theoremstyle{TheoremNum}
    \newtheorem{theoremn}{Theorem}
     \newtheoremstyle{DefinitionNum}
        {}{}              
        {\itshape}                      
        {}                              
        {\bfseries}                     
        {.}                             
        { }                             
        {\thmname{#1}\thmnote{ \bfseries #3}}
    \theoremstyle{DefinitionNum}
    \newtheorem{definitionn}{Definition}
    \newtheoremstyle{LemmaNum}
        {}{}              
        {\itshape}                      
        {}                              
        {\bfseries}                     
        {.}                             
        { }                             
        {\thmname{#1}\thmnote{ \bfseries #3}}
            \newtheorem{lemman}{Lemma}

    \newtheorem{corollaryn}{Corollary}
  \newtheoremstyle{CorollaryNum}
        {}{}              
        {\itshape}                      
        {}                              
        {\bfseries}                     
        {.}                             
        { }                             
        {\thmname{#1}\thmnote{ \bfseries #3}}

  \newtheoremstyle{PropositionNum}
        {}{}              
        {\itshape}                      
        {}                              
        {\bfseries}                     
        {.}                             
        { }                             
        {\thmname{#1}\thmnote{ \bfseries #3}}

\usepackage{tikzscale}

\def\SkipToEndScan#1\EndScan{}
\def\Scan#1#2#3{\ifx#1#2#3\expandafter\SkipToEndScan\fi\Scan#1}
\def\Upper#1{%
\Scan#1aAbBcCdDeEfFgGhHiIjJkKlLmMnNoOpPqQrRsStTuUvVwWxXyYzZ#1#1\EndScan}
\def\Phrase#1 #2/#3/#4=#5 #6/#7/#8.{%
\expandafter\edef\csname#2#3\endcsname{\noexpand\Text{#6#7}}
\expandafter\edef\csname\Upper#2#3\endcsname{\noexpand\Text{\Upper#6#7}}
\expandafter\edef\csname#1#2#3\endcsname{\noexpand\Text{#5 #6#7}}
\expandafter\edef\csname\Upper#1#2#3\endcsname{\noexpand\Text{\Upper#5 #6#7}}
\expandafter\edef\csname#2#4\endcsname{\noexpand\Text{#6#8}}
\expandafter\edef\csname\Upper#2#4\endcsname{\noexpand\Text{\Upper#6#8}}
}
\usepackage{tikz}   

\newcommand{\orb}[1]{{#1\kern-2pt}{\scriptscriptstyle\rm{-orb}}}

\newcommand{\Smash}{\wedge}

\newcommand{\GL}{\operatorname{GL}} 

\newcommand{\HH}{\operatorname{H}}
\newcommand{\EE}{\operatorname{E}}
\newcommand{\PP}{\operatorname{P}}

\makeatletter
\newcommand*{\relrelbarsep}{.386ex}
\newcommand*{\relrelbar}{%
  \mathrel{%
    \mathpalette\@relrelbar\relrelbarsep
  }%
}
\newcommand*{\@relrelbar}[2]{%
  \raise#2\hbox to 0pt{$\m@th#1\relbar$\hss}%
  \lower#2\hbox{$\m@th#1\relbar$}%
}
\providecommand*{\rightrightarrowsfill@}{%
  \arrowfill@\relrelbar\relrelbar\rightrightarrows
}
\providecommand*{\leftleftarrowsfill@}{%
  \arrowfill@\leftleftarrows\relrelbar\relrelbar
}
\providecommand*{\xrightrightarrows}[2][]{%
  \ext@arrow 0359\rightrightarrowsfill@{#1}{#2}%
}
\providecommand*{\xleftleftarrows}[2][]{%
  \ext@arrow 3095\leftleftarrowsfill@{#1}{#2}%
}
\makeatother

\newcommand{\Map}{\operatorname{Map}}

\renewcommand{\star}{\operatorname{St}}
\newcommand{\link}{\operatorname{Lk}}

\newcommand{\rk}{\operatorname{rk}}

\newcommand{\thhh}{\tilde{\operatorname{h}}}
\def\HomotopyOrbit#1on#2/{\ensuremath{#2_{\widetilde{\HH}h#1}}}
\def\RedHomotopyOrbit#1on#2/{\ensuremath{#2_{\thhh#1}}}

\DeclareMathOperator{\Lie}{\Lcal ie}
\DeclareMathOperator{\Lierep}{Lie}

\DeclareMathOperator{\Aff}{Aff}

\DeclareMathOperator{\br}{Br}

\DeclareMathOperator{\Free}{Free}

\DeclareMathOperator{\Sp}{Sp}
\DeclareMathOperator{\hSp}{hSp}

\DeclareMathOperator{\Tor}{Tor}

\def\doCal#1{%
\ifx#1\doAllCalEnd\def\doAllCal{\relax}\else%
 \expandafter\edef\csname#1cal\endcsname{{\noexpand\mathcal #1}}\fi}
\def\doAllCal#1{\doCal#1\doAllCal}
\doAllCal ABCDEFGHIJHLMNOPQRSTUVWXYZ\doAllCalEnd
 
\def\doBar#1{%
\ifx#1\doAllBarEnd\def\doAllBar{\relax}\else%
 \expandafter\edef\csname#1bar\endcsname{{\noexpand\overline{#1}}}\fi}
\def\doAllBar#1{\doBar#1\doAllBar}
\doAllBar ABCDEFGHIJHLMNOPQRSTUVWXYZ\doAllBarEnd

\def\doWiggle#1{%
\ifx#1\doAllWiggleEnd\def\doAllWiggle{\relax}\else%
 \expandafter\edef\csname#1wiggle\endcsname{{\noexpand\tilde{#1}}}\fi}
\def\doAllWiggle#1{\doWiggle#1\doAllWiggle}
\doAllWiggle ABCDEFGHIJHLMNOPQRSTUVWXYZabcdefghijklmnopqrstuvwxyz\doAllWiggleEnd

\usepackage{float}

\newcommand{\hobased}{{{\text{h}}}}

\newcommand{\n}{ {\mathbf{n}} }
\newcommand{\m}{{\mathbf{m}}}

\newcommand{\pgroupsntof}[1]{\Scal_{p}\kern-1pt\left(#1\right)}
\newcommand{\pgroupsof}[1]{{\overline{\Scal}_{p}}\kern-1pt\left(#1\right)}

\newcommand{\reals}{{\mathbb{R}}}
\newcommand{\integers}{{\mathbb{Z}}}

\newcommand{\field}{{\mathbb{F}}}

\Phrase a superfluous//=a prunable//.

\usepackage{datetime}
\newdate{date}{06}{09}{2012}
\date{}

\begin{document} 
\title{ The Action of Young Subgroups on the Partition Complex}

\author{Gregory Z. Arone}
\address{Kerchof Hall, U. of Virginia, P.O. Box 400137,
         Charlottesville VA 22904 USA}
\curraddr{Department of Mathematics, Stockholm University, SE - 106 91 Stockholm, Sweden}         
\email{gregory.arone@math.su.se} 
\author{D. Lukas B. Brantner}
\address{Department of Mathematics, Harvard University, 1 Oxford Street,
         Cambridge MA 02138 USA}
\curraddr{Merton College, Oxford University, Merton Street, Oxford OX1 4JD, United Kingdom}         
\email{brantner@maths.ox.ac.uk}  



\begin{abstract}   
We study  the restrictions,  the strict fixed points, and the strict quotients of the partition \mbox{complex $|\Pi_n|$,}  which is  the $\Sigma_n$-space attached to the poset of proper nontrivial partitions of the set $\{1,\ldots,n\}$.\\
We express the space  of fixed points $|\Pi_n|^G$ in terms of subgroup posets for general   $G\subset \Sigma_n$ and prove a formula for the restriction of $|\Pi_n|$ to Young subgroups $\Sigma_{n_1}\times \dots\times \Sigma_{n_k}$. 
Both results follow by applying a  general method, proven with discrete Morse theory,  for producing equivariant branching rules on lattices with group actions. \\
We uncover surprising links between strict Young quotients of $|\Pi_n|$,  commutative monoid spaces, and the cotangent fibre in derived algebraic geometry. These connections allow us to construct a cofibre sequence relating various strict quotients $|\Pi_n|^\diamond \mywedge{\Sigma_n} (S^\ell)^{\wedge n}$ and give a   combinatorial proof of a splitting in derived algebraic geometry. \\
Combining all our results, we decompose strict Young quotients of $|\Pi_n|$ in terms of ``atoms" $|\Pi_d|^\diamond\mywedge{\Sigma_d} (S^\ell)^{\wedge d}$ for $\ell$ odd and compute their homology. We thereby also generalise  Goerss' computation of the algebraic Andr\'{e}-Quillen homology of trivial square-zero extensions from $\FF_2$ to $\FF_p$ for $p$ an odd prime.
\end{abstract}
\maketitle  
\
\\
\\
\\
\\
\\
\\
\tableofcontents

\newpage \vspace{-4pt}
\section*{Exposition} \vspace{-5pt}
Let $\Pi_{n}$ denote the poset of proper nontrivial partitions of  $\{1,\ldots,n\}$ and write $|\Pi_n|$ for its geometric realisation, the \textit{partition complex}. There is a well-known equivalence $ |\Pi_n|\simeq \bigvee_{(n-1)!} S^{n-3}.$
However, this equivalence does  \textit{not} preserve the natural action of the symmetric group $\Sigma_n$.  

The {equivariant} topology of  the space $|\Pi_n|$ arises in  numerous contexts of interest:  its \mbox{homology} group $\widetilde{\HH}_{n-3}(|\Pi_n|,\ZZ)$ is closely related to the Lie representation, its desuspended Spanier-Whitehead dual parametrises spectral Lie algebras and Goodwillie's Taylor expansion of spaces, and, as we shall prove, the homology of its strict quotients computes the cotangent fibre of certain simplicial commutative rings. Moreover,  $|\Pi_n|$ is equivalent to the complex of non-connected graphs on $n$ vertices, which appears in Vassiliev's  work on knot theory. \label{pageone}
\vspace{2pt}

In \hspace{-2.5pt} this \hspace{-2.5pt} article, we  \hspace{-2.5pt}  describe \textit{general \hspace{-2.5pt} fixed \hspace{-2.5pt} points},\hspace{-1pt}  \textit{Young \hspace{-2.5pt} restrictions}, \hspace{-2.5pt} \mbox{and\hspace{-1.0pt}  \textit{strict \hspace{-4.5pt} Young \hspace{-4pt} quotients} \hspace{-3.5pt} of \hspace{-3.5pt} $|\Pi_n|$.} Our results for  {fixed points} and {restrictions} take the form of \textit{branching rules}, i.e. equivariant equivalences to  wedge sums of simpler spaces. We produce these rules by \mbox{applying a new general algorithm} which takes a $G$-lattice $\mathcal{P}$ and a list of functions $(F_1,\ldots,F_k)$ as input and gives a rule for how to equivariantly collapse a subspace of $|\mathcal{P}-\{\hat{0},\hat{1}\}|$ as output. In the simplest instance, we recover a common generalisation of work of  Bj\"{o}rner-Walker \cite{bjorner1983homotopy}, Welker \cite{welker1990homotopie}, and Kozlov \cite{kozlov1998order}.
\begin{nonexample}
If we feed our algorithm  $\Pi_4$ with its $\Sigma_{2}\times {\Sigma_2}$-action and a suitably chosen pair of functions, it collapses the contractible subcomplex  {drawn with thin lines on the left and gives rise to the bouquet of circles on the right:}\  \vspace{-3pt}
 \begin{figure}[H] 
  \centering
    \includegraphics[width=0.82\textwidth]{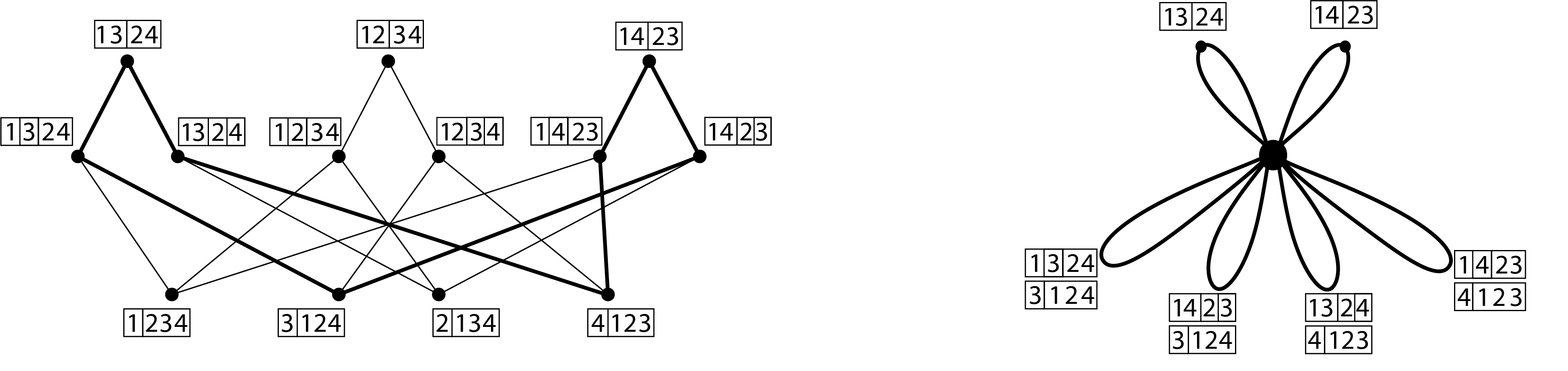}
\end{figure} 
 \end{nonexample}\vspace{-13pt}
\mbox{In general, our branching rule for restrictions of $\Pi_n$ to Young subgroups $\Sigma_{n_1}\times \ldots \times \Sigma_{n_k}$ 
reads\vspace{-1pt}}
$$
|\Pi_n| \xrightarrow{ \ \ \ \simeq\ \ \ } \bigvee_{d|\gcd(n_1, \ldots, n_k)}  \bigvee_{B(\frac{n_1}{d}, \ldots, \frac{n_k}{d})} \Ind^{\Sigma_{n_1}\times \cdots \times {\Sigma_{n_k}}}_{\Sigma_d}\left(\Sigma^{-1} (S^{\frac{n}{d}-1})^{\wedge d} \wedge |\Pi_d|^\diamond\right) .\vspace{-1pt}
$$
Here $B(m_1, \ldots, m_k)$ denotes the set of words in letters $c_1,\ldots,c_k$ which are lexicographically minimal among their own cyclic rotations and involve the letter $c_i$  exactly $m_i$ times. \vspace{3pt} Its size is given by a famous formula due to Witt \cite{witt1937treue} as $\displaystyle \   \frac{1}{m}\sum_{d|\gcd(m_1,\ldots,m_k)} \mu(d){\frac{m}{d} \choose \frac{m_1}{d},\ldots,\frac{m_k}{d}}\   $  \mbox{for $m= m_1+\ldots + m_k$.} \vspace{1pt}

The branching rule allows us to split strict Young quotients of $|\Pi_n|$, and we  decompose them even further by constructing a new cofibre sequence of \textit{strict} orbit spaces for each $d $ and each $n $ even: \vspace{0pt}
$$ \Sigma^2 |\Pi_{\frac{d}{2}}|^\diamond \mywedge{\Sigma_{\frac{d}{2}}} (S^{2n+1})^{\wedge \frac{d}{2}} 
\rightarrow  \Sigma^2 |\Pi_d|^\diamond \mywedge{\Sigma_d} (S^{n})^{\wedge d} \rightarrow  \Sigma |\Pi_d|^\diamond \mywedge{\Sigma_d} (S^{n+1})^{\wedge d}.\vspace{-1pt}$$
\mbox{This sequence is a ``strict'' analogue of the Takayasu cofibration sequence  (cf.\ \cite{takayasu1999stable}, \cite{kuhn2001new}).}

Based on these topological results, we compute the homology of strict Young quotients of $|\Pi_n|$  with coefficients in $\QQ$ and $\FF_p$ for all primes $p$.  From our computations, we can read off the algebraic Andr\'{e}-Quillen homology of trivial square-zero extensions over these fields, thereby generalising work of Goerss over $\FF_2$ and establishing a key tool in the emerging field of deformation  theory in characteristic $p$, cf.\ \cite{brantner2019deformation} \cite{brantner2020purely}.

\section{Statement of Results}
We shall begin by describing our general combinatorial branching algorithm and  then proceed to explain its \mbox{consequences} for fixed points and restrictions of  partition complexes and Bruhat-Tits buildings. Finally, we will discuss several  computations and conceptual connections concerning    strict quotients of partition complexes.

\subsection*{Complementary Collapse} Fix a finite group $G$ acting in an order-preserving manner on a finite lattice $\mathcal{P}$, i.e.\ a finite poset containing all binary meets $x\wedge y$ and joins $x \vee y$. Let  $\mathcal{F}_\mathcal{P}$ denote the collection of nondegenerate chains in $\mathcal{P}$ and write $ \ovA{\mathcal{P}} = \mathcal{P}-\{\hat{0},\hat{1}\}$ for the poset obtained by removing the minimum $\hat{0}$ and the maximum $\hat{1}$ from $\mathcal{P}$. We introduce the following notion:
\begin{definitionn}[\ref{orthogonalityfunction}] 
A function
$F: \mathcal{F}_\mathcal{P} \rightarrow \mathcal{P}$ is called an \textit{orthogonality function} if 
\begin{enumerate}[leftmargin=26pt]
\item $F$ is $G$-equivariant and  increasing  (i.e. $y\leq F(\sigma)$ for  every $\sigma \in \mathcal{F}_\mathcal{P}$ and every  $y\in \sigma$.)
\item For any $\sigma = [y_0 < \dots < y_m] \in  \mathcal{F}_\mathcal{P}$ and $z>y_{m}$, the following subposet   is discrete:
$$\{\ \  y_m  < t< z \  \ \ |\ \  \ t \wedge F(\sigma) =y_m\ \  , \ \ (t\vee F(\sigma)) \wedge z = z \ \   \}. $$
\end{enumerate}
\end{definitionn}
Lists $\mathbf{F} = (F_1,\dots,F_n)$ of orthogonality functions are examples of ``\textit{orthogonality fans}'' (cf.\  Definition \ref{fans}), and there is a notion for when a chain \mbox{$\sigma = [ y_0 < \dots < y_r ] $} is orthogonal to a fan $\mathbf{F}$, written $\sigma \perp \mathbf{F}$ (cf.\ Definition \ref{globalinv}).  
Using  discrete Morse theory  (cf.\  \cite{forman1998morse},  \cite{freij2009equivariant}), we prove:

\begin{theoremn}[\ref{CollapseFansB} (Complementary Collapse)]
Let $\mathbf{F} = (F_1,\dots,F_n)$ be an orthogonality \mbox{fan on $\mathcal{P}$} with $F_1([\hat{0}]) \neq \hat{0},\hat{1}$.
\mbox{There is a $G$-equivariant simple homotopy equivalence}
$$| \ovA{\mathcal{P}}|    \ \xrightarrow{\ \  \simeq \ \  }
 \bigvee_{[ y_0 < \dots < y_r ] \perp \mathbf{F}}  | \ovA{\mathcal{P}}_{(\hat{0},y_0)}|^\diamond \wedge \Sigma| \ovA{\mathcal{P}}_{(y_0,y_1)}|^\diamond \wedge \dots \wedge \Sigma| \ovA{\mathcal{P}}_{(y_{r-1},y_r)}|^\diamond \wedge | \ovA{\mathcal{P}}_{(y_r,\hat{1})}|^\diamond. $$
\end{theoremn}
Here $ \ovA{\mathcal{P}}_{(a,b)}$ denotes the subposet of elements $z$ with $a < z<b$, and  the unreduced suspension of a space $X$ is denoted by $X^\diamond$. If $X$ is pointed, we write $\Sigma X$ for its reduced suspension. \mbox{An \textit{equivariant simple homotopy equivalence} is an equivariant equivalence} which can be obtained by iterated elementary expansions and collapses (cf.\ Definition \ref{elemexp}). 
\vspace{3pt}

 Applying our  theorem to the case where $\mathbf{F}$ consists of a  \textit{single} function $F_1$ with $F_1(\hat{0}) = x$ and $F_1(y) = \hat{1}$ for $y>\hat{0}$, we recover Bj\"{o}rner-Walker's complementation formula (cf.\ \cite{bjorner1983homotopy}) together with its generalisations by Kozlov \cite{kozlov1998order} and Welker \cite{welker1990homotopie}.\vspace{4pt}

Complementary collapse constitutes a powerful tool in  poset topology, as we will demonstrate now. \vspace{-18pt}

\subsection*{Fixed Points} We write $\mathcal{P}_n$ for the lattice of partitions of $\mathbf{n} = \{1,\ldots,n\}$, ordered 
so that $\sigma \leq \tau$ if $\sigma$ refines $\tau$. Set $\Pi_n = \mathcal{P}_n - \{\hat{0},\hat{1}\}$.  Given a subgroup $G\subset \Sigma_n$, it is natural to ask:
\begin{nonquestion}
What is the $W_{\Sigma_n}(G)=N_{\Sigma_n}(G)/G$-equivariant simple homotopy type of  $|\Pi_n|^G$?
\end{nonquestion}

If $G$ acts {transitively} on $\n = \{1,\dots,n\}$, then $\n$ can be identified, as a $G$-set, with the coset space $G/H$ for $H$ the stabiliser of the element $1$, say. For such transitive actions, it is not difficult to show that 
the poset of $G$-invariant partitions of $G/H$ is isomorphic to the poset of subgroups of $G$ that contain $H$. This is Lemma~\ref{lemma: transitive} (it had appeared in the paper of White-Williamson~\cite[Lemma 3]{white1976combinatorial}, and they attributed the result to Klass~\cite{klass1973enumeration}).\\
For general $G\subset \Sigma_n$, this question is more difficult. Group actions on posets fall into two categories: either all orbits are equivariantly isomorphic, in which case the action is called \emph{isotypical}, or they are not, in which case the action is called \emph{non-isotypical}. 

Complementary collapse reduces the above question to  the well-understood transitive case:\vspace{-3pt}
\begin{theoremn}[\ref{fixedpoints}]
If $G$ acts isotypically,    relabel $\n$ so that $G$ is a transitive subgroup of a \mbox{diagonally} embedded  $\Sigma_d\xrightarrow{\Delta} \Sigma_d^{\frac{n}{d}}\subset \Sigma_n$ for $d\ |\ n$.
There is a $W_{\Sigma_n}(G)=N_{\Sigma_n}(G)/G$-equivariant simple   \mbox{equivalence} \vspace{-3pt}
$$|\Pi_n|^G  \ \ \ \xrightarrow{\ \  \simeq \ \  }\ \ \Ind^{W_{\Sigma_n}(G)}_{W_{\Sigma_d}(G)\times  \Sigma_{\frac{n}{d}}} \ \  (|\Pi_d|^G )^\diamond \wedge |\Pi_{\frac{n}{d}}|^\diamond.$$
\end{theoremn}
Here we used the following notation: given a subgroup $H\subset G$ and a pointed $H$-space $X$, the induced $G$-space $\Ind^G_H(X) = G_+\mywedge{H} X$ is   \mbox{the wedge of $|G/H|$ copies of $X$ with  its natural $G$-action.}
\vspace{-13pt}

In the remaining case, we have:\vspace{-4pt}
\begin{lemman}[\ref{easy}]
If $G$ acts non-isotypically, then $|\Pi_n|^G$ is $W_{\Sigma_n}(G)$-equivariantly collapsible.\vspace{-4pt}\end{lemman}
\begin{nonremark}
This lemma also has a very straightforward and direct proof. It has been  observed independently by Markus Hausmann. \vspace{2pt}
\end{nonremark}

In ~\cite{arone2016bredon}, the authors proved that the only $p$-groups $P$ for which the fixed point space $|\Pi_n|^P$ is not contractible are elementary abelian groups acting freely on $\n$.
It was also observed that if $n=p^k$ and $P\cong \FF_p^k$ acts freely and transitively on $\n$, then $\Pi_n^P$ is isomorphic to the poset $\BT(\FF_p^k)$ of proper non-trivial subgroups of $P$, which is closely related 
to the Tits building for $\GL_k(\field_p)$.  
Combining these observations with our Theorem \ref{fixedpoints} and Lemma \ref{easy}, we can complete the calculation and describe all fixed point spaces of the partition complex with respect to  $p$-groups:

\begin{corollaryn}[~\ref{corollary: el abelian}]
Let $P\subset \Sigma_n$ be a $p$-group.
\begin{enumerate}[leftmargin=26pt]
\item If $P\cong \FF_p^k$ is elementary abelian acting freely on $\n$ for  $n=mp^k$ \mbox{(here $p$ may divide $m$),}  {we have}  $\Aff_{\FF_p^k} := N_{\Sigma_{p^k}} ({\FF_p^k}) \cong  \FF_p^k \rtimes \GL_k(\FF_p)$ and  \mbox{$\Aff_{\FF_p^k \wr \Sigma_m}:= N_{\Sigma_n} (\FF_p^k ) \cong (\FF_p^k)^{ m} \rtimes (\GL_k(\FF_p) \times \Sigma_m)$.}\vspace{3pt} \\
There is an $\Aff_{\FF_p^k \wr \Sigma_m}$-equivariant simple  homotopy equivalence\vspace{-3pt}
$$|\Pi_n|^P   \ \ \ \xrightarrow{\ \  \simeq \ \  }\ \ \ 
 \Ind^{\Aff_{\FF_p^k \wr \Sigma_m}}_{\Aff_{\FF_p^k} \times \Sigma_m}  ( |\BT(\FF_p^k)|^\diamond \wedge |\Pi_m|^\diamond).$$ 
As before, $\BT(\FF_p^k)$ is  the poset of proper nontrivial subspaces of  $\FF_p^k$. 
Nonequivariantly, this implies that $|\Pi_n|^P$ is a bouquet of $(m-1)! \cdot p^{k(m-1)+{k\choose 2}}$ spheres of dimension ${m+k-3}.$
\item If the action of $P$ is not of this form, then $|\Pi_n|^P$ is $W_{\Sigma_n}(P)$-equivariantly contractible.
\end{enumerate}
\end{corollaryn}
Our techniques can also be used to examine fixed points under iterated wreath products:\vspace{-4pt}
\begin{lemman}[\ref{lemma: isotypical wreath}]
Suppose that $d \ |\ n$ and that $d$ factors as a product \mbox{$d=d_1\cdots d_l$} of integers \mbox{$d_1,\ldots, d_l>1$} with $l>1$. Consider the iterated wreath product $\Sigma_{d_1}\wr\cdots\wr\Sigma_{d_l}$ as a subgroup of $\Sigma_n$ via the diagonal embedding.
Then the space $|\Pi_n|^{\Sigma_{d_1}\wr\cdots\wr\Sigma_{d_l}}$ is collapsible.
\end{lemman}
\subsection*{Restrictions} We proceed to our results concerning the restrictions of $|\Pi_n|$ to Young subgroups. 
Fix positive integers $n_1,\ldots,n_k$ with $n=n_1+\ldots + n_k$. As before, we write $B(n_1,\ldots, n_k)$ for the set of  words in letters $c_1,\ldots,c_k$ which are lexicographically minimal among their   rotations  and involve the letter $c_i$ precisely $n_i$ times.
\mbox{Words with this minimality property are called \textit{Lyndon words}.}

One can choose orthogonality functions $(F_1,F_2)$ on $\Pi_n$ and apply complementary collapse to obtain:
\begin{theoremn}[\ref{main}]
There is a $\Sigma_{n_1}\times \cdots \times \Sigma_{n_k}$-equivariant simple homotopy equivalence
\begin{equation}\label{equation: main}
|\Pi_n| \ \ \ \xrightarrow{\ \ \simeq \ \  }   \ \ \ \bigvee_{\substack{d|\gcd(n_1, \ldots, n_k)\\ w\in B(\frac{n_1}{d}, \ldots, \frac{n_k}{d})}} \ \Ind^{\Sigma_{n_1}\times \cdots \times {\Sigma_{n_k}}}_{\Sigma_d}\left(\Sigma^{-1} (S^{ \frac{n}{d}-1})^{\wedge d} \wedge |\Pi_d|^\diamond\right).
\end{equation}
\end{theoremn}
The group \mbox{$\Sigma_{n_1}\times \cdots\times \Sigma_{n_k}\subset \Sigma_n$} is the subgroup of permutations in $\Sigma_n$ that preserve the partition $x=\left\{\{1,\ldots,n_1\},\{n_1+1,\ldots,n_1+n_2\},\ldots   \right\} $.

The asserted equivalence is pointed, where the basepoint of  $|\Pi_n|$ is given by the above partition $x$.

If $d>0$ divides all of $n_1, \ldots, n_k$, then $\Sigma_d^{ \frac{n}{d}}$
is a subgroup of $\Sigma_{n_1}\times\cdots \times\Sigma_{n_k}$.  We obtain a diagonal inclusion
$
\Sigma_d \hookrightarrow \Sigma_{n_1}\times \cdots\times \Sigma_{n_k}
$ and will identify $\Sigma_d$ with its image under this map. 

By $\Sigma^{-1} (S^{\frac{n}{d}-1})^{\wedge d}$, we mean a desuspension of the sphere $ (S^{\frac{n}{d}-1})^{\wedge d}$ with its natural  $\Sigma_d$-action. Such a desuspension exists since there is a $\Sigma_d$-equivariant homeomorphism between $S^d$ and the suspension of the reduced standard representation sphere of $\Sigma_d$.
\begin{remark}\label{smallspheres}
Our definition of $\Pi_d$ only works well for $d\ge 2$. For $d=2$, $|\Pi_2|$ is the empty set, which we identify with the $(-1)$-dimensional sphere, since its unreduced suspension $|\Pi_2|^\diamond$ is homeomorphic to $S^0$. We decree that $|\Pi_1|$ is the $(-2)$-dimensional sphere. In practice, this means that the space $|\Pi_1|$ is undefined, but for any suspension $X^\diamond$, we define $X^\diamond \wedge |\Pi_1|^\diamond := X$. \end{remark}
 \vspace{3pt}
\begin{nonexample}
If $\gcd(n_1, \ldots, n_k)=1$, Theorem~\ref{main} says that $|\Pi_n|$ is $\Sigma_{n_1}\times\cdots\times\Sigma_{n_k}$-equivariantly equivalent to a wedge sum 
$
\bigvee_{w \in B(n_1, \ldots, n_k)} ({\Sigma_{n_1}\times \cdots\times\Sigma_{n_k}})_+ \wedge S^{n-3}
$
of copies of $S^{n-3}$,  freely permuted by $\Sigma_{n_1}\times\cdots\times\Sigma_{n_k}$. In the special case  of $\Sigma_{n-1}\times \Sigma_1\subset \Sigma_n$, we note that $|B(n-1, 1)|=1$, and hence conclude that there is a $\Sigma_{n-1}$-equivariant homotopy equivalence 
$|\Pi_n|\simeq_{}{\Sigma_{n-1}}_+ \wedge S^{n-3}.
$
This equivalence is   well-known. For example, see~\cite{donau2012nice} for a closely related statement.

In the case of the subgroup $\Sigma_{n-2}\times \Sigma_2$ for $n$ even, one easily calculates that $|B(n-2, 2)|= \frac{n-2}{2}$ and $|B(\frac{n-2}{2}, 1)|=1$. We conclude that there is a $\Sigma_{n-2}\times \Sigma_2$-equivariant equivalence
\begin{equation}\label{equation: second case}
|\Pi_n|\simeq  \ \ \bigvee_{\frac{n-2}{2}} (\Sigma_{n-2}\times {\Sigma_2})_+ \wedge S^{n-3} \ \ \ \  \vee \ \ \ \  (\Sigma_{n-2}\times {\Sigma_2})_+\mywedge{\Sigma_2} S^{n-3}.
\end{equation}
In the right summand $S^{n-3}$ has an action of $\Sigma_2$: it is homeomorphic to  
$S^{\frac{n}{2}-2}\wedge (\hat{S}^1)^{\Smash{\frac{n}{2}-1}}$, where $S^{\frac{n}{2}-2}$ is a sphere with a trivial action of $\Sigma_2$ while $\hat{S}^1$ is a sphere on which $\Sigma_2$ acts by reflection about a line.  Formula~\eqref{equation: second case} was first obtained by Ragnar Freij, see~\cite{freij2009equivariant}, Theorems 5.3 and 5.5. 

For $n=4$, we obtain the $\Sigma_2\times\Sigma_2$-equivariant equivalence depicted in the illustration on p.\pageref{pageone}:
$$
|\Pi_4| \ \ \xrightarrow{\simeq} \ \ (\Sigma_{2}\times {\Sigma_2})_+ \wedge S^{1}\ \  \vee\ \  (\Sigma_{2}\times {\Sigma_2})_+\mywedge{\Sigma_2} \hat S^{1}.
$$
Here $|\Pi_4|$ is a one-dimensional complex. The map is defined by $\Sigma_2\times\Sigma_2$-equivariantly collapsing the subcomplex drawn in thin lines. The reader is invited to check that this subcomplex is $\Sigma_2\times\Sigma_2$-invariant, and that the quotient space of $|\Pi_4|$ by this subcomplex is indeed homeomorphic to $(\Sigma_{2}\times {\Sigma_2})_+ \wedge S^{1} \vee (\Sigma_{2}\times {\Sigma_2})_+\wedge_{\Sigma_2} \hat S^{1}$.
\end{nonexample}
Our Theorem~\ref{main} is a strengthening of the main result of~\cite{arone1998homology}.  In [op. cit], the authors applied Goodwillie calculus to the Hilton-Milnor theorem, which is a decomposition result for the topological free group generated by a wedge sum of connected spaces. It was possible to prove a weak version of Theorem~\ref{main} which holds only after one (a) makes the group action pointed-free, (b) applies the suspension spectrum functor. Moreover, the map defining this equivalence was not constructed explicitly in [op. cit.]. 

In this work, we first give an explicit point-set level description of the map in ~\eqref{equation: main} as a collapse map and then prove that it is an equivariant simple homotopy  equivalence of spaces. Such an explicit description is desirable for applications.

For example, partition complexes show up in the study of \textit{spectral Lie algebras}, i.e. algebras over the {spectral Lie operad} $\mathbf{Lie}$ defined by Salvatore  \cite{salvatore1998configuration} and Ching \cite{ching2005bar}.  The $n^{th}$ term of $\mathbf{Lie}$ is given by $\Map_{\Sp}(S^1,(S^1)^{\wedge n})\wedge \DD(\Sigma |\Pi_n|^\diamond) $, where $\DD $ denotes Spanier-Whitehead duality. Its structure maps  are constructed  by  ``grafting''  weighted trees. The homotopy  and homology groups of spectral Lie algebras form graded Lie algebras.

In Section \ref{LAPC1} and \ref{LAPC2}, we link the simplicial collapse maps  in Theorem \ref{main} to the structure maps of the spectral Lie operad. This allows us to give a concrete description of  free spectral Lie algebras on many generators. 
Indeed,  suppose that we are given spectra $X_1,\ldots,X_k$. Let  $w$ be a Lie word in letters $c_1,\ldots,c_k$  involving the letter $c_i$ precisely $|w|_i$ times. We can then define a natural  map $F_w:   X_1^{\wedge |w|_1}  \wedge\ldots \wedge X_k^{\wedge |w|_k} \xrightarrow{{F}_w} \Free_{\mathbf{Lie}}(X_1 \vee \ldots \vee X_k)$ to the free spectral Lie algebra on $X_1\vee \ldots \vee X_k$.
The induced map $\pi_\ast(X_1)^{\otimes |w|_1 } \otimes \ldots \otimes \pi_\ast(X_k)^{\otimes |w|_k }  \xrightarrow{ } \pi_\ast(\Free_{\mathbf{Lie}}(X_1 \vee \ldots \vee X_k))$ has the effect of sending an element $(x_1^1 \otimes \ldots \otimes x_1^{|w|_1} \otimes  x_2^1 \otimes  \ldots \otimes  x_2^{|w|_2}\otimes  \ldots )$ to the element  obtained by first replacing, for all $i$, the occurrences of $c_i$ in $w$ by $x_i^1,\ldots,x_i^{|w|_i}$ from left to right,  and then evaluating this Lie product in the Lie algebra $\pi_\ast(\Free_{\mathbf{Lie}}(X_1 \vee \ldots \vee X_k))$.  Theorem \ref{main} implies:\vspace{-3pt}
\begin{corollaryn}[\ref{BBBB}] \label{effectonhomotopy} 
Summing up the induced   maps $$\ovA{F}_w :  \Free_{\mathbf{Lie}}(X_1^{\wedge |w|_1}  \wedge \ldots \wedge X_k^{\wedge |w|_k}) \xrightarrow{{\ \ \ \ \ \ }} \Free_{\mathbf{Lie}}(X_1 \vee \ldots \vee X_k)$$ over all  words $w$ in the Lyndon basis $B_k$ for the free Lie algebra on $k$  letters, \mbox{we obtain an equivalence}
$$\bigvee_{w\in B_k}\ovA{F}_w : \ \ \bigvee_{w\in B_k}\Free_{\mathbf{Lie}}(  X_1^{\wedge|w|_1}  \wedge \ldots \wedge X_k^{\wedge |w|_k}) \xrightarrow{ \ \ \ \simeq \ \ \ } \Free_{\mathbf{Lie}}(X_1 \vee \ldots \vee  X_k).$$
\end{corollaryn}
\begin{remark} 
An equivalence of  this form can also be deduced by applying Goodwillie calculus to the Hilton-Milnor theorem, as established by the first-named author and \mbox{Marja Kankaanrinta  \cite{arone1998homology}.} \mbox{However,}
the maps  arising in this approach are mysterious and it is not a priori clear how they interact with the Lie bracket. Moreover, this 
approach does not give any genuine \mbox{equivariant information.}  These two significant obstructions to     applications are not present  in our explicit approach, which makes   the interaction with the Lie operad and  genuine equivariant phenomena   entirely \mbox{transparent.}
\end{remark}

We can also give an asymmetric decomposition for Young restrictions of $|\Pi_n|$. For this, we  fix a Young subgroup $\Sigma_{A} \times \Sigma_{B_1} \times \dots \times \Sigma_{B_k}\subset \Sigma_n$ and set $B=\cup_{i=1}^k B_i$. \mbox{Complementary collapse implies:}
\begin{theoremn}[\ref{inductivestep} (Symmetry Breaking)] \ \\
There is a simple $\Sigma_{A} \times\Sigma_{B_1} \times  \dots \times \Sigma_{B_k}$-equivariant homotopy equivalence spaces 
$$ |\Pi_n|   \ \longrightarrow \   \bigvee_{ \substack{A = A_1 \coprod \dots \coprod A_r,  \ A_i \neq \emptyset \\ f_i : A_i \hookrightarrow B \\ \mbox{ s.t. } \im(f_{i+1}) \subset \im(f_i)}} \Sigma^{-1}  S^{|A_1|} \wedge\dots \wedge S^{|A_r|}  \wedge |\Pi_B|^\diamond .$$ 
 \end{theoremn}

Partition complexes can be thought of as Bruhat-Tits buildings over ``the field with one element".  In this heuristic picture, Young subgroups correspond to parabolic subgroups. Complementary collapse then has a neat analogous consequence for parabolic restrictions of Bruhat-Tits buildings.\vspace{3pt}

Let $V$ be a finite-dimensional vector space over a finite field $k$.  Fix a flag $\mathbf{A} = [A_0<\dots < A_r] $ of proper nonzero subspaces of $V$, and write 
$P_{\mathbf{A}}$ for the associated parabolic subgroup.
Choose a complementary flag $\mathbf{B}=[B_0<\dots < B_r] $ to $\mathbf{A}$ with corresponding parabolic subgroup $P_\mathbf{B}$. Write $L_{\mathbf{A}\mathbf{B}} = P_\mathbf{A}\cap P_\mathbf{B}$ for the   intersecting Levi. Complementary collapse shows:\vspace{-2pt}
\begin{lemman}[\ref{BTb}] There is a $P_\mathbf{A}$-equivariant simple equivalence\vspace{-3pt}
$$ |\BT(V)| \simeq \Ind^{P_\mathbf{A}}_{L_{\mathbf{A}\mathbf{B}}} (\Sigma^{r} \bigwedge_{i = 0}^{r+1} |\BT(\gr^i(\mathbf{B}))|^\diamond ).\vspace{-3pt}$$
Here $\gr^i(\mathbf{B})=B_i/ {B_{i-1}}$ for $i=1,\ldots ,r$ and we set  $\gr^0(\mathbf{B}) = B_0$  and $\gr^{r+1}(\mathbf{B}) = V/ {B_r}$.
\end{lemman} 
This   gives  a new ``topological'' proof of a  classical   result in modular representation theory:\vspace{-2pt}
\begin{corollaryn}[\ref{steinbergproj}] 
Let $k = \FF_q$ be a finite field and assume that $R$ is any  ring in which the number $\prod_{k=1}^n  ({q^k-1}) $ is invertible. Write $\St = \tilde{\HH}_{n-2}(\BT(\FF_q^n),\ZZ)$ for the integral Steinberg module.
Then $\St \otimes R$ is a projective $R[\GL_n(\FF_q)]$-module. 
\end{corollaryn}

\subsection*{Strict Orbits.} 
Strict quotients of $|\Pi_n|$ by subgroups of $\Sigma_n$ have received some attention over the years. First of all, it is known that the quotient space of $|\Pi_n|$ by the full symmetric group is contractible for $n \ge 3$. This is due to Kozlov, \cite[Corollary 4.3]{kozlov2000collapsibility}. We give a new proof of this result in Corollary \ref{Kozlov} in the body of the paper. Results about quotient spaces by certain subgroups of $\Sigma_n$ were obtained, for example, by Patricia Hersh~\cite{hersh2003lexicographic}  (for groups $\Sigma_2 \wr \Sigma_m$) and Ralf Donau~\cite{donau2012nice} (for subgroups of $\Sigma_1 \times \Sigma_{n-1}$). 

In this paper, we will address the following question:\vspace{-2pt}
\begin{question} Given a Young subgroup $\Sigma_{n_1} \times \ldots \times \Sigma_{n_k} \subset \Sigma_n$, what is the homology and homotopy type of the strict Young quotient $|\Pi_n|/_{\Sigma_{n_1}\times \dots \times \Sigma_{n_k}} ?$ \end{question}
We start by observing the following immediate consequence of Theorem~\ref{main}:
\begin{corollaryn}[\ref{proposition: orbits again}]
 Suppose $n=n_1+\cdots+n_k$.  There is an equivalence $$
|\Pi_n|/_{\Sigma_{n_1}\times\cdots\times\Sigma_{n_k}} \xrightarrow{\ \ \simeq \ \ } \bigvee_{\substack{d|\gcd(n_1, \ldots, n_k) \\ B(\frac{n_1}{d}, \ldots, \frac{n_k}{d})}}  \left(\Sigma^{-1}(S^{\frac{n}{d}-1})^{\wedge d}\mywedge{\Sigma_d} |\Pi_d|^\diamond\right).
$$
More generally, given pointed spaces $X_1,\ldots,X_k$, there is an equivalence
$$ |\Pi_n|^\diamond \mywedge{\Sigma_n}(X_1\vee \ldots \vee X_k)^{\wedge n} \ \ \  \xrightarrow{\ \ \simeq \ \ } \ \ \  \bigvee_{\substack{n=n_1+\ldots+n_k \\ d|\gcd(n_1, \ldots, n_k) \\ w \in B(\frac{n_1}{d}, \ldots, \frac{n_k}{d})}} \left( (S^{\frac{n}{d}-1}\wedge X_1^{\frac{n_1}{d}} \wedge \ldots \wedge X_k^{\frac{n_k}{d}})^{\wedge d}\mywedge{\Sigma_d} |\Pi_d|^\diamond\right). $$
\end{corollaryn}
We are therefore reduced to answering the above questions for  spaces of the form\vspace{3pt} 
\mbox{$
\Sigma^{-1}(S^{\ell})^{\wedge d}\mywedge{\Sigma_d} |\Pi_d|^\diamond
$.}
Fix a prime $p$. To compute the $\FF_p$-homology of the  spaces in question, we consider the graded Mackey functor with $\mu_\ast(T) =\displaystyle  \widetilde{\HH}_*(S^{\ell n} \mywedge{\Sigma_n} T_+ ,\FF_p)$. There is  a spectral sequence of \vspace{-5pt} signature   $  \widetilde{\HH}_s^{\br}( |\Pi_n|^{\diamond}; \mu_t) \displaystyle \Rightarrow   \widetilde{\HH}_{s+t}(S^{\ell n} \mywedge{\Sigma_n} |\Pi_n|^\diamond  ,   \FF_p)$ from   Bredon homology to the   homology of \mbox{strict quotients.}\vspace{3pt}

If $\ell$ is  \textit{odd}, then this Mackey functor $\mu_\ast$ has desirable properties and indeed satisfies the conditions of the main   Theorem $1.1.$ of \cite{arone2016bredon}. Using this, we compute in our final Section \ref{theproof}:
\begin{theoremn}[\ref{theorem: thehomology}] Let $\ell \geq 1$ be an integer, assumed to be odd whenever the prime $p$ is odd.\\
If $n$ is not a power of $p$, then $\displaystyle \widetilde{\HH}_*(  |\Pi_n|^\diamond \mywedge{\Sigma_{n}}S^{\ell n} ,\FF_p)$ is trivial.\\ If $n=p^a$, then $\displaystyle \widetilde{\HH}_* (  |\Pi_{p^a}|^\diamond \mywedge{\Sigma_{p^a}} S^{\ell p^a},\FF_p )$ has a basis consisting of sequences $(i_1, \ldots, i_a)$, where $i_1, \ldots, i_a$ are positive integers satisfying:
\begin{enumerate}[leftmargin=26pt]
\item Each $i_j$ is congruent to $0$ or $1$ modulo $2(p-1)$. \label{congruence}
\item For all $1\le j<a$, we have $1<i_j< p i_{j+1}$.
\item We have $1<i_a\le (p-1)\ell$ (notice that if $p>2$, then~\eqref{congruence} means that the inequality is  strict).
\end{enumerate}
The homological degree of $(i_1, \ldots, i_a)$ is $i_1+\cdots+i_a+\ell+a-1$.
\end{theoremn}

To reduce the case where $\ell$ is even to the case where $\ell$ is odd, a  new conceptual insight is required. The key observation   is that for $p=2$, the above basis has the same form as the algebraic Andr\'{e}-Quillen homology $\AQ^{\FF_2}_\ast(\FF_2 \oplus \Sigma^{\ell} \FF_2)$ of the trivial square zero extension of $\FF_2$ by a generator in simplicial degree $j$ as computed by Goerss (cf.\ \cite{goerss1990andre}).
This similarity suggests that strict quotients of the partition complex are also the {Andr\'{e}-Quillen homology of  suitably defined objects.}\vspace{3pt}

 Indeed, the correct structure to consider here is that of a  \textit{(strictly) commutative monoid space}, i.e. a space $X$ together with a continuous, commutative, and associative multiplication law,  a unit $1$, and a contracting element $0$. \mbox{Heuristically, these monoids give ``simplicial commutative rings over $\FF_1$''.}

More formally, we consider the pointed model category $\mathbf{CMon}^{aug}$ of strictly commutative monoid spaces augmented over $S^0$, the monoid with two distinct elements $0$ and $1$.
Any pointed space $X$ gives rise to an augmented commutative monoid space $S^0 \vee \ovA{X}$ by declaring that $a\cdot b = 0$ unless $a=1$ or $b=1$. We call this  the \textit{trivial square zero extension} of $S^0$ by $X$. There is a natural notion of \textit{Andr\'{e}-Quillen homology}, denoted by $\AQ$, for these commutative monoid spaces.  \vspace{3pt}

The following result establishes the crucial connection to strict orbits of the partition complex:
\begin{lemman}[\ref{strictlie}]
If $X$ is a well-pointed  space, then 
$\AQ(S^0 \vee\ovA{X})\simeq \bigvee_{n\geq 1}\displaystyle \Sigma |\Pi_n|^\diamond \mywedge{\Sigma_n} X^{\wedge n}$.
\end{lemman}

We achieve the reduction  ``from even $\ell$ to odd $\ell$'  by constructing a new cofibre sequence, which we will now outline.
There is a natural notion of suspension $\Sigma^{\otimes}$ in the pointed \mbox{model category  $\mathbf{CMon}^{aug}$.}

For every pointed space, we produce a natural ``EHP-like'' sequence of maps
$$\Sigma^{\otimes} (S^0\vee \ovA{ \Sigma X^{\wedge 2}}) \xrightarrow{\HH} \Sigma^{\otimes}(S^0 \vee \ovA{X}) \xrightarrow{\EE} S^0 \vee \ovA{\Sigma X}.$$
The first map in this sequence is subtle to define and requires us to work with point-set models for strictly commutative monoid spaces.
We then prove:
\begin{theoremn}[\ref{EHPMonoid}]
If  $S^{n}$ is a sphere of even dimension $n\geq 2$, then the EHP sequence 
$$ \Sigma^{\otimes} (S^0 \vee \ovA{S^{2n+1}}) \xrightarrow{\HH}
\Sigma^{\otimes} (S^0 \vee \ovA{S^{n}})\xrightarrow{\EE}  S^0 \vee \ovA{S^{n+1}}$$
is in fact a homotopy cofibre sequence of strictly commutative monoid spaces.
\end{theoremn}
Just like the classical EHP sequence gives rise to  the  Takayasu cofibration sequence via Goodwillie calculus (cf.\ Section 4.2 of \cite{arone1999goodwillie}\cite[Remark 2.1.7]{behrens2012goodwillie}), 
our new EHP sequence for commutative monoid spaces induces a strict analogue of the  Takayasu cofibration sequence by applying $\AQ$:
\begin{theoremn}[\ref{EHPLIE}]
For each $d \in \NN$ and each $\ell \in \NN$ even, there is a cofibre sequence of spaces
$$ \Sigma^2 |\Pi_{\frac{d}{2}}|^\diamond \mywedge{\Sigma_{\frac{d}{2}}} (S^{2\ell+1})^{\wedge \frac{d}{2}} 
\rightarrow  \Sigma^2 |\Pi_d|^\diamond \mywedge{\Sigma_d} (S^{\ell})^{\wedge d} \rightarrow  \Sigma |\Pi_d|^\diamond \mywedge{\Sigma_d} (S^{\ell+1})^{\wedge d}.$$
Here we use the convention that the space on the left is contractible if $d$ is odd.
\end{theoremn}
Combining this  with Corollary \ref{proposition: orbits again}, we can assemble  the space $\Sigma^2 |\Pi_n|^\diamond/_{\Sigma_{n_1}\times \dots \times \Sigma_{n_k}}$ from atomic  blocks  $\Sigma |\Pi_d|^\diamond \mywedge{\Sigma_d} (S^{\ell})^{\wedge d}$ with $\ell$ \textit{odd}. This will be helpful for our homological considerations. \vspace{3pt}

But first, we give a conceptual interpretation of the homology of strict Young quotients of $|\Pi_n|$, thus explaining the aforementioned similarity between the $\FF_2$-homology of  $ 
  |\Pi_d|^\diamond \mywedge{\Sigma_d} (S^{\ell})^{\wedge d}
$  
and  the algebraic Andr\'{e}-Quillen homology of the trivial square-zero extensions of $\FF_2$  \mbox{computed by Goerss.}

Given any ring $R$, we can apply the reduced chains functor $\tilde{C}_\bullet(-,R)$ to a commutative monoid space $X$ and obtain a \textit{simplicial commutative $R$-algebra}. Heuristically, we extend scalars from \mbox{$\FF_1$ to $R$.} This operation  intertwines Andr\'{e}-Quillen chains for strictly commutative monoid spaces with the classical algebraic Andr\'{e}-Quillen chains  for simplicial commutative $R$-algebras \mbox{(cf.\ Lemma \ref{AQSquare}).}
Moreover, the functor $\tilde{C}_\bullet(-,R)$ sends the trivial square-zero extension $S^0\vee \ovA{X}$ to the trivial square zero extension $R\oplus \tilde{C}_\bullet(X,R)$. 
When combined with Lemma \ref{strictlie}, these observations show  that the  singular homology of strict Young quotients of the doubly suspended partition complex $\Sigma |\Pi_n|^\diamond$ computes algebraic Andr\'{e}-Quillen homology: 
\begin{theoremn}[\ref{surprise}]
If $X$ is a well-pointed space and $R$ is a ring, then 
$$  \widetilde{\HH}_* \bigg(\bigvee_{d\geq 1} \Sigma  |\Pi_d|^\diamond \mywedge{\Sigma_d} X^{\wedge d} , R\bigg)\cong \AQ^R_\ast \bigg(R\oplus \tilde{C}_\bullet (X,R)\bigg).$$
\end{theoremn}
The algebraic  Andr\'{e}-Quillen cohomology  groups of trivial square-zero extensions of $R$  are of particular interest since they parametrise the operations which act naturally on the Andr\'{e}-Quillen cohomology of any simplicial commutative $R$-algebra.\vspace{5pt}

Theorem \ref{surprise} allows us to  observe that our splitting in Corollary \ref{proposition: orbits again} and our cofibration sequence in \mbox{Theorem \ref{EHPLIE}}  imply corresponding results for the algebraic Andr\'{e}-Quillen homology over any ring. Over $\FF_2$, these were first proven by Goerss \cite{goerss1990andre}. We find it remarkable that our combinatorial methods have these consequences  in derived algebraic geometry. 
\mbox{Combining our results, we  prove:}\vspace{-3pt}
\begin{theoremn}[\ref{finalthecohomology}]
Fix integers $\ell_1,\dots , \ell_k\geq 0$ and consider the reduced homology group $$\widetilde{\HH}_*\bigg(\bigvee_d \Sigma |\Pi_d|^\diamond \mywedge{\Sigma_d} (S^{\ell_1}\vee \dots \vee S^{\ell_k})^{\wedge d},R\bigg).$$
This group is given by the algebraic Andr\'{e}-Quillen homology ${\AQ}^R_\ast(R\oplus (\Sigma^{\ell_1} R \oplus \ldots \oplus \Sigma^{\ell_k} R))$ of the trivial square zero extension of $R$ by generators $x_{ 1}, \dots, x_{ k}$ in simplicial  degrees $\ell_1,\dots, \ell_k$.\vspace{3pt}

For $R=\QQ$, the above homology group has a basis indexed by pairs $(e,w)$. Here $w\in B_k(n_1,\ldots,n_k)$ is a Lyndon word. We have $e=0$ if \ $|w|:= \sum_i (1+\ell_i)n_i-1$ is odd and  $e\in \{0,1\}$ if $\ |w|$ is even. 
 The homological degree of $(e,w)$ is $(1+e)|w| + e$ and it lives in multi-weight $(n_1  (1+e) ,\ldots,n_k  (1+e))$.\vspace{3pt}

For $R=\FF_p$, the above homology group has a basis indexed by sequences $(i_1, \ldots, i_a, e,w)$, where \mbox{$w\in B_k(n_1,\ldots,n_k)$} is a Lyndon word and   $e$ lies in  $ \{0,\epsilon\}$. Here
 $\epsilon = 1$ if $p$ is odd and $|w|$ is even or if $|w|=0$. Otherwise, we set $\epsilon = 0$. The sequence $i_1, \ldots, i_a$ consists of positive integers satisfying:
\begin{enumerate}[leftmargin=26pt]
\item Each $i_j$ is congruent to $0$ or $1$ modulo $2(p-1)$. \label{congruence}
\item For all $1\le j<a$, we have  $1<i_j< p i_{j+1}$.
\item We have  $1<i_a\le (p-1)(1+e)|w|+\epsilon$.
\end{enumerate}
The homological degree of $(i_1, \ldots, i_a,e,w)$ is $i_1+\cdots+i_a+(1+e)|w| + e+a$ and it lives in multi-weight $(n_1 p^a (1+e) ,\ldots,n_k  p^a (1+e))$.  Note that $a=0$ is allowed.
\end{theoremn}
A sequence $(i_1,\ldots,i_a,e,w)$ satisfying the conditions in the theorem above is called \textit{allowable} with respect to $R=\QQ$ or $R=\FF_p$, respectively.
We can read off the answer to the question raised above: \vspace{-6pt}\vspace{-6pt}\vspace{-6pt}
\begin{corollaryn}[\ref{strqu}]
Let $n=n_1+\ldots+n_k$.

The vector space $\widetilde{\HH}_*(|\Pi_n|/_{\Sigma_{n_1} \times \ldots \times \Sigma_{n_k}},\QQ)$ has a basis consisting of sequences $(e,w\in B(m_1,\ldots,m_k))$ which are  allowable with respect to $\QQ$ and satisfy $ m_i (1+e)= n_i$ for all $i$.  

The vector space $\widetilde{\HH}_*(|\Pi_n|/_{\Sigma_{n_1} \times \ldots \times \Sigma_{n_k}},\FF_p)$ has a basis consisting of all  $\FF_p$-allowable sequences    $(i_1,\ldots,i_a,e,w\in B(m_1,\ldots,m_k))$  satisfying 
$m_i p^a (1+e)  = n_i$ for all $i$.

The sequence $(i_1,\ldots,i_a,e,w)$ sits in homological degree  $i_1+\ldots +i_a +(1+e) |w| + e + a -2$.
\end{corollaryn}
\vspace{-2pt}
We have seen in Corollary \ref{proposition: orbits again} that strict Young quotients of   partition complexes split as wedge sums of   spaces of the form  $\Sigma^{-1}(S^{\ell})^{\wedge d }\mywedge{\Sigma_d} |\Pi_d|^\diamond$, which sit in the cofibre sequences in \mbox{Theorem \ref{EHPLIE}.}

We study the homotopy type of these spaces and prove that $\Sigma^{-1}(S^{\ell})^{\wedge d }\mywedge{\Sigma_d} |\Pi_d|^\diamond$ is:\vspace{-3pt}
\begin{enumerate}[leftmargin=26pt] 
\item rationally contractible unless $d=1$ or $d=2$ and $\ell$ even (this is a special case of \mbox{Theorem \ref{finalthecohomology}}).
\item equivalent to $\Sigma^\ell\reals P^{\ell-1}$ if $d=2$ (this is well-known).
\item equivalent to the $p$-localisation of the strict quotient $\Sigma^{-1} (S^{\ell})^{\wedge p}/_{\Sigma_p}$ if $d=p$ is an odd prime, and $\ell$ is odd (cf.\ Proposition \ref{proposition: prime}). 
\item equivalent to the $p$-localisation of the homotopy cofibre of the map $S^{\ell p-1}\rightarrow  S^{\ell p-1}/_{\Sigma_p}$ if $d=p$ is an odd prime, and $\ell$ is even (cf.\ Proposition \ref{proposition: prime}).
\item contractible if $\ell=1$ or if $\ell=2$ and $d$ is an odd prime (cf.\ Lemma \ref{lemma: contractible} and Corollary \ref{corollary: p-local}).
\end{enumerate} Hence, we  describe the homotopy type of $\Sigma^{-1}(S^{\ell})^{\wedge d}\mywedge{\Sigma_d} |\Pi_d|^\diamond$ when $d$ is  prime or $\ell=1$.
Together with Corollary \ref{proposition: orbits again}, this allows us to  describe  $ |\Pi_{2p^2}|/_{\Sigma_{p^2}\times\Sigma_{p^2}}$ for any prime  $p$.
 Moreover, we\vspace{-3pt}
\begin{enumerate}[leftmargin=26pt] 
\item describe the torsion-free part of $\widetilde{\HH}_*(|\Pi_n|/_{\Sigma_{n_1}\times\cdots\times\Sigma_{n_k}},\ZZ)$ in all cases (but this is not new). 
\item compute the groups $\widetilde{\HH}_*(|\Pi_n|/_{\Sigma_{n_1}\times\cdots\times\Sigma_{n_k}},\QQ)$ and $\widetilde{\HH}_*(|\Pi_n|/_{\Sigma_{n_1}\times\cdots\times\Sigma_{n_k}},\FF_p)$ \mbox{for all primes $p$.}
\item prove that $\widetilde{\HH}_*(|\Pi_n|/_{\Sigma_{n_1}\times\cdots\times\Sigma_{n_k}})$  has $p$-primary torsion only for primes $p$ that satisfy   the inequality \mbox{$p\le \gcd(n_1, \ldots, n_k)$ (cf.\ Lemma~\ref{lemma: torsion}).}
\item describe the homotopy type of $|\Pi_n|/_{\Sigma_{n_1}\times\cdots\times\Sigma_{n_k}}$ when $\gcd(n_1, \ldots, n_k)$ equals $1$ or a prime number (cf.\ Corollary~\ref{corollary: gcd one} and~Proposition~\ref{proposition: prime}). 
\item classify Young subgroups for which the quotient is a wedge of spheres (Corollary~\ref{corollary: wedge of spheres}). 
\end{enumerate}
  
\subsection*{Historical Review}
This paper subsumes the preprint \cite{arone2015branching}, a part of the second-named author's thesis \cite{brantnerthesis}, and additional joint work.
We give a brief historical review. The short discrete Morse-theoretic proof of Theorem \ref{fixedpoints} concerning the fixed points of partition complexes was found by the second-named author in the spring of 2014 and presented at a short presentation in Bonn in May 2015.   In late August of the same year, the first-named  author made his preprint publicly available. Here, the first-named author follows a very different path: He first establishes Theorem \ref{fixedpoints} for certain wreath products. Using this, he  proceeds to prove Theorem  \ref{main} by reducing it to the algebraic Hilton-Milnor theorem on the level of cohomology. From this, he then deduces Theorem \ref{fixedpoints} for all subgroups. Strictly speaking, only the non-simple versions of these two statements were covered. The preprint also contains several other statements, in particular a version of Lemma \ref{lemma: contractible} and Corollary \ref{corollary: wedge of spheres}.  The preprint motivated the second-named author to strengthen his discrete Morse theoretic methods, develop the theory of orthogonality fans, and thereby give the direct combinatorial proof of the strengthened version of Theorem \ref{main} presented in this paper. In this approach, fixed points and restrictions are computed independently by applying a general combinatorial technique.  
\subsection*{Acknowledgements}
The first-named author's initial impetus for this work was provided by Ralf Donau's paper~\cite{donau2012nice} and his desire to understand its connection with~\cite{arone1998homology}, and he is grateful to Donau for several useful email exchanges, and in particular for sharing  the results of some of his computer-aided calculations of the homology groups of quotient spaces of $|\Pi_n|$.

The second-named author's interest in  the partition complex was initially sparked by the study of the Morava $E$-theory of spectral Lie algebras in his thesis, and he would like to thank his advisor Jacob Lurie for his  guidance and support. The second-named author was  supported by an ERP scholarship   and by the Max Planck Institute \mbox{for Mathematics in Bonn.}

Both authors  thank the Hausdorff Research Institute for Mathematics in Bonn for its hospitality, and the anonymous referee for 
many helpful comments that led to an improvement of this paper.

\newpage 
\section{Topological Preliminaries} \label{topp}
We  will now set up the basic  framework used in  the subsequent sections of this article.\vspace{-7pt}
\subsection{Combinatorial Models for Spaces} We briefly review the links between several different  models for spaces  commonly used in homotopy theory and combinatorial topology.\\
Write $\mathbf{Top}$ for the category of  compactly generated weak Hausdorff spaces. We will simply call its objects ``spaces''. All limits and colimits will be implicitly computed in this category rather than the category of mere topological spaces  
(cf.\  \cite{mccord1969classifying}, where compactly generated weak Hausdorff spaces are called  ``compact spaces'').

In homotopy theory, we often use the category $\sSet$ of \textit{simplicial sets}, i.e. contravariant functors from the  category  of nonempty linearly ordered finite  sets $\Delta$  to the category of sets.
The Yoneda embedding $ i: \Delta  \hookrightarrow  \sSet $ sends the ordered set $[n]:=\{0<\ldots<n\}$ to the simplicial  $n$-simplex.

In combinatorial topology, the following notion is commonly used:
\begin{definition}\label{scomplex}
A \textit{simplicial complex} is a pair $(V,F)$ consisting of a set $V$ of {vertices} and a set \mbox{$F\subset \euscr{P}(V)$} of  {finite} nonempty subsets of $V$, called {faces},  such that $F$ is closed under passing to nonempty subsets and contains all singletons. A \textit{morphism $(V,F) \rightarrow (V',F')$ of simplicial complexes}  is a map   $V\rightarrow V'$ sending subsets in $F$ to subsets in $F'$. \mbox{Let $\sCpl$ be the resulting category.}
\end{definition}

Let $\Fin_+$ be the category of nonempty finite sets. We can define a functor $  \Fin_+ \rightarrow  \sCpl$ by sending a  set $B$ to the  simplicial complex $(B,\euscr{P}(B))$  modelling a simplex with $B$ vertices. 

Yet another common model is given by  the category $\CW$ of $\CW$-complexes.

In order to link these models, we use the following gadget (cf.\ \cite{grandis2000higher}, \cite{lawvere1988toposes}, \cite{rosicky2003left}): 
\begin{definition}
The category $\symsset$ of \textit{symmetric simplicial sets} is given by the category of contravariant functors from $\Fin_+$ to $\Set$. 
\end{definition}

There is a natural diagram\hspace{15pt} \ \ \ \ {} \ \  \hspace{15pt}\vspace{-5pt}
\begin{diagram}
& \Delta & \rTo^U & \Fin_+ & & \\
& \dInto & & \dInto & \rdInto(3,2)^F & \\
& \sSet & \rTo^L & \symsset & \rTo^{|-|} & &  \CW& \rTo &   \mathbf{Top} .
\end{diagram}

The vertical arrows are given by Yoneda embeddings, the functor $U$  forgets the order, the functor $L$ is the colimit-preserving extension making the diagram commute, the functor $F$ sends a finite set $B$ to the simplex on $B$ vertices, and the functor $|-|$  extends $F$ \mbox{in a colimit-preserving way.}

Every simplicial complex $X$ gives  a symmetric simplicial set $B \mapsto \Map_{\sCpl}((B,\euscr{P}(B)),X)$, and this assignment is in fact fully faithful.\\
Writing $\Po$ for the category of posets, there is a \textit{nerve functor} $N_\bullet: \Po\rightarrow \sSet$ (defined by considering posets as   categories) and an \textit{order complex functor} $N: \Po\rightarrow \sCpl$ (defined by sending a poset $P$ to the simplicial complex whose vertices are the elements of  $P$ and whose faces are  all subsets which are chains in $P$). These constructions are  compatible -- the following \mbox{diagram commutes:} \vspace{-5pt} 
\begin{equation}\label{diagramofmodels}\begin{diagram}
\Po & \rTo^{N} & \sCpl & & \\
\dTo^{N_\bullet} & & \dTo & & \\
\sSet & \rTo & \symsset & \rTo^{|-|}& \CW & \rTo &  \mathbf{Top} .
\end{diagram}\end{equation}

We will abuse notation and denote all arrows landing in $\CW$ or in $\mathbf{Top}$ by \vspace{5pt}$|-|$.\\For the rest of this section, we fix a finite group $G$. We invite the reader to recall the notion of a  \mbox{$G$-$\CW$-complex} from \cite{luck1989transformation} -- note that this is \textit{not} the same as a $G$-object in $\CW$-complexes. We write $\Po^G$, $\sCpl^G$, $\sSet^G$, and $\symsset^G$ for the categories of objects with $G$-action in the undecorated versions of these respective categories.
One can then obtain a $G$-equivariant version of diagram (\ref{diagramofmodels}) on the preceding page. A  similar  diagram exists in the pointed setting. \vspace{-5pt}
\subsection{Basic Examples}
Many posets have an initial and/or final element which we will denote by $\hat0$ and/or $\hat1$. Given a poset $\Pcal$, we define  $\ovA{\Pcal}:=\Pcal-\{\hat 0 , \hat 1\}$. 
\begin{example}\label{example: simplex}
Let $\Ccal_m$ be the linearly ordered set of integers between $1$ and $m$. Then $|\Ccal_m|=\Delta^{m-1}$ is the standard $(m-1)$-dimensional simplex. Its points can be parametrised  by tuples of numbers $t_0,\ldots,t_{m-1}\geq 0$ with $t_0+\ldots+t_{m-1}=1$. Alternatively, we can specify points in $\Delta^{m-1}$  by tuples  $s_0,\ldots,s_{m-1}$ satisfying $0 \leq s_1\leq \ldots \leq s_{m-1}\leq 1$.
 The simplicial boundary of $\Ccal_{m}$ is the simplicial complex consisting of all the increasing chains in $\Ccal_m$, except for the maximal chain \mbox{$[1<\cdots<m]$}. Its geometric realisation is the topological boundary $\partial \Delta^{m-1}$ \mbox{of $\Delta^{m-1}$}, i.e. the sphere $S^{m-2}$.
\end{example} 
\vspace{5pt}

\begin{minipage}{0.75\textwidth}
\begin{example}\label{example: sphere} Let $\Bcal_m$ be the poset of subsets of $\m = \{1,\dots, m\}$.
 It is easy to see that there is an isomorphism of posets $\Bcal_m \simeq (\Bcal_1)^m$. Since $|\Bcal_1|\simeq [0,1]$, there is a homeomorphism $|\Bcal_m| \simeq [0,1]^m$. Next, we consider  \mbox{$\Bcal_m-\{\hat0\}$} -- the poset of non-empty subsets of $\m$, and $\ovA{\Bcal}_m = \Bcal_m-\{\hat0, \hat1\}$ -- the poset of proper, non-empty subsets of $\m$. We draw $\ovA{\Bcal_4}$:
\vspace{5pt}
\end{example}
\end{minipage} 
\begin{minipage}{0.3 \textwidth} \hspace{8pt} \ \vspace{ 5pt}
\includegraphics[width=0.65 \textwidth]{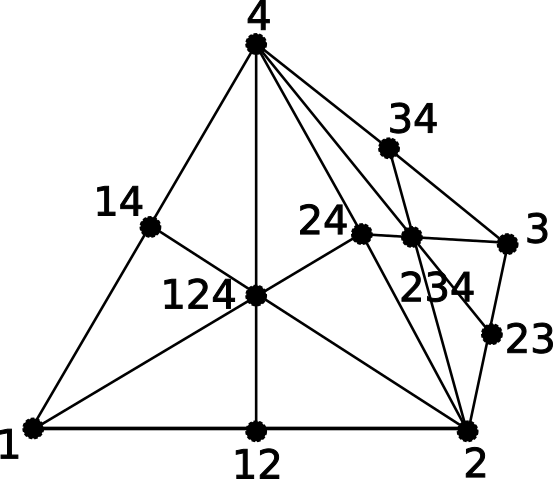}
\end{minipage}   
It is well-known, and not difficult to check, that the order complexes of $\Bcal-\{\hat 0\}$ and $\ovA{\Bcal}_m$ are isomorphic to the barycentric subdivisions of $\Ccal_{m}$ and of its simplicial boundary, respectively (see Example~\ref{example: simplex}). There is a $\Sigma_m$-equivariant homeomorphism $|\ovA{\Bcal}_m| \simeq S^{m-2}$ to the doubly desuspended representation sphere of the standard representation of $\Sigma_m$. It is well-known and not difficult to prove that this action is homeomorphic to the action of $\Sigma_m$ on the unit sphere in $\reals^{m-1}$, where $\reals^{m-1}$ is endowed with the reduced standard action of $\Sigma_m$. \vspace{-5pt}

\subsection{Equivariant Homotopy Theory}
Let $G$ be a finite group. All of our spaces will be homeomorphic, as $G$-spaces, to  geometric realisations of  $G$-simplicial sets.  
A (finite) pointed $G$-space $X$ can be written as a (finite) homotopy colimit of spaces  $G/H_+$, where $H$ is an isotropy group of $X$. 

A $G$-map $f\colon X \rightarrow Y$ is a $G$-equivariant weak equivalence if it induces weak equivalences of fixed point spaces $f^H\colon X^H \rightarrow Y^H$ for every subgroup $H$ of $G$. It is well-known that for nice spaces, an equivariant weak equivalence is an equivariant homotopy equivalence.  
It is in fact not necessary to check the fixed points condition for every subgroup of $G$ --  isotropy groups of $X$ and $Y$ suffice.

Suppose that $K$ and  $H$ are subgroups of $G$ and that $X$ is a pointed space with an \mbox{action of $H$.} One may form the $G$-space $\Ind^G_H X := G_+\wedge_H X$ and consider the space of $K$-fixed points $(\Ind^G_H X)^K$ with its action by the Weyl group $W_{G}(K)=N_G(K)/K$. Write $N_G(K; H)=\{g\in G\mid g^{-1}Kg\subset H\}$ and note that $H$ acts on this set from the right.
 The following is well-known: \vspace{-1pt}
\begin{lemma}\label{lemma: double cosets} 
There is a $W_G(K)$-equivariant  homeomorphism $(\Ind^G_H X)^K \cong \bigvee_{[g]\in N_G(K; H)/H} X^{g^{-1}Kg}$, where the Weyl group $W_G(K)$ acts through the natural left action of $N_G(K)$ on the set $N_G(K; H)$.
Hence if $K$ is not subconjugate to $H$, then $(\Ind^G_H X)^K \cong *$.\vspace{-1pt}
\end{lemma}

We will now provide some tools which make it possible to compute the effect of Theorem \ref{main} on fixed points. This can be used to give an alternative proof of some of our results in Section \ref{fixedpoints}, where we have chosen a more direct approach. Let $G\subset \Sigma_d\subset \Sigma_n$, where $d\ |\ n$ and $\Sigma_d$ is embedded diagonally in $\Sigma_n$ (observe the change in the role of $G$).  
\begin{lemma}\label{lemma: closed system} Let $\sigma\in N_{\Sigma_n}(G; \Sigma_d)$ be an element of $\Sigma_n$ that conjugates $G$ to a subgroup of $\Sigma_d$. Then there is an element $\eta\in \Sigma_d$ such that  conjugation by $\eta$ on $G$ is the \mbox{same as conjugation by $\sigma$.} \vspace{-3pt}\vspace{-3pt}\vspace{-3pt}\vspace{-3pt}
\end{lemma} \vspace{-3pt}\vspace{-3pt}\vspace{-1pt}
\begin{proof}
An inclusion of $G$ into $\Sigma_n$ is the same thing as an effective action of $G$ on $\n$. If the inclusion factors through $\Sigma_d$, then $\n$ is isomorphic, as a $G$-set, to a disjoint union of $\frac{n}{d}$ copies of the set $\mathbf d$, which is equipped with an effective action of $G$. Specifying an element $\sigma$ of $\Sigma_n$ that conjugates $G$ into a subgroup of $\Sigma_d$ is the same as specifying a second action of $G$ on ${\mathbf d}$ (let ${\mathbf d}'$ denote $\mathbf d$ with the second action of $G$), a group isomorphism $\sigma_G\colon G\rightarrow G$, and a bijection 
$
\sigma\colon \coprod_{\frac{n}{d}}{\mathbf d}\rightarrow \coprod_{\frac{n}{d}}{\mathbf d}'
$
that is equivariant with respect to the isomorphism $\sigma_G$. Such an isomorphism exists if and only if $\coprod_{\frac{n}{d}}{\mathbf d}$ and $\coprod_{\frac{n}{d}}{\mathbf d}'$ are isomorphic permutation representations of $G$. But this is possible if and only if $\mathbf d$ and ${\mathbf d}'$ are isomorphic permutation representations of $G$, which means that the isomorphism $\sigma_G$ can be realised as conjugation by an element of $\Sigma_d$. \vspace{-3pt}
\end{proof}
Let $C_W(G)$ denote the centraliser of $G$ in $W$.
\begin{corollary}\label{corollary: centraliser}
Let $G\subset \Sigma_d\subset W\subset \Sigma_n$  with $\Sigma_d$  embedded diagonally in $\Sigma_n$ and $W$  any intermediate subgroup. \mbox{Then inclusion induces an isomorphism $C_{W}(G)/C_{\Sigma_d}(G) \stackrel{\cong}{\rightarrow} N_{W}(G; \Sigma_d)/\Sigma_d$ of sets.}\vspace{-3pt}\vspace{-3pt}\vspace{-1pt}
\end{corollary}  
\begin{proof}
We have $N_{W}(G; \Sigma_d)=\{w\in W\mid w^{-1} G w\subset \Sigma_d\}$. By Lemma~\ref{lemma: closed system} for every such $w$ one can find an $\eta\in \Sigma_d$ such that conjugation by $\eta$ on $G$ coincides with conjugation by $w$, hence  $w\eta^{-1}\in C_{W}(G)$. Thus every element of $N_{W}(G; \Sigma_d)/\Sigma_d$ has a representative in $C_{W}(G)$. \mbox{Two elements} of $C_{W}(G)$ differ by an element of $\Sigma_d$ if and only if they differ by an element of $C_{\Sigma_d}(G)$. \vspace{-3pt} 
\end{proof} 
\begin{corollary}\label{corollary: simplified fixed points}
Suppose $G\subset \Sigma_d\subset W\subset\Sigma_n$ with $\Sigma_d$  embedded diagonally in $\Sigma_n$, and $W$ any intermediate subgroup. Let $X$ be a space with an action of $\Sigma_d$. There is a homeomorphism
$(\Ind^W_{\Sigma_d} X)^G  \cong    \Ind^{C_W(G)}_{C_{\Sigma_d}(G)}X^G$.
In particular, if $X^G$ is contractible, then so is $(W_+ \wedge_{\Sigma_d} X)^G$. \vspace{-3pt} \vspace{-1pt}
\end{corollary}
\begin{proof}
By Lemma~\ref{lemma: double cosets} there is a homeomorphism
$(W_+ \wedge_{\Sigma_d} X)^G\cong \bigvee_{N_W(G; \Sigma_d)/\Sigma_d} X^{w^{-1}Gw}.$
By Corollary~\ref{corollary: centraliser}, $N_W(G; \Sigma_d)/\Sigma_d\cong C_{W}(G)/C_{\Sigma_d}(G)$. In particular, $w$ can always be chosen in the centraliser of $G$, so $X^{w^{-1}Gw}= X^G$. It follows that $
(W_+ \wedge_{\Sigma_d} X)^G\cong \bigvee_{C_{W}(G)/C_{\Sigma_d}(G)} X^{G}.$
The right hand side is isomorphic to $C_W(G)_+\wedge_{C_{\Sigma_d}(G)} X^G$ but the latter presentation is better in the sense that it gives the correct action of $C_W(G)$.    \vspace{-3pt} 
\end{proof}
 
\subsection{Simple Equivariant Homotopy Theory}
We briefly review the basics of simple equivariant homotopy theory. 
We begin by looking at simplicial complexes and recall a notion from \cite{kozlov2013topology}:

\begin{definition}
An inclusion $(V,F)\subset (V',F')$ of $G$-simplicial complexes is called an \textit{elementary} $G$-\textit{collapse}  if there is a 
$\sigma \in F'$ such that \vspace{-3pt}
\begin{enumerate}[leftmargin=26pt]
\item There is exactly one face in $F'$ which  properly contains $\sigma$.
\item For every $g\in G$ with $g\sigma \neq \sigma$, there does not exist a simplex which  contains both \mbox{$g\sigma$ and $\sigma$.}
\item $F$ is obtained from $F'$ by deleting all faces  which contain $g\sigma$ for some $g\in G$.
\end{enumerate}
\vspace{-3pt}

\end{definition}

There is a corresponding notion for $G$-$\CW$-complexes (cf.\ \cite{luck1989transformation}).  
\mbox{Let $D^k$ be the $k$-dimensional disc.}

\begin{definition}\label{elemexp}\index{Elementary expansion}
An \textit{elementary expansion} consists of a pushout of $G$-$\CW$-complexes  
\begin{diagram}[small]
G/H \times D^{n-1} & \rTo_f & X \\
\dInto & \SEpbk & \dInto^\iota \\
G/H \times D^{n} & \rTo & Y 
\end{diagram}
where $D^{n-1} \rightarrow D^n$  includes the lower hemisphere $D^{n-1}  \hookrightarrow S^n$ of the bounding sphere  $S^n$ of $D^n$, and
 $f(G/H \times \partial D^{n-1}) \subset X_{n-2}$, $ f(G/H \times D^{n-1}) \subset X_{n-1}$, where $X_k$  {denotes the $k$-skeleton of $X$.}
\end{definition}
Given a subspace $X$ of $Y$ for which $X\hookrightarrow Y$ is an elementary expansion, we call any strong $G$-equivariant deformation retract $Y\rightarrow X$ an \textit{elementary collapse}.\index{Elementary collapse}
An elementary collapse between $G$-simplicial complexes induces an elementary collapse between their geometric realisations. 
\begin{definition}
A $G$-simplicial complex (or $G$-$\CW$-complex) is said to be \textit{collapsible} if it can be mapped to the point by a finite number of elementary collapses.
\end{definition}
\begin{definition}
A $G$-map $f:X\rightarrow Y$ between $G$-$\CW$-complexes is a \textit{simple  homotopy equivalence} if it is $G$-homotopic to a finite composition of expansions and collapses.
\end{definition}   
\begin{proposition}\label{con}
Let $A$ be a contractible sub-$G$-$\CW$-complex of a $G$-$\CW$-complex $X$. Then $X/A$ carries a natural $G$-$\CW$-structure and  the quotient map $ X\rightarrow X/A$ is a simple   equivalence. \vspace{-5pt} \vspace{-2pt}
\end{proposition}
\begin{proof}
Recall from  \cite[Section 4]{luck1989transformation}
that the   \textit{Whitehead group} is an abelian group $\Wh^G(Y)$ attached to any $G$-$\CW$-complex $Y$, and the \textit{Whitehead torsion} is an element \mbox{$\tau^G(f)\in \Wh^G(Y)$} for any equivariant map $f:X\rightarrow Y$.
An equivariant equivalence $f$ is   simple if and only if $\tau^G(f)=0$, and the result  follows from additivity of the Whitehead torsion, cf.\   \cite[Theorem 4.8]{luck1989transformation}.  \vspace{-5pt}
\end{proof}
\subsection{The Join}
We will now review the join operation in several contexts. We begin with spaces: 
\begin{definition}
The \textit{join} of two  spaces $X$ and $Y$ is
given by $ X*Y = X \times I \times Y /_\simeq$,
where $\simeq$ identifies $(x,y,0) \simeq(x,y',0)$ and $(x,y,1) \simeq(x',y,1)$
for all $x,x'\in X$ and $ y,y'\in Y$ .

The join is a model for the double mapping cylinder (i.e. homotopy pushout) of the diagram
$
X\leftarrow X\times Y \rightarrow Y.$
This operation is associative up to natural isomorphism.
\end{definition}

If $X,Y$ are $G$-$\CW$-complexes, then  $X*Y$ (computed in $\mathbf{Top}$) inherits   a $G$-$\CW$-structure.

\begin{notation} Write $X^\diamond = S^0 \ast X$ for the unreduced suspension of $X$ with basepoint $0$ (the `south pole'). It  models the pointed homotopy cofibre of   $X_+\rightarrow S^0$. If $X$ is a $G$-$\CW$-complex, there is a simple equivariant equivalence $X^\diamond\simeq (X\times \Delta^1)/_{\stackrel{(x,0)\sim (x',0)}{_{(x,1)\sim (x',1)}}}$ since $\{1\}\ast X$ is contractible.
\end{notation}

It is well-known that if $X$ and $Y$  are pointed, then there is a simple  equivalence $
X*Y\stackrel{\simeq}{\rightarrow} S^1\wedge X \wedge Y.$ This can be generalised slightly to a situation where only one of the spaces is pointed:
\begin{lemma}\label{lemma: join} Let $X$ be a $G$-$\CW$-complex and $(Y,y_0)$ be a pointed $G$-$\CW$-complex. Then there are equivariant simple equivalences $X*Y\stackrel{\simeq}{\rightarrow} X*Y/_{X*\{y_0\}} \stackrel{\simeq}{\rightarrow} X^\diamond \wedge Y$. 
\end{lemma}  \vspace{-5pt} \vspace{-3pt}
\begin{proof} 
Consider the two spaces
$X*Y/ X*\{y_0\} = X \times I \times Y /\simeq_1$
where $\simeq_1$ is spanned by   \vspace{-3pt} \vspace{-3pt}
$$(x,t,y_0) \simeq_1 (x',t',y_0), \ \ \  (x,0,y) \simeq_1 (x,0,y'), \ \ \  (x,1,y) \simeq_1 (x',1,y)  \ \ \ \forall x,x'\in X, y,y' \in Y\vspace{-3pt} $$   
and
$((X\times \Delta^1)/_{\stackrel{(x,0)\sim (x',0)}{ {(x,1)\sim (x',1)}}} ) \wedge Y  = X \times I \times Y /\simeq_2$
where $\simeq_2$ is spanned by   \vspace{-3pt}  
$$(x,t,y_0) \simeq_2 (x',t',y_0)  \simeq_2 (x,0,y) , \ \ \  (x,1,y) \simeq_2 (x',1,y)  \ \ \ \forall x,x'\in X, y,y' \in Y \vspace{-3pt}  $$ 
These  quotients are equal.
Cones are contractible, so the claim follows from Proposition \ref{con}.   \vspace{-3pt}
\end{proof}

It follows that the expression $\Sigma^{-1}(S^{\frac{n}{d}-1})^{\wedge d} \wedge |\Pi_d|^\diamond$ in the Theorem \ref{main} can be replaced with $S^{n-d-1} * |\Pi_d|$. This form naturally occurs in the proof via discrete Morse theory. \vspace{3pt}
\begin{examples}\label{examples: joins}
The empty set acts as a unit for the join operation, i.e. $X* \emptyset \cong \emptyset * X  \cong  X$.

A special case of Lemma \ref{lemma: join} is the well-known homeomorphism $S^m * S^n  \cong S^{m+n+1}$. We think of the empty set as the $(-1)$-dimensional sphere. Then the formula remains valid for $m, n\ge -1$.

The join of two simplices $\Delta^i * \Delta^j$ is homeomorphic to $\Delta^{i+j+1}$. It is also  easy to see that $\Delta^0 * S^j$ is homeomorphic to $\Delta^{j+1}$. It follows that for all $i\ge 0$, $j\ge -1$, $\Delta^i * S^j$ is homeomorphic to $\Delta^{i+j+1}$.  
\end{examples}
 
There is  a version of the join operation for the simplicial complexes from Definition \ref{scomplex};  the  symmetry of this operation is one of the key advantages  of simplicial complexes.
 \begin{definition}
The simplicial join $(V_1,F_1)*(V_2,F_2)$ of two simplicial complexes $(V_1,F_1)$, $(V_2,F_2)$ is defined as  $(V_1 \coprod V_2, F_1*F_2)$, where $F_1*F_2$ is the set of   $S\subset V_1\coprod V_2$ with $S\cap V_i \in F_i$ for $i=1,2$. 
\end{definition}
Geometric realisation takes simplicial to space-level join. \mbox{There is  a corresponding notion for posets:}

\begin{definition}
The join of two posets $\Pcal$ and $\Qcal$ is defined to be the  poset $\Pcal * \Qcal$ whose underlying set is given by $\Pcal \coprod \Qcal$ and whose partial order is defined by keeping the old order on $\Pcal$ and $\Qcal$, and declaring that every element of $\Pcal$ is smaller than every element of $\Qcal$.
\end{definition}
It is easy to see that the order complex of $\Pcal * \Qcal$ is isomorphic to the simplicial join of the two order complexes. It follows that $|\Pcal * \Qcal|\simeq |\Pcal| * |\Qcal|$. \vspace{-3pt}

\subsection{Stars and Links} \label{starsandlinks}
Let $\Pcal$ be a poset. Let $\sigma=[x_0<\cdots < x_k]$ be a non-empty chain in $\Pcal$. \begin{definition} The \textit{star} $\star(\sigma)$ of $\sigma$ is the poset of elements of $\Pcal$ that are comparable \mbox{with  each $x_i$.}\\
\mbox{The \textit{link}  $\link(\sigma)$ of $\sigma$ is the subposet of $\star(\mathcal{P})$ consisting of elements that are distinct from all $x_i$. }
\end{definition}
 The star decomposes as a join
$
\star(\sigma)\cong \mathcal{P}_{<x_0}* \{x_0\} * \mathcal{P}_{(x_0, x_1)}* \{x_1\}*\cdots* \{x_k\} * \mathcal{P}_{>x_k}$. Hence $|\star(\sigma)|$ is contractible.
There is a similar decomposition for the  link as
$
\link(\sigma)\cong \mathcal{P}_{<x_0}*  \mathcal{P}_{(x_0, x_1)}*\cdots* \mathcal{P}_{>x_k}
. \vspace{-3pt}  $ 

\subsection{Indexed Wedges} We recall the theory of indexed wedges (cf.\ Section 2.2.3.\ of \cite{hill2009non}).
Given a $G$-set $J$, we write $\mathcal{B}_JG$ for the category with objects
$\Ob(\mathcal{B}_JG) = J$ and morphisms $\Mor_{\mathcal{B}_JG}(j,j')  = \{ h \in G \ | \ h \cdot j = j'\}$.
The composition law is evident.\\
Given a functor   $X: \mathcal{B}_JG \rightarrow \mathbf{\sSet}_\ast$ with $j \mapsto X_j$,
the wedge sum
$ \bigvee_{j \in J} X_j$ inherits a natural   $G$-action with $g\cdot( x\in X_j)  := X_{(j\xrightarrow{g\cdot} gj)} (x) \in X_{gj}$. The space $\bigvee_{j \in J} |X_j|$ inherits a  $G$-$\CW$-structure.

Clearly, we can rewrite this indexed wedge sum as an induction:
\begin{proposition} \label{indexedwedge}
There is a simple $G$-equivalence
\mbox{$ \bigvee_{[j] \in J/G}  \Ind_{\Stab(j)}^G |X_j | \xrightarrow{\simeq} \bigvee_{j \in J} |X_j|.$}
\end{proposition}  
\begin{corollary}
If $X\xrightarrow{\upalpha} Y$ is transformation  of functors $ \mathcal{B}_JG \rightarrow \sSet_\ast$ for which $ |X_j| \rightarrow |Y_j|$   is a simple $\Stab(j)$-equivalence   all $j\in J$, 
then 
$ \bigvee_{j\in J }|X_j| \longrightarrow  \bigvee_{j\in J }|Y_j| $
is a simple $G$-equivalence.  \vspace{-3pt} 
\end{corollary} 
\subsection{Removing Simplices}\label{remove}
We start by observing the following basic fact:
\begin{proposition} \label{square}\label{propos}
Assume that we are given a square of simplicial $G$-sets \vspace{-3pt}  \begin{diagram} A & \rTo & C \\ \dInto & & \dInto \\B & \rTo & D \end{diagram}  
such that for each simplicial degree $d$, the map of sets $(B-A)_d \rightarrow (D-C)_d$ is bijective.
Taking vertical quotients induces an isomorphism of simplicial $G$-sets.
\end{proposition}
\vspace{0pt}
\begin{notation} 
If $S$ is a $G$-stable set of nondegenerate simplices in  the $G$-simplicial set $X$,  let $X^{-S}$ be the simplicial subset  of    simplices which do not contain a simplex in $S$. \mbox{If $S=\{\sigma\}$,  set \mbox{$X^{-\sigma } = X^{-\{\sigma\}}$.}} \end{notation} 
We will now remove simplices from the nerve  of a $G$-poset $\mathcal{P}$. Here, we drop $N_\bullet$ from our notation whenever we geometrically realise, e.g.\ write $|\mathcal{P}^{-S}|$ instead of $|N_\bullet(\mathcal{P})^{-S}|$.  
 
\begin{lemma}\label{triple}
Let $S$ be a  $G$-stable family of strictly increasing chains in $\mathcal{P}$ such  that no two chains in $S$
lie  in a larger chain in $\mathcal{P}$.
 The following diagram induces isomorphisms on vertical quotients:\vspace{-0pt}
\begin{diagram}\textstyle
\coprod_{\sigma \in S}  N_\bullet(\St(\sigma))^{-\sigma} &\rTo& N_\bullet(\mathcal{P})^{-S}\\
\dInto & &\dInto  \\ 
\textstyle \coprod_{\sigma\in S}  N_\bullet(\St(\sigma)) & \rTo  &  N_\bullet(\mathcal{P})
\end{diagram} 
We obtain cellular homeomorphisms $ |\mathcal{P}|/_{|\mathcal{P}^{-S}|} \cong   \bigvee_{\sigma \in S} |\St(\sigma)|/_{|\St(\sigma)^{-\sigma}|}$. 
\end{lemma}
\begin{proof}
Any  simplex  in $N_d(\mathcal{P})- N_d(\mathcal{P})^{-S}$  contains exactly one $\sigma \in S$ and thus has unique preimages lying  in $N_d(\star(\sigma))-N_d(\star(\sigma))^{-\sigma}$.  Proposition \ref{propos} proves the claim.
\end{proof}

Let $\sigma= [ y_0 < \dots <y_r]$ be a chain in   $\mathcal{P}$. We can swap the order of the join-factors in the star $\St(\sigma)$ (cf.\ Section \ref{starsandlinks}) and obtain a $ \Stab(y_0)\cap \dots \cap \Stab(y_r) $-equivariant diagram:\vspace{3pt}
\begin{diagram}
 \partial \Delta^{r} \ast| \link(\sigma)| & \rTo &\Delta^{r}\ast |\link(\sigma)|\\
\dTo^{\cong} & & \dTo^{\cong}  \\
|\star(\sigma)^{-\sigma}| & \rTo & |\star(\sigma)|
\end{diagram}\vspace{5pt}
By Lemma \ref{lemma: join}, there are \vspace{-5pt} $  \Stab(y_0)\cap \ldots \cap \Stab(y_r)$-equivariant simple equivalences given by $$|\St(\sigma)|/_{|\St(\sigma)^{- \sigma}|} \simeq S^r\wedge (|\mathcal{P}_{<y_0}| \ast \ldots \ast  |\mathcal{P}_{>y_r}|)^\diamond  \simeq   | {\mathcal{P}}_{<y_0}|^\diamond \wedge  \Sigma | {\mathcal{P}}_{(y_0,y_1)}|^\diamond \wedge  \dots  \wedge  \Sigma | {\mathcal{P}}_{(y_{r-1},y_r)}|^\diamond \wedge   |{\mathcal{P}}_{>y_r}|^\diamond.$$  

\begin{corollary}\label{takingout}
Under the above conditions, there are simple $G$-equivalences 
$$\hspace{-10pt} | {\mathcal{P}}|/_{| {\mathcal{P}}^{-S}|} \cong \bigvee_{\sigma = [y_0< \dots < y_r] \in S}\hspace{-10pt} |\St(\sigma)|/_{|\St(\sigma)^{- \sigma}|} \simeq \bigvee_{\sigma = [y_0< \dots < y_r] \in S} \hspace{-10pt} | {\mathcal{P}}_{<y_0}|^\diamond \wedge  \Sigma | {\mathcal{P}}_{(y_0,y_1)}|^\diamond \wedge  \dots  \wedge  \Sigma | {\mathcal{P}}_{(y_{r-1},y_r)}|^\diamond \wedge   |{\mathcal{P}}_{>y_r}|^\diamond$$
Here $h\in G$ acts on the indexed wedge  by sending $| {\mathcal{P}}_{<y_0}|^\diamond \wedge  \Sigma | {\mathcal{P}}_{(y_0,y_1)}|^\diamond \wedge  \dots  \wedge  \Sigma | {\mathcal{P}}_{(y_{r-1},y_r)}|^\diamond \wedge   |{\mathcal{P}}_{>y_r}|^\diamond$ to $| {\mathcal{P}}_{<h\cdot y_0}|^\diamond \wedge  \Sigma | {\mathcal{P}}_{(h\cdot y_0,h\cdot y_1)}|^\diamond \wedge  \dots  \wedge  \Sigma | {\mathcal{P}}_{(h\cdot y_{r-1},h\cdot y_r)}|^\diamond \wedge   |{\mathcal{P}}_{>h\cdot y_r}|^\diamond$.\vspace{-3pt}
\end{corollary}

\subsection{Suspensions and Products}  \label{clock} 
We let $S^1_\bullet= N_\bullet ([1])/_{N_\bullet ([1])^{-[\hat{0}<\hat{1}]}}$ be the simplicial circle, where $[1]$ denotes the poset $\{\hat{0}<\hat{1}\}$. 
If $X$ is a $G$-simplicial set, we define $\Sigma X = S^1_\bullet \wedge X$. By Lemma \ref{lemma: join}, we obtain   a simple equivariant equivalence  
$\Sigma |X|^\diamond = S^1 \wedge (S^0 \ast |X|) \simeq |\{0\} \ast X \ast \{1\}/_{(\{0\} \ast X \ast \{1\})^{-[\hat{0}<\hat{1}]}}|$. We will make use of  this well-known simplicial model for  $\Sigma |X|^\diamond$ in the rest of this \vspace{3pt} work.

Using our models for suspensions, we observe (cf.\ Theorem $5.1.$ in \cite{walker1988canonical} for a related claim):
\begin{proposition} 
For $i=1,\ldots,r$, let $\mathcal{P}_i$ be a poset with distinct minimal and maximal elements $\hat{0}^i \neq \hat{1}^i$. There is a cellular homeomorphism 
$ \Sigma |\ovA{\mathcal{P}_1\times \ldots \times \mathcal{P}_r}|^\diamond \xrightarrow{\cong} \Sigma |\ovA{\mathcal{P}_1}|^\diamond \wedge \ldots \wedge \Sigma |\ovA{\mathcal{P}_r}|^\diamond $ \mbox{given by}
$$\hspace{-10pt}  \Bigg( \begin{diagram}
\{\hat{0}^i\}_{i=1}^r&\leq&\{x^i_0\}_{i=1}^r&\leq&\ldots&\leq&\{x^i_n\}_{i=1}^r &\leq&\{\hat{1}\}_{i=1}^r  \\
t_{-1} & + &t_0&+&\ldots &+ &t_n &+ &t_{n+1} =1
\end{diagram} \Bigg)  \ \ \mapsto  \ \ \Bigg(\begin{diagram}
\hat{0}^i & \leq & x^i_0 &\leq &\ldots&\leq & x^i_n & \leq& \hat{1}^i  \\
t_{-1} & + & t_0&+&\ldots &+ &t_n &+ &t_{n+1} = 1
\end{diagram}\Bigg)_{i=1}^r \ .$$
\end{proposition}\vspace{-5pt}
\begin{proof}
This follows by observing that there is a natural isomorphism of simplicial sets
$$N_\bullet(\mathcal{P}_1\times \ldots \times \mathcal{P}_r)/_{N_\bullet(\mathcal{P}_1\times \ldots \times \mathcal{P}_r)^{-[\{\hat{0}^i\}< \{\hat{1}^i\}]}} \cong \bigwedge_i N_\bullet(\mathcal{P}_i)/_{N_\bullet(\mathcal{P}_i)^{-[\hat{0}^i<\hat{1}^i]}}.\vspace{-5pt}\vspace{-5pt}\vspace{-5pt}$$ \end{proof}\vspace{-5pt}\vspace{-5pt}
\subsection{Explicit Collapse}\label{clong}
Let $\mathcal{P}$ be a $G$-poset with distinct minimal element $\hat{0}$ and maximal element $\hat{1}$. Given  a chain   $\sigma= [y_0 < \ldots < y_r]$ in $\ovA{\mathcal{P}} = \mathcal{P}-\{\hat{0},\hat{1}\}$, we write $\St(\sigma)$ for the star  of \mbox{$\sigma$ in $\ovA{\mathcal{P}}$.}
We  represent   points $P$  in $\Sigma^2 ( |\St(\sigma)|/_{|\St(\sigma)^{-\sigma}|} ) \simeq    |\hat{0}\ast\St(\sigma)\ast\hat{1}|/_{|(\hat{0}\ast\St(\sigma)\ast\hat{1})^{-[\hat{0}<\sigma<\hat{1}]}|}$ by \vspace{5pt}
$$\hspace{-22 pt}\Bigg(\begin{diagram} \hat{0}&<& x_0^0<\ldots<x^0_{n_0}  &< & y_0 & <&  x_0^1<\ldots<x^1_{n_1} &< &y_1  &< &  \ldots&< & y_r  &< & x^{r+1}_0< \ldots < x^{r+1}_{n_{r+1}}&<&\hat{1} \\
s_{-1}&+&t_0^0+\ldots+t^0_{n_0}  &+ &s_0 &+&  t_0^1+\ldots+t^1_{n_1} &+ &s_1 &+&  \ldots  &+ &s_r &+& t^{r+1}_0+ \ldots + t^{r+1}_{n_{r+1}}& +&s_{r+1}& = 1\  \end{diagram}\Bigg)\vspace{5pt}$$ 
Such an expression represents the basepoint precisely if one of the parameters \vspace{3pt} $s_i$ vanishes. 

We  represent points in  $ \Sigma |\ovA{\mathcal{P}}_{(\hat{0},y_0)}|^\diamond  \wedge \Sigma |\ovA{\mathcal{P}}_{(y_0,y_1)}|^\diamond \wedge  \ldots \wedge \Sigma |\ovA{\mathcal{P}}_{(y_r,\hat{1})}|^\diamond $ by $(r+2)$-tuples \vspace{4pt}
$$\Bigg\{  \Bigg(\begin{diagram} 
y_{i-1} &< &z^i_0&<&\ldots&<&z^i_{n_i}& < &y_i &&\\
\ell_{-1}^i &+ &\ell_0^i& +& \ldots &+ & \ell^i_{n_i}&+ &\ell_{n_i+1}^i &=& 1 \ 
 \end{diagram} \Bigg)\Bigg\}_{i=0}^{r+1} \vspace{4pt}$$
where we use the convention $y_{-1} = \hat{0}$ and $y_{r+1} = \hat{1}$. Such a tuple represents the basepoint   if some $l^i_{-1}$ or $l^i_{n_i+1}$ vanishes.
We now define  $\Sigma^2 ( |\St(\sigma)|/_{|\St(\sigma)^{-\sigma}|} ) \rightarrow \Sigma |\ovA{\mathcal{P}}_{(\hat{0},y_0)}|^\diamond   \wedge  \ldots \wedge \Sigma |\ovA{\mathcal{P}}_{(y_r,\hat{1})}|^\diamond $ by sending points $P$ represented  as above to  \vspace{4pt}
$$\Bigg\{ \Bigg(\begin{diagram} 
y_{i-1} &< &x^i_0&<&\ldots&<&x^i_{n_i}& < &y_i &&\\
\frac{s_{i-1}}{C^i}&+ &\frac{t_0^i}{C^i}& +& \ldots &+ &\frac{ t_{n_i}^i}{C^i}&+ &\frac{s_{i} }{C^i} &=& 1\ 
 \end{diagram} \Bigg)\Bigg\}_{i=0}^{r+1} \vspace{4pt} $$
whenever all $C^i := s_{i-1} + t_0^i + \ldots + t^i_{n_i}+s_i$ are nonzero for all $i$. If one of the $C^i$ vanishes, we map $P$ to the basepoint. It is not hard to check that this map is well-defined, continuous, cellular, $ \Stab(y_0)\cap \dots \cap \Stab(y_r) $-equivariant, and that it has a  continuous inverse. \vspace{3pt}

Given a $G$-stable family $S$ of strictly increasing chains in $\ovA{\mathcal{P}}$ such that no two chains have a common refinement, we obtain an explicit description of the   collapse maps in the 
 simple   equivalence \vspace{3pt}
$$\hspace{-7pt}\Sigma^2 ( |\ovA{\mathcal{P}}| /_{ |  \overline{\mathcal{P}}^{-S}|})\cong \bigvee_{\sigma = [y_0< \dots < y_r] \in S}\hspace{-10pt} \Sigma^2|\St(\sigma)|/_{|\St(\sigma)^{- \sigma}|} \xrightarrow{\simeq} \bigvee_{\sigma = [y_0< \dots < y_r] \in S}\hspace{-10pt} \Sigma |\ovA{\mathcal{P}}_{(\hat{0},y_0)}|^\diamond  \wedge \Sigma |\ovA{\mathcal{P}}_{(y_0,y_1)}|^\diamond \wedge  \ldots \wedge \Sigma |\ovA{\mathcal{P}}_{(y_r,\hat{1})}|^\diamond\vspace{3pt}$$
implied by Corollary \ref{takingout}. This doubly suspended   collapse map is  relevant in our applications to spectral Lie algebras. As in Section \ref{remove}, it  is $G$-equivariant, where $h\in G$ acts on the right by taking the wedge sum of action maps  $\Sigma |{\ovA{\mathcal{P}}}_{(\hat{0},y_0)}|^\diamond  \wedge   \ldots  \wedge \Sigma  |\ovA{\mathcal{P}}_{(y_r,\hat{1})}|^\diamond \rightarrow  \Sigma |{\ovA{\mathcal{P}}}_{(\hat{0},h \cdot y_0)}|^\diamond  \wedge   \ldots \wedge  \Sigma |\ovA{\mathcal{P}}_{(h \cdot y_r,\hat{1})}|^\diamond$.
\vspace{4pt}

It is not difficult to also give an explicit description of the unsuspended collapse map 
$$ |\ovA{\mathcal{P}}| /_{ |  \overline{\mathcal{P}}^{-S}|}   \longrightarrow \bigvee_{\sigma = [y_0< \dots < y_r] \in S} |\ovA{\mathcal{P}}_{(\hat{0},y_0)}|^\diamond  \wedge \Sigma |\ovA{\mathcal{P}}_{(y_0,y_1)}|^\diamond \wedge  \ldots \wedge \Sigma   |\ovA{\mathcal{P}}_{(y_{r-1},y_r)}|^\diamond\wedge  |\ovA{\mathcal{P}}_{(y_r,\hat{1})}|^\diamond.$$ \vspace{3pt} We  represent points $P$ in  the star $ |\St(\sigma)|/_{|\St(\sigma)^{-\sigma}|}  $  of a chain  $\sigma = [y_0<\ldots<y_r]$ by expressions\vspace{4pt}
$$\hspace{-17 pt}\Bigg(\begin{diagram} \ \ \hat{0}&<& x_0^0<\ldots<x^0_{n_0}  &< & y_0 & <&  x_0^1<\ldots<x^1_{n_1} &< &y_1  &< &  \ldots&< & y_r  &< & x^{r+1}_0< \ldots < x^{r+1}_{n_{r+1}}&<&\hat{1} \ \ \\
&&t_0^0+\ldots+t^0_{n_0}  &+ &s_0 &+&  t_0^1+\ldots+t^1_{n_1} &+ &s_1 &+&  \ldots  &+ &s_r &+& t^{r+1}_0+ \ldots + t^{r+1}_{n_{r+1}}&  =  & 1\   \end{diagram}\Bigg)\vspace{4pt}$$ 
The unsuspended collapse map then sends such a  point $P$ to the $(r+2)$-tuple\vspace{4pt}
$$\Bigg\{ \Bigg(\begin{diagram} 
y_{i-1} &< &x^i_0&<&\ldots&<&x^i_{n_i}& < &y_i &&\\
\frac{s_{i-1}}{C^i}&+ &\frac{t_0^i}{C^i}& +& \ldots &+ &\frac{ t_{n_i}^i}{C^i}&+ &\frac{s_{i} }{C^i} &=& 1 \ 
 \end{diagram} \Bigg)\Bigg\}_{i=0}^{r+1} \vspace{4pt} \vspace{4pt}$$
whenever all $C^i := s_{i-1} + t_0^i + \ldots + t^i_{n_i}+s_i$ are nonzero for all $i$, where we use the additional convention that $s_{-1} = s_{r+1} = 0$. If one of the $C^i$ vanishes, we again map $P$ to the basepoint. 
We used well-known models for the two unreduced suspensions  $|\ovA{\mathcal{P}}_{(\hat{0},y_0)}|^\diamond$ and  $|\ovA{\mathcal{P}}_{(y_r,\hat{1})}|^\diamond$.
The unsuspended collapse maps again  interact well with the action by $G$.
 \vspace{4pt}

We will also often consider the collapse maps obtained by postcomposing the above maps to wedge summands with the ``projection'' to an individual  summand.
\subsection{Spectra}
In our study of spectral Lie algebras, we shall use the closed symmetric monoidal model category $(\mathbf{Sp},\wedge,S)$ of \textit{$S$-modules},  endowed with the smash product,  as model for spectra (cf.\   \cite{elmendorf2007rings}).  We urge the reader to not get too distracted by this refined
homotopical notion as many of our later arguments will merely require the  homotopy category $\hSp$ of spectra.

There is an adjunction 
$ \Sigma^\infty : \mathbf{Top}_\ast \leftrightarrows \mathbf{Sp} : \Omega^\infty$
whose left component $\Sigma^\infty$ is monoidal. 
Given $S$-modules $X,Y\in \mathbf{Sp}$, we  let $\Map_\mathbf{Sp}(X,Y)$ be the {mapping spectrum}  between $X$ and $Y$ and write $\DD(X)=\Map_\mathbf{Sp}(X,S)$ for the Spanier-Whitehead dual of $X$. Using the functor $\Sigma^\infty$, we can observe that the category $\mathbf{Sp}$ is both tensored and cotensored over $\mathbf{Top}_\ast$. We will   suppress  $\Sigma^\infty$ from our notation whenever the context implies that we are working in spectra.
We will now establish two results which tell us that several strict constructions are homotopically well-behaved:

\begin{proposition}\label{sw1}
Assume that the map $f:X\rightarrow Y$ of spectra is 
\begin{enumerate}[leftmargin=26pt]
\item either obtained by applying $\Sigma^\infty$ to a  weak equivalence of well-pointed spaces in $\mathbf{Top}_\ast$
\item or  a weak equivalence   with $X$ cofibrant and $Y$ a suspension spectrum of a well-pointed space.
\end{enumerate}
Then $\DD(f):\DD(Y) \rightarrow \DD(X)$ is a weak equivalence of $S$-modules.
\end{proposition}
\begin{proof}
Let $S_c \rightarrow S$ be a cofibrant replacement of the sphere spectrum. 

For $(1)$, Theorem 24.3.1 of \cite{may2006parametrized} shows that $S_c \wedge X \rightarrow S_c \wedge Y$ is a weak equivalence. The spectra $S_c \wedge X$, $S_c \wedge  Y$ are seen to be cofibrant by combining  Theorem VII.$4.6.$ in \cite{elmendorf2007rings} with Proposition $10.3.18 (i)$ of \cite{may2006parametrized}. Since $S$ is fibrant,   $F(S_c \wedge  X,S) \rightarrow F(S_c \wedge  Y,S)$ is a weak equivalence.  
The left horizontal  arrows below are weak equivalences by Lemma $4.2.7$ of \cite{hovey1999model}:\vspace{-0pt}
\begin{diagram}
F(X,S)  &\rTo^\simeq & F(S_c,F(X,S) )  & \rTo^\cong & F(S_c \wedge  X,S) \\
\dTo & & \dTo & & \dTo^\simeq  \\ 
F(Y,S) & \rTo^\simeq&  F(S_c,F(Y,S)) &\rTo^\cong &  F(S_c \wedge  Y,S) 
\end{diagram}
For claim $(2)$, we first note that smashing with $S_c$ preserves weak equivalences by  Theorem III.$3.8.$ of \cite{elmendorf2007rings}  and then  follow a parallel argument. \vspace{-3pt}
\end{proof}
\begin{proposition}
Let $X,Y\in \mathbf{Top}_\ast$ be well-pointed spaces. Fix weak equivalences from cofibrant \textit{spectra} $\widetilde{\DD(Y)}\rightarrow \DD(Y)$, $\widetilde{X}\rightarrow X$, $\widetilde{Y}\rightarrow Y$. Then, we have:
\begin{enumerate}[leftmargin=26pt]
\item  
$\widetilde{X} \wedge \widetilde{Y} \rightarrow X \wedge Y  $ is a  weak equivalence in $\mathbf{Sp}$.
\item   $\widetilde{X} \wedge \widetilde{\DD(Y)} \rightarrow X \wedge \DD(Y)$ is a  weak equivalence of spectra whenever $X$ is a $\CW$-complex. \vspace{-2pt}
\end{enumerate}
\end{proposition}
\begin{proof} We again let $S_c \rightarrow S$ be a cofibrant replacement of $S$. 
For $(1)$,  consider the following square: \vspace{-5pt}\vspace{-6pt}
\begin{diagram}
(S_c \wedge \widetilde{X}) \wedge (S_c \wedge  \widetilde{Y}) &  &\rTo^\simeq & &(S_c \wedge X) \wedge (S_c \wedge Y) \\
\dTo^\simeq  & & & & \dTo^\simeq  \\
  \widetilde{X}\wedge  \widetilde{Y}&  &\rTo&&  X\wedge Y
\end{diagram}\vspace{-5pt}

The vertical arrows are weak equivalences by  the very strong commutative monoid axiom of \cite{muro2015unit}, asserting that  smashing $S_c \rightarrow S$ with \textit{any} spectrum gives a weak equivalence. Theorem III.$3.8.$ of \cite{elmendorf2007rings} implies that smashing with $S_c$ preserves  weak equivalences. The top horizontal arrow is a weak equivalence as it is obtained by smashing two weak equivalences between cofibrant spectra ($S_c \wedge X$, $S_c \wedge Y$ are  cofibrant by the same argument \mbox{as in   Proposition \ref{sw1}).}

For $(2)$, we consider the diagram \vspace{-5pt}
\begin{diagram}
(S_c \wedge \widetilde{X}) \wedge   \widetilde{\DD(Y)} &   &\rTo^\simeq  & &   (S_c \wedge {X}) \wedge   \widetilde{\DD(Y)} &  &\rTo^\simeq  & & (S_c \wedge X) \wedge  \DD(Y)  \\
\dTo^\simeq  & & & & \dTo^\simeq  & & & & \dTo^\simeq  \\
 \widetilde{X}\wedge  \widetilde{\DD(Y)} & & \rTo & &   {X}\wedge  \widetilde{\DD(Y)} & & \rTo & & X\wedge \DD(Y)
\end{diagram}
The vertical arrows are equivalences as smashing with $S_c\rightarrow S$ preserves these. The top right  hori-zontal  arrow is an equivalence by   \cite[Theorem III.3.8]{elmendorf2007rings} as $S_c \wedge X$ \mbox{is a cell $S$-module.} The top left  horizontal   arrow  is obtained by smashing weak equivalences between cofibrant $S$-modules.
\end{proof}

\newpage

\newpage 
\section{Complementary Collapse}\label{c4}
We will now present an algorithm which collapses large subcomplexes of order complexes attached to lattices and thereby produces equivariant simple homotopy equivalences to wedge sums of smaller spaces. Fix a finite group $G$ throughout this section.\vspace{-5pt}
\subsection{A Reminder on Discrete Morse Theory} We review  the basics of discrete Morse theory: 
\begin{definition}\label{perfectmatch}
A  \textit{G}--equivariant matching on a simplicial $G$-complex $(V,F)$ with fixed point $x$ consists of an equivalence relation $\sim$ on the face set $F$ satisfying the following conditions:
\begin{itemize}[leftmargin=26pt]
\item The relation $\sim$ is $G$-invariant, i.e. $\sigma \sim \tau$ implies $g\sigma \sim g\tau$ for all $g\in G$.
\item Every equivalence class is either equal to $\{x\}$ or has the form $\{\sigma^-, \sigma^+\}$, where $\sigma^- $ is a face of codimension $1$ of $\sigma^+$.
\end{itemize}
Such a matching is   \textit{acyclic} if there does  \textit{not} exist a chain 
\mbox{$\sigma_1^- < \sigma_1^+  > \sigma_2^- < \sigma_2^+ > \dots > \sigma_n^- < \sigma_n^+ > \sigma_1^-$}
with $n>1$ and all $\sigma_i$ distinct.
The following statement is due to Forman \cite{forman1998morse} and Freij \cite{freij2009equivariant}:
\end{definition}
\begin{theorem} 
If there exists  a $G$-equivariant acyclic matching on a simplicial $G$-complex \mbox{$X=(V,F)$} with fixed point $x$, then there is a $G$-equivariant collapse
$|X| \simeq_G \{x\} $.
\end{theorem}

\subsection{Complementary Collapse against Points}\vspace{-5pt}
A finite $G$-lattice $\mathcal{P}$ is a $G$-poset whose underlying poset is a finite lattice, which means that  every two elements $x,y$ have a meet $x\wedge y$ and a join $x\vee y$. 
Fix a finite  $G$-lattice $\mathcal{P}$ and write $\ovA{\mathcal{P}} := \mathcal{P} - \{\hat{0}, \hat{1}\}$, where $\hat{0}$ is the least and $\hat{1}$ the largest \mbox{element of $\mathcal{P}$}.
\begin{definition}
The complement of a $G$-stable  $x\in \ovA{\mathcal{P}}$ is 
\mbox{$x^\perp := \{y \in \ovA{\mathcal{P}} \ | \ x \wedge y = \hat{0}, x \vee y = \hat{1} \}$.}
\end{definition}
Before stating complementary collapse in full generality, we will present a special case:

\begin{theorem} \label{CC1}
If $x \in \ovA{\mathcal{P}}$ is a $G$-stable vertex, 
then $N(\ovA{\mathcal{P}})^{-x^\perp}$
is  $G$-equivariantly collapsible.
Here $N(\ovA{\mathcal{P}})^{-x^\perp}$ is the  complex containing all simplices in $N(\ovA{\mathcal{P}})$ which do not \mbox{contain a vertex in $x^\perp$.}
\end{theorem}
Using Corollary \ref{takingout} and Proposition \ref{indexedwedge}, we immediately obtain:
\begin{theorem}\label{Maesterkey} \label{CC2}
If $x\in \ovA{\mathcal{P}}$ is a $G$-stable element for which $x^\perp$ is discrete, then there are simple homotopy equivalences of $G$-spaces
$$
|\ovA{\mathcal{P}} |  \longrightarrow   \bigvee_{y\in x^\perp}   |\ovA{\mathcal{P}}_{<y}|^\diamond  \wedge   |\ovA{\mathcal{P}}_{>y}|^\diamond   \longrightarrow  \bigvee_{[z]\in x^\perp/G}  \Ind^G_{\Stab(z)} ( |\ovA{\mathcal{P}} _{<z}|^\diamond \wedge | \mathcal{P}_{>z}|^\diamond ) $$
\end{theorem}
\begin{remark}
The contractibility of $N(\ovA{\mathcal{P}})^{-x^\perp}$ and the resulting branching rule is originally due to Bj\"{o}rner-Walker  \cite{bjorner1983homotopy} in the nonequivariant and Welker \cite{welker1990homotopie} in the equivariant case. The collapsibility in the nonequivariant case also follows from work by Kozlov \cite{kozlov1998order} on \mbox{nonevasiveness.}
\end{remark}
\subsection{Orthogonality Fans}\vspace{-5pt}
We will now prove a generalisation of Theorems \ref{CC1} and \ref{CC2}   in  which the ``reference simplex" $x$ is allowed to vary across the poset. We introduce some natural notation:
\begin{notation} Write $\mathcal{F}_\mathcal{P}$ for the collection of nonempty chains $\sigma = [x_0<\ldots<x_n]$ with $x_i \in \mathcal{P}$.
Given {any} chain $\sigma$ and any element $z$, we let $\sigma^{<z}$  denote the subchain of $\sigma$ spanned by all elements $x$ in $\sigma$ which satisfy $x<z$. We write $\sigma<z$  if  all elements $x$ in $\sigma$  satisfy $x<z$ and define $[\sigma<z]$ to be the chain obtained by adding $z$ to $\sigma$. We define the subchain  $\sigma^{>z}$, the  condition $z<\sigma$, and the chain $[z<\sigma]$ analogously. 

If $F:\mathcal{F}_\mathcal{P}\rightarrow \mathcal{P}$ is any function and $y\in \mathcal{P}$, we can define two new functions 
$$F^{\leq y}: \mathcal{F}_{[\hat{0},y]} \rightarrow [\hat{0},y],  \ \ \ \ \ \ \ \ \ \ \ \ \ \ F^{\geq y}: \mathcal{F}_{[y,\hat{1}]} \rightarrow [y,\hat{1}]$$ 
$$\ \ \ \  \ F^{\leq y}(\sigma) := F(\sigma) \wedge y, \ \ \ \ \  \ \ \ \ \ \ \ \ \ \ F^{\geq y}(\sigma) := F([\hat{0}<\sigma])\vee y$$ 
The variation of the reference simplex will be parametrised by the following structure:
\end{notation}
\begin{definition}\label{fans}
A list of  functions $(F_1,\dots,F_r)$ from  $\mathcal{F}_\mathcal{P}$ to $\mathcal{P}$
is an \textit{orthogonality fan} if\mbox{ $r=0$ or }
\begin{enumerate}[leftmargin=26pt]
\item $F_i$ is $G$-equivariant and  increasing for all $i$  (i.e. $y\leq F_i(\sigma)$ for all $\sigma \in \mathcal{F}_\mathcal{P}$ and all $y\in \sigma$.)
\item The subposet $F_1([\hat{0}])^\perp$ is discrete.
\item If $ F_1([\hat{0}])\neq \hat{1}$ , then we have for any $y\in F_1([\hat{0}])^\perp$:\\
The list  \ \ \ \ \ \ \ $(F_2^{\leq y},\dots,F_r^{\leq y}) $ is an orthogonality fan on the $\Stab_y$-lattice $[\hat{0},y]$.\\
The list $(F_1^{\geq y},F_2^{\geq y},\dots,F_r^{\geq y}) $ is an orthogonality fan on the $\Stab_y$-lattice $[y,\hat{1}]$.
\end{enumerate}
\end{definition}

\begin{definition}\label{globalinv} 
A (possibly empty) chain $\sigma=[y_0<\ldots<y_k]$ in $\ovA{\mathcal{P}}$ is said to be \textit{invisible} for some orthogonality fan $\mathbf{F}=(F_1,\dots,F_r)$ if 
$r=0$, or $F_1([\hat{0}])=\hat{1}$,  or  there is a (necessarily unique) element  $y\in [\sigma < \hat{1}]$ with
\begin{enumerate}[leftmargin=26pt]  \item $y\perp F_1([\hat{0}])$
\item $\sigma^{<y}\mbox{ is } \ \ \ \ \ \ \  (F_2^{\leq y},\dots,F_r^{\leq y})\mbox{--invisible}$
\item $\sigma^{>y}\mbox{ is } (F_1^{ \geq y},F_2^{ \geq y},\dots,F_r^{\geq y})\mbox{--invisible}$.
\end{enumerate}
\end{definition}
These two recursive definitions terminate.
An $\mathbf{F}$-invisible chain $\sigma$ is said to be \mbox{\textit{$\mathbf{F}$-orthogonal}  if it is} minimally invisible, i.e. if  none of its proper subchains are $\mathbf{F}$-invisible. Write $\sigma \perp \mathbf{F}$ and observe:
\begin{lemma}\label{containsorthogonal}
Every $\mathbf{F}$-invisible chain $\sigma$ contains a unique orthogonal chain.
\end{lemma}\vspace{-5pt}
\begin{proof}
Let $\tau_1,\tau_2$ be two  orthogonal subchains of an $\mathbf{F}$-invisible chain $\sigma$.
If $F_1([\hat{0}])=\hat{1}$ or $r=0$, then both $\tau_i$ are empty.
Otherwise pick  $y$ in $[\sigma< \hat{1}]$ with $y \perp F_1([\hat{0}])$. Since $\tau_i$ is invisible for $i=1,2$, the chain  $[\tau_i<\hat{1}]$ must contain $y$. The chain $\tau_i^{>y}$ in $\sigma^{> y}$ must be $
(F_1^{\geq y},\dots, F_r^{\geq y})$-orthogonal. By induction, this implies $\tau_1^{>y} = \tau_2^{> y}$. Similarly, we conclude that $\tau_1^{<y} = \tau_2^{< y}$.  
\end{proof}\vspace{-5pt}
Before stating the main theorem of this section, we unravel this definition in two cases of interest.
\begin{example} 
If the orthogonality fan $\mathbf{F}=(F_1)$ consists of a single function, then   a chain $\sigma = [y_0<\ldots<y_k]$ in $\ovA{\mathcal{P}}$ is $(F_1)$-invisible if and only if there are indices $i_1<\ldots < i_\ell$ with
\begin{minipage}{0.6\textwidth}$_{}$ \vspace{-5pt}
\begin{enumerate}[leftmargin=26pt]
\item For $t=1, \ldots,  \ell$, we have \\ 
  $_{} \ \ \ \ \ \ \ \ \ \ \ \ \ \ y_{i_t} \wedge F_1([\hat{0} < y_{i_1}<\ldots< y_{i_{t-1}}]) = y_{i_{t-1}}$ \\ 
 $_{} \ \ \ \ \ \ \ \ \ \ \ \ \ \ y_{i_t} \vee F_1([\hat{0} < y_{i_1}<\ldots< y_{i_{t-1}}]) = \hat{1}$ 
\item $F_1([\hat{0} <y_{i_1}<\ldots<y_{i_\ell}]) \in \{y_{i_\ell},\hat{1}\}$.
\end{enumerate}
\end{minipage}
\begin{minipage}{0.4\textwidth}$_{}$ \\ \\ \ \ \ \ \ \ \ \ 
\includegraphics[width=0.95\textwidth]{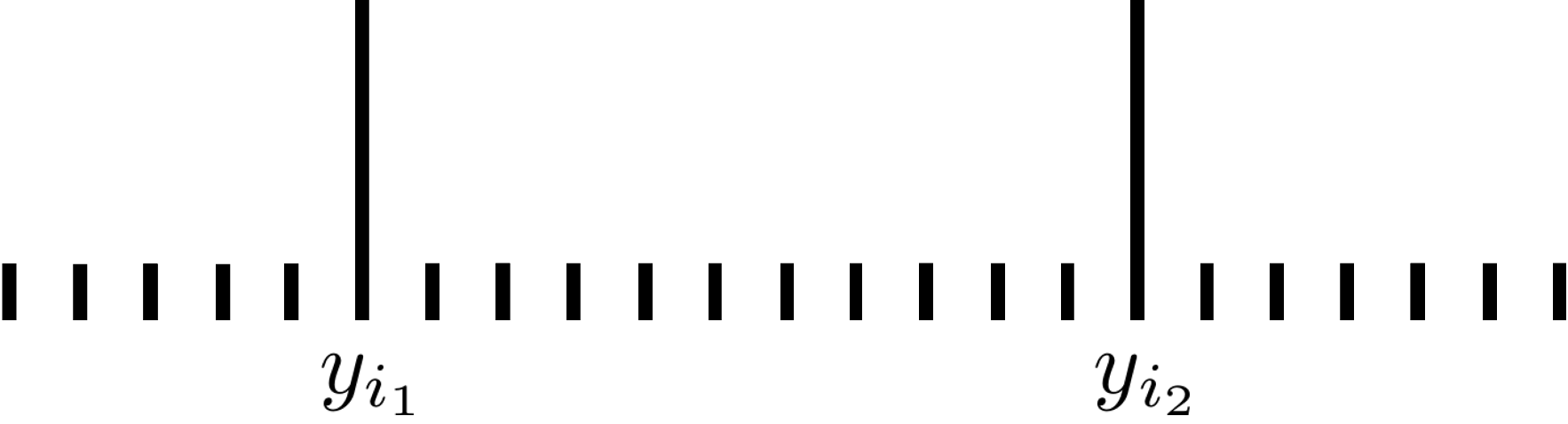}
\end{minipage}  
\ \vspace{-5pt}

The sequence $[y_{i_1}< \ldots < y_{i_\ell}]$ is then orthogonal. 
We use the convention $y_{i_0} = \hat{0}$. For the empty sequence of indices ($\ell = 0$), the first two conditions are automatically satisfied whereas the third condition reads $F_1([\hat{0}]) \in \{\hat{0},\hat{1}\}.$ The empty sequence is orthogonal precisely if $F_1([\hat{0}])\in \{\hat{0},\hat{1}\}$. 
\end{example}  
\begin{example}\label{exb2}
If   $\mathbf{F}=(F_1,F_2)$, then   $\sigma = [y_0<\ldots<y_k]$   is $(F_1,F_2)$-invisible \mbox{if and only if there are}
\mbox{ $j^1_1<\ldots< j^1_{m_1} \ \ \ < i_1< \ \ \ j^2_1<\ldots< j^2_{m_2}  \ \ \ <i_2 <\ \ \ \ldots \ \ \ < i_\ell <  \ \ \  j^{\ell+1}_1<\ldots< j^{\ell+1}_{m_{\ell+1}}$}
\mbox{such that:}
\begin{enumerate}[leftmargin=26pt]
\hspace{-15pt} \vspace{5pt} \begin{minipage}{0.6\textwidth} 
  \item For $t=1, \ldots,  \ell$, we have \\ 
  $_{} \ \ \ \ \ \ \ \ \ \ \ \ \ \ y_{i_t} \wedge F_1([\hat{0} < y_{i_1}<\ldots< y_{i_{t-1}}]) = y_{i_{t-1}}$ \\ 
 $_{} \ \ \ \ \ \ \ \ \ \ \ \ \ \ y_{i_t} \vee F_1([\hat{0} < y_{i_1}<\ldots< y_{i_{t-1}}]) = \hat{1}$ 
 \item $F_1([\hat{0} <y_{i_1}<\ldots<y_{i_\ell}]) \in \{y_{i_\ell},\hat{1}\}$ 
\end{minipage}\hspace{-35pt} 
\begin{minipage}{0.4\textwidth}$_{}$ \\ \\ \ \ \ \ \  \ \ \ 
\includegraphics[width=0.95\textwidth]{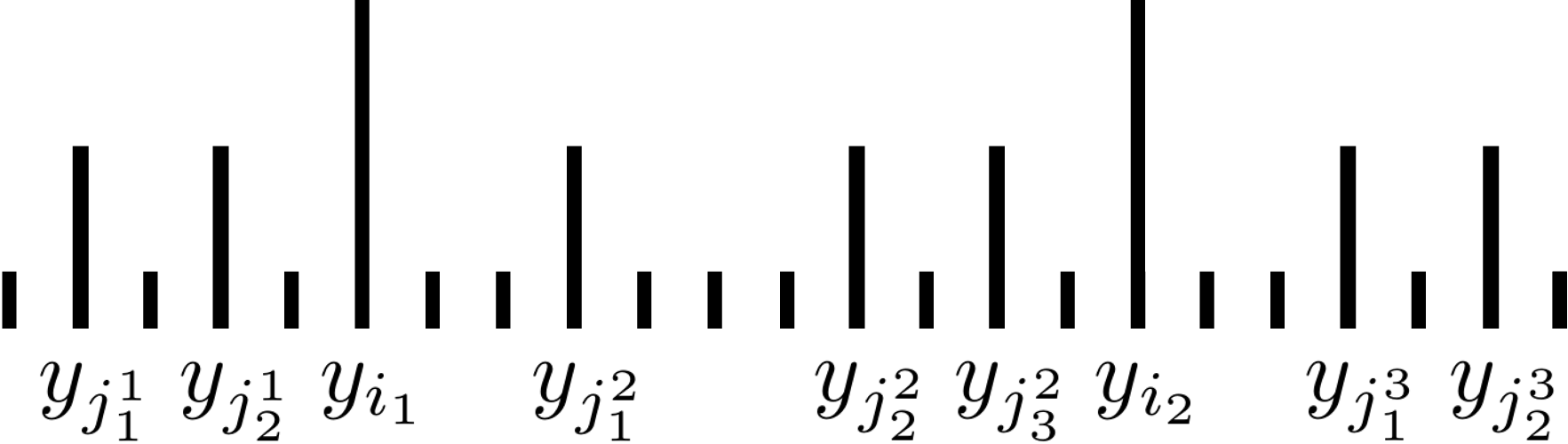}
\end{minipage}
\vspace{-5pt}
\item For $t = 1,\ldots,\ell+1$ and $s=1, \ldots,  m_t$, we have \\
$_{} \ \ \ \ \ \ \ \ \ \ \ \ \ \  y_{j_s^t} \wedge \bigg(F_2([\hat{0} < y_{i_1}< \ldots < y_{i_{t-1}}\ \ \  <\ \ \  y_{j_1^t}<\ldots< y_{j_{s-1}^t}]) \wedge y_{i_t}\bigg) = y_{j_{s-1}^t}\\ _{}
 \ \ \ \ \ \ \ \ \ \ \ \ \ \ y_{j_s^t} \vee\bigg(F_2([\hat{0} < y_{i_1}<\ldots < y_{i_{t-1}} \ \ \ <\ \ \  y_{j_1^t}<\ldots< y_{j_{s-1}^t}]) \wedge y_{i_t}\bigg) = y_{i_t}$
\item  \mbox{For $t = 1,\ldots,\ell+1$, we have $F_2([\hat{0} < y_{i_1}<\ldots < y_{i_{t-1}} < y_{j_1^t}<\ldots< y_{j_{m_t}^t}])\wedge y_{i_t} \in \{y_{j_{m_t}^t},{y_{i_t}}\}$.}
\end{enumerate}
We use the conventions  that $y_{i_0} = \hat{0}$ and $y_{j^t_0} = y_{i_{t-1}}$ for all $t$. Moreover, we set $y_{i_{\ell+1}} = \hat{1}$ if  $F_1([\hat{0} <y_{i_1}<\ldots<y_{i_\ell}])  = y_{i_\ell}$
and $y_{i_{\ell+1}} =  y_{i_\ell}$ if $F_1([\hat{0} <y_{i_1}<\ldots<y_{i_\ell}])  = \hat{1}$.
\end{example}

We can now state the two main theorems of this section:
\begin{theorem}[Complementary Collapse against Fans]\label{CollapseFansA}
Let $\mathbf{F}$ be an orthogonality fan on a finite $G$-lattice $\mathcal{P}$ with $F_1([\hat{0}]) \neq \hat{0},\hat{1}$. Let $N(\ovA{\mathcal{P}})^{-\mathbf{F}^\perp}$ be the simplicial complex obtained from $N(\ovA{\mathcal{P}})$ by deleting all  $\mathbf{F}$-invisible chains. 
Then  $N(\ovA{\mathcal{P}})^{-\mathbf{F}^\perp}$ collapses $G$-equivariantly  to the point $F_1([\hat{0}]) $.
\end{theorem}
\begin{theorem}\label{CollapseFansB}
If $\mathbf{F}$ is an orthogonality fan on a finite $G$-lattice $\mathcal{P}$ with $F_1([\hat{0}]) \neq \hat{0},\hat{1}$, then there is a simple $G$-equivariant homotopy  equivalence obtained by collapsing the subcomplex $N(\ovA{\mathcal{P}})^{-\mathbf{F}^\perp}$:
$$| \ovA{\mathcal{P}}| \    \xrightarrow{\ \ \simeq \ \ }
 \bigvee_{[ y_0 < \dots < y_r ] \perp \mathbf{F}}  | \ovA{\mathcal{P}}_{(\hat{0},y_0)}|^\diamond \wedge \Sigma| \ovA{\mathcal{P}}_{(y_0,y_1)}|^\diamond \wedge \dots \wedge \Sigma| \ovA{\mathcal{P}}_{(y_{r-1},y_r)}|^\diamond \wedge | \ovA{\mathcal{P}}_{(y_r,\hat{1})}|^\diamond $$
\end{theorem}

Complementary collapse for fans can be used to deduce  the weaker   statements in Theorem  \ref{CC1} and Theorem  \ref{CC2}:
\begin{proof}[Proof of Theorems \ref{CC1}, \ref{Maesterkey}]
Let $x$ be a $G$-stable element in a finite \mbox{$G$-lattice $\ovA{\mathcal{P}}$.}
\mbox{Consider the function } $F$ with $F([\hat{0}]) =x$ and  $F(\sigma) = \hat{1}$ if $\sigma \neq [\hat{0}]$.
If $x^\perp$ is discrete, we can apply Theorem \ref{CollapseFansA} to the fan $(F)$ consisting of only {one} function $F$.
A chain $\sigma$ is $F$-invisible iff it contains some $y\perp x$ and it is $F$-orthogonal iff it is of the form $\sigma = [y]$ for $y\perp x$.
The general proof  presented in Section \ref{theproof2} in fact demonstrates Theorem  \ref{CC1} without the assumption that \mbox{$x^\perp$ is discrete.}
\end{proof}
\begin{remark}
One could use Theorem \ref{Maesterkey} and induction to prove the mere existence of the equivalence in  \mbox{Theorem \ref{CollapseFansB}}. However, the equivalence produced in this way would not be easily accessible due to its inductive definition. On the other hand, the equivalence asserted in \mbox{Theorem \ref{CollapseFansB}} is obtained by collapsing a large subcomplex all at once -- the involved maps are therefore entirely transparent. Moreover, the collapsibility  of this large subcomplex asserted in Theorem \ref{CollapseFansA} does not follow from \mbox{Theorem \ref{CollapseFansB}}, but can be of independent interest.
\end{remark}
Since the axiom $(3)$ in Definition \ref{fans} is difficult to check, we introduce the following simpler notion:
\begin{definition}\label{orthogonalityfunction}
Let $G$ be a finite group and $\mathcal{P}$ a finite $G$-lattice  with face set $\mathcal{F}_\mathcal{P}$. 

A function
$F: \mathcal{F}_\mathcal{P} \rightarrow \mathcal{P}$ is called an \textit{orthogonality function} if 
\begin{enumerate}[leftmargin=26pt]
\item $F$ is $G$-equivariant and  increasing  (i.e. $y\leq F(\sigma)$ for all $\sigma \in \mathcal{F}_\mathcal{P}$ and all $y\in \sigma$.)
\item For all $\sigma = [y_0 < \dots < y_m] \in  \mathcal{F}_\mathcal{P}$ and any $z>y_{m}$, the following subposet   is discrete:
$$\{\ \  y_m  < t< z \  \ \ |\ \  \ t \wedge F(\sigma) =y_m\ \  , \ \ (t\vee F(\sigma)) \wedge z = z \ \   \}. $$
\end{enumerate}
\end{definition}
Orthogonality functions will provide us with many examples of orthogonality fans:
\begin{lemma}\label{fansfromfunctions}
If $\mathbf{F} = (F_1,\dots,F_r)$ is a list of orthogonality functions, then $\mathbf{F}$ is an orthogonality fan in the sense of Definition \ref{fans}.
\end{lemma}
\begin{proof}
The first axiom of orthogonality fans is evidently satisfied. The second axiom follows by condition $(2)$ of Definition \ref{fansfromfunctions} for $F_1$ with $\sigma = [\hat{0}]$ and $z=1$. To verify the third axiom, we fix some nonzero $y\perp F_1([\hat{0}])$. We observe that $F_i^{\geq y}$ and $F_i^{\leq y}$ are again orthogonality functions for all $i$. By induction, this implies that $(F_1^{\geq y},F_2^{\geq y},\dots , F_r^{\geq y})$ and 
$(F_2^{\leq y},\dots , F_r^{\leq y})$ are both orthogonality fans on the relevant lattices. 
\end{proof}

\subsection{Proof of Complementary Collapse against Fans}\label{theproof2}
As before, let $\mathbf{F} = (F_1,\dots,F_r)$ be an orthogonality fan on a  finite $G$\mbox{-lattice $\mathcal{P}$.} The following technical gadget will allow us to ``scan'' a chain $\sigma$ in order to find out whether or not it is  $\mathbf{F}$-invisible:

\begin{definition}
The \textit{orthogonality tree} $\mathbf{T}_\mathbf{F}(\sigma)$ of a chain $\sigma = [y_0<\dots< y_k ]$ in $\ovA{\mathcal{P}}$ is an empty {or} planar rooted tree whose nodes $\mathbf{w}$ are labelled  by pairs $(Z\in I_\mathbf{w}) $
consisting of an interval \mbox{$I_\mathbf{w} = [y_\upalpha, y_\omega]$} in $\mathcal{P}$
and a \ $\Stab_{y_\upalpha} \cap \Stab_{y_\omega} $- stable point $Z\in I_\mathbf{w}$. 

The tree is defined by the following recursion:
\begin{enumerate}[leftmargin=26pt]
\item If $r=0$, then the tree is empty.
\item Otherwise, we create a root $\mathbf{v}$ of the tree  $\mathbf{T}_\mathbf{F}(\sigma)$ and label it by
$(F_1([\hat{0}])\in [\hat{0} , \hat{1}]) $.

\begin{itemize}[leftmargin=8pt]
\item If $F_1(\hat{0})=\hat{1}$ or there does not exist a vertex of $[\sigma<\hat{1}]$ lying in $F_1([\hat{0}])^\perp$, we stop.
\item Otherwise assume $y$ is the necessarily unique vertex of $[\sigma<\hat{1}]$ which lies in $F_1([\hat{0}])^\perp$. \\
Let $L$ be the orthogonality tree of the chain $\sigma^{<y}$ for the fan  \ \ \ \ \ \ \ $(F_2^{\leq y},\dots,F_r^{\leq y})$.\\
Let $R$ be the orthogonality tree of the chain $\sigma^{>y}$ for the fan $(F_1^{\geq y}, F_2^{\geq y},\dots,F_r^{\geq y})$.\\
$_{}$ \ \ \ \   Here $[\hat{0},y]$ and $[y,\hat{1}]$ are considered as $\Stab(y)$-lattices.

We create the labelled rooted planar tree $\mathbf{T}_\mathbf{F}(\sigma)$ by declaring  the root of $L$ to be the left child and the root of $R$ to be the right child of  $\mathbf{v}$. Note that $L$  may be empty. \end{itemize}\end{enumerate} \end{definition}
We call a vertex of the orthogonality tree a \textit{leaf} if it has no children.\\

\begin{minipage}{0.35\textwidth}
\begin{example}
Assume that $\mathbf{F}$ consists of \mbox{two functions $  (F_1,F_2)$.}
In the illustration in Example  \ref{exb2}, we considered an invisible chain $\sigma$ with $y_{i_1}<y_{i_2}$ (indicated by long bars) and  for which  $y_{j^1_1}<y_{j^1_2}$, 
$y_{j^2_1}<y_{j^2_2}<y_{j^2_3}$, and $y_{j^3_1}<y_{j^3_2}$ (indicated by short bars). 

The orthogonality tree  $\mathbf{T}_{\mathbf{F}}(\sigma)$ of $\sigma$ is  drawn on the right. The leftmost dotted 
node  is included   if $F_2([\hat{0}<y_{j^1_1} <y_{j^1_2}])\wedge y_{i_1}) = y_{j^1_2}$. The existence of the two other dotted \mbox{nodes depends on similar conditions.}
\end{example} \end{minipage} \ \ \ 
\begin{minipage}{0.6\textwidth} \includegraphics[width=1.25\textwidth]{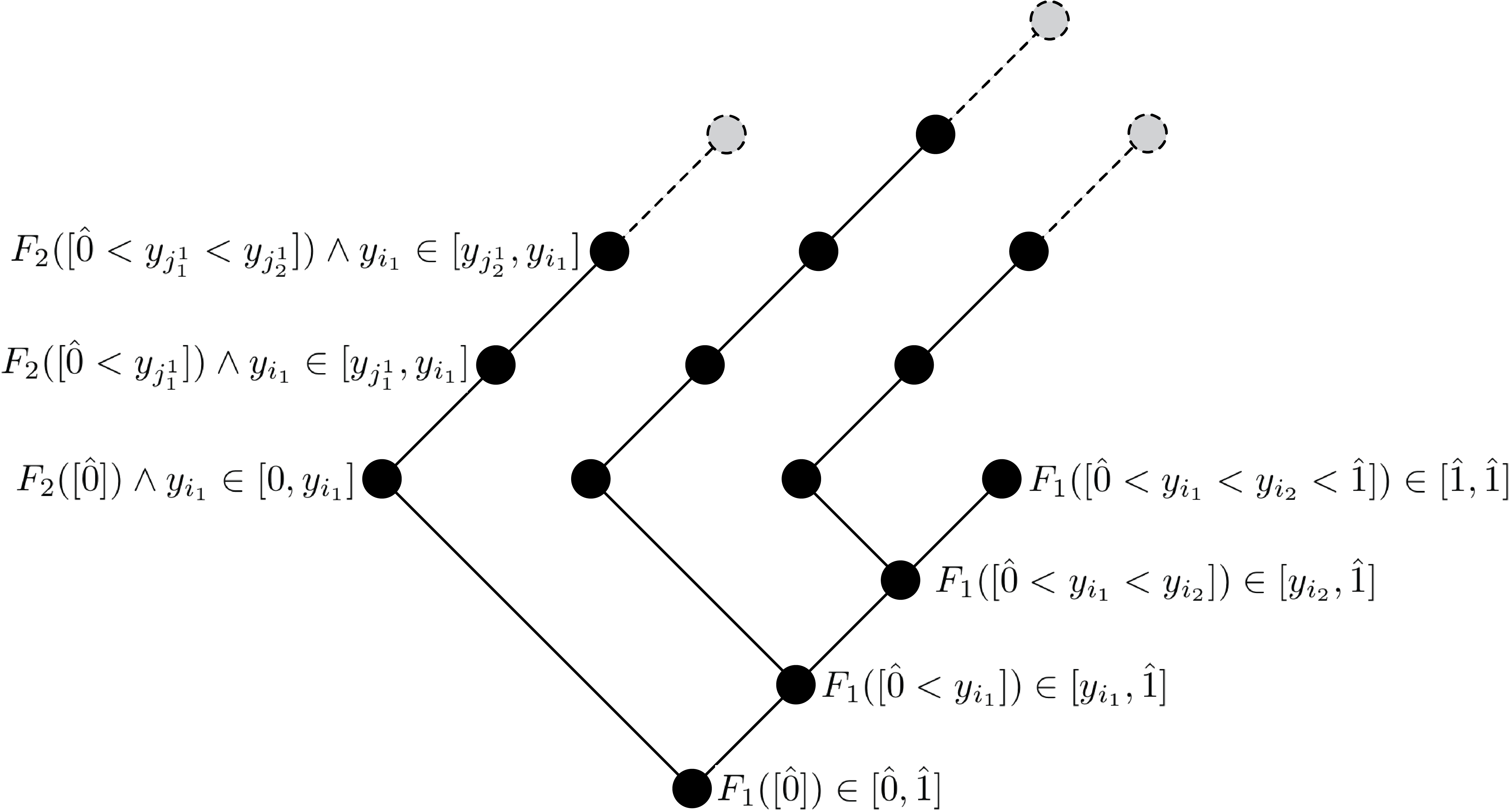}
\end{minipage}
 \ \\

\begin{definition}\label{local}
Let $\sigma$ be a chain and $\mathbf{w}$ a  {leaf} of $\mathbf{T}_{\mathbf{F}}(\sigma)$  with  label $(Z\in  [y_\upalpha, y_\omega] )$. The simplex $\sigma$ is said to be $\mathbf{F}$-\textit{invisible at }$\mathbf{w}$ if  we have $Z = y_{\omega}$.  
\end{definition}
\begin{lemma}\label{iffinv}
A chain $\sigma$  is invisible for an orthogonality fan $\mathbf{F}$ (cf.\ Definition \ref{globalinv}) if  $\sigma$ is \mbox{$\mathbf{F}$-{invisible at every leaf}} $\mathbf{w}$ of its orthogonality tree $\mathbf{T}_{\mathbf{F}}(\sigma)$ (cf.\ Definition \ref{local}).
\end{lemma}\vspace{-5pt}
\begin{proof}
If $r=0$ or $F_1([\hat{0}])=\hat{1}$,  then the equivalence is obvious, and we may therefore assume without restriction that $F_1([\hat{0}])\neq \hat{1}$. \\Suppose $\sigma$ is $\mathbf{F}$-invisible. Then there is a vertex $y$ in $[\sigma < \hat{1}]$ with $y\perp F_1([\hat{0}])$ such that $\sigma^{> y}$ is $(F_1^{\geq y}, \dots  F_r^{\geq y})$-invisible and such that $\sigma^{< y}$ is $(F_2^{\leq y}, \dots  F_r^{\leq y})$-invisible.
By induction, this happens if and only if the two orthogonality trees $L$ and $R$ used in the definition of $\mathbf{T}_\mathbf{F}(\sigma)$ only have invisible leaves, which in turn is equivalent to $\mathbf{T}_\mathbf{F}(\sigma)$ only having invisible leaves.
\end{proof}\vspace{-5pt}
The next two statements follow by similarly straightforward inductions:\vspace{-2pt}
\begin{lemma}\label{sub}
If $\sigma\leq \tau$ is a subsimplex, then $\mathbf{T}_{\mathbf{F}}(\sigma) \leq \mathbf{T}_{\mathbf{F}}(\tau)$ is naturally a (labelled) subtree.\vspace{-2pt}
\end{lemma}
\begin{lemma}\label{sametree}
\mbox{Fix a simplex $\sigma$ with tree $\mathbf{T}_{\mathbf{F}}(\sigma)$ whose \textit{leaf} $\mathbf{w}\in \mathbf{T}_{\mathbf{F}}(\sigma)$  has label $(Z\in [y_\upalpha, y_\omega])$. }\vspace{-5pt}
\begin{itemize}[leftmargin=26pt]
\item \mbox{Adding   $x\in (y_\upalpha, y_\omega)$ to $\sigma$ with
$x\wedge Z \neq y_\upalpha$ or $x\vee Z \neq y_\omega$ 
gives  a $\sigma^+ \geq \sigma$ with equal   tree.}\vspace{-2pt}
\item Removing $x\in (y_\upalpha,y_\omega)$ from $\sigma$ gives  a simplex $\sigma^- \leq \sigma$ \mbox{with equal  tree.}\vspace{-2pt}
\end{itemize}
\end{lemma}
We can now prove complementary collapse for fans:\vspace{-2pt}
\begin{proof}[Proof of Theorem \ref{CollapseFansA}] 
We  define a $G$-equivariant perfect matching with fixed point $F_1([\hat{0}])$ (cf.\ Definition \ref{perfectmatch}).
Fix a nondegenerate $\ell$-simplex $\sigma = [y_0 < \dots < y_k]$ in $N(\overline{\mathcal{P}})^{-\mathbf{F}^\perp}$\mbox{  other than $F_1([\hat{0}]) $}. Let  $\mathbf{w}=\mathbf{w}_{\sigma}$ be the leftmost \textit{leaf} of the orthogonality tree $\mathbf{T}_\mathbf{F}(\sigma)$ such that $\sigma$ is \textit{not} \mbox{$\mathbf{F}$-invisible at $\mathbf{w}$} (it exists by Lemma \ref{iffinv}). 
Write $( Z\in[y_{\upalpha},  y_{\omega}] )$ for the label of $\mathbf{w}$.
\mbox{Since $\sigma$ is not $\mathbf{F}$-invisible at $\mathbf{w}$,} we have strict inequalities $y_{\upalpha}<Z< y_{{\omega}}$. We use the convention that $y_{-1} = \hat{0}$ and $y_{k+1} = \hat{1}$.\\
Let $i\geq \upalpha$ be the largest index with $y_i \wedge Z= y_{\upalpha}$.  Note that $i<{\omega}$ and
observe that $y_i \vee Z < y_{\omega}$  as otherwise $\mathbf{w}$ would  {not} be a leaf.
Let $j\geq i$ be maximal with  $y_j \leq y_i\vee Z $.  We have $j<\omega$.
\\
We call  $(i,j,\mathbf{T})=(i(\sigma), j(\sigma),\mathbf{T}_\mathbf{F}(\sigma))$ the \textit{structure triple} of $\sigma$.
\\
The element $(y_{i} \vee  Z )\wedge  y_{j+1}$ is larger than $y_{\upalpha}$ since it contains  $y_{j+1} \wedge Z $ and smaller than $y_\omega$ as it is contained in 
$ y_{i} \vee  Z < y_{\omega}$.
\\
If $y_{j} <(y_{i} \vee  Z) \wedge y_{j+1}$, match $\sigma$ and \mbox{$\sigma^+ := [\ldots <y_{\upalpha}<\ldots < y_{j} <  (  (y_{i} \vee  Z )\wedge y_{j+1}  )< y_{j+1} < \ldots ]$}.
If  $y_{j} =(y_{i} \vee  Z) \wedge y_{j+1}$,  match  $\sigma$ and $\sigma^- :=[ \ldots <y_{\upalpha} <\ldots  < y_{j-1} <   y_{j+1} <\ldots  ]$.

In the first case, we consider the orthogonality tree of $\sigma^+$. Since $y_{\upalpha}<((y_i \vee Z)\wedge y_{j+1})\wedge Z  $, we know by Lemma \ref{sametree} that $\sigma$ and $\sigma^+$ have the same orthogonality tree. Hence $\mathbf{w}$ is also the leftmost non-invisible node of the orthogonality tree for $\sigma^+$ and it is also labelled by $( Z\in [y_{\upalpha},  y_{\omega}])$.
We can now observe that $i(\sigma^+)=i(\sigma) $, $j(\sigma^+) = j(\sigma)+1$, and hence $(\sigma^+)^-=\sigma$.

In the second case, we first observe that
$j>i$ as otherwise we would have
$y_{i} =(y_{i} \vee Z) \wedge y_{i+1}$ which is absurd as $y_{\upalpha} = y_{i} \wedge Z$ and $y_{\upalpha} <( (y_{i}\vee Z) \wedge y_{i+1} ) \wedge Z$. We therefore remove a vertex $y_j$ in the open interval $(y_{\upalpha},y_{\omega})$ and again conclude by Lemma \ref{sametree} that   the orthogonality trees of $\sigma$ and $\sigma^-$ are equal .
We therefore have $i(\sigma^-) = i(\sigma) $ , $j(\sigma^-)= j(\sigma)   -1$, and  $\mathbf{T}(\sigma) = \mathbf{T}(\sigma^-)$.  We conclude that $(\sigma^-) ^+ = \sigma$.
We have thus defined a matching with fixed point $F(\hat{0})$, and it is evidently $G$-equivariant.

To see that the matching is acyclic, assume for the sake of contradiction that we are given a cycle $ \sigma_1< \sigma_1^+> \  \sigma_2=d_{t_1}(\sigma_1^+)< \  \sigma_2^+>\ \sigma_3=d_{t_2}(\sigma_2^+)<\  \dots \ \sigma_N^+>\   \sigma_1 =d_{t_N}(\sigma_N^+)  $ of distinct nondegenerate simplices in $N(\ovA{\mathcal{P}})^{-\mathbf{F}^\perp}$ for $N>1$. Here $d_t$ denotes the $t^{th}$ face map which forgets   the $t^{th}$ element of a given chain.

Let $(i_s,j_s,\mathbf{T}_s)$ be the structure triple of $\sigma_s$.
We have observed above that the triple attached to $\sigma_s^+$ is $(i_s,j_s+1,\mathbf{T}_s)$. 
By Lemma \ref{sub}, we have  $\mathbf{T}_{s+1} \leq \mathbf{T}_{s}$. Since the above is a cycle, the orthogonality tree $ \mathbf{T}_{s}$ must therefore be constantly equal to  $\mathbf{T}$, say. Let $\mathbf{w}$ be the leftmost non-invisible node of $\mathbf{T}$ with label $(Z\in [y_{\upalpha},  y_{\omega}])$. 

We will now examine how $i$ and $j$ change as $s$ increases. Fix $s$.
By definition, the number $i_{s+1}$ is  the largest number with
$\upalpha < i_{s+1}(< \omega)$ such that the $i_{s+1}$-th vertex of  \mbox{$\sigma_{s+1} = d_{t_s}(\sigma_s^+)$intersects $Z$ in $y_{\upalpha}$.}

Since $\sigma_{s+1} \neq \sigma_{s}$, we know that $t_s \neq j_s+1$.
If $t_s \leq i_s$ then $i_{s+1} =   i_{s}-1 $.
If $t_s >i_s $ and $t_s \neq j_s+2$, then $\sigma_{s+1}$ is an upper simplex in the matching, a contradiction. 
If $t_s = j_s+2$, then  $(i_{s+1},j_{s+1})  =(i_{s},j_{s}+1)$.

The function $(i_s,j_s,\mathbf{T}_s)$ hence cannot visit the same point twice --  the above cycle cannot exist.
\end{proof}
\begin{proof}[Proof of Theorem \ref{CollapseFansB}]Combine Theorem \ref{CollapseFansA}, Proposition \ref{containsorthogonal}, and Corollary \ref{takingout}.
\end{proof}

\newpage

\section{Lie Algebras and the Partition Complex } \label{LAPC} \vspace{-5pt}
We  will now discuss the connection between   partition complexes $|\Pi_n|$ and the theory of classical and spectral Lie algebras. Moreover, we shall  explain  some straightforward algebraic consequences of our later topological results on Young restrictions of $|\Pi_n|$ in Section \ref{Youngsection}.\vspace{-4pt}
\subsection{Free Lie Algebras }\label{section: hilton} We begin by recalling the following basic notion:
\begin{definition}\label{LAdef}
A \textit{Lie algebra} (over $\ZZ$) is an abelian group $\mathfrak{g}$ together with a binary operation $[-,-] $ satisfying $[u,v] = - [v,u]$ and $[u, [v, w]] +[w, [u, v]] +[v, [w, u]] =0$ \mbox{for all $u,v,w \in \mathfrak{g}$.} \vspace{-3pt}
\end{definition}

Given an abelian group $V$, we let $\Lie[V]$ be the free Lie algebra  generated by $V$. 
If $V$ is free with  $\integers$-basis $\{c_s \}_{s\in S}$, we write $\Lie [S]$, and if $S  = \mathbf{k}=\{1,\ldots,k\}$,  we will use the notation $\Lie[c_1,\ldots,c_k]$.

\vspace{2pt}
In fact, $\Lie[S]$ can be constructed as the quotient of the free non-associative and non-unital algebra on $\{c_s \}_{s\in S}$ by the two-sided ideal generated by the antisymmetry  and the Jacobi relation.

Throughout the paper, \textit{monomials} are understood to be in non-commutative and non-associative variables. We will also refer to them as parenthesised monomials. Elements of $\Lie[S]$ can be represented as linear combinations of parenthesised monomials in the  letters $\{c_s\}_{s\in S}$. As is customary in the context of Lie algebras, we will use square brackets to indicate the parenthesisation. For example, the monomial $[[c_1, [c_2, c_1]], c_2]$ represents an element of $\Lie[c_1, c_2]$. 

The symmetric group $\Sigma_S$ acts additively on $\Lie[S]$: given a monomial $w$ in letters $\{c_s \}_{s\in S}$ and a permutation $h \in \Sigma_S$,  the monomial $h \cdot w$ is obtained by replacing each occurrence of $c_s$ by $c_{h(s)}$.
\begin{definition}  A monomial in the Lie algebra $\Lie[c_1, \ldots, c_k]$ has multi-degree $(n_1, \ldots, n_k)$ if it contains $n_i$ copies of the letter $c_i$  for  all $i $.  For example, $[[c_1, [c_2, c_1]], c_2] $  has multi-degree $(2,2)$.\vspace{0pt}
\end{definition}

\mbox{There is a second common notion of Lie algebras over the integers:}
\begin{definition}\label{isotropic}
A \textit{totally isotropic Lie algebra} (over $\ZZ$) is an abelian group $\mathfrak{g}$   with a binary operation $[-,-] $ satisfying $[u,u] = 0$ and $[u, [v, w]] +[w, [u, v]] +[v, [w, u]] =0$ \mbox{for all $u,v,w \in \mathfrak{g}$.}
\end{definition}
\begin{remark} Some sources use the term ``Lie algebra" for what we call ``totally isotropic Lie algebra", and the term ``quasi-Lie algebra" for what we call ``Lie algebra".  \vspace{1pt}
\end{remark}

Write $\Lie^i[V]$ for the free totally isotropic Lie algebra on an abelian group. If $V$ is free on some set $S$, we use the notation $\Lie^i[S]$, and if $S=\mathbf{k}$, we write $\Lie^i[c_1,\ldots,c_k]$.
It is evident that  the free Lie algebra $\Lie[V]$ on an abelian group $V$ maps to the free totally isotropic \mbox{Lie algebra $\Lie^i[V]$ on $V$.}\vspace{1pt}

 {Totally isotropic Lie algebras} have the   property that the free totally isotropic Lie algebra $\Lie^i[V]$  on a free abelian group $V$ is  again  a free abelian group (cf.\   \cite[Corollary 0.10]{reutenauer2003free}). This is \textit{not} true for   free Lie algebras  in the sense of \mbox{Definition \ref{LAdef}.} For example, the free Lie algebra  $\Lie[c_1]$ on one generator $c_1$ has underlying $\ZZ$-module  $\ZZ \oplus \ZZ/2$ generated by $c_1$ and  $[c_1,c_1]$ with $2[c_1,c_1]=0$.\vspace{1pt}

However,  totally isotropic Lie algebras are overly restrictive in our context: first, they cannot be defined as algebras over an operad. Moreover, several natural examples of interest (related to algebraic Andr\'{e}-Quillen homology  groups  or  spectral Lie algebras over   $\FF_2$, cf.\ \cite{goerss1990andre}, \cite{camarena2016mod}) only satisfy the weaker antisymmetry axiom $[u,v] = -[v,u]$. 
We will therefore focus on  Lie algebras. \vspace{-13pt}

\subsection{Lyndon Words}\label{Lyndon}\vspace{-6pt}
\mbox{ We will use the following classical notion  (cf.\ \cite{shirshov1958free}, \cite{chen1958free}):}
\begin{definition}\label{Lyndon word}  
A word $w$ in letters $c_1,\dots,c_k$ is  a \textit{(weak) Lyndon word} if it is (weakly)  smaller than any of its cyclic rotations in the lexicographic order with $c_1<\dots < c_k$. Write  $B(n_1,\dots,n_k)$ (or $B^w(n_1,\dots,n_k)$) for the set of  (weak) Lyndon words which involve $c_i$  precisely $n_i$ times. \vspace{1pt}
\end{definition}
A Lyndon word $w$ of length $\ell>1$ can be written uniquely as $w = u v$ with $u<v$ both Lyndon words and $u$ as long as possible  -- this is  the \textit{(left) standard factorisation}\index{Standard factorisation}. Given any two Lyndon words $u<v$, the word $w=uv$ is again a Lyndon word. The factorisation $w=u  v$  is standard if and only if   $v$ is a  letter \textit{or} it has standard factorisation $v=xy$ with $x\leq u$.

There is a unique injection $\phi$ from the set of Lyndon words in $c_1,\dots,c_k$ to  $\Lie[c_1,\dots,c_k]$ sending each letter $c_i$ to $c_i\in \Lie[c_1,\dots,c_k]$ and   satisfying $\phi(w) = [\phi(u),\phi(v)]$ for each standard factorisation $w = u v$.  Elements in the image of $\phi$ are  called \textit{basic monomials}, and we will often   identify Lyndon words with \mbox{their image under $\phi$.}

\begin{remark} Basic monomials  \textit{do not} generate  $\Lie[c_1,\dots,c_k]$ due to the presence of self-brackets. However,  their images  under  $\Lie[c_1,\dots,c_k] \rightarrow \Lie^i[c_1,\dots,c_k]$ form a $\ZZ$-basis (cf.\ \cite{reutenauer2003free}). \hspace{3pt}
\end{remark}

Any  {weak} Lyndon word $w$ can be written uniquely as $w=u^d$ with $u$ a Lyndon word. We call $d$ the \textit{period} of $w$. This gives an identification
$B^w(n_1,\dots,n_k) = \coprod_{d|\gcd(n_1,\dots,n_k)} B(\frac{n_1}{d},\dots,\frac{n_k}{d} ) $.
\vspace{2pt}

Now let $S$ be a finite set and assume $g:S\rightarrow \{1,\dots,k\}$ is a map whose fibre $C_i$ over $i$ has size $n_i$.
\begin{definition}\label{laLyn}
An $(S,g)$-\textit{labelling} of a weak Lyndon word $w=u^d$ in letters $c_1,\ldots,c_k$ is represented by a \mbox{bijection $f$} from $S$ to the  letters of $w$ with the property that for each $s\in C_i$, the symbol $f(s)$ is  $c_i$. Two bijections are considered to represent the same labelling  if one can be obtained from the other by permuting the various copies of $u$ in $w$.

 Write $B^w_{(S,g)}(n_1,\dots,n_k) $ for the set of $(S,g)$-labelled weak Lyndon words in  $B^w(n_1,\dots,n_k) $. This set comes endowed with a natural (left) $\Sigma_{C_1} \times \dots \times \Sigma_{C_k}$-action: given a permutation $h$ and a weak Lyndon word $w$ whose labelling is represented by the bijection $f$ from $S$ to the letters of $w$, we form a new labelled word $h\cdot w$ \mbox{by changing $f$ to $f\circ h^{-1}$.}\vspace{0pt}
\end{definition}

\mbox{Labelled words can be used to  define  \textit{multilinear}   monomials containing each generator once:}
\begin{definition}\label{resolution}\index{Resolution of a labelled Lie word}
Let $w$ be a Lyndon word  with an $(S,g)$-labelling represented by \mbox{the bijection $f$.} The $(S,g)$-{\it resolution} $\tilde w^f\in \Lie[S]$ of $w$ is represented by the multilinear monomial obtained from $w$ by replacing  each letter in $w$ with its preimage  in $S$ under $f$.

We observe that this resolution procedure intertwines the action on labels in Definition \ref{laLyn} with the action on words in   Section \ref{section: hilton}, i.e.  if $h \in \Sigma_{C_1} \times \dots \times \Sigma_{C_k}$, then 
$h \cdot \tilde{w}^{f} = \widetilde{h\cdot w}^{h^{-1}(f)}$.

\begin{remark}\label{standardchoice}
Given a decomposition $n_1+\ldots+n_k=n$, there is a standard choice for $(S,g)$: we take $S = \mathbf{n}= \{1,\dots,n\} $ for $n = \sum n_i$ and use the unique order-preserving  map $g:\mathbf{n} \rightarrow \mathbf{k}$ with $|g^{-1}(i)|=n_i$ for all $i$. We  write $B^w_{\mathbf{n}}(n_1,\dots,n_k) := B^w_{(S,g)}(n_1,\dots,n_k)$.
The quotient of 
 $B^w_{(S,g)}(n_1,\ldots,n_k)$ by the group $\Sigma_{n_1}\times \ldots \times \Sigma_{n_k}$ can be identified with the set of weak Lyndon words $B^w(n_1,\dots,n_k) = \coprod_{d|\gcd(n_1,\dots,n_k)} B(\frac{n_1}{d},\dots,\frac{n_k}{d} ) $ containing the $i^{th}$ generator $n_i$ times.
 
Each orbit in the $\Sigma_{n_1} \times \ldots \times \Sigma_{n_k}$-set $B^w_{\mathbf{n}}(n_1,\dots,n_k)$ contains a unique labelled weak Lyndon word $w=u^d$ with the property that for each  copy of $u$, all labels are congruent mod $d$ and increase on all occurrences of a letter $c_i$ in this copy from left to right. We call this the \textit{standard labelling} of the underlying weak Lyndon word $w=u^d $.
The various copies of $u$ then partition $n$ as $S_1 = \{1,d+1, \ldots\}$, $S_2 = \{2,d+2, \ldots\}$, \ldots. via the labelling.
 The stabiliser of $w$ in  $B^w_{(S,g)}(n_1,\ldots,n_k)$ is equivalent to $\Sigma_d$, embedded diagonally as $\Sigma_d \rightarrow \Sigma_d^{ \frac{n}{d}} \rightarrow \Sigma_{n_1}\times \ldots \times \Sigma_{n_k}$. 
\end{remark}
\end{definition}
\vspace{-10pt}

\subsection{The Algebraic Lie Operad}
\begin{definition} Given a finite set $S$, the \textit{Lie representation} $\Lierep_S$ of the symmetric group $\Sigma_S$ is given by the  submodule of $\Lie[S]$  spanned by all words  which contain each generator $c_s$   once.\end{definition}
It is well-known that $\Lierep_S$ is a free abelian group.
The $\Sigma_S$-modules $\Lierep_S$ assemble to   the \textit{Lie operad}. Hence every bijection \mbox{$  \coprod_{i\in D} S_i \stackrel{ \ \ \phi\ \ }{\rightarrow} S$} gives  rise to a homomorphism 
\mbox{$
s_\phi\colon\Lierep_{D}\otimes \bigotimes_{i \in D} \Lierep_{S_i}\rightarrow \Lierep_S
$}
which is natural in the bijection $\phi$ and  with respect to \mbox{permutations of the sets $S_i$.}

Let us describe the map $s_\phi$ explicitly. Recall that $\Lierep_S$ is generated by  multilinear monomials in {variables $\{c_s \}_{s\in S}$.} Let $w_D \in \Lierep_{D}$ and $\{ \ w_{S_i} \in \Lierep_{S_i} \ |  \ i\in D\ \} $ be represented by monomials.
Then $s_\phi(w_D \otimes \bigotimes_{i\in D} w_{S_i})\in \Lierep_S$ is obtained by taking $w_d$,  substituting the monomial $w_{S_i}$ for the variable corresponding to $i\in D$, and finally replacing \mbox{each letter $c_s$ with $s\in S_i$ by the letter $c_{\phi(s)}$.}

\begin{remark}
 {Lie algebras in the sense of Definition \ref{LAdef}  are just algebras over the Lie operad (cf.\ \cite{fresse346koszul}, 1.1.11.)}.
The well-known formula for a free algebra over an operad gives a natural isomorphism  $\epsilon\colon\  \bigoplus_{n} \ \Lierep_{\mathbf{n}} \myotimes{\integers[\Sigma_n]} V^{\otimes n} \xrightarrow{ \ \ \cong \ \ } \Lie[V]$ (cf.\ ~\cite{markl2007operads}, \mbox{Propositions 1.25., 1.27.)}.

If $w$ is a multilinear monomial in letters $c_1, \ldots c_n$, then $\epsilon$ sends $w\otimes v_1\otimes\cdots\otimes v_n$ to $w(v_1, \ldots, v_n)$.    \vspace{-5pt}\end{remark}

\subsection{The Algebraic Branching Rule}\label{section: lie} 
 We will now explain an algebraic branching rule for  the $\Sigma_n$-module $\Lierep_{\mathbf{n}}$, which will follow from our later topological  Theorem \ref{main}  by applying homology.\vspace{0pt}
 
  Fix   a decomposition \mbox{$n=n_1+\ldots+n_k$}. Let $S=\mathbf{n}$ and $g:\mathbf{n}\rightarrow \mathbf{k}$ be  order-preserving  with $|g^{-1}(i)| = n_i$.
Suppose $w= u^d\in B^w_{\mathbf{n}}(n_1,\dots,n_k)$ is an $\mathbf{n}$-labelled weak Lyndon word with \mbox{period $d$.} The different copies of $u$ in $w$  induce a partition $\mathbf{n} = \coprod_{S_i \in \mathbf{d}_w} S_i$ via the labelling. Here $\mathbf{d}_w$ denotes the set of classes of this partition. It has cardinality $d$, but no preferred identification with $\mathbf{d}=\{1,\ldots,d\}$. 
For each $S_i\in \mathbf{d}_w$, we can restrict the labelling  $f$ on $w$ to an $(S_i,g|_{S_i})$-labelling $f_i$ on $u$ and form a word \mbox{$\tilde{u}^{f_i} \in \Lierep_{S_i}$} by replacing each  $c_j$ with $j\in \mathbf{k}$ \mbox{by the corresponding   $c_s$ for $s\in S_i$.}

Let  $\uptheta_w: \Lierep_{\mathbf{d}_w} \rightarrow  \Lierep_{\mathbf{n}}$ be defined
as follows: given a multilinear  monomial $w_d\in  \Lierep_{\mathbf{d}_w}  $ in letters $\{ c_{S_i} \}_{S_i \in \mathbf{d}_w }$, \mbox{replace each  $c_{S_i}$ in $w_d$ with  $\tilde{u}^{f_i}$.}
The group $\Sigma_{\mathbf{d}_w}$ acts  naturally on $\Lierep_{\mathbf{d}_w}$. It acts on $\mathbf{n}$ (and hence   $\Lierep_{\mathbf{n}}$)  by letting $h \in \Sigma_{\mathbf{d}_w}$ send   $x\in S_i$ to the unique $h\cdot x \in h(S_{i})$ for which $x$ and $h\cdot x$ label corresponding letters in the respective copies of $u$.  The map $\uptheta_w$ is $\Sigma_{\mathbf{d}_w}$-equivariant.

\begin{example}\label{ex3}
\mbox{For $n_1=n_2=4$ and $d=2$,  consider the $\mathbf{n}$-labelled Lyndon word $w=u^2$ given by\vspace{-10pt}} 
\begin{diagram}
1 & \  \ &3& \ \  & 5 & \ \ &7 & \  \ &2 &  \  \ & 4 &  \  \ &6 & \ \  & 8\vspace{-10pt}\ \\
c_1& \ \  & c_1&   \  \  &c_2&  \  \   &c_2&  \  \ &c_1 &  \ \   &c_1& \  \  &c_2 & \  \ &  c_2 \vspace{-5pt}
\end{diagram} 
The Lyndon word $u$ corresponds to  $[[c_1,[c_1, c_2]], c_2]\in B(2,2)$. The induced partition on $\mathbf{8}$ has classes $S_1 = \{1,3,5,7\}$ and $S_2=\{2,4,6,8\}$. The group $\Sigma_{\mathbf{d}_2}\cong \Sigma_2$ acts on $\mathbf{8}$ via $(1,2)(3,4)(5,6)(7,8)$.

The homomorphism $\uptheta_w:\Lierep_2\rightarrow \Lierep_8$ is defined follows. Suppose $w_2(c_{S_1}, c_{S_2})$ is a monomial representing an element of $\Lierep_{\mathbf{d}_2}$. Then $w_2$ is sent to $w_2([[c_1,[c_3, c_5]], c_7], [[c_2,[c_4, c_6]], c_8]).$
In particular our $\Sigma_2$-equivariant homomorphism acts as follows on words of length $2$:  \vspace{-5pt}
$$ [c_2, c_1]\mapsto [[[c_2,[c_4, c_6]], c_8],[[c_1,[c_3, c_5]], c_7]] \ \ \ \ \ \ \ \ \ \ \ \ 
[c_1, c_2]\mapsto [[[c_1,[c_3, c_5]], c_7], [[c_2,[c_4, c_6]], c_8]].\vspace{-2pt}$$
\end{example}

Summing  up  $\uptheta_w$ over all labelled weak Lyndon words $w$, we obtain an equivariant homomorphism 
$ \bigoplus \uptheta_w  :\bigoplus_{w \in B^w_{\mathbf{n}}(n_1,\ldots,n_k)}\Lierep_{\mathbf{d}_w} \longrightarrow \Lierep_{\mathbf{n}} $, where 
$\Sigma_{n_1}\times \ldots \times \Sigma_{n_k}$ acts on the left as follows: given a permutation $h$ and a word $w_d \in \Lierep_{\mathbf{d}_w}$,  define  $h\cdot w_d \in \Lierep_{\mathbf{\mathbf{d}_{h \cdot w}}} $
 by replacing the letter $c_{S_i}$ corresponding to    \mbox{$S_i \in \mathbf{d}_w$} by the letter $c_{h (S_i)}$ corresponding to    \mbox{$h(S_i) \in \mathbf{d}_{h\cdot w}$.}  As before, $h \cdot w$ is the labelled weak Lyndon word obtained from $w$ by precomposing the labelling with $h^{-1}$ \mbox{(cf.\ Definition \ref{laLyn}).}
\begin{example}
In the above Example \ref{ex3}, we take $h = (124) (68)$. Then $h \cdot w$ is given by\vspace{-2pt}
\begin{diagram}
2 & \  \ &3& \ \  & 5 & \ \ &7 & \  \ &4 &  \  \ & 1 &  \  \ &8 & \ \  & 6\vspace{-10pt}\ \\
c_1& \ \  & c_1&   \  \  &c_2&  \  \   &c_2&  \  \ &c_1 &  \ \   &c_1& \  \  &c_2 & \  \ &  c_2 \vspace{-7pt}
\end{diagram} 
\mbox{Hence $h\cdot \left([c_{S_2},c_{S_1}] \in \Lierep_{\mathbf{d}_w}\right) = \left([c_{h(S_2)},c_{h(S_1)}] \in \Lierep_{\mathbf{d}_{h\cdot}}\right)$ for $h(S_1) = \{2,3,5,7\}$, $h(S_2) = \{1,4,6,8 \}$.}
\end{example}
\vspace{3pt}
{The following  will be proven on p.\pageref{algebraicbranchingproof}
by applying homology to  Theorem \ref{main}:}
\begin{lemma}[Algebraic Branching Rule] \label{equation: branching for lie} The map $\bigoplus  \uptheta_w$ gives $\Sigma_{n_1}\times\cdots\times \Sigma_{n_k}$-isomorphisms\vspace{-2pt}
$$
\bigoplus_{d|\gcd(n_1, \ldots, n_k)}    \ \  \bigoplus_{u\in B\left(\frac{n_1}{d},\ldots, \frac{n_k}{d}\right)}
\integers[\Sigma_{n_1}\times\cdots\times \Sigma_{n_k}] \myotimes{\integers[\Sigma_d]}
\Lierep_d\ \ \xrightarrow{\  \cong \   } \ \ \bigoplus_{w \in B^w_{\mathbf{n}}(n_1,\ldots,n_k)}\Lierep_{\mathbf{d}_w}\ \   \xrightarrow{\  \cong   \ } \ \  \Lierep_n.\vspace{-3pt}$$
\end{lemma}

\subsection{The Algebraic Hilton-Milnor Theorem}
In topology, the Hilton-Milnor theorem is a  homotopy decomposition of the space $\Omega\Sigma(X_1\vee\ldots\vee X_k)$ as a weak product indexed by basic monomials in $k$ variables $X_1, \ldots, X_k$. Equivalently, one can read the theorem as a  homotopy decomposition of the pointed free topological group generated by $X_1, \ldots, X_k$, where $X_i$ are connected spaces.  

The free Lie algebra over $\ZZ$ (in the sense of Definition \ref{LAdef}) on a direct sum has an analogous decomposition, which we will call the \textit{algebraic} Hilton-Milnor theorem. There are differences between the two versions: the algebraic version gives an isomorphism, rather than an equivalence, \mbox{and it does not require a connectivity hypothesis.}  {In our later Corollary \ref{BBBB}, we will in fact generalise the algebraic Hilton-Milnor theorem   to spectral Lie algebras. }\vspace{5pt}

We fix free abelian groups $V_1, \ldots, V_k$ and an order-preserving map  $g:\mathbf{m}\rightarrow \mathbf{k}$ with $m_i = |g^{-1}(i)|$. Every Lyndon word $u\in B (m_1, \ldots, m_k)$  has a unique  $(\mathbf{m},g)$-labelling for which the labels of each letter $c_i$  increase from left to right in $u$. Write $\tilde{u}$ for the corresponding resolution \mbox{(cf.\ Definition \ref{resolution}).} We define a map of abelian groups
$
\beta_u\colon V_1^{\otimes m_1}\otimes \cdots\otimes V_k^{\otimes m_k} \rightarrow \Lie[V_1 \oplus \cdots\oplus V_k]$ by setting $
\beta_u (\otimes_{i=1}^k\otimes_{j=1}^{m_i} v_i^j )=\tilde u(v_1^1, v_1^2, \ldots, v_1^{m_1}, v_2^1, \ldots,v_2^{m_2}, \ldots, v^{1}_k,\ldots, v_k^{m_k})
$, where $v_i^j\in V_i$. Since $\tilde u$ is  multilinear, this homomorphism is well-defined.
\begin{example}
The Lyndon word $u=[[c_1,[c_1, c_2]], c_2]$ determines the group homomorphism $\beta_u\colon V_1^{\otimes 2}\otimes V_2^{\otimes 2}\rightarrow \Lie[V_1\oplus V_2]$ defined by
$
\beta_u(v_1^1 \otimes v_1^2 \otimes v_2^1\otimes v_2^2)= \tilde u (v_1^1, v_1^2, v_2^1, v_2^2) = [[v_1^1, [v_1^2, v_2^1 ]], v_2^2]
$.
\end{example}\vspace{3pt}

The map $\beta_u$ extends to a map of Lie algebras \mbox{
$\tilde\beta_u\colon \Lie[V_1^{\otimes m_1}\otimes \cdots\otimes V_k^{\otimes m_k}] \rightarrow \Lie[V_1 \oplus \cdots\oplus V_k].$}
Using the algebraic branching rule in Lemma \ref{equation: branching for lie} , we can deduce the following result:
\begin{theorem}[Algebraic Hilton-Milnor Theorem]\label{theorem: hilton-milnor}
The homomorphism of abelian groups 
$$
\oplus \ \tilde\beta_u \ : \ \underset{\underset{u\in B(m_1,\ldots, m_k)}{m_1, \ldots, m_k\ge 0}}{\bigoplus} \Lie[V_1^{\otimes m_1}\otimes \cdots\otimes V_k^{\otimes m_k}] \rightarrow \Lie[V_1 \oplus \cdots\oplus V_k].\vspace{-5pt}
$$
obtained by taking the direct sum of all maps $\tilde\beta_u$  is an  isomorphism.\vspace{-5pt} 
\end{theorem}
\begin{proof}
We  construct a commutative square\vspace{-1pt} 
\begin{diagram} 
\underset{\underset{u\in B(m_1,\ldots, m_k)}{m_1, \ldots, m_k\ge 0}}{\bigoplus}\ \underset{d\geq 1}{\bigoplus}\  \Lierep_d \myotimes{\Sigma_d}(V_1^{\otimes m_1}\otimes \ldots\otimes V_k^{\otimes m_k})^{\otimes d} & \rTo & &  \underset{n\ge 1}{\bigoplus}\Lierep_n\myotimes{\Sigma_n}(V_1 \oplus \ldots\oplus V_k)^{\otimes n}\\
  & & &    \\
  \dTo  & & & \dTo \\
\underset{\underset{u\in B(m_1,\ldots, m_k)}{m_1, \ldots, m_k\ge 0}}{\bigoplus}\Lie[V_1^{\otimes m_1}\otimes \cdots\otimes V_k^{\otimes m_k}] & \rTo^{\oplus \tilde\beta_u} &  &  \Lie[V_1 \oplus \cdots\oplus V_k] 
\end{diagram} 
The vertical isomorphisms are given by our operadic description of free Lie algebras.
The  top  map on the summand corresponding to  $m_1, \ldots, m_k$,    $u\in B(m_1,\ldots, m_k)$, $d\geq 1$ 
is obtained as follows. \mbox{First, apply $\Sigma_{dm_1 }\times  ... \times \Sigma_{dm_k }$-orbits to the following map  induced by algebraic branching for $\Lierep_n$:}\vspace{-4pt} 
$$
(\integers[\Sigma_{dm_1 }\times\cdots\times \Sigma_{d m_k }] \myotimes{\integers[\Sigma_d]}
\Lierep_d) \otimes (V_1^{\otimes m_1  }\otimes \ldots\otimes V_k^{\otimes  m_k  })^{\otimes d}  \longrightarrow \Lierep_n \otimes V_1^{\otimes d m_1}\otimes \ldots\otimes V_k^{\otimes d m_k}\vspace{-8pt} $$
Here  $n = dm_1+ \ldots + dm_k$.
Then compose with  $ \Lierep_n\myotimes{\integers[\Sigma_{dm_1}\times\cdots\times \Sigma_{dm_k}]} (V_1^{\otimes dm_1}\otimes \ldots\otimes V_k^{\otimes dm_k}) \rightarrow \Lierep_n\myotimes{\Sigma_n}(V_1 \oplus \ldots\oplus V_k)^{\otimes n}$.
A routine diagram chase shows that the square in question commutes. The top horizontal map is an isomorphism by the algebraic branching rule in Theorem \ref{equation: branching for lie}.
\end{proof}
\begin{remark}
There is a corresponding decomposition for totally isotropic Lie algebras, which is  folklore and can, for example,  be easily deduced   from Example $8.7.4$ in Neisendorfer's book \cite{neisendorfer2010algebraic}.
\end{remark}

\subsection{The Spectral Lie Operad}\label{section: lie to partitions}
There is a close connection between the Lie representation and the cohomology of the partition complex: given a finite set $S$, there is an isomorphism of \mbox{$\Sigma_S$-representations} \mbox{$\Lierep_S\ \cong  \     \sgn_S   \otimes \ {\widetilde{\HH}}^{|S|-1}(\Sigma |\Pi_S|^\diamond,\ZZ) , $}
where $\sgn_S$ denotes the sign representation. This  was first proven by H. Barcelo~\cite{barcelo1990action} and later generalised, and perhaps simplified, by Wachs~\cite{wachs1998co}. Salvatore \cite{salvatore1998configuration}  and
Ching \cite{ching2005bar} have subsequently endowed the collection of  spectra $\{\Map_{\Sp}(S^1,(S^1)^{\wedge S})\wedge \DD(\Sigma |\Pi_S|^\diamond) \}_{S}$ with the structure of an operad in a way that is compatible with the operad structure on $\{\Lierep_S\}_S$.  \vspace{3pt}

We will now explain these relations. Let $\mathcal{P}_S$ be the poset of partitions of $S$, so that \mbox{$\Pi_S = \mathcal{P}_S-\{\hat{0},\hat{1}\}$}. The group $\Sigma_S$ acts by letting a permutation $h\in \Sigma_S$ send a partition  $\coprod_i S_i$ to the partition $\coprod_i h(S_i)$.

For $|S|>1$, we use the simplicial model $N_\bullet (\mathcal{P}_S)/_{N_\bullet(\mathcal{P}_S-\hat{0}) \cup N_\bullet(\mathcal{P}_S-\hat{1})}$ 
for $\Sigma |\Pi_S|^\diamond$ (cf.\ Section \ref{clock}). For $|S|=1$, we model  $\Sigma |\Pi_S|^\diamond=S^0$ by the simplicial $0$-circle.
The set  of $k$-simplices of the simplicial model for $\Sigma |\Pi_S|^\diamond$ is given by   $ \{[\hat{0}\leq x_1\leq \ldots\leq x_{k-1}\leq \hat{1}]\} \coprod \{\ast\} $, i.e. by the set of all chains starting in $\hat{0}$ and ending in $\hat{1}$ together with an additional  basepoint $\ast$. \vspace{3pt}

In order to prove  \mbox{$\Lierep_S \ \cong \  \sgn_S \ \otimes \  {\widetilde{\HH}}^{|S|-1}(\Sigma |\Pi_S|^\diamond,\ZZ)  \  \cong \ {\widetilde{\HH}}_0\left(\Map_{\Sp}(S^1,(S^1)^{\wedge S})\wedge \DD(\Sigma |\Pi_S|^\diamond)  ,\ZZ\right)$} and construct the operadic structure maps, 
we recall  the language of weighted trees from \cite{ching2005bar}:
\begin{definition}\label{n=1}
A \textit{(rooted) tree} consists of a finite poset $T$ containing a unique minimal \mbox{element $r$,} the \textit{root}, and another element $b>r$ with $b\leq t$ for all $t\in T-\{r\}$, such that:\vspace{-4pt}
\begin{enumerate}[leftmargin=26pt]
\item  If $u,v,$ and $t$ are any elements of $T$ with $u\leq t$ and $v\leq t$, then $u\leq v$ or $v\leq u$.
\item   If $t,u$ are elements in $T$ with $t<u $ and $r\neq t$, then there exists a $v$ with $t<v $ and $u\nleq v$.\vspace{-4pt}\end{enumerate} 
Maximal elements are called  \textit{leaves}.
Elements which are neither roots nor leaves are \textit{internal nodes}. An \textit{edge} is an inequality $(u<v)$ such that there does not exist a  $t$ with $u<t<v$. 
Given $t\in T$, write $i(t)$ for the set of \textit{incoming edges} $(t<v)$. A \textit{planar structure} on $T$ is given by a total order on each of the sets $i(t)$ of incoming edges.
The tree $T$ is \textit{binary} if every internal node has two incoming edges.
Given a set $S$, an \textit{$S$-labelling} on $T$ is a bijection $\iota$ from $S$ to the set of leaves of $T$.

A \textit{weighting} on  $T$ consists of an assignment of nonnegative numbers to all edges of $T$ such that the ``distance'' from the root to any leaf is $1$. Write $w(T)\subset [0,1]^{\{{edges\ of\ T}\}}$ for the space of all weightings and $\partial w(T)$ for the subspace of weightings which have at least one vanishing weight. 
\end{definition} \vspace{3pt}

Weighted $S$-labelled trees can be used to model the space $\Sigma |\Pi_S|^\diamond$. Starting with the simplicial model introduced above, we observe that points  in $\Sigma |\Pi_S|^\diamond$  are either equal to the basepoint or can be represented by pairs of the form  $\left([\hat{0} = x_0\leq x_1\leq \ldots\leq x_{k-1}\leq x_k=\hat{1}],t_0+\ldots+t_k = 1\right)$.\vspace{-5pt}
\begin{itemize}[leftmargin=26pt] 
\item Given a chain of partitions $\sigma = [\hat{0} = x_0\leq x_1\leq \ldots\leq x_{k-1}\leq x_k=\hat{1}]$ of $S$,  we let $T_\sigma$ be the set of equivalence classes of all partitions $x_i$  that are not classes in the preceding partition $x_{i-1}$ (here $x_{-1}$ has no classes by convention), partially ordered under refinement \textit{from coarsest to finest}, together with an {additional} minimal element $r$, the root. This poset is a tree and it carries a canonical $S$-labelling.
\item Given $(t_0,\ldots,t_k)\in [0,1]^k$ with $t_0+\ldots+t_k = 1$, we give $T_\sigma$ the unique weighting for which the distance from the root  $r$ to each class in $x_i$ is precisely $t_i+\ldots+t_k$.\vspace{-10pt}
\begin{center}
\includegraphics[width=1.1\textwidth]{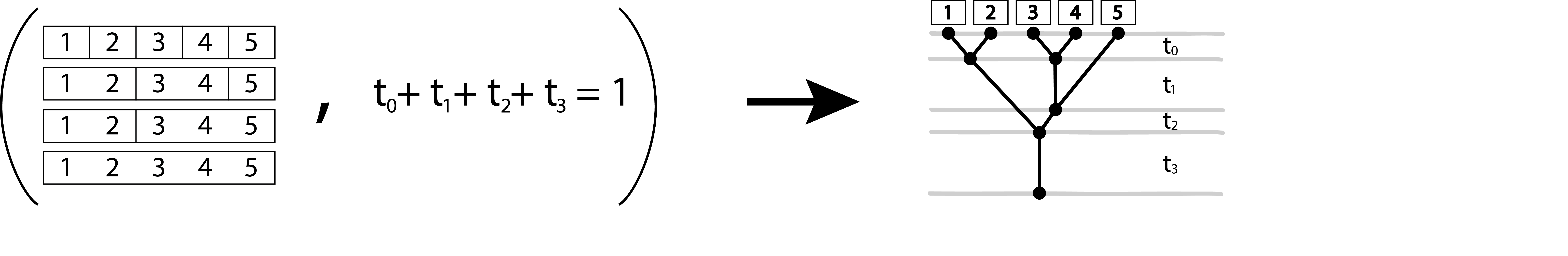}\vspace{-5pt}
\end{center}
\end{itemize}
Clearly, different weighted $S$-labelled trees can represent the same point in the space $\Sigma |\Pi_S|^\diamond$. In particular, any tree with weight $0$ on a leaf edge or the root edge represents the basepoint.

Given a permutation $h\in \Sigma_S$ and an $S$-labelled weighted tree $T$, we define a new $S$-labelled weighted tree $h\cdot T$ by precomposing the labelling function from $S$ to the set of leaves with $h^{-1}$; in other words, we replace each  label $s\in S$ by $h(s)$. This gives a concrete description of the $\Sigma_S$-action on $\Sigma|\Pi_S|^\diamond$ induced by the $\Sigma_S$-action on the poset of partitions $\Pi_S$.\vspace{5pt}
\label{expl}

In order to define    \textit{spectral Lie algebras}, we  need one further ingredient from \cite{ching2005bar}:
\begin{definition} 
Assume that we are given a bijection of finite sets $\phi\ \colon  \coprod_{i\in D} S_i \xrightarrow{\  \  \cong\ \  }  S$.
Let $T$ be an $S$-labelled weighted  tree  representing a point in $\Sigma |\Pi_S|^\diamond$. 

The \textit{ungrafting map} $ u_\phi:  \Sigma |\Pi_S|^\diamond  \ \longrightarrow \  \Sigma |\Pi_{D}|^\diamond  \wedge \bigwedge_{S_i \in D} \Sigma |\Pi_{S_i}|^\diamond $ has the following two properties:
\begin{enumerate}[leftmargin=26pt]
\item If we can find ``partitioning edges'' $\{ (v_i^-<v_i)\}_{i\in   D}$ such that for each $i\in D$, the leaf labels  above $v_i$ are given by $\phi(S_i)$, then we create a  tree $T'$ by removing all vertices and edges from $T$ which lie above the upper vertex $v_i$ of \textit{any} partitioning edge. Weight $T'$ by giving  ``non-partitioning'' edges the same weights as in $T$ and then scaling the weights of the partitioning edges $\{   (v_i^-<v_i)\}_{ i\in   D}$ appropriately. For each $i\in D$, we label the leaf $v_i$ in $T'$  by $i\in D$.\\
Furthermore, for each $i\in   D$, we create an $S_i$-labelled weighted tree $T_i$ from $T$ by only taking those vertices and edges in $T$ which lie above the lower vertex $v_i^-$ of the $i^{th}$ partitioning edge $(v_i^-<v_i)$ and then simultaneously scaling the weights inherited from $T$ appropriately.\\
{The ungrafting map $u_\phi$  sends $T$ to the point represented by $\left(T',\{T_i \}_{i \in  D}\right)$.}\vspace{5pt}

\item If, on the other hand, such partitioning edges \textit{cannot} be found, then the ungrafting map $u_\phi$ sends an $S$-labelled weighted tree $T$  to the basepoint.
\end{enumerate}
\end{definition}

\begin{definition}\label{spectralliedef} 
Given a nonempty finite set $S$, set $\mathbf{Lie}_S := \Map_{\mathbf{Sp}}(S^1 , (S^1)^{\wedge S}) \mywedge{} \DD(\Sigma |\Pi_S|^\diamond) $. For $S=\emptyset$,   define  $\mathbf{Lie}_{\emptyset} = 0$. This assignment gives  a \textit{symmetric sequence}, i.e.  a functor   from the category $\Fin^{\cong}$ of finite sets and bijections to $\mathbf{Sp}$: given a bijection $f:S_1 \xrightarrow{\cong} S_2$,  define \mbox{$\mathbf{Lie}_{S_1} \xrightarrow{\mathbf{Lie}_{f} } \mathbf{Lie}_{S_2}$} by smashing  $(S^1)^{\wedge f}: (S^1)^{\wedge S_1}  \rightarrow (S^1)^{\wedge S_2}$ with $\DD\left(\Sigma |\Pi_{f^{-1}}|^{\diamond}: \Sigma |\Pi_{S_2}|^{\diamond} \rightarrow  \Sigma |\Pi_{S_1}|^{\diamond}\right)$.

We endow $\mathbf{Lie}$  with the structure of an operad as follows: given  a bijection $\phi\colon \coprod_{i\in D} S_i \stackrel{\cong}{\rightarrow} S$,   define $\mu: \mathbf{Lie}_{D} \wedge \bigwedge_{i\in D} \mathbf{Lie}_{S_i}   \longrightarrow \mathbf{Lie}_S   $ as the composite
\begin{diagram}\bigg(\Map_{\mathbf{Sp}}(S^1 , (S^1)^{\wedge D}) \wedge \DD(\Sigma |\Pi_{D}|^\diamond)  \bigg) \ \   \wedge \ \ \bigwedge_{i\in D}  \bigg(  \Map_{\mathbf{Sp}}(S^1 , (S^1)^{\wedge S_i })\wedge  \DD( \Sigma |\Pi_{S_i}|^\diamond)\Bigg) \\
\dTo \\ 
 \\ \Map_{\mathbf{Sp}}(S^1 , (S^1)^{\wedge D}) \wedge 
\bigwedge_{i\in D}    \Map_{\mathbf{Sp}}(S^1 , (S^1)^{\wedge S_i }) \ \ \wedge \ \ 
\DD(\Sigma |\Pi_{D}|^\diamond)  \wedge \bigwedge_{i\in D}  \DD( \Sigma |\Pi_{S_i}|^\diamond) \\
 \dTo \\ 
 \\
\Map_{\mathbf{Sp}}(S^1 , (S^1)^{\wedge S}) \wedge \DD(\Sigma|\Pi_S|^\diamond)    
\end{diagram}
The top map is simply given by reordering factors. The left component of the lower map is given in an evident manner by ``plugging in''. The right component of the lower map is obtained by applying $\DD(  - )$ to the \textit{``ungrafting map''}\vspace{0pt}
$u_\phi:  \Sigma |\Pi_S|^\diamond  \ \longrightarrow \  \Sigma |\Pi_{D}|^\diamond  \wedge \bigwedge_{i \in D} \Sigma |\Pi_{S_i}|^\diamond $  defined above. \vspace{3pt}

The \textit{spectral Lie operad} $\underline{\mathbf{Lie}}$ is a cofibrant replacement of $\mathbf{Lie}$ in the model category of constant-free operads in $\mathbf{Sp}$ (\cite{batanin2017homotopy}, cf.\ also Proposition 3.6.(2) in \cite{caviglia2017model}). 

 In particular, there is a map  $\underline{\mathbf{Lie}} \rightarrow \mathbf{Lie}$  of operads in $\mathbf{Sp}$ which is a trivial fibration in the functor category  $\Fun(\Fin^{\cong}, \mathbf{Sp})$.
The spectral Lie operad $\underline{\mathbf{Lie}}$ is \textit{$\Sigma$-cofibrant} (cf.\ \cite[Proposition 4.3]{berger2003axiomatic}), which means that it is cofibrant in the projective model structure on $\Fun(\Fin^{\cong}, \mathbf{Sp})$. This implies that the associated monad  $X \mapsto \bigoplus_n \ {\underline{\mathbf{Lie}}}_n \mywedge{\Sigma_n}X ^{\wedge n}$ on $\mathbf{Sp}$  preserves weak equivalences.

We let $\Alg_{\Lierep}(\Sp)$ be the $\infty$-category of algebras for the induced monad  $\Free_{\Lierep}$ on the underlying $\infty$-category $\Sp$ of $\mathbf{Sp}$. Its objects  are called \textit{spectral Lie algebras}.
The induced monad    on the homotopy category  $\hSp$ is simply given by \mbox{$X \mapsto \bigoplus_n \ \mathbf{Lie}_n \mywedge{\hobased \Sigma_n}X ^{\wedge n}.$}\end{definition}
 
\subsection{Algebraic vs. Spectral Lie Operad} \label{c3po} We will  construct a $\Sigma_S$-equivariant isomorphism $\Lierep_S \cong \pi_0(\mathbf{Lie}_S) \cong  \tilde{\HH}_0(\mathbf{Lie}_S,\ZZ)\cong  \tilde{\HH}_0(\underline{\mathbf{Lie}}_{\hspace{0.3 pt}S},\ZZ)$ and show that it is compatible with the operadic structures. The existence of such an isomorphism was established abstractly via Koszul duality in Example $9.50$ of \cite{ching2005bar} -- we shall give an explicit construction. 
For the rest of this section, we will work  in the \textit{homotopy category} $\hSp$ of spectra. Here $\mathbf{Lie}_S$ and $\underline{\mathbf{Lie}}_{\hspace{0.3 pt}S}$ are canonically isomorphic.\label{subT}

Let $T$ be a binary $S$-labelled tree whose partial order is denoted by $\prec$.
Write $(\Sigma |\Pi_S|^\diamond)_T\subset \Sigma |\Pi_S|^\diamond$ for the open subspace of all points which can be represented  (cf.\ p.\pageref{expl}) by an $S$-labelled weighted tree that has  nonzero weights on all edges and is  isomorphic to $T$ (as an $S$-labelled tree).
 Specifying a point in $(\Sigma |\Pi_n|^\diamond)_T$ is equivalent to giving a weight in $w(T)$ all of whose components are positive. There is a homeomorphism \mbox{$\left( \Sigma |\Pi_S|^\diamond \right)/_{\Sigma |\Pi_S|^\diamond-(\Sigma |\Pi_S|^\diamond)_T} \cong w(T)/_{\partial w(T)}$.}\vspace{5pt}\label{standardorder}

Fix a binary planar $S$-labelled tree  $T$. We   define the \textit{standard order} $<$ (refining $\prec$) on the non-root nodes  $v_1< \ldots <  v_{2|S|-1}$   of $T$  by first listing the vertex above the root, then  the non-root nodes of the left subtree $T_L$ (ordered by recursion), and finally the non-root  nodes of the right subtree  $T_R$ (ordered by recursion).   
Consider the subspace $V\subset \RR^{2|S|-1}$ cut out by the following equations \vspace{-5pt}  $$ \hspace{-5pt} \bigg\{x_{i_1}+\ldots + x_{i_k} = 1 \ \ \mbox{whenever} \ \ (r \prec v_{i_1}), (v_{i_1} \prec v_{i_2}) , \ldots, (v_{i_{k-1}} \prec v_{i_k})   \mbox{ are edges of $T$ and } v_{i_k} \mbox{ is a leaf }   \bigg\}\vspace{-5pt}$$
The space $V$ is an intersection of transversely intersecting hyperplanes (cf.\ Lemma $3.9$ in \cite{ching2005bar}) and  hence has dimension $\dim(V) = |S|-1 $. There is a natural identification $w(T) \cong [0,1]^{2|S|-1} \cap V$.  Consider the map $ [0,1]^{|S|-1} \cong \{(t_1,\ldots,t_{2|S|-1})\in [0,1]^{2|S|-1} \ | \   t_i  = 1 \mbox{ if $v_i$ is a leaf in $T$} \}    \longrightarrow    w(T) $
 \vspace{-2pt} $$ \{t_i\}_{i=1}^{2|S|-1}  \ \ \mapsto \ \  \bigg\{t_i \cdot \prod_{ v_j \prec v_i} (1-t_j)\bigg\}_{i=1}^{2|S|-1} \vspace{-1pt}.$$
In other words, we send a tuple $(t_1,\ldots,t_{2|S|-1})$ to the unique weighting on $T$ with the property that if $v_i$ is the $i^{th}$ node (in standard order) with parent $v_i^-\prec v_i$, then the ratio $\frac{d(v_i^-,v_i)}{1-d(r,v_i^-)}$ between the distance   from $v_i^-$ to $v_i$ and the distance  from $v_i^-$ to  any leaf above $v_i$ is exactly $t_i$.

Passing to quotients by the respective boundaries gives rise to a homeomorphism $   S^{|S|-1} \cong w(T)/_{\partial w(T)}$.\vspace{5pt}
\begin{definition}\label{amap} Let $T$ be a binary planar $S$-labelled tree.
The map $a_T$ is defined as the composite \vspace{-13pt}
$$a_T: \Sigma |\Pi_S|^\diamond \longrightarrow  (\Sigma |\Pi_S|^\diamond)/_{(\Sigma |\Pi_S|^\diamond)-(\Sigma |\Pi_S|^\diamond)_T}  \cong w(T)/_{\partial w(T)} \cong S^{|S|-1}$$\vspace{-1pt}\vspace{-15pt}  

The map  $b_T:S^{|S|} \longrightarrow (S^{1})^{\wedge S} $ sends the $k^{th}$ coordinate in $S^{|S|} $ to the coordinate in $ (S^{1})^{\wedge S}$  which corresponds to the  label $s_k\in S$ of the $k^{th}$ leaf from the \mbox{left in $T$.} \vspace{3pt}
\end{definition}
\begin{minipage}{0.85\textwidth}
\begin{definition} The \textit{signature} $\sgn(T)$ of a binary planar tree $T$ is $\sgn(T) = 1$ if $T$ has no internal nodes. If $T$  is obtained by first identifying the roots of a left subtree $T_L$ and a right subtree $T_R$ and then adding a new minimal element $r$, the root, then  we set $\sgn(T) = (-1)^{ (\#\ internal  \ nodes \ of\ T_L) \cdot (\#\ {leaves} \ of \ T_R)} \cdot \sgn(T_L) \cdot  \sgn(T_R)$. \end{definition}
\end{minipage} \ \ 
\begin{minipage}{0.15\textwidth}   \includegraphics[width=0.95\textwidth]{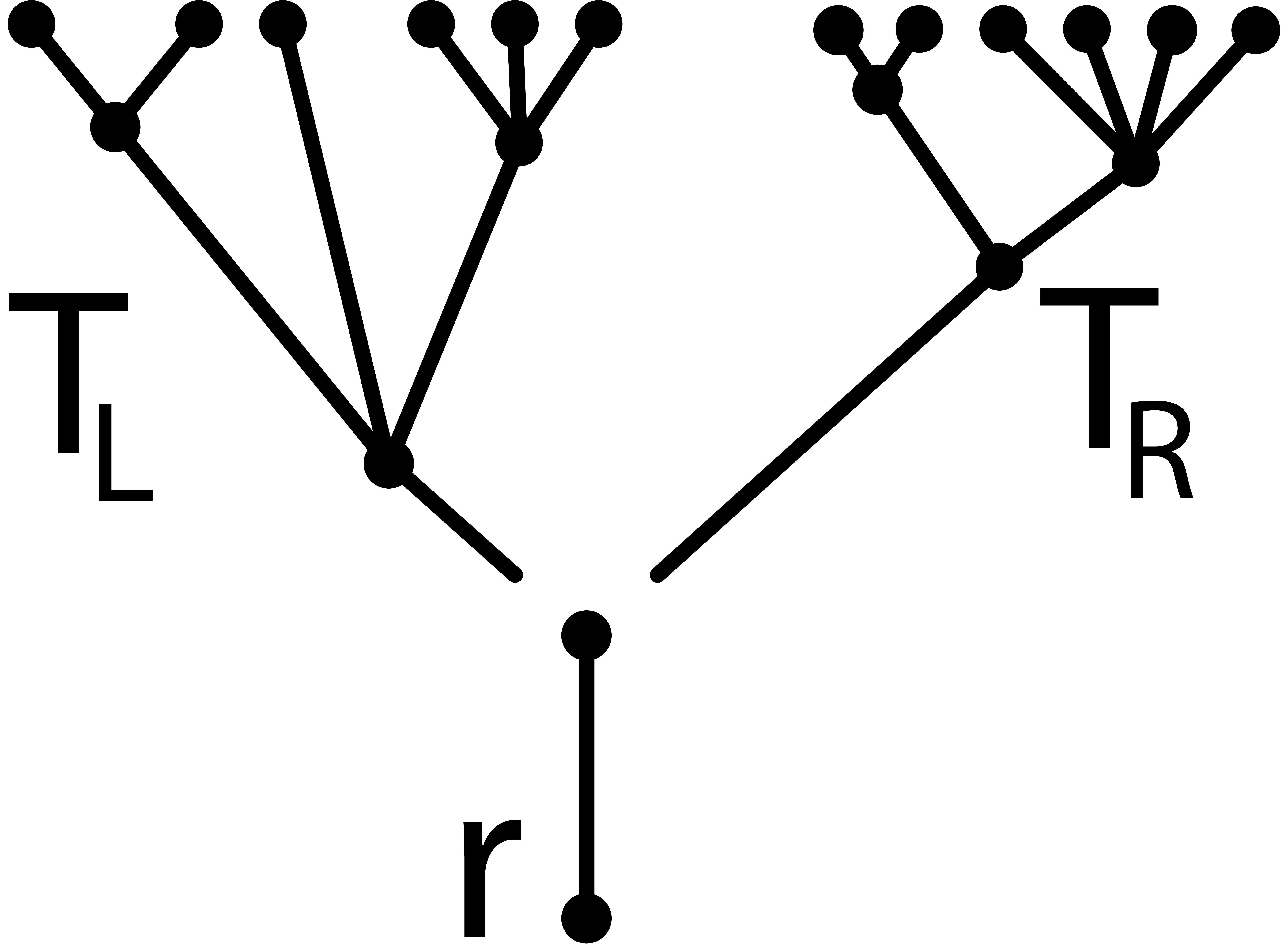}

\end{minipage}

Alternatively, $\sgn(T)$ is the sign of the shuffle permutation of the non-root nodes of $T$ required to transform the standard ordering    into the ordering which first lists all leaves  and then lists all internal nodes, where these two sets are internally still ordered as in the standard order.
\vspace{4pt}

Every multilinear  monomial $w$ in letters  $\{c_s \}_{s\in S}$ canonically gives rise to a binary planar $S$-labelled tree $T_w$: if $w = c_s$ is a letter, then $T_w$ has a root and a leaf labelled by $s$. If $w=[u,v]$, then $T_w$ is obtained by first identifying the roots of $T_u$ and  $T_v$, then adding a new minimal element, and finally placing $T_u$ to the left of $T_v$. We can now attach homology classes to  words:\vspace{2pt}
\begin{definition}\label{gammamaps}
\mbox{Given a multilinear monomial $w$ in letters $\{c_s \}_{s\in S}$,  define a map $\gamma_w $ in $\hSp$ as} 
$$ S^{0} \xrightarrow{\sgn(T_w)} \ S^0 {\simeq} \    \Map_{\mathbf{Sp}}(S^1,S^{|S|}) \wedge \DD(  S^{|S|-1})\xrightarrow{b_{T_w}\wedge  \DD( a_{T_w}) } \Map_{\mathbf{Sp}}(S^1,(S^1)^{\wedge S})  \wedge  \DD(\Sigma  |\Pi_S|^\diamond)  = \mathbf{Lie}_S.  $$ 
We define $x_w \in {\widetilde{\HH}}_0(\mathbf{Lie}_S,\ZZ)$  by pushing forward the fundamental class in $ \widetilde{\HH}_{0}(S^0,\ZZ)$ along $\gamma_w$.
\end{definition} 
If $h\in \Sigma_S$ and $w\in \Lierep_S$ is represented by a  multilinear monomial, we observe that  $\gamma_{h\cdot w} = \mathbf{Lie}_h \circ \gamma_w$.\vspace{4pt}

\begin{proposition}\label{propc}
If $w_d\in \Lierep_{\mathbf{d}}$  and $v_1\in \Lierep_{S_1}, \ldots , v_d \in \Lierep_{S_d}$ are represented by multilinear  monomials, then we can form the element   $w_d(v_1,\ldots,v_d)\in \Lierep_S$ for $S = \coprod_{i\in \mathbf{d}} S_i$ by substitution. 
The following diagram commutes \vspace{-3pt}in $\hSp$:  
\begin{diagram}
S^0 & \rTo^{\gamma_w \wedge \gamma_{v_1} \wedge \ldots \wedge \gamma_{v_d} \ \  } & &  \mathbf{Lie}_{\mathbf{d}}\wedge  \mathbf{Lie}_{S_1}\wedge  \ldots \wedge\mathbf{Lie}_{S_d} \\
 & \rdTo_{\gamma_{w(v_1,\ldots,v_d)}}(3,2) &  & \dTo_\mu \\
  & &  & \mathbf{Lie}_{S}
\end{diagram}
\end{proposition}  \vspace{0pt}
\begin{proof} We will first  explain the following two squares:
$$\hspace{-15pt}
\begin{diagram}
\Sigma |\Pi_S|^\diamond & \rTo & & \Sigma|\Pi_{\mathbf{d}}|^\diamond \wedge \Sigma|\Pi_{S_1}|^\diamond \wedge \ldots \wedge  \Sigma|\Pi_{S_d}|^\diamond \\
\dTo^{a_{T_{w(v_1,\ldots,v_d)}}} && & \dTo^{a_{T_w} \wedge a_{T_{v_1}} \wedge \ldots \wedge  a_{T_{v_d}} } \\ 
S^{|S|-1}& \rTo  & & S^{|\mathbf{d}|-1} \wedge S^{|S_1|-1}\wedge\ldots \wedge  S^{|S_d|-1}
\end{diagram} \ \ \ \ \ \ \ \ \ \  \ \ \  \begin{diagram}
S^{|S_1|}\wedge \ldots \wedge S^{|S_d|} & \rTo & S^{|S|} \\
\dTo_{b_{T_{v_1}}\wedge \ldots \wedge b_{T_{v_d}}} & & \dTo_{b_{T_{w(v_1,\ldots,v_d)}}} \\
(S^1)^{\wedge S_1 }\wedge \ldots \wedge (S^1)^{\wedge S_d} & \rTo & (S^1)^{\wedge S} 
\end{diagram}$$
We begin with the left square. Its top horizontal map is given by tree ungrafting (cf.\ Definition \ref{spectralliedef}). The lower horizontal map  is induced by the shuffle permutation of the internal nodes of $T_{w(v_1,\ldots,v_d)}$ which changes the standard order to the order that firsts lists the $|d|-1$ internal nodes of $T_w\subset T_{w(v_1,\ldots,v_d)}$ (in standard order), then the $|S_1|-1$ internal nodes of  $T_{v_1} \subset T_{w(v_1,\ldots,v_d)}$ (in standard order), and so on. Since ungrafting preserves the ratios appearing in the definition of the maps $a_T$ above, the left square commutes

We move on to the right square above. Its lower map is canonical. The top horizontal map rearranges the various $S^{|S_i|}$'s so that the corresponding trees $T_{v_i} $ appear from left to right in  $ T_{w(v_1,\ldots,v_d)}$.  
Applying $\DD( -)$ to the first square and $S^{-1} \wedge (-)$ to the second,  {smashing both, and rearranging the terms gives  a commutative square}\vspace{3pt}
\begin{diagram}
  \Map_{\mathbf{Sp}}(S^1,S^{|\mathbf{d}|}) \wedge S^{1-|\mathbf{d}|}  \wedge \Map_{\mathbf{Sp}}(S^1,S^{|S_1|}) \wedge S^{1-|S_1|} \wedge  \ldots   & & \rTo && \Map_{\mathbf{Sp}}(S^1,S^{|S| })  \wedge S^{1-|S|} \\
\dTo   & &&& \dTo \\ 
 \Map_{\mathbf{Sp}}(S^1,S^{\wedge \mathbf{d}})\wedge \DD(\Sigma |\Pi_{\mathbf{d}}|^\diamond) \wedge  \Map_{\mathbf{Sp}}(S^1,(S^1)^{\wedge S_1 }) \wedge \DD(\Sigma |\Pi_{S_1}|^\diamond) \wedge  \ldots  & & \rTo && \Map_{\mathbf{Sp}}(S^1, (S^1)^{\wedge S})\wedge \DD(\Sigma  |\Pi_S|^\diamond)   
\end{diagram}\vspace{3pt}

The ``down-right'' path of this diagram is   $ \mu \circ( \gamma_w \wedge \gamma_{v_1} \wedge \ldots \wedge \gamma_{v_d}) \circ (\sgn(T_w)\sgn(T_{v_1})\ldots \sgn(T_{v_d}))$.

Assume now that the letters $c_1,\ldots,c_d$ appear in the order $c_{\sigma(1)},\ldots,c_{\sigma(d)}$ in the word $w$ .
The top path $S^0\rightarrow S^0$ of this diagram is then homotopy equivalent to the map \vspace{-5pt}
$$ S^{-1} \wedge S^{1-|\mathbf{d}|} \wedge (S^{|S_{\sigma(1)}|} \wedge S^{1-|S_{\sigma(1)}|}) \wedge \ldots \wedge (S^{|S_{\sigma(d)}|} \wedge S^{1-|S_{\sigma(d)}|})\ \ \  \longrightarrow \ \ \   S^{-1} \wedge S^{|S|} \wedge S^{1-|S|} \vspace{-5pt}$$ 
which shuffles negative spheres other than the leftmost $S^{-1}$ to the right (while preserving their internal order). The sign of this map is equal to the sign of a corresponding permutation $\tau$ of the nodes of $T_{w(v_1,\ldots,v_d)}$: internal nodes corresponds to a copies of $S^{-1}$ and leaves \mbox{correspond to copies of $S^{1}$.}

We can now observe that $\sgn(T_w)\sgn(T_{v_1})\ldots \sgn(T_{v_d}) \sgn(\tau) = \sgn(T_{w(v_1,\ldots,v_n)})$ since both are equal to the sign of the permutation of the non-root nodes of $T_{w(v_1,\ldots,v_n)}$ which moves all leaves to the left while preserving the internal standard order of leaves and internal nodes. \vspace{-5pt}
\end{proof}

\begin{corollary} \label{thisco}
If $v_1,v_2,v_3$ are monomials in pairwise disjoint sets of letters $S_1,S_2,S_3$, then \vspace{-5pt}$$\gamma_{[v_1,v_2]} = -\gamma_{[v_2,v_1]} \ \ \ \ \mbox{ and } \ \ \ \ \gamma_{[v_1,[v_2,v_3]]} + \gamma_{[v_2,[v_3,v_1]]} + \gamma_{[v_3,[v_1,v_2]]} = 0 \ .  \vspace{-5pt}$$
\end{corollary}
\begin{proof}
For the first claim,  we use Proposition \ref{propc} for $w=[c_1,c_2]$  and $w=[c_2,c_1]$ to see that  $\gamma_{[v_1,v_2]} $ and $\gamma_{[v_2,v_1]} $ agree up to precomposition by the  map $\mathbf{Lie}_{\mathbf{2}} \rightarrow \mathbf{Lie}_\mathbf{2}$ induced by the nontrivial permutation in $\Sigma_2$. This map is homotopic to multiplication by $(-1)$.

For the second claim, we use  Proposition \ref{propc} for the words $w=[c_1,[c_2,c_3]]$,  $w=[c_2,[c_3,c_1]]$, and $w=[c_3,[c_1,c_2]]$ to observe that it suffices to show that    the map 
$\Map_{\mathbf{Sp}}(S^1,S^3) \wedge \DD(\Sigma |\Pi_3|) \xrightarrow{1+\sigma+\sigma^2}  \Map_{\mathbf{Sp}}(S^1,S^3) \wedge  \DD(\Sigma |\Pi_3|)  
$ is zero, where $\sigma$ denotes the action map induced by the cycle $(1\ 2 \ 3) \in \Sigma_3$. Since this cycle has even sign, the induced maps $1,\sigma,\sigma^2: S^3\rightarrow S^3$ are in fact all homotopic, and it therefore suffices to prove that $\DD(\Sigma |\Pi_3|)   \xrightarrow{1+\sigma+\sigma^2}     \DD(\Sigma |\Pi_3|) 
$
is null. This fact was proven by the second-named author and Antol\'{i}n-Camarena and is written in Proposition $5.2.$ of \cite{camarena2016mod}\vspace{-5pt}
\end{proof} 
Proposition \ref{propc} and Corollary \ref{thisco} imply that the assignments $w\mapsto  \gamma_w \mapsto x_w$ from Definition \ref{gammamaps} descend  to  $\Sigma_n$-equivariant homomorphisms $\Lierep_n \rightarrow \pi_0(\mathbf{Lie}) \rightarrow \widetilde{\HH}_0(\mathbf{Lie},\ZZ)$ -- our later branching rule implies that they are in fact isomorphisms. 
Using Proposition \ref{propc}, we deduce:
\begin{corollary}\label{commutes}The maps \mbox{$w\mapsto \gamma_w \mapsto  x_w$ give  isomorphisms of operads $\Lierep \cong  \pi_0(\mathbf{Lie}) \cong {\widetilde{\HH}}_0(\mathbf{Lie}_{},\ZZ)$.}\vspace{2pt}
\end{corollary} 

Hence, the homology and homotopy groups of  spectral Lie algebras form \textit{graded Lie algebras}, i.e. \mbox{$\ZZ$-graded} abelian groups $\mathfrak{g}_\ast$ with a bilinear product $[-,-]:\mathfrak{g}_i \times \mathfrak{g}_j \rightarrow \mathfrak{g}_{i+j}$ such that for all homogeneous \mbox{ $u,v,w$,} the  graded antisymmetry relation $[u,v] = -(-1)^{\deg(u)\deg(v)}[v,u]$ and the graded Jacobi identity
$(-1)^{\deg(u)\deg(w)} [u,[v,w]] + (-1)^{\deg(w)\deg(v)}[w,[u,v]] +(-1)^{\deg(v)\deg(u)} [v,[w,u]] = 0$ hold true.
Equivalently, graded Lie algebras  are algebras over the algebraic Lie operad $\{\Lierep_S\}$ in the category of $\ZZ$-graded abelian groups (endowed with the Koszul symmetric monoidal structure).

An easy variation of our arguments establishes the corresponding statements for the shifted spectral Lie operad (whose $n^{th}$ term is given by $\DD(\Sigma |\Pi_n|^\diamond)$) and shifted Lie algebras.\vspace{-5pt}

\subsection{Collapse and Ungrafting} \label{LAPC1}
Let $x$ be a proper nontrivial partition of the finite set $S$ given by $S=\coprod_{S_i\in \mathbf{d}_x}S_i$, where $\mathbf{d}_x$ denotes the set of classes of $x$. We observe evident {isomorphisms of posets $(\mathcal{P}_S)_{[\hat{0},x]} \cong \prod_{S_i \in \mathbf{d}_x} \mathcal{P}_{S_i}$ and $(\mathcal{P}_S)_{[x,\hat{1}]} \cong \mathcal{P}_{\mathbf{d}_x}$.} We consider the following two maps:

\begin{enumerate}[leftmargin=26pt]
\item \hspace{-2pt} The map $\nu_x: \Sigma |\Pi_S|^\diamond \rightarrow\Sigma |(\Pi_S)_{(\hat{0},x )}|^\diamond \wedge \Sigma |(\Pi_S)_{(x,\hat{1})}|^\diamond \rightarrow   \Sigma |\Pi_{ \mathbf{d}_x }|^\diamond \wedge  ( \bigwedge_{S_i\in \mathbf{d}_x} \Sigma |\Pi_{S_i}|_{}^\diamond  )   $ uses the   collapse map for $x$ (cf.\ Section \ref{clong}) and the product map for $(\Pi_S)_{(\hat{0},x)}$  \mbox{(cf.\ Proposition \ref{clock}).}
\item The ungrafting map $ u_x: \Sigma |\Pi_S|^\diamond  \rightarrow   \Sigma |\Pi_{ \mathbf{d}_x }|^\diamond \wedge  ( \bigwedge_{S_i\in \mathbf{d}_x} \Sigma |\Pi_{S_i}|_{}^\diamond  )$ corresponding to the bijection $\phi:\coprod_{S_i\in \mathbf{d}_x}S_i = S$ (cf.\ Definition \ref{spectralliedef}).\vspace{3pt}
\end{enumerate}
\begin{proposition}\label{propbb}
There is a homotopy $   {K}_x: [0,1]\times \Sigma |\Pi_S|^\diamond \longrightarrow  \Sigma |\Pi_{ \mathbf{d}_x }|^\diamond \wedge  ( \bigwedge_{S_i\in \mathbf{d}_x} \Sigma |\Pi_{S_i}|_{}^\diamond  ) $ from the collapse map  $ \nu_x$ to the ungrafting map $ u_x  $  for $x$. \vspace{2pt}

If $h \in \Sigma_S$, then $ {K}_{h \cdot x}(t,h \cdot T) = h\cdot  {K}_x(t,T) \in \Sigma |\Pi_{ \mathbf{d}_{h\cdot x} }|^\diamond \wedge  ( \bigwedge_{h(S_i)\in \mathbf{d}_{h\cdot x}} \Sigma |\Pi_{h(S_i)}|_{}^\diamond  ) $ for all  $t, T$. We have used the natural  action map $ 
\Sigma |\Pi_{ \mathbf{d}_{ x} }|^\diamond \wedge  ( \bigwedge_{S_i\in \mathbf{d}_{ x}} \Sigma |\Pi_{S_i}|_{}^\diamond  ) \xrightarrow{h\cdot} \Sigma |\Pi_{ \mathbf{d}_{h\cdot x} }|^\diamond \wedge  ( \bigwedge_{h(S_i)\in \mathbf{d}_{h\cdot x}} \Sigma |\Pi_{h(S_i)}|_{}^\diamond  )$. 
\end{proposition} \vspace{-2pt}
\begin{proof}
We begin by describing the map $\nu_x$ explicitly. Let $T$ be a weighted $S$-labelled tree in $\Sigma |\Pi_S|^\diamond$. \vspace{-10pt}

\begin{minipage}{0.65 \textwidth} 
Assume  that  cutting $T$ along edges $\{(v_i^- \prec v_i)\}_{S_i\in \mathbf{d}_x}$ partitions $S$ as $\coprod_{S_i\in \mathbf{d}_x}S_i$ and that for $m=\max_i d(r,v_i^-)$  and $M=\min_i d(r,v_i)$, we have $m<M$; this is precisely what it means for $T$ to lie in the interior of a  simplex corresponding to a chain of \mbox{partitions $[ \ldots<x_{i-1}<x<x_{i+1}<\ldots ]$ containing $x$.}
 
 For each $S_i\in \mathbf{d}_x$,  we form an $S_i$-labelled weighted tree $T_i$ by 
\mbox{restricting to the corresponding component \mbox{above the ``$m$-line'',} adding a root   if necessary, and
 }

\end{minipage}
\begin{minipage}{0.35 \textwidth} \hspace{8pt} \ \vspace{ 15pt} \hspace{-10pt}
\includegraphics[width=0.95 \textwidth]{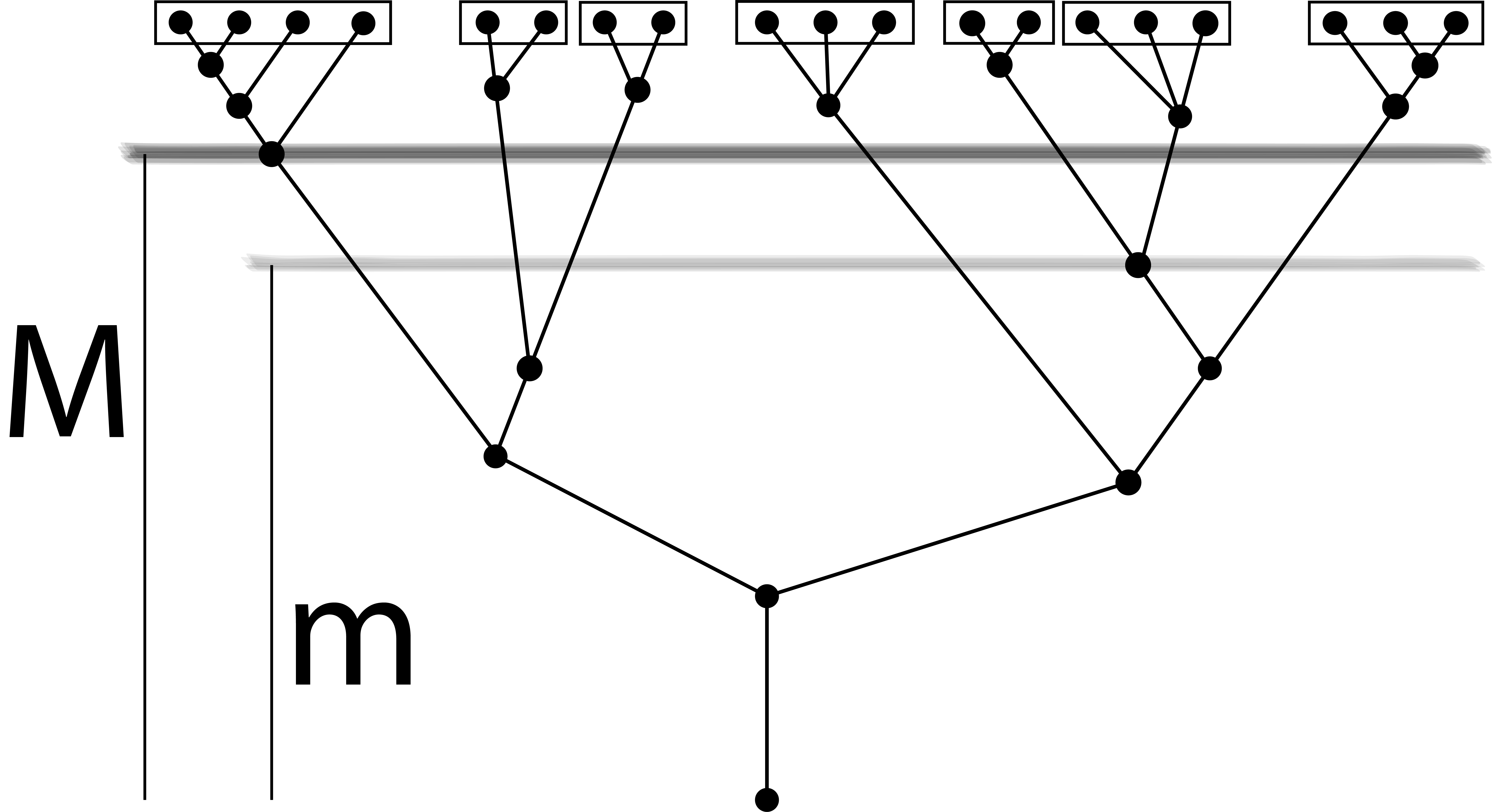}
\end{minipage} \vspace{-4pt}
\vspace{0pt}
 
then rescaling the weights. Form a weighted $\mathbf{d}_x$-labelled tree $T'$ by first restricting to the component below the ``$M$-line'', then adding leaves where necessary, and finally rescaling the weights. The collapse map $\nu_x$ from Section \ref{clong} then sends $T$ to $(T',\{ T_i \}_{S_i \in \mathbf{d}_x})$. 

If $m\geq M$ or there do not even exist such edges $(v_i^- \prec v_i)$,    then    $\nu_x$ sends $T$ to the basepoint.  \vspace{3pt}

Let $(\Sigma |\Pi_S|^\diamond)_{x}\subset \Sigma |\Pi_S|^\diamond $ be the open subspace of all weighted $S$-labelled trees for which cutting along some edges $\{(v_i^- \prec v_i)\}_{S_i\in \mathbf{d}_x}$ with  nonzero weights $\{e_i\}_{S_i\in \mathbf{d}_x}$ partitions  the set $S$ as \mbox{$S=\coprod_{S_i\in \mathbf{d}_x}S_i$}. The ungrafting map $u_x$  and the collapse  map  $\nu_x$  naturally factor through  maps $\ovB{u}_x, \ovB{\nu}_x: ( \Sigma |\Pi_S|^\diamond )/_{ \Sigma |\Pi_S|^\diamond - (\Sigma |\Pi_S|^\diamond)_{x}} \longrightarrow      \Sigma |\Pi_{\mathbf{d}_x}|^\diamond \wedge ( \bigwedge_{S_i\in \mathbf{d}_x} \Sigma |\Pi_{S_i}|_{}^\diamond  )  $. 
We shall now produce an explicit  homotopy from   $\ovB{\nu}_x$ to $\ovB{u}_x$ (and hence also from $\nu_x$ to $u_x$). 
 \vspace{-1pt}

Given an $S$-labelled weighted tree  $T \in (\Sigma |\Pi_S|^\diamond)_{x }$, we again write $e_i$ for the length of the edge $(v_i^- \prec v_i)$ and let  $f_i<1$ be the distance from the root $r$ of $T$  to $v_i^-$.  
 We set $\displaystyle a = \min_i \left( \frac{e_i}{1-f_i} \right)$ and $b =\displaystyle \max_i \left(  f_i \right)$. Moreover, we define $M_1 = \frac{a}{1+ab-b}$ and $m_1 = b M_1$; these   will
be the new heights of the $M$- and $m$-lines. \vspace{1pt} Note that $M_1\leq 1$.

Given some $t\in [0,1]$, we define a new $S$-labelled weighted tree $\ovA{H}_x(t,T)$ by taking the tree $T$ and:\vspace{-2pt}
\begin{itemize}[leftmargin=26pt]
\item Scaling the weights  of edges lying above the node $v_i$ by a factor of $(1-t) + t\frac{1-m_1}{1-f_i}$.
\item Giving the edge $(v_i^- \prec v_i)$ below $v_i$ the weight  
 $e_i +  t  \cdot \left(    m_1 - M_1f_i  +   e_i \frac{f_i-m_1}{1-f_i} \right) $. 
\item Scaling all edges  below some $v_i^-$ by a factor of $(1-t) + t M_1 \leq 1$.\vspace{1pt}
\end{itemize}
An easy check shows that summing up the numbers  in $\ovA{H}_x(t,T)$ along an increasing path from the root to any leaf \mbox{gives $1$.} To see that we have  
defined a valid weighting for all $t$, we need to show that all weights are nonnegative. This is clear for the first and last case in the definition of $\ovA{H}_x(t,T)$. For the middle case, it suffices to prove the claim for $t=1$. 
We observe that $1-(\frac{1-m_1}{1-f_i})(1-e_i-f_i) \geq M_1$ (with equality attained for at least one $i$) and that $M_1\cdot f_i  \leq m_1 $ (with equality attained for at least one $i$).
Hence the weight of $(v_i^- \prec v_i)$ in $\ovA{H}_x(1,T)$ is at least $M_1-m_1\geq 0$.

The assignment $(t,T)\mapsto \ovA{H}_x(t,T)$ is readily checked to be a well-defined and continuous function $[0,1]\times (\Sigma |\Pi_S|^\diamond)_{x } \rightarrow (\Sigma |\Pi_S|^\diamond)_{x }$.  Define  $\overline{K}_x: [0,1] \times  (\Sigma |\Pi_S|^\diamond)_{x}  \longrightarrow  ( \bigwedge_{S_i\in \mathbf{d}_x} \Sigma |\Pi_{S_i}|_{}^\diamond  )   \wedge \Sigma |\Pi_{\mathbf{d}_x}|^\diamond$ by 
postcomposing $\ovA{H}_x$ with  $(\Sigma |\Pi_S|^\diamond)_{x} \hookrightarrow ( \Sigma |\Pi_S|^\diamond )/_{ \Sigma |\Pi_S|^\diamond - (\Sigma |\Pi_S|^\diamond)_{x}} \xrightarrow{\overline{\nu}_x}     ( \bigwedge_{S_i\in \mathbf{d}_x }\Sigma |\Pi_{S_i}|_{}^\diamond  )   \wedge \Sigma |\Pi_{\mathbf{d}_x}|^\diamond$.

The function $\ovA{K}_x$ can  be continuously extended to a function  $$\ovA{K}_x: [0,1] \times ( \Sigma |\Pi_S|^\diamond )/_{ \Sigma |\Pi_S|^\diamond - (\Sigma |\Pi_S|^\diamond)_{x}} \longrightarrow  ( \bigwedge_{S_i\in \mathbf{d}_x }\Sigma |\Pi_{S_i}|_{}^\diamond  )   \wedge \Sigma |\Pi_{\mathbf{d}_x}|^\diamond. \vspace{-2pt} $$  
 
Indeed, let  \vspace{1pt}$(t^n,T^n)\in  [0,1] \times  (\Sigma |\Pi_S|^\diamond)_{x} \subset  [0,1] \times ( \Sigma |\Pi_S|^\diamond )/_{ \Sigma |\Pi_S|^\diamond - (\Sigma |\Pi_S|^\diamond)_{x}} $ \mbox{be a sequence such that} $t^n \rightarrow t$ and with $T^n\in (\Sigma |\Pi_S|^\diamond)_{x} $ \mbox{converging to the \vspace{1pt} basepoint in $( \Sigma |\Pi_S|^\diamond )/_{ \Sigma |\Pi_S|^\diamond - (\Sigma |\Pi_S|^\diamond)_{x}}$. \hspace{-1pt}We will}  show that   $\ovA{K}_x(t^n,T^n) = \ovB{\nu}_x(\ovA{H}_x(t^n,T^n))$  \vspace{1pt}converges to the basepoint in \mbox{$ (\bigwedge_{S_i\in \mathbf{d}_x }\Sigma |\Pi_{S_i}|_{}^\diamond  )   \wedge \Sigma |\Pi_{\mathbf{d}_x}|^\diamond$.}

\mbox{Using   notation from above, we let $e_i^n, f_i^n, a^n, b^n, m^n_1, M^n_1, $ be the various numbers attached to $T^n$. }\vspace{5pt}
 
In $\ovA{H}_x({t_n},T_n)$, the distance $M^n$ from the root $r$ to the lowest vertex $v_i$ is given by \vspace{-5pt}
$$M^n:=\min_i  \bigg( (1-t^n)(e_i^n+f_i^n) + t^n(m^n_1+\frac{e_i^n}{1-f_i^n}  (1-m_1^n))\bigg) \leq (1-t^n) +t^nM^n_1 \vspace{-5pt} $$
The distance from the root $r$  to the highest  $v_i^-$ in $\ovA{H}_x(t_n,T_n)$ is simply
$ m^n_{}:= \left((1-t^n)+t^n M_1^n\right) b^n $.

After possibly decomposing the sequence into subsequences, we may assume without restriction that one of the following cases holds true:
\begin{enumerate}[leftmargin=26pt] 
\item $m^n \geq M^n$ for all $n$.
\item $m^n <  M^n$ for all $n$ and the sequence $f_i^n$ tends to $1$ for some $i$.
\item \mbox{$m^n <  M^n$ for all $n$, there is a $\kappa>0$ with  $f_j^n<1-\kappa$ for all $n, j$, and  $e_i^n\longrightarrow 0$       for some $i$.} 
\item $m^n <  M^n$ for all $n$, there is a $\kappa >0$ with  $f_j^n<1-\kappa$ and $\kappa < e_j^n$  for all $n, j$, and the length $r^n$ of the root edge of $T^n$  tends to zero.
\item $m^n <  M^n$ for all $n$, there is a $\kappa >0$ with  $f_j^n<1-\kappa$ and $\kappa < e_j^n$  for all $n, j$, and the length of the leaf edge of $T^n$ labelled by $s\in S_i$ tends to zero.\vspace{-2pt}
\end{enumerate}
We write $\ovA{K}_x(t^n,T^n) = (T^n{}',\{T_i^n \}_{S_i\in \mathbf{d}_x})$.\\
For case $(1)$, we observe that the tree $\ovA{K}_x(t^n,T^n)$ is constantly equal to the basepoint.

For case $(2)$, we note that   the $i^{th}$ leaf-edge of $T^n{}'$   has length 
$1-\frac{((1-t^n)+t^nM_1^n) f_i^n}{M^n} \ \  \geq  \ \ 1-f_i^n \ \ \longrightarrow \ \ 0$

For case $(3)$, we check that  the length of the root edge of the $i^{th}$ tree $T^n_i$ is bounded above by  $\frac{e_i^n}{1-m^n}$. This number converges to $0$ since  $e_i^n \longrightarrow  0$ and $1-m^n \geq 1- b^n > \kappa $.

\mbox{For case $(4)$, we compute   $M^n >\kappa$. The length of the root edge of $T^n{}'$ thus tends to zero.}

For case $(5)$, we use  again that $1-m^n \geq \kappa$ to see that the length of the leaf edge in $T^n_i$ labelled by $s$ tends to zero. This concludes the proof of the existence of the continuous extension $\ovA{K}_x$.\vspace{3pt}

Define $K_x$ as  \mbox{ $[0,1] \times   \Sigma |\Pi_S|^\diamond  \rightarrow  [0,1] \times ( \Sigma |\Pi_S|^\diamond )/_{ \Sigma |\Pi_S|^\diamond - (\Sigma |\Pi_S|^\diamond)_{x}} \xrightarrow{\ovB{K}_x} ( \bigwedge_{S_i\in \mathbf{d}_x }\Sigma |\Pi_{S_i}|_{}^\diamond  )   \wedge \Sigma |\Pi_{\mathbf{d}_x}|^\diamond $}.

We clearly have $ {K}_x(0,T) =  {\nu_x}(0,T) $ for all $T$. To compute $ {K}_x(1,T) =  {u}_x(T)$, we note  that for $T$ a weighted $S$-labelled tree in $(\Sigma |\Pi_S|^\diamond)_{x }$, our above observations show that  all nodes $v_i$ (in the above notation) lie above the ``$M_1$-line'' in  $\ovA{H}_x(1,T)$ (with at least one on this line). Similarly,  all nodes   $v_i^-$ lie below the ``$m_1$-line'' (with at least one having distance $m_1$ from the root). We can then read off the identification $ { K}_x(1,T)= {u}_x(T)$ from our explicit description of $\nu_x$ in the beginning of this proof.
Finally, the equivariance claim  is evident by the symmetry of the definition of $\ovA{H}_x$.
 \vspace{-5pt}  
\end{proof}
\subsection{Binary Chains} \label{binarypage}\label{LAPC2} Let $S$ be a finite set.
We will  analyse collapse maps for the following chains:
\begin{definition}\label{definition: binary}
An increasing chain of partitions $\sigma=[x_1\leq \cdots\leq x_r]$ of $S$ is   {\it binary} if for all $i$, each equivalence class in $x_i$ is the union of \textit{at most two} equivalence classes in $x_{i-1}$. For $i=0$, this means that each component of $x_1$ has at most two elements. However, note that we do \textit{not} require $x_r$ to have  only  two equivalence classes.    \vspace{-5pt}
\end{definition}
\subsubsection*{Binary Chains of Length One}
We begin our analysis with the following simple observation:
\begin{proposition}\label{veryeasy}
Let $[x\leq y]$ be a binary chain. Write $N_{[x,y]}$ for the set of classes of $y$ which are obtained by merging two classes of $x$. There is an isomorphism of posets $(\Pi_S)_{(x,y)} \cong \ovA{\Bcal}_{N_{[x,y]}}$ between the interval $(x,y)$ in $\Pi_S$ and the poset of proper nonempty subsets of $N_{[x,y]}$.\end{proposition}
By Example \ref{example: sphere}, there is a homeomorphism $|(\Pi_S)_{(x,y)}| \cong S^{N_{[x,y]}-2}$ to the doubly desuspended standard representation sphere of $\Sigma_{N_{[x,y]}}$.  Conventions for $|N_{[x,y]}|=0, 1$  are as  in \mbox{Remark \ref{smallspheres}.}

We can describe the induced homeomorphism $\Sigma |(\Pi_S)_{(x,y)}|^\diamond \xrightarrow{\cong} \Sigma  |\ovA{\Bcal}_{N_{[x,y]}}|^\diamond \xrightarrow{\cong} S^{ N_{[x,y]}}$ by sending an $S$-labelled weighted  tree $T$ in $\Sigma |(\Pi_S)_{(x,y)}|^\diamond$  to the point in $S^{ N_{[x,y]}} = I^{N_{[x,y]}}/_{\partial  I^{N_{[x,y]}}}$ whose coordinate  at $n\in N_{[x,y]}$ is   the distance from the root   to the maximal node lying under all \mbox{elements of $n$.}

\subsubsection*{Binary Chains ending in $\hat{1}$}
Let $\sigma = [\hat{0}= x_0\leq \ldots \leq x_r = \hat{1}]$ be  a  binary chain of partitions starting in $\hat{0}$ and ending in $\hat{1}$. Moreover, assume that we are given a planar structure on the $S$-labelled rooted tree $T_\sigma$ associated with this chain (cf.\ Section \ref{section: lie to partitions}).
Consider the map  
 $$\upalpha_\sigma : \Sigma |\Pi_S|^\diamond \longrightarrow  \Sigma |(\Pi_S)_{(x_0,x_1)}|^\diamond  \wedge \ldots   \wedge\Sigma |(\Pi_S)_{(x_{r-1},x_{r})}|^\diamond  \cong S^{N_{[x_{0},x_1]}} \wedge \ldots  \wedge S^{N_{[x_{r-1},x_{r}]}} \cong S^{|S|-1}.  $$
which is obtained by composing the collapse map $ \Sigma |\Pi_S|^\diamond \rightarrow \Sigma^2 ( |\St(\sigma)|/_{|\St(\sigma)^{-\sigma}|} )$ explained in Section  \ref{clong} with the homeomorphism  {$S^{N_{[x_{0},x_1]}} \wedge \ldots  \wedge S^{N_{[x_{r-1},x_{r}]}}$}$\cong S^{|S|-1}$ defined by sending the coordinate corresponding to the $i^{th}$ internal node in 
standard order (on the set $\cup_{i=1}^{r} N_{[x_{i-1},x_i]}$ 
of internal nodes  of $T$, cf.\ Section \ref{standardorder})  to the $i^{th}$ coordinate on the right.
We can link    $\upalpha_\sigma$ to the  map $a_{T_\sigma}$ which we have associated with the binary planar tree $T_\sigma$ in in Definition  \ref{amap}:\vspace{2pt}
\begin{proposition}\label{propaa}
There is a homotopy $ L_\sigma:[0,1]\times \Sigma |\Pi_S|^\diamond \rightarrow   S^{|S|-1} $ from $\upalpha_\sigma$ to   \mbox{$  a_{T_\sigma}$.} 
If $h\in \Sigma_S$, then $L_{h\cdot \sigma} (t,h\cdot T) = L_{\sigma}(t, T)$ for all $t,T$. Here $T_{h\cdot \sigma}$ inherits a planar structure from $T_{\sigma}$. 
\end{proposition}	\vspace{-5pt}
\begin{proof}
We proceed by induction on $r$, the statement being trivially  true for $r=0, 1$.
For $r\geq 2$, we assume that the tree $ T_\sigma$ is constructed by first identifying the roots of a left subtree $T_{\sigma_1}$ and a right subtree  $T_{\sigma_2}$  and then  adding an additional minimal element $r$, the root. For $i=1,2$, let $S_i$  be the set of labels of $T_{\sigma_i}$.  Write $\sigma_i =  [x_0^i\leq x_1^i\leq \ldots \leq x_{r}^i]$ for the binary chain of partitions of $S_i$ obtained by restricting the   partitions of $\sigma$ to $S_i$.
We consider the following two composite maps: 
$$
 \Sigma |\Pi_S|^\diamond  \xrightrightarrows[u_{x_r}]{ \   \nu_{x_r} \   } \Sigma |\Pi_{\mathbf{2}}|^\diamond \wedge    \bigg( \bigwedge_{i=1}^2 \Sigma |\Pi_{S_i}|_{}^\diamond \bigg)   
 \xrightrightarrows[\ \  \id\wedge a_{T_{\sigma_1}} \wedge a_{T_{\sigma_2}}\ \ ]{\id \wedge \upalpha_{\sigma_1} \wedge   \upalpha_{\sigma_2} } \Sigma |\Pi_{\mathbf{2}}|^\diamond    \wedge   \bigg( \bigwedge_{i=1}^2  S^{|S_i|-1}   \bigg). 
$$ 
Here $\upalpha_{\sigma_i}$ is the collapse map associated with the chain $\sigma_i$, the map $a_{T_{\sigma_i}}$ is defined in terms of the $S_i$-labelled binary planar tree $T_{\sigma_i}$ (cf.\ Definition \ref{amap}), the symbol $ \nu_{x_r}$ denotes the collapse map for  the \mbox{partition $x_r$}, and the map $ u_{x_r}$ is the ungrafting map with respect to $x_r$ (cf.\ Section \ref{LAPC1}).

Using our explicit description of collapse maps in Section \ref{clong}, we see that the top composition agrees with $\upalpha_\sigma$. The  bottom composition  evidently gives the map $a_{T_\sigma}$.
By Proposition \ref{propbb}, there is a homotopy $ {K}_{x_r}$ from $\nu_{x_r}$ to $u_{x_r}$. By induction hypothesis, there are homotopies $L_{\sigma_i}$ from $\alpha_{\sigma_i}$ to $a_{T_{\sigma_i}}$. Setting $L_{\sigma} :=( \id \wedge L_{\sigma_1} \wedge L_{\sigma_2}) \circ  {K}_{x_r}$ gives the desired homotopy. 
\end{proof}
 
\subsubsection*{General Binary Chains}\label{GeneralBinary} To analyse the collapse map associated with general binary chains, let $\sigma = [\hat{0} = x_0\leq  x_1\leq \ldots\leq x_r] $ be a binary chain of proper nontrivial partitions of $S$ such that $x_r$ partitions $S$ as $S=\coprod_{S_i \in\mathbf{d}_{x_r}} S_i$, where $\mathbf{d}_{x_r}$ is the set of classes of $x_r$. Write $x^i_j$ for the partition obtained by restricting $x_j$ to the set $S_i$. 
Suppose that for each $i$, we have chosen a planar structure on the $S_i$-labelled binary tree $T_{\sigma_i}$ associated with the binary chain $\sigma_i=[\hat{0}^i\leq x_1^i \leq \cdots \leq x_r^i \leq \hat{1}^i]$, where $\hat{0}^i$ and $\hat{1}^i$ denote the minimal and maximal partition on   $S_i$.
Write $\upalpha_{ {\sigma_i}}: \Sigma |\Pi_{S_i}|^\diamond \rightarrow S^{|S_i|-1}$ for the corresponding \mbox{collapse map described above Proposition \ref{propaa}.}\vspace{5pt}

Observe that  $ \Sigma |\Pi_S|^\diamond \xrightarrow{\nu_{x_r}}  \Sigma |\Pi_{\mathbf{d}_{x_r}}|^\diamond  \wedge    ( \bigwedge_{S_i\in\mathbf{d}_{x_r}} \Sigma |\Pi_{S_i}|_{}^\diamond )   \xrightarrow{ \id  \wedge (\bigwedge_i \upalpha_{\sigma_i}) } \Sigma |\Pi_{\mathbf{d}_{x_r}}|^\diamond \wedge ( \bigwedge_{S_i\in\mathbf{d}_{x_r}} S^{|S_i|-1} ) $
factors through  $\Sigma |\Pi_S|^\diamond \rightarrow \Sigma |(\Pi_S)_{(\hat{0},x_0)}|^\diamond  \wedge \ldots \wedge \Sigma |(\Pi_S)_{(x_r,\hat{1})}|^\diamond \simeq ( \bigwedge_{S_i\in\mathbf{d}_{x_r}} S^{|S_i|-1} )   \wedge \Sigma |\Pi_{\mathbf{d}_{x_r}}|^\diamond$, where the first map is the collapse map associated with the chain $[x_0<\ldots<x_r] $. \vspace{5pt}

Combining the two homotopies from Proposition \ref{propbb} and Proposition \ref{propaa}, we deduce:
\begin{proposition}[Collapse and ungrafting] \label{ta}There is a homotopy $$M_{\sigma}:[0,1] \times \Sigma |\Pi_S|^\diamond \longrightarrow  \Sigma |\Pi_{\mathbf{d}_{x_r}}|^\diamond \wedge ( \bigwedge_{S_i\in\mathbf{d}_{x_r}} S^{|S_i|-1} ) $$ from   $M_{\sigma}(0,-) = (\id \wedge (\bigwedge_i \upalpha_{\sigma_i}) ) \circ \nu_{x_r}$ (defined in terms of collapses) to  \mbox{$M_{\sigma}(1,-)=(\id \wedge a_{T_{\sigma_i}}) \circ u_{x_r}$} (defined using  ungrafting of trees). \vspace{3pt}

If $h\in \Sigma_S$, then $M_{h{\cdot}\sigma}(t,h{\cdot}T) = h{\cdot}M_{ \sigma}(t,T) \in \Sigma |\Pi_{\mathbf{d}_{h{\cdot}x_r}}|^\diamond \wedge ( \bigwedge_{h(S_i)\in\mathbf{d}_{h{\cdot}x_r}} S^{|h(S_i)|-1} ) $ for all $t,T$. Here we have used the  map $ 
\Sigma |\Pi_{ \mathbf{d}_{ x} }|^\diamond \wedge  ( \bigwedge_{S_i\in \mathbf{d}_{ x}} S^{|S_i|-1} ) \xrightarrow{h\cdot} \Sigma |\Pi_{ \mathbf{d}_{h\cdot x} }|^\diamond \wedge  ( \bigwedge_{h(S_i)\in \mathbf{d}_{h\cdot x}} S^{|h(S_i)|-1}  )$ and have given the  binary trees associated with $h\cdot \sigma$  the planar structures induced by the planar trees $T_{\sigma_i}$.
\end{proposition} 
\vspace{7pt}
In the unsuspended case, the collapse map takes the form\vspace{3pt}
$$ \hspace{-10pt} |\Pi_S|  \longrightarrow  |(\Pi_S)_{(\hat{0},x_1)}|^\diamond  \wedge\Sigma |(\Pi_S)_{(x_1,x_2)}|^\diamond  \wedge \ldots \wedge \Sigma |(\Pi_S)_{(x_{r-1},x_r)}|^\diamond\wedge  |(\Pi_S)_{(x_r,\hat{1})}|^\diamond \cong  \Sigma^{-1}  (\hspace{-2pt}\bigwedge_{S_i\in \mathbf{d}_{x_r}} \hspace{-5pt}S^{|S_i|-1})   \wedge  |\Pi_{\mathbf{d}_{x_r}}|^\diamond.   $$
Here  $\Sigma^{-1} \bigwedge_{S_i \in \mathbf{d}_{x_r}} \Sigma |\Pi_{S_i}|_{}^\diamond  $ denotes a desuspension of  the sphere  $\bigwedge_{S_i\in \mathbf{d}_{x_r}} S^{|S_i|-1} $. 

The unsuspended collapse map interacts with the group action in the expected manner.

\newpage  
\section{Restrictions}\label{c5} We will  now use the full strength of complementary collapse against orthogonality fans to compute Young restrictions of the partition complex and parabolic restrictions of Bruhat-Tits buildings.

\vspace{-6pt}
\subsection{Young Restrictions of the Partition Complex}\label{Youngsection}
Let $\mathcal{P}_S$ be the lattice of partitions of the finite set $S$, ordered under refinement, so that $\Pi_S = \ovA{\mathcal{P}}_S = \mathcal{P}_S-\{\hat{0},\hat{1}\} $. Fix a map $g:S\rightarrow \mathbf{k}$ and write $C_i = g^{-1}(i)$. Let $\Sigma_g = \Sigma_{C_1} \times \dots \times \Sigma_{C_k}$ be the associated Young subgroup.

\vspace{-3pt}
\subsubsection*{An orthogonality fan on the partition complex}\ \label{previossection}

Given a chain of partitions $\sigma = [ y_0 < \dots < y_m] \in  \mathcal{F}_{\mathcal{P}_S}$ and a class $K$ of $y_m$, we  attach a word \mbox{$w_K = w_K(\sigma)\in F \langle c_1, \dots ,c_k\rangle$}  in the free group on $k$ generators to $K$:
\begin{itemize}[leftmargin=26pt]
\item If $m=0$, we attach the word
$c_1^{|K\cap C_1|} \dots c_k^{|K\cap C_k|}$ to $K$.
\item If $m>0$, we first use the chain $[y_0< \dots < y_{m-1}]$ to attach a word to all classes in $y_{m-1}$. \\We then let $w_K$ be the product of all words attached to $y_{m-1}$-classes which are \mbox{contained in $K$,} multiplied in ascending lexicographical order (where $c_1<\dots<c_k$).
\end{itemize}

\begin{example}\vspace{-2pt}
For $S = \{1,\dots,6\}$, $k=3$, and $C_1=\{1,2\}$, $C_2 = \{3,4,5\}$, $C_3=\{6\}$, we send the chain 
$ [\{1|23|4|5|6\}<\{1|23|45|6\}< \{123|456\}  ] $
to the words
$ c_1^2 c_2,\  c_2^2 c_3$.
\end{example}\vspace{2pt}

For $\sigma=  [ y_0 < \dots < y_m]$ a chain, we record the words attached to $y_m$-classes  in ascending lexicographic order as $w_a<w_{b_1}<\dots < w_{b_s}$, where possibly $s=0$.
Let $A$ be the set of $y_m$-classes whose associated word is the minimal $w_a$. Define $B_1,\dots,B_s$ in a similar manner and set $B= \cup_i B_i$. 

Define $F_1(\sigma) = F_1(S,g) (\sigma)$ to be the partition obtained from $y_m$ by merging all classes in $B$.\\
Define $F_2(\sigma)= F_2(S,g) (\sigma)$ to be the partition obtained from $y_m$ by merging all classes in $A$.\vspace{-5pt}
\begin{figure}[hb]
  \centering
  \includegraphics[scale=0.13]{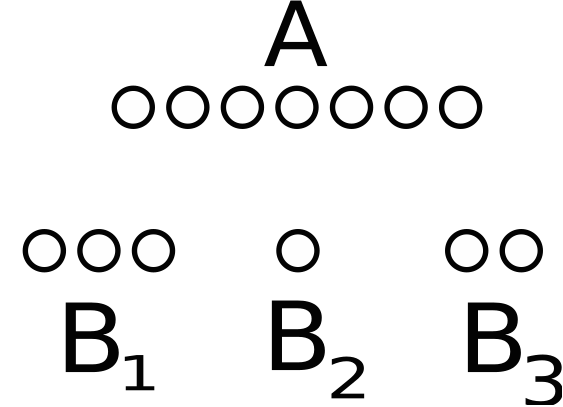} \ \ \ \ \ \ \ \ \ \ \ \ \ \ \ \ \ \ 
  \includegraphics[scale=0.13]{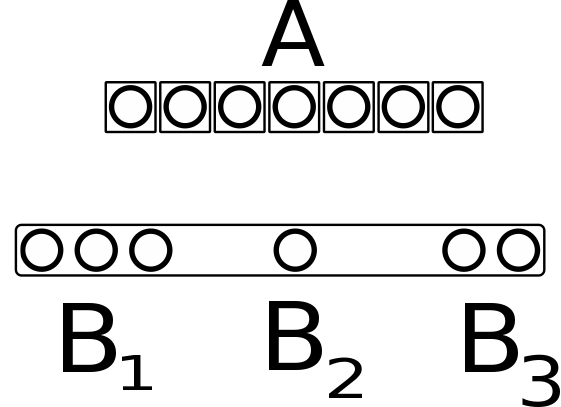} \ \ \ \ \ \ \ \ \ \ \ \ \ \ \ \ \ \ 
  \includegraphics[scale=0.13]{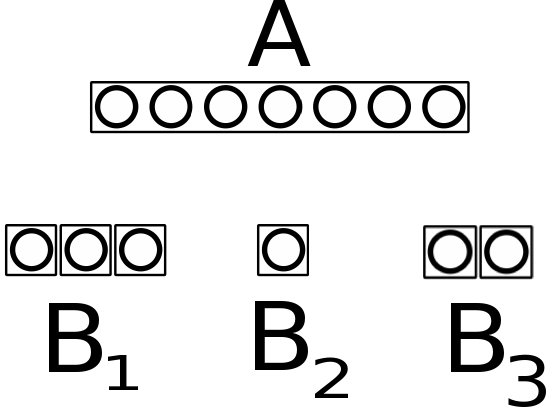}\vspace{-10pt}
\caption*{We sketch $\sigma, F_1(\sigma),$ and $F_2(\sigma)$. The bullets  represent classes of $y_m$.}
\end{figure}\vspace{-10pt}\vspace{2pt}
\begin{theorem}\label{isafan}
\mbox{The pair $\mathbf{F}(S,g)  = (F_1(S,g),F_2(S,g))$ is an orthogonality fan on the $\Sigma_g$-lattice $\mathcal{P}_S$.}
\end{theorem}\vspace{-5pt}
\begin{proof} 
By Lemma \ref{fansfromfunctions}, it suffices to check that $F_1$ and $F_2$ are \textit{orthogonality functions} in the sense of Definition \ref{orthogonalityfunction}.
The functions $F_1$, $F_2$ are clearly increasing and equivariant.\\
To check axiom $(3)$, let $\sigma=[y_0< \dots < y_{m}]$ be a chain of partitions in $\mathcal{P}_S$ and let $z>y_m$. 
Define the sets of $y_m$-classes $A$ and $B_1,\ldots,B_s$ as above. Let $A'\subset A$ be the collection of $y_m$-classes which are merged with a class in $B=\cup B_i$ in $z$. 
Similarly, write $B' \subset \cup B_i$ for the family of classes  in $B$ which are merged with a class in $A$ in the partition $z$.
For $F_1$,  we observe a natural injection
$$  \{y_m < y < z  \ \ \  |\ \  \ y\wedge F_1(\sigma) = y_m,\ \  (y\vee F_1(\sigma))\wedge z =z  \} \ \ \hookrightarrow  \ \ \{A'\xrightarrow{f} B' \ \ | \ \ \  \forall a\in A' : \ \ a\simeq_z f(a)  \} 
 $$
\begin{minipage}{0.8\textwidth}
 obtained by sending  $a\in A'$ to the unique $f(a)\in B'$ with $a\simeq_{y} f(a)$. We draw a suggestive example on the right.
The subposet is therefore discrete and $F_1$ an orthogonality function. The statement for $F_2$ follows from a parallel argument.
\end{minipage} \hspace{5pt} 
\begin{minipage}{0.2 \textwidth} 
\includegraphics[width=0.7\textwidth]{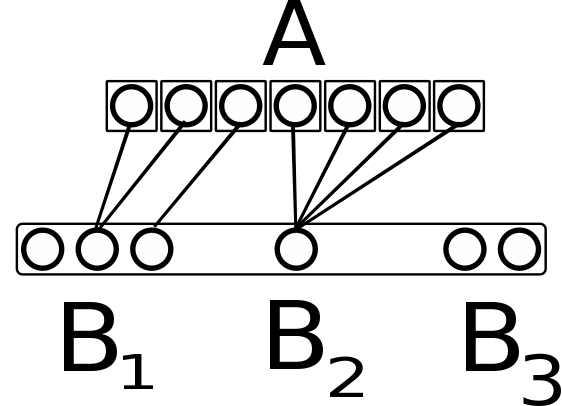}  
\end{minipage}   
\end{proof}
 
\vspace{-15pt}
\subsubsection*{Orthogonal Chains from Labelled Lyndon Words}
We will now give a convenient description of the chains of partitions which are orthogonal to the fan $\mathbf{F}(S,g)$ from Theorem \ref{isafan}.   
\mbox{Let $\mathbf{k}^\ast$ be the set of} words of finite length in $\{1,\ldots,k\}$, ordered lexicographically, and consider the free group $F\langle c_\upalpha \rangle_{\upalpha \in \mathbf{k}^\ast}$. This free group inherits a lexicographic order based on the order on $\mathbf{k}^\ast$. A word $w\in F\langle c_\upalpha \rangle_{\upalpha \in \mathbf{k}^\ast}$ is called a \textit{weak Lyndon word} if $w\leq \tilde{w}$ for any cyclic rotation $\tilde{w}$ of $w$.
\begin{definition}
The \textit{reduction function} $(-)': F\langle c_\upalpha \rangle_{\upalpha \in \mathbf{k}^\ast} \rightarrow F\langle c_\upalpha \rangle_{\upalpha \in \mathbf{k}^\ast}$   takes a word \mbox{$w=c_{\upalpha_1} \ldots c_{\upalpha_m}$} and replaces a string
$c_{\upalpha_i} c_{\upalpha_{i+1}}$  by $c_{\upalpha_i \upalpha_{i+1}}$ whenever $\upalpha_i$ is minimal in $\{\upalpha_1,\ldots,\upalpha_m\}$ and $\upalpha_{i+1}>\upalpha_{i}$.
\end{definition}
Given $w \in F\langle c_\alpha \rangle$, we define a sequence $w_0, w_1, w_2 \ldots$ by $w_0 = w'$ and $w_i = w_{i-1}'$ for $i>0$.\vspace{-0pt}  
\begin{example} \label{exampleword} For $k=3$ and $w=c_1c_1 c_2 c_1 c_1 c_3 c_2 c_3$, we have:\vspace{-7pt}  
$$w_0 = c_1c_{12} c_1 c_{13} c_2 c_3, \ \ \ \ w_1 = c_{112} c_{113} c_2 c_3, \ \ \ \    w_2 = c_{112113} c_2 c_3, \ \ \ \  w_3 = c_{1121132} c_3, \ \ \ \   w_4= c_{11211323}$$
\end{example}\vspace{-5pt}  
\begin{proposition}\label{preserveslyndon}
The reduction function $(-)'$ preserves weak Lyndon words in $F\langle c_\alpha \rangle$.
\end{proposition}\vspace{-5pt}  
\begin{proof}
Given a weak Lyndon word $w=c_{\upalpha_1} \ldots c_{\upalpha_m}$, any rotation $\widetilde{w'}$ of $w'$ lifts to a   rotation $\widetilde{w}$ of $w$ satisfying $(\widetilde{w})' = \widetilde{w'}$ (since $w$ is weak Lyndon). Observe that if  $\widetilde{w'}<{w'}$, \mbox{then  $\widetilde{w}<w$.}
\end{proof}\vspace{-5pt}  
\begin{lemma}
If $w = u^d$ is a weak Lyndon word for which $u$ is not a letter, then $w' \neq w$.
\end{lemma} \vspace{-5pt}  
\begin{proof}
If $w=c_{\upalpha_1} \ldots c_{\upalpha_m}$, then $\upalpha_1$ is minimal in $\{\upalpha_1,\ldots \upalpha_k\}$. Hence either $w=c_{\upalpha_1}^d$ or $w'\neq w$.
\end{proof}\vspace{-5pt}  
\begin{corollary}
If $w$ is weak Lyndon, then $w_0,w_1 \ldots$ stabilises in a  word of the form $c_\upalpha^i$.
\end{corollary}

We use this construction to produce chains of partitions. First, we mildly extend terminology from  Section \ref{Lyndon}.
 Fix a finite set $S$ and a map $g: S\rightarrow \mathbf{k}^\ast$.    An $(S,g)$-labelling of a weak Lyndon word $w = u^d$ in letters $\{c_\upalpha \}_{\upalpha \in \mathbf{k}^\ast}$ is a bijection between $S$ and \mbox{the   letters of $w$ such} that each $s\in S$ labels  a symbol $c_{g(s)}$. Two bijections give the same labelling if they agree \mbox{after permuting the copies of $u$.}\ \\

\vspace{-2pt}\vspace{-2pt}
\begin{minipage}{0.51\textwidth} \label{previouspage}
Let $w$ be a weak Lyndon word in letters $\{c_\upalpha \}_{\upalpha \in \mathbf{k}^\ast}$ together with an $(S,g)$-label $f$. For each $i$, we have  a canonical function from $S$ to the letters of $w_i$. 

Let $x_i$  be the partition on $S$ which identifies two points if they have the same image under this function. Let $S_{x_i} = S/x_i$ be the set of classes of $x_i$ and $g_{x_i}:S_{x_i}\rightarrow \mathbf{k}^\ast$ be the natural function recording the ``type'' $\upalpha$  of the letter $c_{\upalpha}$ in $w_i$ corresponding to a given class in $S_{x_i} $. \mbox{Then \mbox{$w_i$ is naturally  $(S_{x_i},g_{x_i})$-labeled, and $x_i,S_{x_i}$, and  $g_{x_i}$ are well-defined for all $i$.} }
\end{minipage}  \ \ 
\begin{minipage}{0.2 \textwidth} \vspace{-20pt} 
  \includegraphics[scale=0.5]{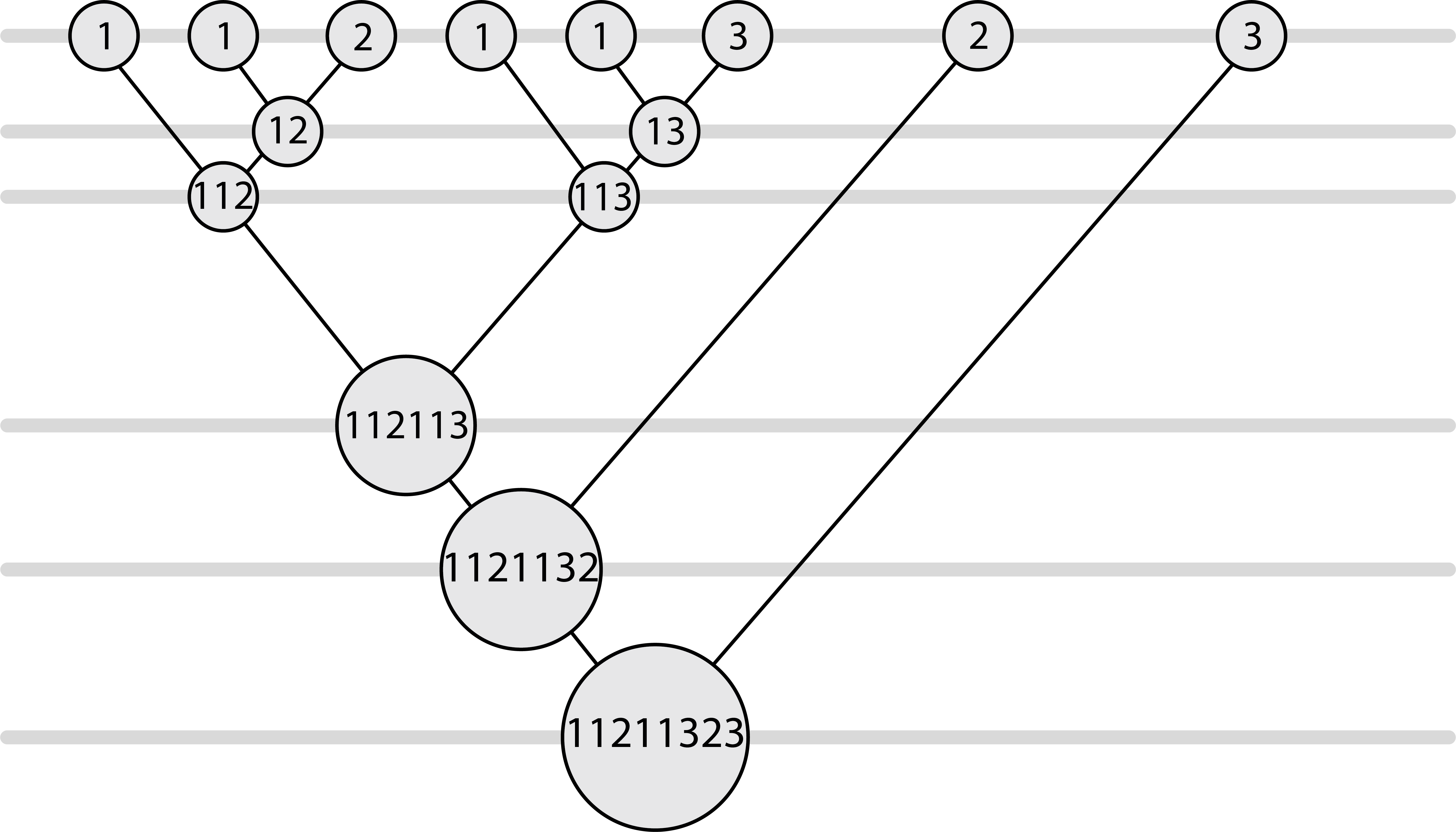} 
\end{minipage}   
\ \\ 

The illustration shows the chain of partitions attached to   $w=c_1c_1 c_2 c_1 c_1 c_3 c_2 c_3$ (cf.\ Example \ref{exampleword}).

\begin{definition}\label{chainfromlyndon} The   \textit{binary chain  of partitions $\sigma_w$ attached to a labelled weak Lyndon word} $w$ is given by  $\sigma_w=[x_0<x_1<\ldots<x_r]$, where $r$ is chosen maximal with $x_r \neq \hat{1}$. 
\end{definition} 
Let now $g:S\rightarrow \mathbf{k}$ be a map with fibres $C_i = g^{-1}(i)$ of size $n_i$.   Using Definition \ref{laLyn}, we prove:
\begin{lemma}\label{thewords}
The rule $w\mapsto \sigma_w$ gives a $\Sigma_{C_1} \times \dots \times \Sigma_{C_k}$-bijection 
$B^w_{(S,g)}(n_1,\dots,n_k) \xrightarrow{\cong} \mathbf{F}^\perp(S,g).$
\end{lemma}
\begin{proof} \vspace{-5pt} We freely use notation from above.
Equivariance is evident.
For bijectivity, we induct on $n=n_1+\ldots + n_k$. Assume without restriction that $n_1 >0$. 
If $n=n_1$, then the map identifies the unique weak Lyndon word $c_1^n$ with the  unique (empty) orthogonal chain.\\  
If $n>n_1$, \vspace{2pt}  we can remove all occurrences of $c_1$ by iterated application of the \mbox{reduction function $(-)'$}.  
Indeed, fix \mbox{$w\in B^w_{(S,g)}(n_1,\dots,n_k)$}. We will show that the chain $\sigma_w$ associated with $w$ is $\mathbf{F}(S,g)$-orthogonal in the sense of Definition \ref{globalinv}. Pick $m$ minimal such that $w_{m}$ \mbox{does not contain $c_1$.} We have a  ``type function'' $g_{x_m}: S_{x_m} \rightarrow \{\ldots , 1^22,\ldots,1^2k, 12,\ldots,1k,2,\ldots,k\}$ with finite linearly ordered image, and the word $w_m$ is an $(S_{x_m},g_{x_m})$-labeled weak Lyndon word in letters $\{\dots , c_{1^22},\ldots,c_{1^2k}, c_{12},\ldots,c_{1k},c_2,\ldots,c_k \}$ by Proposition \ref{preserveslyndon}. 
Note that   \mbox{$x_m \perp F_1([\hat{0}])$} and  $(\sigma_w)^{<x_m} = [x_0 <\dots<x_{m-1}] \perp F_2^{\leq x_m}$, using that  $F_1([\hat{0}])$  is  the partition identifying
 all elements labelled by $2, 3,\ldots $, while $F_2([\hat{0}])$ is the partition identifying   all elements labelled by $1$.
Observe that $(\sigma_w)^{>x_m}$ is $(F_1^{\geq x_m}, F_2^{\geq x_m})$-orthogonal precisely if $\sigma_{w_m}$ is an $\mathbf{F}(S_{x_m},g_{x_m})$-orthogonal chain in the poset of partitions  on $S_{x_m}$. This holds by induction as $w_m$ is shorter   than $w$. Hence $\sigma_w$ is $\mathbf{F}(S,g)$-orthogonal.\vspace{3pt}

After having shown  that the map $B^w_{(S,g)}(n_1,\dots,n_k) \xrightarrow{ } \mathbf{F}^\perp(S,g)$ is well-defined,  we will now prove bijectivity.
Sending $w$ to $(x_m, [x_0 <\dots<x_{m-1}],w_m)$ gives  a \mbox{$\Sigma_{C_1} \times \dots \times \Sigma_{C_k}$-equivariant map}
$$ D: B^w_{(S,g)}(n_1,\dots,n_k) \xrightarrow{ } \coprod_{\substack{y\perp F_1([\hat{0}]) \\ [z_0 <\dots<z_{m-1}] \perp F_2^{\leq y}}}  B^w_{(S_y,g_y)}( \ldots , n_{1^2 2}, \ldots ,  n_{1^2 k} , n_{1 2}, \ldots ,  n_{1 k},  n_{2} , \ldots ,  n_{k}) \vspace{-5pt} $$ 
In order to prove that the map $D$ is bijective, we produce an inverse $E$. Given $y\perp F_1([\hat{0}])$, a chain $[z_0 <\dots<z_{m-1}] \perp F_2^{\leq y}$, and a   word $w\in B^w_{(S_y,g_y)}( \ldots , n_{1^2 2}, \ldots ,  n_{1^2 k} , n_{1 2}, \ldots ,  n_{1 k},  n_{2} , \ldots ,  n_{k})$, we first produce a word $\ovB{w}$ by replacing all letters of the form $c_{1^a i}$ by the string $c_1^a c_i$ and observe that $\ovB{w}$ is again a weak Lyndon word.
We define an $(S,g)$-labelling of $\ovB{w}$:  if $s\in S$ maps to $[s]\in S_y$ with $g_y([s]) =  {1^i a}$, then we can pick a unique $t\in [s]$  with  $g(t) = a$.
We then declare that  $s$ labels the $(r+2)^{nd}$ letter from the right in the corresponding string \mbox{$c_1 c_1 \ldots c_1 c_a$ in $\ovB{w}$} if $r$ is the minimal index with $s\sim t$ in the partition $z_r$ (using the convention $z_{-1} = \hat{0}$). \vspace{3pt} \mbox{We have defined the inverse $E$.}

A parallel construction can be carried out on chains of partitions: given any partition $y\perp F_1([\hat{0}])$ of $S$, we write $S_y$ for the collection of equivalence classes of $y$ and obtain a ``type function''  \mbox{$g_y:S_y \rightarrow \{\ldots , 1^22,\ldots,1^2k, 12,\ldots,1k,2,\ldots,k\}$} from $g$.  By Definition \ref{globalinv} of orthogonality, we obtain a $\Sigma_{C_1} \times \dots \times \Sigma_{C_k}$-equivariant bijection $ \mathbf{F}(S,g)^\perp \xrightarrow{\cong}  \coprod_{\substack{y\perp F_1([\hat{0}]) \\ [z_0 <\dots<z_{m-1}] \perp F_2^{\leq y}}}  \mathbf{F}(S_y,g_y)^\perp$.\vspace{3pt}

Combining the two bijections for words and chains, we obtain the following commutative square whose lower horizontal    map is bijective by induction:\vspace{0pt} 
\begin{diagram}\hspace{-25pt}
B^w_{(S,g)}(n_1,\dots,n_k) & \rTo& & \mathbf{F}(S,g)^\perp   \\
\\ \hspace{-25pt}\dTo^\cong  & && \dTo^\cong \\  
\hspace{-25pt} \coprod_{\substack{y\perp F_1([\hat{0}]) \\ [z_0 <\dots<z_{m-1}] \perp F_2^{\leq y}}} \hspace{-15pt} B^w_{(S_y,g_y)}( \ldots , n_{1^2 2}, \ldots ,  n_{1^2 k} , n_{1 2}, \ldots ,  n_{1 k},  n_{2} , \ldots ,  n_{k})
& \rTo^{\cong} & &\hspace{-15pt}\coprod_{\substack{y\perp F_1([\hat{0}]) \\ [z_0 <\dots<z_{m-1}] \perp F_2^{\leq y}}} \hspace{-15pt} \mathbf{F}(S_y,g_y)^\perp
\end{diagram} \end{proof}  
\vspace{-20pt}

\subsubsection*{Topological Branching Rule}\vspace{-10pt}
Fix a decomposition $n = n_1+ \ldots + n_k $. Let $S = \mathbf{n}$ and $g:\mathbf{n} \rightarrow \mathbf{k}$ be the unique order-preserving function with $|g^{-1}(i)| = n_i$.
Given any $(\mathbf{n},g)$-labelled weak Lyndon word $w = u^d \in B_{\mathbf{n}}^w(n_1 , \dots , n_k)$, we  have attached  a binary chain of partitions  \mbox{$\sigma_w  = [x_0<\ldots<x_r] \in  \mathbf{F}^\perp(\mathbf{n},g)$} in Definition \ref{chainfromlyndon}. 
The  copies of $u$ partition $\mathbf{n}$ into $d$ sets via the labelling; we write $\mathbf{d}_w$ for the set of classes of this partition. 
Observe that for each \mbox{$S_i \in \mathbf{d}_w$,} the $S_i$-labelled tree associated with the $(S_i, g|_{S_i})$-labelled Lyndon word $u$ carries a  planar structure. Using our analysis   in Section \ref{GeneralBinary}, we can describe  {the collapse map associated with $\sigma_w$ as} $$|\Pi_n| \longrightarrow   | (\Pi_n)_{(\hat{0},x_0)}|^\diamond \wedge \Sigma| (\Pi_n)_{(x_0,x_1)}|^\diamond \wedge \dots \Sigma| (\Pi_n)_{(x_{r-1},x_r)}|^\diamond \wedge | (\Pi_n)_{(x_r,\hat{1})}|^\diamond\xrightarrow{\simeq} \Sigma^{-1} (S^{\frac{n}{d}-1})^{\wedge {\mathbf{d}_w}} \wedge |\Pi_{\mathbf{d}_w}|^\diamond$$

Combining this observation with Complementary Collapse in Theorem \ref{CollapseFansB}, our Lemma \ref{thewords} concerning weak Lyndon words and orthogonal chains, and Remark \ref{standardchoice}, we obtain: 
\begin{theorem}\label{main}
For $n=n_1 + \dots + n_k$, there  are simple   $\Sigma_{n_1} \times \dots \times \Sigma_{n_k}$- equivariant  equivalences
$$ |\Pi_n|  \ \xrightarrow{\ \ \simeq \ \ }  \hspace{-3pt}  \bigvee_{w \in B_{\mathbf{n}}^w(n_1 , \dots , n_k)}  \hspace{-12pt}   \Sigma^{-1} (S^{\frac{n}{d}-1})^{\wedge {\mathbf{d}_w}} \wedge |\Pi_{\mathbf{d}_w}|^\diamond     \ \xrightarrow{\ \ \simeq \ \ } \hspace{-3pt} \bigvee_{\substack{d \ | \ \gcd(n_1 , \dots , n_k) \\ B(\frac{n_1}{d}, \dots, \frac{n_k}{d})}} \hspace{-12pt}   \Ind^{\Sigma_{n_1} \times \dots \times \Sigma_{n_k}}_{\Sigma_d}  (\Sigma^{-1} (S^{\frac{n}{d}-1})^{\wedge d} \wedge |\Pi_d|^\diamond  )  $$
\end{theorem}
\subsection{Branching and Ungrafting}
\mbox{Theorem \ref{main} has  consequences for  spectral Lie algebras.}

{Fix a decomposition $n=n_1+\ldots +n_k$} and let $g:\mathbf{n}\rightarrow \mathbf{k}$ be the increasing function with $|g^{-1}(i)| = n_i$. 

Let $w=u^d\in B_{\mathbf{n}}^w(n_1,\ldots,n_k)$ be a weak Lyndon word  with labelling represented by $f$. The various copies of $u$ give rise to a partition  $x_w$, written as $\mathbf{n}=\coprod_{S_i\in \mathbf{d}_w} S_i$. For each $S_i\in \mathbf{d}_w$, the Lyndon word $u$ has an  $(S_i,g|_{S_i})$-labelling represented by the restriction $f_i$ of $f$ to $S_i$. We write $T_{i,w}$ for the 
 $S_i$-labelled binary planar tree $T_{\tilde{u}^{f_i}}$ corresponding to the Lie\vspace{3pt} monomial $\tilde{u}^{f_i}$ obtained \mbox{by resolving $u$.} 
An $\mathbf{n}$-labelled weighted tree is said to be \textit{grafted from $w$} if it has nonzero weights on all edges and we can find edges $\{ (v_i^-<v_i) \}_{S_i \in \mathbf{d}_w}$ such that for each $S_i$,  the component of $T$  above $v_i^-$ is isomorphic to the $S_i$-labelled tree $T_{i,w}$. 
We write $(\Sigma |\Pi_n|^\diamond)_{(x_w,T_{i,w})}$ for the open subspace of $\Sigma |\Pi_n|^\diamond$ consisting of all points which can be represented by trees   grafted from $w$. 
\begin{proposition} Every point in $\Sigma |\Pi_n|^\diamond$ can lie in at most one subspace $(\Sigma |\Pi_n|^\diamond)_{(x_w,T_{i,w})}$. \vspace{-5pt}  \vspace{-2pt} 
\end{proposition}
\begin{proof}
Suppose a point in $(\Sigma |\Pi_n|^\diamond)_{(x_w,T_{i,w})}$ is represented by the tree $T$ having nonzero weights on all edges. We  attach a weak Lyndon word in letters $c_1,\ldots,c_k$ to   the nodes of $T$ as follows. First, we assign the letter $c_{g(s)}$ to each  leaf labelled by $s\in \mathbf{n}$. Then, we label all non-leaves recursively: if the node $t$ has incoming nodes $t_1,\ldots,t_k$, we multiply the words attached to $t_1,\ldots,t_k$ in lexicographically increasing order and attach this product to $t$. The nodes $v_i$ are then characterised by the property that for all nodes above them, the attached words are \textit{strict} Lyndon words, whereas for all nodes strictly below them, they are only  weak Lyndon words. The word $u$ is the label \mbox{of any $v_i$,} and hence $w = u^d$ is determined.   \vspace{-5pt}  \end{proof}
 
We obtain a continuous $\Sigma_{n_1} \times \ldots \times \Sigma_{n_k}$-equivariant  map \vspace{-2pt}
$$\Sigma |\Pi_n|^\diamond \xrightarrow{\ \ \ \ \ \ \ \ \ \ \   } \bigvee_{w\in B_{\mathbf{n}}^w(n_1,\ldots,n_k)}  \Sigma |\Pi_n|^\diamond/_{\Sigma |\Pi_n|^\diamond- (\Sigma |\Pi_n|^\diamond)_{(x_w,T_{i,w})}}.$$
For each $w$, the homotopy 
$M_{\sigma_w}:[0,1] \times \Sigma |\Pi_n|^\diamond \rightarrow  \Sigma |\Pi_{\mathbf{d}_{x_k}}|^\diamond \wedge ( \bigwedge_{S_i\in\mathbf{d}_{x_k}} S^{|S_i|-1} ) $
from \mbox{Proposition \ref{ta}}  sends $[0,1] \times \left(\Sigma |\Pi_n|^\diamond- (\Sigma |\Pi_n|^\diamond)_{(x_w,T_{i,w})}\right)$ to the basepoint. Together with its established equivariance properties, this implies the existence of a $\Sigma_{n_1} \times \ldots \times \Sigma_{n_k}$-equivariant homotopy\vspace{-1pt}
$$M:[0,1] \times \Sigma |\Pi_n|^\diamond   \xrightarrow{\ \ \ \ \ \ \ \ \ \ \   }  \   \bigvee_{w \in B_{\mathbf{n}}^w(n_1 , \dots , n_k)}  \hspace{-12pt} \Sigma |\Pi_{\mathbf{d}_w}|^\diamond  \wedge    (S^{\frac{n}{d}-1})^{\wedge {\mathbf{d}_w}}\vspace{-1pt} $$
whose composition with the projection to $ \Sigma |\Pi_{\mathbf{d}_w}|^\diamond \wedge (S^{\frac{n}{d}-1})^{\wedge {\mathbf{d}_w}}   $ is $M_{\sigma_w}$ for all $w\in B_{\mathbf{n}}^w(n_1,\ldots,n_k)$.
The map $M(0,-)$ appears in Theorem \ref{main} and is an equivariant equivalence. We deduce:
\begin{corollary}\label{branchingforungrafting}
There is a  $\Sigma_{n_1} \times \ldots \times \Sigma_{n_k}$-equivariant  homotopy equivalence \vspace{-2pt}$$M(1,-): \Sigma |\Pi_n|^\diamond  \xrightarrow{\ \ \ \ \ \simeq \ \ \ \ \   } \bigvee_{w \in B_{\mathbf{n}}^w(n_1 , \dots , n_k)}  \hspace{-12pt}  \Sigma  |\Pi_{\mathbf{d}_w}|^\diamond \wedge (S^{\frac{n}{d}-1})^{\wedge {\mathbf{d}_w}}\vspace{-2pt}$$
such that composing with the projection to the summand corresponding to  $w$ \mbox{gives  
  $(\id \wedge \bigwedge a_{T_{i,w}}) \circ u_{x_w}$.}
\end{corollary}

We can now deduce the algebraic branching rule by applying homology: \vspace{-5pt}
\begin{proof}[Proof of Lemma \ref{equation: branching for lie}] \label{algebraicbranchingproof}
Let  $w=u^d\in   B_{\mathbf{n}}^w(n_1,\ldots,n_k)$. Write $x_w$  for the partition  \mbox{$\mathbf{n}=\coprod_{S_i \in  \mathbf{d}_w} S_i$} induced by the various copies of $u$ in $w$ via the labelling. For each $S_i \in \mathbf{d}_w$,   let $f_i$ be the $(S_i,g|_{S_i})$-labelling on $u$ obtained by restriction and write $T_{i,w} = T_{\tilde{u}^{f_i}}$.
By \mbox{Corollary \ref{commutes},} \vspace{3pt} we can obtain the map  $\uptheta_w$ appearing in the algebraic branching rule by applying $\tilde{\HH}_0(-,\ZZ)$ to the composite   $\vartheta_w$ given by
$\vspace{3pt}\mathbf{Lie}_{\mathbf{d}_w} \xrightarrow{\id \wedge \bigwedge \gamma_{\tilde{u}^{f_i}}} \mathbf{Lie}_{\mathbf{d}_w}\wedge   \bigwedge_{S_i \in \mathbf{d}_w} \mathbf{Lie}_{S_i} \xrightarrow{ \ \   \ \ } \mathbf{Lie}_n$.
\mbox{After} pre- and postcomposing with suitable equivalences,  the map $\bigoplus \vartheta_w$ is  readily obtained by applying $\Map_{\mathbf{Sp}}(S^1, (S^1)^{\wedge n}) \wedge \DD(-)$ to the equivalence $\Sigma |\Pi_n|^\diamond \xrightarrow{\ \simeq \ }  \bigvee_{w \in B_{\mathbf{n}}^w(n_1 , \dots , n_k)}  \hspace{-0pt}  \Sigma  |\Pi_{\mathbf{d}_w}|^\diamond \wedge (S^{\frac{n}{d}-1})^{\wedge {\mathbf{d}_w}}$ in Corollary \ref{branchingforungrafting}.
\end{proof}

\subsection{Free Lie Algebras on Multiple Generators}\label{freelie}
We will now describe free Lie algebras on direct sums of spectra $X_1,\ldots, X_k$ in terms of free  Lie algebras on individual spectra.

Assume that we are given a \textit{strict} Lyndon word $w   \in B(|w|_1,\ldots, |w|_k)$ of length $|w| = \sum_i |w|_i$. 
Let  $S=\{1,\ldots,|w|\}$ and write $g: S \rightarrow \mathbf{k}$ for the unique order-preserving function with $|g^{-1}(i)| = |w|_i$. We give $w$ the $(S,g)$-label $f$ which labels all copies of $c_i$ by elements in $g^{-1}(i)$ in increasing order from left to right.
Write $\tilde{w} = \tilde{w}^f \in \Lierep_{|w|}$ for the  Lie monomial obtained by resolving $w$.\vspace{0pt}

\mbox{In Section \ref{c3po}, we have constructed a map $\gamma_{\tilde{w}}: S^0 \rightarrow  {\mathbf{Lie}}_{|w|}$ in the homotopy category  of spectra.}  {Since $\underline{\mathbf{Lie}}_{|w|}$ is nonequivariantly equivalent to a wedge of $0$-spheres, we can pick a morphism} \mbox{$\upgamma_{\tilde{w}} : S_c \rightarrow \underline{\mathbf{Lie}}_{|w|}$ in $\mathbf{Sp}$} from a connected
space of representatives for $\gamma_{\tilde{w}}$.  Here $S_c$ is a cofibrant replacement of the unit.
Inducing up, we obtain a map $ \Ind_1^{\Sigma_{|w|_1} \times \ldots \times \Sigma_{|w|_k}} S_c \rightarrow  \underline{\mathbf{Lie}}_{|w|}$ in the category\vspace{3pt} $\Fun(B(\Sigma_{|w|_1} \times \ldots \times \Sigma_{|w|_k}), \mathbf{Sp})$.

Given cofibrant $X_1,\ldots,X_k\in \mathbf{Sp}$, we apply  $  ( - \wedge (X_1^{\wedge  {|w|_1} } \wedge \ldots \wedge X_k^{\wedge  { |w|_k} })^{ } )_{\Sigma_{|w|_1}\times \ldots \times \Sigma_{|w|_k}}$
and obtain $$ S_c \wedge X_1^{\wedge  {|w|_1} } \wedge \ldots \wedge X_k^{\wedge  { |w|_k} }  \rightarrow  \underline{\mathbf{Lie}}_{|w|}\mywedge{\Sigma_{|w|_1}\times \ldots \times \Sigma_{|w|_k}} (X_1^{\wedge  {|w|_1} } \wedge \ldots \wedge X_k^{\wedge  { |w|_k} }) \rightarrow \underline{\mathbf{Lie}}_{|w|} \mywedge{\Sigma_{|w|}} (X_1\vee \ldots \vee  X_k)^{\wedge |w|} $$
The final spectrum includes into $\bigoplus_n \ {\underline{\mathbf{Lie}}}_n \mywedge{\Sigma_n}(X_1\vee \ldots \vee  X_k)^{\wedge n}$. \vspace{3pt} Passing to underlying \mbox{$\infty$-categories} gives a \mbox{morphism $F_w:  X_1^{\wedge  {|w|_1} } \wedge \ldots \wedge X_k^{\wedge  { |w|_k} } \rightarrow  \Free_{ {\Lierep}}(X_1\vee \ldots \vee X_k) $ in $\Sp$.} By Corollary \ref{commutes} and a straightforward diagram chase, the effect of $F_w$ on homotopy \vspace{3pt} is as \mbox{explained on p.\pageref{effectonhomotopy}.} 

\begin{corollary}\label{BBBB}
Inducing up the maps $F_w$ and summing  over all Lyndon words gives an equivalence of spectra
$$\bigvee_{w\in B_k} \ovA{F}_w  \ : \   \ \bigvee_{w\in B_k}\Free_{  {\Lierep}}( X_1^{\wedge |w|_1} \wedge \ldots \wedge X_k^{\wedge |w|_k})\ \  \xrightarrow{\simeq}\ \  \Free_{ {\Lierep}}(X_1 \vee\ldots \vee X_k) .$$
\end{corollary} 
\begin{proof}
It suffices to prove that for each $(n_1,\ldots,n_k)$ with $n=\sum_i n_i$, the  component of $\bigvee_{w\in B_k}  \ovA{F}_w$ of multi-degree $(n_1,\ldots,n_d)$
is an equivalence. In $\mathbf{Sp}$, we can obtain this component by applying \mbox{$(- \wedge X_1^{\wedge n_1} \wedge  \ldots \wedge  X_k^{\wedge n_k})_{ \Sigma_{n_1} \times \ldots \times \Sigma_{n_k}  }$} to the following  map of spectra with  $\Sigma_{n_1} \times \ldots \times \Sigma_{n_k} $-action:\vspace{2pt}
$$\hspace{-20pt}\bigvee_{\substack{d\ | \ n \\ w\in B_k(\frac{n_1}{d},\ldots,\frac{n_k}{d})}}   \hspace{-15pt}\Ind_{\Sigma_d}^{\Sigma_{n_1} \times \ldots \times \Sigma_{n_k}}(\underline{\mathbf{Lie}}_{\hspace{0.4pt} d}\wedge (S_c)^{\wedge {d}}) \xrightarrow{ \id \wedge \upgamma_{\tilde{w}}^{\wedge d}} \hspace{-15pt}\bigvee_{\substack{d\ | \ n \\ w\in B_k(\frac{n_1}{d},\ldots,\frac{n_k}{d})}} \hspace{-15pt} \Ind_{\Sigma_d}^{\Sigma_{n_1} \times \ldots \times \Sigma_{n_k}} \left(\underline{\mathbf{Lie}}_{\hspace{0.4pt} d}  \wedge  \underline{\mathbf{Lie}}_{\hspace{0.4pt} |w|}^{\wedge d} \right) \longrightarrow  \underline{\mathbf{Lie}}_{\hspace{0.4pt} n}  $$

Lemma $9.20.$ of \cite{arone2011operads} implies  that
$ \Ind_{\Sigma_d}^{\Sigma_{n_1} \times \ldots \times \Sigma_{n_k}}(\underline{\mathbf{Lie}}_{\hspace{0.4pt} d}\wedge (S_c)^{\wedge {d}}) \wedge X_1^{\wedge n_1} \wedge  \ldots \wedge  X_k^{\wedge n_k}$ 
and \mbox{$\underline{\mathbf{Lie}}_{\hspace{0.4pt} n}  \wedge X_1^{\wedge n_1} \wedge  \ldots \wedge  X_k^{\wedge n_k}$} are both  $\Sigma_{n_1} \times \ldots \times \Sigma_{n_k}$-cofibrant. It is therefore sufficient to prove that applying $(- \wedge X_1^{\wedge n_1} \wedge  \ldots \wedge  X_k^{\wedge n_k})_{ h(\Sigma_{n_1} \times \ldots \times \Sigma_{n_k})  }$ to the above map  gives an equivalence. For this, it is in turn enough to check that the above equivariant map is a weak equivalence of underlying spectra, i.e. an isomorphism in the homotopy category. Here, this map can be obtained, after precomposing with an equivalence,  by applying $ \Map_{\Sp}(S^1,(S^1)^{\wedge n}) \wedge \DD( -) $ to the equivalence 
$     \Sigma |\Pi_n|^\diamond \xrightarrow{ }  \bigvee_{\substack{d\ | \ n \\ w\in B_k(\frac{n_1}{d},\ldots,\frac{n_k}{d})}} \Ind_{\Sigma_d}^{\Sigma_{n_1} \times \ldots \times \Sigma_{n_k}} \left(\Sigma |\Pi_{\mathbf{d}}|^\diamond \wedge (S^{|w|-1})^{\wedge d}\right)  \vspace{5pt} $
established in \mbox{Corollary \ref{branchingforungrafting}.} By Proposition \ref{sw1}, this implies the result.\vspace{-5pt}
\end{proof} 
 
\begin{remark} Exactly the same strategy can also be used to construct an equivalence of spectra {$\bigvee_{w\in B_k} \ovA{G}_w  \ : \   \ \bigvee_{w\in B_k}\Free_{  \Sigma {\Lierep}}(S^{1-|w|} \wedge X_1^{\wedge |w|_1} \wedge \ldots \wedge X_k^{\wedge |w|_k})\ \  \xrightarrow{\simeq}\ \  \Free_{ \Sigma {\Lierep}}(X_1 \vee\ldots \vee X_k)  $} for the shifted Lie operad whose $n^{th}$ term is  given by $\DD(\Sigma | \Pi_n|^\diamond)$. Here,   $\ovA{G}_w$ is the free map induced by the map of spectra $G_w$, which   is defined by replacing $\gamma_{\tilde{w}} $ by $\DD(a_{T_{\tilde{w}}})$ in the definition of $F_w$. The effect of $G_w$ on homotopy groups is given by the shifted variant of the rule described on p.\ \pageref{effectonhomotopy}.  
\end{remark}

\subsection{Symmetry Breaking}
Complementary collapse   also gives  an asymmetric version of Theorem \ref{main}.
Fix a partition $S = A \cup B_1 \cup \dots \cup B_k$  corresponding to $g:S\rightarrow \{1,\dots, k+1\}$. Write $x$ for the partition which identifies all points in $B = \displaystyle \cup_i B_i$.
Set $F'_1( [\hat{0}] ) = x$ and $ F'_1(\sigma )= \hat{1}$ else.
Let $F_2' = F_2(S,g)$ be as in section \ref{Youngsection} on p.\pageref{previossection} for the orbits ordered as $A<B_1<\dots<B_k$.
\mbox{Theorem \ref{isafan}}  shows that $\mathbf{F}'=(F_1',F_2')$ is an orthogonality fan, and 
we deduce:
\begin{theorem}[Symmetry Breaking]\label{inductivestep}
\mbox{There is a simple $(\Sigma_{A} \times\Sigma_{B_1} \times  \dots  )$-equivariant  equivalence}
$$ |\Pi_n|   \ \longrightarrow \   \bigvee_{ \substack{A = A_1 \coprod \dots \coprod A_{\ell}, \ A_i \neq \emptyset \\ f_i : A_i \hookrightarrow B \\ \mbox{ s.t. } \im(f_{i+1}) \subset \im(f_i)  }} \Sigma^{-1} S^{|{A_1}|} \wedge\dots \wedge  S^{|{A_\ell}|}  \wedge |\Pi_B|^\diamond \vspace{-4pt}$$ 
 \end{theorem}
\begin{proof} 
Choosing a chain $[z_1<\dots<z_\ell<y]$ orthogonal to $\mathbf{F}'$ amounts to the following data: 
\begin{itemize}[leftmargin=26pt] 
\item Choose a function $f:A \rightarrow B$ corresponding to a partition $y=y_f \perp x$. 
\item  Choose a set $A_1$ containing one $f$-preimage for each point in $f(A)\subset B$. We obtain a partition $z_1 = z_{A_1}$ by identifying each point in $A_1$ with its image under $f$.
\item Choose a set $A_2 \subset A-A_1$ containing one $f$-preimage for each point in $f(A-A_1)$. We obtain a partition $z_2$ taking $z_1$ and merging all points in $A_2$ with their image under $f$.  \end{itemize}
Proceeding in this way gives  all $[z_1<..<z_\ell<y]\perp \mathbf{F}'$ -- the claim follows by \mbox{Theorem \ref{CollapseFansB}.}\vspace{-5pt}
\end{proof}

\begin{remark}
Symmetry breaking can also be used to give an inductive proof of Theorem \ref{main}. The disadvantage of this approach is  that it is hard to describe the involved collapse maps.\vspace{-5pt}
\end{remark}

\subsection{Parabolic Restrictions of Bruhat-Tits Buildings}
Let $V$ be  a finite-dimensional vector space over a finite field $k$.  Fix a flag $\mathbf{A} = [A_0<\dots < A_r] $
with   parabolic subgroup $P_{\mathbf{A}}$. We define an orthogonality fan $(F)$ of \mbox{length $1$} by
$F([ B_0 < \dots < B_i] ) =   A_{r-i} \vee B_i  $.
A flag $[ C_0< \dots <C_r]$ is then  $(F)$-orthogonal if  
$C_i \perp A_{r-i}$ for $i=0,\ldots,r$.
Choosing a  flag $\mathbf{B}$  complementary to $\mathbf{A}$ with parabolic $P_\mathbf{B}$ and intersecting Levi $L_{\mathbf{A}\mathbf{B}} = P_\mathbf{A}\cap P_\mathbf{B}$, we read off from Theorem \ref{CollapseFansB}:
\begin{lemma}\label{BTb} There is a simple  $P_\mathbf{A}$-equivalence
$|\BT(V)| \cong \Ind^{P_\mathbf{A}}_{L_{\mathbf{A}\mathbf{B}}} (\Sigma^{r} \bigwedge_{i = 0}^{r+1} |\BT(\gr^i(\mathbf{B}))|^\diamond )$.
Here $\gr^i(\mathbf{B})=B_i/ {B_{i-1}}$ for $i=1,\ldots ,r$ and we set  $\gr^0(\mathbf{B}) = B_0$  and $\gr^{r+1}(\mathbf{B}) = V/ {B_r}$.
\end{lemma} 
Lemma \ref{BTb} gives a new proof of a nontrivial result in modular representation theory, cf.\ \cite{steinberg1957prime}. 
Write $\St_n$ for the \textit{integral Steinberg module}, i.e. the $\GL_n(\FF_p)$-module given by  $\tilde{\HH}_{n-2}(\BT(\FF_p^n),\ZZ)$.  
\begin{corollary}\label{steinbergproj} Let $k = \FF_q$ be a finite field and assume that $R$ is any  ring in which the number $\prod_{k=1}^n  ({q^k-1}) $ is invertible.
Then $\St \otimes R$ is a \vspace{-2pt} projective $R[\GL_n(\FF_q)]$-module. \vspace{-1pt}
\end{corollary}
\begin{proof}
Recall that whenever $H\subset G$ are finite groups with $\frac{|G|}{|H|}$ invertible in $R$, then an $R[G]$-module $M$  is projective if and only if its restriction  to $R[H]$ is a projective $R[H]$-module.

Let  $P_{\mathbf{A}}\subset \GL_n(\FF_q)$ be the subgroup of upper triangular invertible matrices stabilising the flag $\langle e_1 \rangle, \langle e_1, e_2 \rangle, \ldots$ and write $H\subset P_{\mathbf{A}}$ for the subgroup of matrices with all diagonal \mbox{entries equal to $1$.} Note that    the unipotent radical $H$  is also  the Sylow $p$-subgroup, where $p=\chara(\FF_q)$.
Since $H$ intersects the Levi $L_{\mathbf{A}\mathbf{B}}$ corresponding to any complementary flag $\mathbf{B}$  trivially, 
Lemma $\ref{BTb}$ and the double coset formula imply that the restriction of $\St \otimes R$ to $R[H]$ is a free  $R[H]$-module. The claim follows since $\frac{|\GL_n(\FF_q)|}{|H|} = \prod_{k=1}^n  ({q^k-1}) $ by a well-known counting argument.\end{proof}

\newpage

\newpage
\section{Fixed Points}
Given a subgroup $G\subset \Sigma_n$, the subspace $|\Pi_n|^G$ of $G$-\textit{fixed points} carries a natural action of the  {Weyl group} $W_{\Sigma_n}(G) = N_{\Sigma_n}(G)/G$. We will now provide an answer to the following question:
\begin{nonquestion}
What is the $W_{\Sigma_n}(G)$-equivariant simple homotopy type of $|\Pi_n|^G$?
\end{nonquestion}

In order to present our analysis, we single out a special class of subgroups:
\begin{definition}
A subgroup $G\subset \Sigma_n$ is   \textit{isotypical}  if all $G$-orbits are equivariantly isomorphic.
\end{definition}

Our result reduces the general case of the question raised above to the transitive situation:
\begin{theorem}\label{fixedpoints}
If $G\subset \Sigma_n$ acts isotypically on $\{1,\dots,n\}$, we may assume after relabelling   that $G$ is a transitive subgroup of   $\Sigma_d\xrightarrow{\Delta} \Sigma_d^{ \frac{n}{d}}\subset \Sigma_n$ for $d\ |\ n$ and $\Delta$ the diagonal embedding.\\
Then there is a $W_{\Sigma_n}(G)=N_{\Sigma_n}(G)/G$ 
-equivariant simple homotopy equivalence 
$$|\Pi_n|^G  \ \ \ \xrightarrow{\ \  \simeq \ \  }\ \ \Ind^{W_{\Sigma_n}(G)}_{W_{\Sigma_d}(G)\times  \Sigma_{\frac{n}{d}}} \ \  (|\Pi_d|^G )^\diamond \wedge |\Pi_{\frac{n}{d}}|^\diamond$$
\end{theorem}
We again use the convention that if $\mathcal{P}$ is a poset with one element, then $|\overline{\mathcal{P}}|^\diamond \wedge X^\diamond = X$ for any $X$.
\begin{lemma}\label{easy}
If $G$ acts non-isotypically, then $|\Pi_n|^G$ is $W_{\Sigma_n}(G)$-equivariantly collapsible.\end{lemma}
\begin{remark}
We could also view $|\Pi_n|^G$ as a $N_{\Sigma_n}(G)$-space or a $C_{\Sigma_d}(G)^{ \frac{n}{d}}$-space. In this case, the above result implies equivalences
$$|\Pi_n|^G  \cong \Ind^{N_{\Sigma_n}(G)}_{N_{\Sigma_d}(G)\times  \Sigma_{\frac{n}{d}}} \ \  (|\Pi_d|^G)^\diamond \wedge |\Pi_{\frac{n}{d}}|^\diamond , \ \ \ \ \ \ \ |\Pi_n|^G  \cong \Ind^{C_{\Sigma_d}(G)^{ \frac{n}{d}}}_{C_{\Sigma_d}(G)} \ \  (|\Pi_d|^G)^\diamond \wedge |\Pi_{\frac{n}{d}}|^\diamond $$
\end{remark}

\begin{proof}[Proof of Theorem \ref{fixedpoints} and Lemma \ref{easy}]
Write $x\in \Pi_n$ for the partition of $\{1,\dots,n\}$ into $G$-orbits. 
We assume that $G$ is nontrivial and that the action is not transitive. This implies  that $x\neq \hat{0},\hat{1}$.

\begin{claim}
The group $G$ is isotypical if and only if $x^\perp \neq \emptyset$.\vspace{-5pt}
\end{claim}
\begin{proof}[Proof of Claim]
Let $y\in x^\perp$ be a partition with corresponding equivalence relation $\simeq_y$ on $\{1,\dots,n\}$. Given two elements $a,b$ with $a\simeq_y b$, we observe that for each $c\in \Orb_G(a)$, there is a $d \in \Orb_G(b)$ with $c\simeq_y d$, and that this $d$ must be unique as $x \wedge y = \hat{0}$. We obtain an $G$-equivariant function $\Orb_G(a) \rightarrow \Orb_G(b)$ which by symmetry is bijective. Since $y \vee x = \hat{1}$, this implies that all orbits are isomorphic $G$-sets.\\
For the converse, assume that the action is isotypical with orbits $O_1,\ldots,O_k$. We pick $G$-equivariant isomorphisms
$f_i: O_i \rightarrow O_{i+1}$ and observe that the finest partition $y$ of $\{1,\dots,n\}$ satisfying  $(a\in O_i)\simeq_y (f_i(a)\in O_{i+1})$ for all $a$ and all $i$ satisfies $y\perp x$.\vspace{-5pt}
\end{proof}
\begin{claim}\label{transitive}
The action of $W_{\Sigma_n}(G)$ on $x^\perp$ is transitive.\vspace{-5pt}
\end{claim}
\begin{proof}[Proof of Claim]
Given $y,z \in x^\perp$, we define $\sigma \in \Sigma_n$ by setting $\sigma(i) = j $ if and only if $i,j$ lie in the same $G$-orbit and there is an $g\in G$ with $i \simeq_y g(1)$ and $j \simeq_z g(1)$. A simple argument shows that this gives a permutation $\sigma$ in $C_{\Sigma_d}(G)^{ \frac{n}{d}} \subset N_{\Sigma_n}(G)$ satisfying $\sigma(i) \simeq_z \sigma(j)$ if and only if $i\simeq_y j$, which means that $\sigma \cdot z = y$.
\end{proof}
We can now deduce  Theorem \ref{fixedpoints} and Lemma \ref{easy}.
The Lemma follows immediately from Theorem \ref{CC2} since $x^\perp = \emptyset$. 
In order to prove Theorem \ref{fixedpoints}, we take $z\in x^\perp$ to be the $G$-invariant partition with $i \simeq_z (i+kd)$ for all \mbox{numbers $i,k$.} The transitivity of the $W_{\Sigma_n}(G)$-action on $x^\perp$, Theorem \ref{CC2}, and Proposition \ref{indexedwedge} together imply the existence of a  simple $W_{\Sigma_n}(G)$-equivariant equivalence  
 {$ |\Pi_n|^G \cong \Ind^{W_{\Sigma_n}(G)}_{\Stab(z)} \ \  (|\Pi_{n,<z}|^G)^\diamond\wedge (| \Pi_{n,>z} |^G)^\diamond$.}
The result   follows by observing the identifications 
$ \Stab(z) = W_{\Sigma_d}(G)\times  \Sigma_{\frac{n}{d}}$ and $ \Pi^G_{n,<z} \cong \Pi_{\frac{n}{d}}$, and $\Pi^G_{n,>z} \cong \Pi_d^G$.
\end{proof}

It remains to analyse the transitive actions. This case is the subject of a lemma of Klass ~\cite{klass1973enumeration}:
\begin{lemma}\label{lemma: transitive}
 If $G\subset \Sigma_n$ is transitive  and $H$ is the stabiliser of $1$, then the poset of $G$-invariant partitions of $\{1,\dots,n\}$ is isomorphic to the  poset of subgroups $H\subseteq K \subseteq G$. The non-trivial partitions correspond to subgroups $H\subsetneq K \subsetneq G$.
\end{lemma} 

\begin{proof} 
We  identify $\mathbf{n}$ with the $G$-set $G/H$. Let $x$ be a $G$-invariant partition of $G/H$. \\
Let $K=\{g\in G\mid g(eH)\simeq_x eH\}$. Then $K$ is a subgroup of $G$. Indeed, suppose that we have $eH\simeq_x g_1H\simeq_x g_2H$. Multiplying the first equivalence by $g_2$ we obtain that $g_2H \simeq_x g_2g_1H$. Since $g_2H\simeq_x eH$, it follows that $g_2g_1\in K$. This shows that $K$ is closed under multiplication. A similar argument shows that $K$ is closed under taking inverses. Clearly, $H\subseteq K$. It is easy to see that a finer partition will lead to a smaller group $K$.

Conversely, let $K$ be a group, $H\subseteq K\subseteq G$. Let $x$ be the partition of $G/H$ defined by the rule $g_1H \simeq_x g_2H$ if and only if $g_2^{-1}g_1\in K$. It is easy to check that $x$ is a well defined $G$-invariant partition, and that we have defined two maps between posets that are inverse to each other.  
\end{proof}

We can obtain specific contractibility results for certain subgroups:

Let us suppose that $d \ |\ n$ and $d$ factors as a product $d=d_1\cdots d_l$ of integers $d_1,\ldots, d_l>1$ with $l>1$. Consider the iterated wreath product $\Sigma_{d_1}\wr\cdots\wr\Sigma_{d_l}$ as a subgroup of $\Sigma_n$ via the inclusions
\[
\Sigma_{d_1}\wr\cdots\wr\Sigma_{d_l}\hookrightarrow \Sigma_d\xrightarrow{\Delta} (\Sigma_d)^{ \frac{n}{d}}\hookrightarrow \Sigma_n.
\]
\begin{lemma}\label{lemma: isotypical wreath} 
The space $|\Pi_n|^{\Sigma_{d_1}\wr\cdots\wr\Sigma_{d_l}}$ is  
collapsible. 
\end{lemma}

\begin{proof}Without restriction, we may assume by Theorem \ref{fixedpoints} that $d=n$.\\
We observe that $\Sigma_{d_1} \wr \dots \wr \Sigma_{d_l}= N_{\Sigma_{d_1} \wr \dots \wr \Sigma_{d_l}}((\Sigma_{d_1} \wr \dots \wr \Sigma_{d_{l-1}} )^{ d_l})$.
Hence, we have
$$|\Pi_n|^{\Sigma_{d_1} \wr \dots \wr \Sigma_{d_l}} = (|\Pi_n|^{(\Sigma_{d_1} \wr \dots \wr \Sigma_{d_{l-1}} )^{ d_l}} )^{N_{\Sigma_{d_1} \wr \dots \wr \Sigma_{d_l}}((\Sigma_{d_1} \wr \dots \wr \Sigma_{d_{l-1}} )^{ d_l})} $$
By Lemma \ref{fixedpoints}, the space $|\Pi_n|^{(\Sigma_{d_1} \wr \dots \wr \Sigma_{d_{l-1}})^{{d_l}}}$ is $W_{\Sigma_{d_1} \wr \dots \wr \Sigma_{d_l}}((\Sigma_{d_1} \wr \dots \wr \Sigma_{d_{l-1}} )^{ d_l})$-equivariantly collapsible since the action of $(\Sigma_{d_1} \wr \dots \wr \Sigma_{d_{l-1}})^{{d_l}}$ on $\{1,\ldots,n\}$ is not isotypical.
\end{proof}

We can classify the isotropy groups of simplices in $|\Pi_n|$ which act isotypically:
\begin{lemma}\label{lemma: isotropy}
Consider the action of $\Sigma_{n_1}\times \cdots \times\Sigma_{n_k}$ on $\Pi_n$. Let $G\subset \Sigma_{n_1}\times \cdots \times\Sigma_{n_k}$ be an isotropy group of a simplex in $\Pi_n$. Suppose that $G$ acts isotypically on $\n$. Then $G$ is conjugate to a group of the form $\Sigma_{d_1}\wr\cdots\wr \Sigma_{d_l}$, where $l\ge 1$, and $d_1\cdots d_l | \gcd(n_1, \ldots, n_k)$. As usual, the group $G$ acts diagonally on blocks of size $d_1\cdots d_l$. The possibility that $G$ is the trivial group is included.
\end{lemma}
\begin{proof}
Let $\sigma = [x_0 < \cdots < x_i]$ be a chain of partitions of $\n$. Let $G$ be the stabiliser of this chain in $\Sigma_{n_1}\times \cdots \times\Sigma_{n_k}\subset \Sigma_n$. We will analyse the general form of $G$. The group $G$ acts on the set of equivalence classes of each $x_j$. Recall that $x_i$ is the coarsest partition of the chain. Let us say that two equivalence classes of $x_i$ are of the same type  if they are in the same orbit of the action of $G$. This is an equivalence relation on the set of equivalence classes of $x_i$. We say that there are $s$ different types of equivalence classes of $x_i$, and there are $t_1$ classes of type $1$, $t_2$ classes of type $2$ and so forth. It is easy to see that in this case $G$ is isomorphic to a group of the form $G\cong K_1\wr\Sigma_{t_1} \times \cdots \times K_s\wr \Sigma_{t_s}$
Here $K_j$ is the group of automorphisms (in $\Sigma_{n_1}\times \cdots \times\Sigma_{n_k}$) of the restriction of the chain  $\sigma = [x_0 < \cdots < x_i]$ to an equivalence class of $x_i$ of type $j$. We conclude that if $G$ acts isotypically on $\n$, then either $G$ is trivial, or $s=1$ and $G\cong K_1\wr \Sigma_{t_1}$. Furthermore, if $G$ acts isotypically on $\n$, then $K_1$ has to act isotypically on a component of $x_i$. The lemma follows by induction.
\end{proof}
Homotopy theory often deals with computations ``one prime at a time", and sometimes  $p$-local questions  about a given $G$-space $X$ can   be answered using only its fixed point spaces under $p$-groups. For example, Theorem $6.4$ of \cite{dwyer1998sharp} (cf.\ also Theorem $A$ of  \cite{webb1991split} and Proposition 4.6. of \cite{arone2016bredon}) specifies conditions under which the Bredon homology $\widetilde{\HH}_\ast^{\br}(X,\mu)$ of $X$ with coefficients in a Mackey functor $\mu$ can be computed in terms of the fixed points  of $X$ \mbox{under $p$-groups.} \vspace{3pt}

The fixed point spaces $|\Pi_n|^P$ of partition complexes under $p$-groups  are therefore of distinguished importance. The first-named author, William  Dwyer, and Kathryn Lesh have shown in  Proposition $6.2.$ of \cite{arone2016bredon} that $|\Pi_n|^P$ is contractible unless the $p$-group $P$ is abelian and acts freely. 
Our Theorem \ref{fixedpoints} allows us to compute all non-contractible fixed point spaces of partition complexes under $p$-groups:
\begin{corollary}\label{corollary: el abelian}
Let $P\subset \Sigma_n$ be a $p$-group.
\begin{enumerate}[leftmargin=26pt]
\item If $P\cong \FF_p^k$ is elementary abelian acting freely on $\n$ for  $n=mp^k$ \mbox{(here $p$ may divide $m$),}  {we have}  $\Aff_{\FF_p^k} := N_{\Sigma_{p^k}} ({\FF_p^k}) \cong  \FF_p^k \rtimes \GL_k(\FF_p)$ and  \mbox{$\Aff_{\FF_p^k \wr \Sigma_m}:= N_{\Sigma_n} (\FF_p^k ) \cong (\FF_p^k)^{ m} \rtimes (\GL_k(\FF_p) \times \Sigma_m)$.}\vspace{3pt} \\
There is an $\Aff_{\FF_p^k \wr \Sigma_m}$-equivariant simple  homotopy equivalence\vspace{-5pt}
$$|\Pi_n|^P \ \ \ \xrightarrow{\ \  \simeq \ \  }\ \ \  \Ind^{\Aff_{\FF_p^k \wr \Sigma_m}}_{\Aff_{\FF_p^k} \times \Sigma_m}  ( |\BT(\FF_p^k)|^\diamond \wedge |\Pi_m|^\diamond).\vspace{-5pt}$$ 
As before, $\BT(\FF_p^k)$ denotes  the poset of proper nontrivial subspaces of the vector space $\FF_p^k$. 
Nonequivariantly, this implies that $|\Pi_n|^P$ is a bouquet of $(m-1)! \cdot p^{k(m-1)+{k\choose 2}}$ spheres of dimension ${m+k-3}.$
\item If the action of $P$ is not of this form, then $|\Pi_n|^P$ is $W_{\Sigma_n}(P)$-equivariantly contractible.
\end{enumerate}
\end{corollary}
\begin{proof}
Case $(2)$ is Proposition $6.2.$ of \cite{arone2016bredon}. 

For $(1)$, we observe that the action is \mbox{isotypical.} Theorem \ref{fixedpoints} (in its instance with normaliser actions) gives an $N_{\Sigma_n}(\FF_p^k)$-equivariant equivalence \mbox{$|\Pi_n|^{\FF_p^k}  {\simeq}  \Ind^{N_{\Sigma_n}(\FF_p^k)}_{N_{\Sigma_{p^k}}(\FF_p^k)\times  \Sigma_{m}} (|\Pi_{\FF_p^k}|^{\FF_p^k} )^\diamond \wedge |\Pi_{m}|^\diamond$.}\vspace{-5pt}
 By Lemma \ref{lemma: transitive}, we have an isomorphism of posets $\Pi_{p^k}^{\FF_p^k} \cong \BT(\FF_p^k)$, which implies  \mbox{$|\Pi_n|^{\FF_p^k} \cong |\BT(\FF_p^k)|$}. 
 
 The standard action of   $\FF_p^k \rtimes \GL_k(\FF_p)$  on  $\FF_p^k$ induces an injection \mbox{$\FF_p^k \rtimes \GL_k(\FF_p) \hookrightarrow \Sigma_{p^k}$,} and it is well-known and not hard to check that its image is precisely given by $N_{\Sigma_{p^n}}(\FF_p^n)$.
 
More generally, we consider  the semidirect  product $(\FF_p^k)^{ m} \rtimes (\GL_k(\FF_p) \times \Sigma_m) $, where $\GL_k(\FF_p)$ acts diagonally via the standard action and $\Sigma_m$ permutes the various coordinates of  $(\FF_p^k)^{ m} $. The standard action of this group  on   $\mathbf{n} = \coprod_{i=1}^m \FF_p^k$ gives rise to an embedding \mbox{$(\FF_p^k)^{ m} \rtimes (\GL_k(\FF_p) \times \Sigma_m) \ \subset\ \Sigma_{n}$.} We can check on generators that 
$(\FF_p^k)^{ m} \rtimes (\GL_k(\FF_p) \times \Sigma_m) \ \subset \ N_{\Sigma_n}(\FF_p^k)$. Given any  $\tau \in N_{\Sigma_n}(\FF_p^k)$, we can find  $\sigma \in (\FF_p^k)^{m} \rtimes (\GL_k(\FF_p) \times \Sigma_m)$ for which $\sigma \tau$ sends each $\FF_p^k$-orbit  to itself, fixes the zero elements in each copy of $\FF_p^k$, and restricts to the identity on the first copy. These conditions imply that $\sigma \tau$ is the identity and hence $\tau \in (\FF_p^k)^{ m} \rtimes (\GL_k(\FF_p) \times \Sigma_m)$. Claim $(1)$ follows.\vspace{3pt}

For the nonequivariant  claim in $(1)$, we count that the index of $(\FF_p^k \rtimes \GL_k(\FF_p) )\times \Sigma_m$ inside 
 $(\FF_p^k)^{ m} \rtimes (\GL_k(\FF_p) \times \Sigma_m)$ is given by $(p^k)^{m-1}$ and combine this with  the well-known facts that $|\Pi_m| $ is nonequivariantly equivalent to a wedge sum of $(m-1)!$ spheres of dimension $m-3$, that $|\BT(\FF_p^k)|$ is  nonequivariantly equivalent to a wedge of $p^{{k\choose 2}}$ spheres of dimension $k-2$.
\end{proof}
 \begin{minipage}{0.85 \textwidth}\
\begin{example}
For   $p=5$, $k=2$, and $ m=3$, we have $n = 75$ and the space
$|\Pi_{75}|^{\FF_5^2}$ of $\FF_5^2$-invariant partitions of the set $\{1,\ldots,75\}\cong \coprod_3 \FF_5^2$ illustrated on the right  splits nonequivariantly as a bouquet of $6250$ copies of the $2$-dimensional sphere $S^2$.\\
\end{example} 
\end{minipage}  \ \ \ \ \ \ 
 \begin{minipage}{0.3 \textwidth}
\includegraphics[width=0.5 \textwidth]{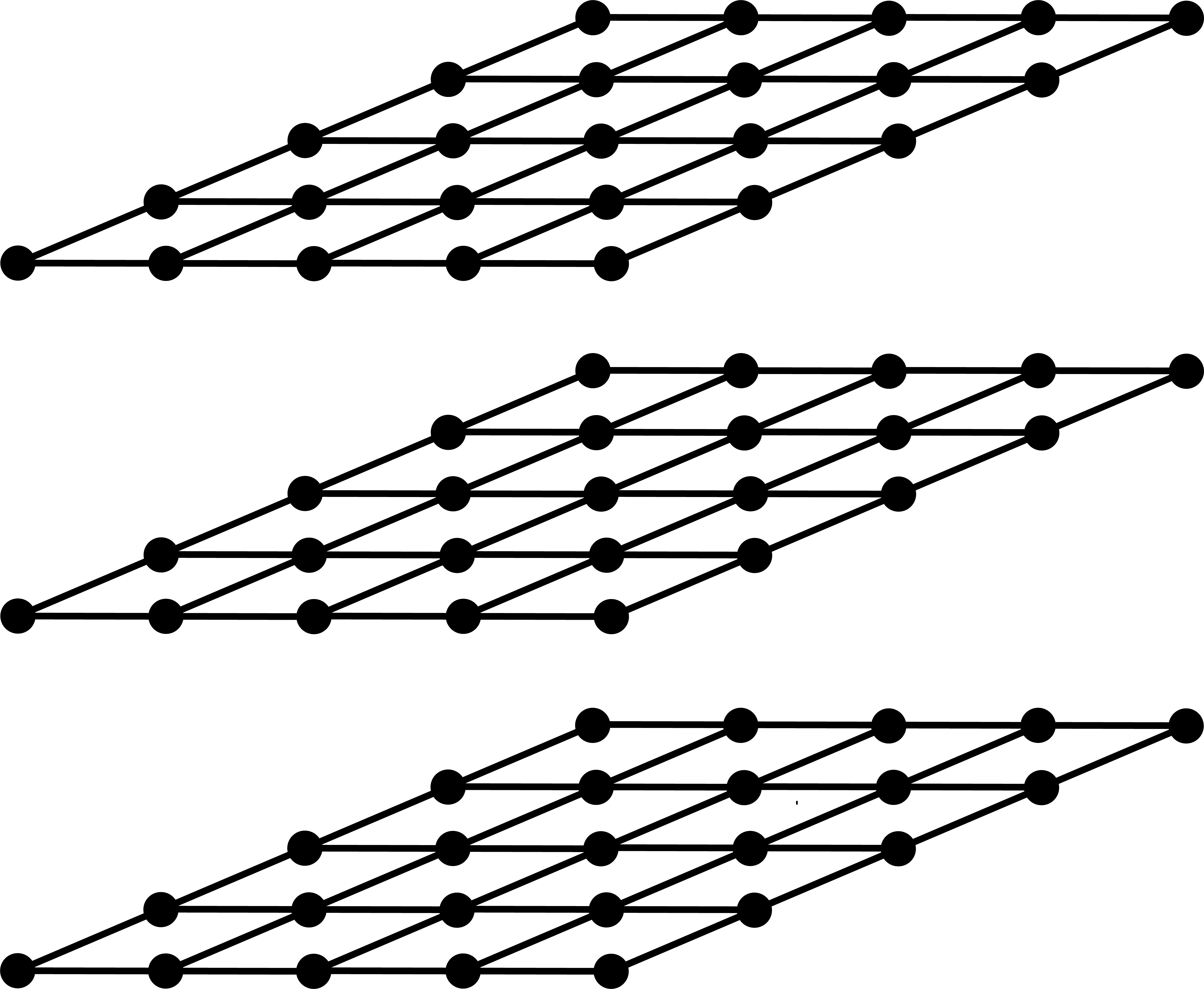}
\end{minipage}

\newpage
\section{An EHP Sequence for Commutative Monoid Spaces}
In this section, we will construct an EHP-like sequence for strictly commutative monoid spaces. This sequence will be a key tool in our subsequent study of strict Young quotients of the partition complex and our computation of the the algebraic Andr\'{e}-Quillen homology of trivial square-zero extensions over $\FF_p$ for $p$ odd. \vspace{-5pt} 
 
\subsection{Commutative Monoid Spaces and Simplicial Commutative Monoids}\label{indexingmonoid}
We begin by setting up the theory of commutative monoid spaces. 
Heuristically, it is often helpful to think of commutative monoid spaces as simplicial commutative rings over the field with one element.

We use point-set models in order to  facilitate our later definition of a new EHP-like sequence. Write $\mathbf{Top}_\ast$ for the category of pointed (compactly generated weak Hausdorff) spaces, endowed with the usual Quillen model structure. 
Our definitions and results will be given for \textit{graded}  commutative monoid spaces -- this will allow us to decompose the EHP-like sequence into  \mbox{graded components.}

We fix a commutative indexing monoid $J$ with identity $0\in J$ (in sets) once and for all. We write $\mathbf{Top}_\ast^J$ for the category of functors $X_{(-)}: J\rightarrow \mathbf{Top}_\ast$ and give it the model structure whose fibrations, cofibrations, and weak equivalences are defined pointwise. Objects $X\in \mathbf{Top}_\ast^J$ will be called  \textit{$J$-graded spaces}. We will often think of them as pointed spaces $X$ together with a wedge decomposition $X= \bigvee_{j\in J} X_j$, by gluing together all basepoint $0_j\in X_j$ to a single point $0$.

The  category  $\mathbf{Top}_\ast^J$ is endowed with a symmetric monoidal structure given by \textit{Day convolution}: we set $(X \wedge Y)_k  := \bigvee_{a+b = k} X_a \wedge Y_b $. The unit is given by $S^0$ concentrated in degree $0\in J$.
\begin{definition}
 
A  \textit{$J$-graded commutative monoid space}  is a commutative monoid object in the symmetric monoidal category $\mathbf{Top}_\ast^J$. We write $\mathbf{CMon}^J$ for the category of $J$-graded commutative monoid spaces and let $\mathbf{T}$ be the monad on $\mathbf{Top}_\ast^J$ building the free such monoid.
\end{definition}
 
We will now unravel this definition. A   $J$-graded monoid space $R$ consists of a  pointed space $R$ together with a decomposition $\bigvee_{j\in J} R_j= R$, a continuous multiplication map $\cdot : R\times R\rightarrow R$, and a distinguished element $1\in R_0$, such that the following three conditions hold true: 
\begin{enumerate}[leftmargin=26pt]
\item The multiplication is associative and commutative.
\item The element $1$ acts as a unit and the basepoint $0$ satisfies $0\cdot x = 0$ for all $x\in R$.
\item The multiplication respects the grading, i.e. $x \cdot y \in R_{j_1+j_2}$ for $x\in R_{j_1}$ and $y\in R_{j_2}$.
\end{enumerate}  
\begin{theorem}[Schw\"anzl --Vogt \cite{schwanzl1991categories}] 
The category $\mathbf{CMon}_{}^J$ of $J$-graded commutative monoid spaces admits the structure of a cofibrantly generated  
model category, where a map $f$ is a fibration or weak equivalences if and only if the underlying map of in $\mathbf{Top}_\ast^J$ has the corresponding property.
\end{theorem} 
  
Any cofibration $f$ in $\mathbf{CMon}^J$ can be written as a retract of a transfinite  composition of \textit{relatively free} maps $A\rightarrow B$, i.e. maps given by pushouts\vspace{-2pt} 
\begin{diagram}
& \mathbf{T}(X) & \rTo^{ \mathbf{T}(f)} &  &\mathbf{T}(Y) \\
& \dTo & & \SEpbk &  \dTo \\
& A & \rTo & &B
\end{diagram}
for $f : X \rightarrow Y $  pointwise a coproduct of generating cofibration of pointed spaces. 

Strict pushouts of $J$-graded commutative monoid spaces are computed by relative smash products: 
\begin{proposition}
If
$C\leftarrow A \rightarrow B$
is a span of $J$-graded commutative monoid spaces, then $B\mywedge{A} C$ is the pushout of $B \leftarrow A \rightarrow C$. Here $B\mywedge{A} C$ is defined as the coequaliser $\coeq(B\wedge A \wedge C \rightrightarrows B\wedge C) $, endowed with  the multiplication
$ (b_1,c_1) \cdot (b_2,c_2) = (b_1b_2,c_1c_2)$.
\end{proposition}

We can place $\mathbf{CMon}_{}$  in the framework of Batanin-Berger \cite{batanin2017homotopy} and White \cite{white2017model}. 
By (an evident graded generalisation of) Theorem 5.4. in \cite{white2017model} and Lemma 2.3.  in \cite{hovey1998monoidal},  the category  $\mathbf{Top}_\ast^J$ satisfies  White's  {strong commutative monoid axiom} and the monoid axiom, respectively.

We say that a  $J$-graded commutative monoid space is \textit{well-pointed} if  each of its components is a well-pointed space.  
Proposition $3.5$ in  \cite{white2017model} implies that if $M\rightarrow N$ is a cofibration of $J$-graded commutative monoid spaces and $M$ is well-pointed, then so is $N$. In particular, every cofibrant such monoid space is well-pointed. The following result appears as Theorem 4.18 in \cite{white2017model} (generalising Theorem 3.1 in \cite{batanin2017homotopy}):
\begin{theorem}[Batanin-Berger, White]
The model structure on  $\mathbf{CMon}^J$  is \textit{relatively left proper}. This means that given a  cofibration $A\xrightarrow{g} C$ and a weak equivalence $A\xrightarrow{f} B$ in $\mathbf{CMon}^J$ with $A,B$ well-pointed, the following pushout  gives rise to a weak equivalence $h$:
\begin{diagram}
A & \rTo^f  & B\\
\dTo^g & \SEpbk  &  \dTo \\
C & \rTo_h  &  D.
\end{diagram}

\end{theorem}

Relative left properness allows us to explicitly compute certain homotopy pushouts:
\begin{lemma} \label{pushout}
Assume we are given a pushout of $J$-graded commutative monoid spaces 
\begin{diagram}
A & \rTo & B\\
\dTo^j &  \SEpbk & \dTo \\
C & \rTo & D
\end{diagram}
with $j$ a cofibration and $A$, $B$, $C$ well-pointed. Then the square is a \textit{homotopy pushout} square.
\end{lemma}

\begin{proof}
We apply  the usual argument for left proper model categories in the relative setting.
\end{proof}

We  introduce some variants. Let $S^0$ be the monoid with two elements \mbox{  in degree $0$ (the unit of $J$).}

Define the category  $\mathbf{CMon}^{J,aug}$ of \textit{augmented $J$-graded commutative monoid spaces} as the overcategory $\mathbf{CMon}^J_{/S^0}.$ It inherits a natural model category structure from $\mathbf{CMon}^{J}$ such that the forgetful functor preserves fibrations, cofibrations, and weak equivalences.

We write $\mathbf{CMon}^{J,nu}$ for the category of \textit{nonunital} $J$-graded commutative monoid spaces, which are algebras over the monad $\mathbf{T}^{>0}(X) = \bigvee_{n>0} X^{\wedge n}/_{\Sigma_n}$. Once more, this category is endowed with a model category structure whose fibrations and weak equivalences are defined on the level \mbox{of spaces.}

The \textit{augmentation ideal functor} $I_{(-)} :\mathbf{CMon}_{ }^{J,aug} \rightarrow \mathbf{CMon}_{ }^{J,nu}$ takes an augmented monoid space $A\rightarrow S^0$ and assigns the preimage of the basepoint $0\in S^0$.
Every $J$-graded nonunital commutative monoid space $X$ gives rise to a unital augmented monoid space $S^0\vee  {X}$ by adding a disjoint unit $1$. The functors $I$ and $S^0 \vee (-)$ do \textit{not} assemble to an equivalence.

We can  carry out a similar construction for the symmetric monoidal model category $(\sSet_
\ast,\wedge, S^0)$ of pointed simplicial sets, following Quillen to define a model structure on its category $\mathbf{SCM}^J_{ }$ of  $J$-graded simplicial commutative monoids.
Weak equivalences and fibrations are  defined using the forgetful functor to $\sSet_\ast^J$. 
Again, there is are augmented/nonunital \mbox{versions  $\mathbf{SCM}^{J,aug}$, $\mathbf{SCM}^{J,nu}$.}

\subsection{Extension of Scalars}
Given a ring $R$ and a commutative monoid  space $X$, we will produce a simplicial commutative $R$-algebra $R\otimes X$. Heuristically speaking, we extend scalars from \mbox{$\FF_1$ to $R$.}
Consider the Quillen adjunction $\tilde{F}_R: \mathbf{sSet}_\ast  \leftrightarrows\mathbf{sMod}_R:U$, where  $U$ is the forgetful functor and \mbox{$\tilde{F}_R(\ast \rightarrow X)  =  \coker (\Free_{\Mod_R}(\ast) \rightarrow \Free_{\Mod_R}(X) )$}
is the reduced levelwise free $R$-module construction. This above adjunction is monoidal.
Passing to $J$-graded monoids, we obtain Quillen adjunctions
$ \mathbf{SCM}^J_{ } \leftrightarrows \mathbf{SCR}^J_{R}$ and
\mbox{$ \mathbf{SCM}^{J,aug}_{ } \leftrightarrows \mathbf{SCR}^{J,aug}_{R}$}, whose constituent left and right components will also be denoted by $\tilde{F}_R$ and $U$, respectively. Here $\mathbf{SCR}^J_{R}$  denotes the model category of   $J$-graded simplicial commutative $R$-algebras, i.e. commutative monoids in the category of functors $J\rightarrow \mathbf{sMod}_R$ (equipped with Day convolution) with fibrations and weak equivalences defined pointwise on underlying simplicial sets. The model category $\mathbf{SCR}^{J,aug}_{R}$ is a corresponding augmented variant. Note that no Koszul sign rule is imposed ``in the $J$-direction''.
\begin{definition}
The functor $R\myotimes{ }(-): \bfCMon^{J}_{ } \xrightarrow{\Sing_\bullet} \mathbf{SCM}^{J}_{ }\xrightarrow{\tilde{F}_R} \mathbf{SCR}^{J}_{R} $
given by composing singular chains  with the levelwise free $R$-module construction is called \textit{extension of scalars to $R$}.
We denote the corresponding functor $ \bfCMon^{J,aug}_{ } \xrightarrow{}\mathbf{SCR}^{J,aug}_{R}$ by the same name.
\end{definition}

\begin{proposition}\label{preserves}
The functor $R\otimes (-)$ preserves weak equivalences and homotopy colimits.
\end{proposition}
\begin{proof}
If $M\rightarrow N$ is a weak equivalence of $J$-graded commutative monoid spaces, then it is a weak equivalence of underlying $J$-graded spaces. The morphisms $\Sing_\bullet (M) \rightarrow \Sing_\bullet(N)$ and $(\tilde{F}_R \circ \Sing_\bullet) (M) \rightarrow (\tilde{F}_R \circ \Sing_\bullet)(N)$
are weak equivalences of simplicial sets and hence also weak equivalences of simplicial commutative monoids and simplicial $R$-algebras, respectively.

To see preservation of homotopy colimits, we observe that the functor $\Sing_\bullet: \bfCMon_{ }^J \rightarrow \mathbf{SCM}_{ }^J$ is the right half of a Quillen \textit{equivalence}, which implies that its right derived functor $R\Sing_\bullet$ preserves homotopy colimits.  Since  monoid spaces are fibrant, $\Sing_\bullet$ computes its own right derived functor and thus preserves homotopy colimits.
The functor $\tilde{F}_R$ is left Quillen, and so its left-derived functor $L\tilde{F}_R$ preserves homotopy colimits.
If $Q\rightarrow \id$ is the cofibrant replacement functor on $\mathbf{SCM}^J$, we use that $\tilde{F}_R$ preserves weak equivalences to see that
$L\tilde{F}_R(X)\cong\tilde{F}_R(QX) \xrightarrow{\sim} \tilde{F}_R(X) $
is a weak equivalence.
Hence, given   \mbox{$D: I \rightarrow \mathbf{SCM}_{ }^J$,} we compute $\tilde{F}_R(\myhocolim{I} D) \xleftarrow{\sim}  L\tilde{F}_R(\myhocolim{I} D)  \simeq  \myhocolim{I}   ((L\tilde{F}_R)\circ D ) \xrightarrow{\sim} \myhocolim{I} (\tilde{F}_R\circ D) $.
In the last step, we used that the pointwise equivalence of diagrams
$ \tilde{F}_R \circ Q \circ D \rightarrow  \tilde{F}_R \circ  D $
induces an equivalence on homotopy colimits.
\end{proof}    
\subsection{Monoid Modules over Commutative Monoid Spaces}
A monoid $R\in \mathbf{CMon}^J$ is simply an algebra object in $(\mathbf{Top}_\ast^J, S^0, \wedge)$, and so there is   a naturally defined theory of $R$-modules. We write $\mathbf{MMod}^J_R$ for the resulting category and call its members \textit{$J$-graded $R$-monoid modules}.
We unpack the definition: a $J$-graded $R$-{monoid module}  is a  pointed space $M$ together with a decomposition $M= \bigvee_{j\in J} M_j$ and continuous action maps
$ \cdot: R\times M \rightarrow M$ such that:
\begin{enumerate}[leftmargin=26pt]  
\item The action of $R$ on $M$ is associative.
\item The element $1\in R$ acts as the identity and  $0\cdot m = 0_M$ for all $m\in M$.
\item The multiplication respects the gradings, i.e. $r\cdot m \in M_{j_1+j_2}$ for $r\in R_{j_1}$ and $m \in M_{j_2}$.
\end{enumerate}  
We glued together all basepoints $0_{M,j}$ of $M_j$ to a single point $0_M$.

We can also think of $J$-graded $R$-monoid modules as  algebras over the monad $X \mapsto R\wedge X $.
\begin{warning}  
If $R$ is a topological ring, then the notion of an $R$-module is (evidently) \textit{not} equivalent to the notion of a monoid module over the underlying commutative monoid space.
\end{warning}    
The category  $\mathbf{MMod}^J_R$ carries a symmetric monoidal structure given by the relative smash product
$M\mywedge{R} N := \coeq(M\wedge R \wedge N \rightrightarrows M\wedge N) $. By applying Theorem $2.3.$ in  \cite{schwede2000algebras}, we have:   
\begin{theorem} [Schwede-Shipley]
The category $(\mathbf{MMod}_R^J, \mywedge{R}, R)$ of $R$-monoid objects in $\mathbf{Top}_\ast^J$ carries the structure of a cofibrantly generated monoidal model category where a map $f$ is a fibration or weak equivalences iff the underlying map of pointed spaces has this property. Moreover, $\mywedge{R}$ satisfies the monoid axiom.    
\end{theorem}
Again, cofibrations are  retracts of transfinite compositions of \textit{relatively free}  maps $A\rightarrow B$. In this context, these are maps which can be obtained as pushouts 
\begin{diagram}
&R\wedge X  & \rTo^{R\wedge f} & &R\wedge Y \\
&\dTo & & \SEpbk & \dTo \\
&A & \rTo & & B
\end{diagram}
for $f:X \rightarrow Y$ a pointwise coproduct of generating cofibrations of pointed spaces. We observe:
\begin{proposition}
The pushout of a diagram $M_1\leftarrow M_0\rightarrow M_2$ in $\mathbf{MMod}^J_R$ is given by taking the  pushout $\displaystyle M_1 \coprod_{M_0} M_2 $ in $\mathbf{Top}_\ast^J$ and endowing it with the evident \mbox{$R$-multiplication.}
\end{proposition}   
Given a $J$-graded $R$-monoid module $M$ and a $J$-graded pointed space $X$, we endow the $J$-graded space $M \wedge X$ with the structure of an $R$-monoid module by acting only on the  $M$-factor.

We call a $J$-graded $R$-monoid module \textit{well-pointed} if its underlying $J$-graded space \mbox{has this property.}
\begin{lemma}\label{flat}
Smashing with a cofibrant $J$-graded $R$-monoid module $M$ preserves weak equivalences between well-pointed $J$-graded $R$-monoid modules.   
\end{lemma}
\begin{proof}
  Say that $M$ has property ($\ast$) if it satisfies the conclusion of the theorem.
Since retracts of modules with property ($\ast$) have property ($\ast$), we may assume that $0\rightarrow M$ is given by a transfinite composition
$ 0 = M_0 \rightarrow M_1 \rightarrow \dots \rightarrow M $, where $M_i$ is obtained from $M_{i-1}$  through a pushout
\begin{diagram}
 &R\wedge A_i & \rTo^{ R\wedge f_i} &  &R\wedge B_i \\
&\dTo & & \SEpbk & \dTo \\
&M_{i-1} & \rTo && M_i &  \ \  .
\end{diagram}
Here the top map $f_i:A_i \rightarrow B_i $ is a coproduct of generating cofibration of $J$-graded pointed spaces.
 Since the filtered colimit of weak equivalences of $J$-graded pointed spaces is a weak equivalence, it suffices to check that if $M_{i}$ satisfies ($\ast$), then so does $M_{i+1}$.\\
We first observe that if $X$  is a  well-pointed $J$-graded space,  then $R\wedge X $  has property  ($\ast$).\\
We will now prove that if  $M_i$ satisfies ($\ast$), then so does $M_{i+1}$.
Let $S\rightarrow T$ be a weak equivalence of well-pointed $J$-graded $R$-monoid modules. We consider the cube
\begin{diagram}
S\mywedge{R} ( R\wedge A_i ) & & \rTo&  & S\mywedge{R} (  R\wedge B_i )& & \\
\dTo&   \rdTo^{\simeq}& && \dTo & \rdTo^{\simeq} & &  \\
&  &  T\mywedge{R}(   R\wedge A_i  )& \rTo & &  &T\mywedge{R} (  R\wedge B_i)&\\
& & & \ \   \SEpbk &  & & \dTo\\
S\mywedge{R} M_{i} &\rTo & & &  S\mywedge{R} M_{i+1} &  &  \\
&\rdTo^{\simeq}& \dTo & & \ \ \ \  \ \ \ \ \ \ \  \SEpbk& \rdTo   & & & &\\
& & T\mywedge{R} M_{i} & & \rTo& &  T\mywedge{R} M_{i+1}  &. & & 
\end{diagram} 
The three indicated arrows are weak equivalences by our previous argument and the induction hypothesis. The horizontal maps in the top square are given by smashing a cofibration of $J$-graded  well-pointed  spaces with the $J$-graded well-pointed  spaces $S$ and $T$ respectively, hence these two maps are cofibrations of $J$-graded well-pointed spaces.
Since pushouts of $J$-graded monoid modules are computed in $J$-graded pointed spaces, we conclude that the front and back square are  {homotopy pushouts} of pointed spaces. This implies that the lower right map is a weak equivalence of $J$-graded spaces, hence of $J$-graded $R$-monoid modules.  
\end{proof}
\subsection{Derived Smash Product of Monoid Modules}
We will now define an explicit model for the \textit{derived relative smash product} of graded monoid modules. For this, it will be convenient to use $ \Delta^k = \{ 0 \leq t_1 \leq \dots \leq t_k \leq 1 \ \ | \ \ t_i\in \RR \}$ as a model for the $k$-simplex. We think of $\Delta^k_+$ as a $J$-graded pointed space concentrated in degree $0$.

Fix a well-pointed   $R\in \mathbf{CMon}^J$ and  two well-pointed monoid modules $M,N \in \mathbf{MMod}_R^J$.
For $k\in \NN$, we write points in the $J$-graded space $\Delta^k_+ \wedge M \wedge R^{\wedge k} \wedge N$ as
$$\Bigg\{  \Bigg(\begin{diagram}
\  0&\leq & t_1&\leq & \dots& \leq &  t_k &\leq &1 \ \\
  \ m&,     & r_1&   ,   & \dots&   ,  &r_k& ,&n \  
\end{diagram} \Bigg) \ \ \ \ \bigg| \ \  \ t_i\in \RR,  r_i\in R, m\in M, n\in N\Bigg\} $$
For the sake of readability, we will often drop the second clause of this expression. We define:
 
\begin{definition}\label{homotopypush}
Given
$M,N\in \mathbf{MMod}_R$, we define the space 
$$M \mywedgetwo{R}{\hobased} N =|\Barr_\bullet(M,R,N)| =\bigg( \bigvee_{k \geq 0}\Delta^k_+  \wedge M \wedge R^{\wedge n} \wedge N   \bigg)\bigg/\sim$$ 
Here $\sim$ divides out by face and degeneracy maps.
More explicitly, we quotient out by: \\ \\
$ \Bigg(\begin{diagram}
\  0&\leq& \dots &\leq & t_i&=& t_{i+1}& \leq&  \dots &\leq&1 \ \\
  \ m&,&\dots&,& r_i&, &  r_{i+1}&, &\dots&,&n \  
\end{diagram} \Bigg)
  \ \ \sim \ \  \Bigg(\begin{diagram}
 0& \leq&  \dots& \leq &  t_i& \leq &  \dots & \leq  & 1\\
 \ m&,&\dots&,& r_i \cdot r_{i+1}&,& \dots&,&n \ 
\end{diagram}\Bigg)$
\\ \\ \\ \ 
$\Bigg(\begin{diagram}
\  0  &\leq & \dots &\leq & t_i& \leq& t_{i+1}&   \leq &\dots& \leq &1 \ \\
  \ m&,& \dots&,&               r_i&, & 1&,&\dots&,&n \  
\end{diagram} \Bigg)
  \ \   \sim \ \  \Bigg(\begin{diagram}
 0& \leq   & \dots &\leq &  t_i& \leq  & t_{i+2} & \dots &\leq  & 1\\
 m&,&\dots&,& r_i&, &   r_{i+2} &, \dots  &, & n \ 
\end{diagram}\Bigg)$.
The space $\displaystyle M \mywedgetwo{A}{\hobased} N$ is naturally  a $J$-graded $R$-monoid module by multiplication  {from the left}.
\end{definition}
If the unit $S^0 \rightarrow R$ is a cofibration of   spaces and both $M$ and $N$ are well-pointed, then $M \mywedgetwo{R}{\hobased} N$ is well-pointed since the map $S^0\rightarrow M \mywedgetwo{R}{\hobased} N$ is the realisation of a levelwise cofibration between pointwise good $J$-graded simplicial pointed spaces.
 
\subsection{Cofibrant Replacement of Monoid Modules}
We can use the construction $``M\mywedgetwo{R}{\hobased} N"$ to define an explicit  {cofibrant replacement} functor.  
For this, let $R\in \mathbf{CMon}^J$ be a well-pointed graded commutative monoid space such that the unit $S^0\rightarrow R$ is a cofibrant map of spaces.
\begin{proposition} 
Given a well-pointed $J$-graded   $R$-monoid module $M$,  the monoid module $R \mywedgetwo{R}{\hobased} M $ is  cofibrant and the multiplication map
$ R \mywedgetwo{R}{\hobased}  M \rightarrow M  $ is a weak equivalence. 
\end{proposition}
\begin{proof}
For each $K\in \ZZ_{\geq 0}$, consider  $D_K = | \sk_K \Barr_\bullet(R,R,M)| =  \bigg( \bigvee_{k= 0}^K\Delta^k_+  \wedge  R \wedge R^{\wedge k} \wedge M   \bigg)\bigg/\sim$.
An  $R$-monoid module structure on $D_K$ can be defined using the  leftmost  copy of $R$.
We obtain a transfinite composition of $R$-monoid modules
$ R \wedge M= D_0 \hookrightarrow D_1 \hookrightarrow \dots \hookrightarrow R \mywedgetwo{R}{\hobased} M$.
Since transfinite compositions of cofibrations are cofibrations and the monoid module $D_0$ is  cofibrant, it suffices to prove that all of the above maps are cofibrations in order to establish the first claim.

Let $L_{K} \subset R^{\wedge K}$ denote the subspace  of all points
$\{ (r_1 ,   \dots ,  r_k)    \ |\  r_i\in R , \ \exists i \ \mbox{s.t.}\  r_i=1\}$.
There is a natural pushout of $J$-graded $R$-monoid modules 
\begin{diagram}
(\partial \Delta^K _+ \wedge R\wedge R^{\wedge K} \wedge M) \coprod_{\partial \Delta^K_+ \wedge R \wedge L_{K}\wedge M}   ( \Delta^K_+ \wedge R\wedge L_{K}\wedge M) & \rInto & & \Delta^K_+  \wedge R\wedge R^{\wedge K} \wedge M  \\
  & & &  \\
\dTo & & & \dTo \\
D_{K-1}& \rInto& &  D_{K} & .
\end{diagram} 
Since cofibrations are closed under cobase change, it is enough to prove that the top horizontal  map is a cofibration of $J$-graded monoid modules.  This map is obtained by 
starting with the cofibration of $J$-graded pointed spaces $(\partial \Delta^K  \wedge R^{\wedge K} \wedge M) \coprod_{\partial \Delta^K_+ \wedge L_{K}\wedge M}   ( \Delta^K_+  \wedge L_{K}\wedge M) \rightarrow \Delta^K_+  \wedge R^{\wedge K} \wedge M $
and applying the left Quillen functor $R\wedge(-)$ -- it is therefore a cofibration. The first claim follows.

To prove that the map $ R \mywedgetwo{R}{\hobased}  M \rightarrow M  $ is an equivalence, it suffices to show that this holds true for the underlying map of spaces. We note that this map is obtained from the augmented simplicial space
$X_\bullet:  \Barr_\bullet(R,R,M) \rightarrow M$.
The unit of $R$ gives rise to  a {backwards contracting homotopy} $S: X_n \rightarrow X_{n+1}$. This implies that we have a weak equivalence after applying $\myhocolim{\Delta^{op}}$. Since $S^0\rightarrow R$ is a cofibration, this homotopy colimit is  computed by the geometric realisation.   
\end{proof}
Since $M \mywedgetwo{R}{\hobased} N \cong M \mywedgetwo{R}{}(R \mywedgetwo{R}{\hobased} N) $, we see that the derived smash product deserves its name. 
\subsection{Derived Pushout of Commutative Monoid Spaces }
We construct the derived pushout:
\begin{definition}
If $B \leftarrow A \rightarrow C$ is a span of well-pointed   $J$-graded commutative monoid spaces, we endow 
the space $B \mywedgetwo{A}{\hobased} C$ with a multiplication satisfying
\\
\\
$\Bigg(\begin{diagram}
   0 &\leq  &    t_{i_1 }&\leq& \dots & \leq &  t_{i_n}& \leq& 1\\
\ b_1&,& a_{i_1}&, & \dots &,& a_{i_n}&,& c_1 \ 
\end{diagram} \Bigg) \ \ \ \ \ \cdot \ \ \ \ \ \ 
\Bigg( \begin{diagram}
 0 & \leq &    t_{j_1 }&\leq& \dots & \leq  & t_{j_m}& \leq & 1\\
\ b_2& ,& a_{j_1}&, & \dots &, &a_{j_m}&, & c_2 \  
\end{diagram} \Bigg)   \ \ \ = \ \ \  \\ 
$
$$ _{} \ \ \ \ \ \ \ \ \ \ \ \ \  \ \ \ \ \ \ \ \ \ \   \ \ \ \ \ \ \ \ \ \ \ \ \ \ \  \ \ \ \ \ \ \ \ \ \ \ \ \ \ \ \ \  \ \ \ \ \ 
\Bigg( \begin{diagram}
  0 & \leq &    t_{1}&\leq& \dots & \leq &  t_{n+m} & \leq & 1\\
\ b_1b_2&,& a_{1}&, & \dots &, & a_{n+m}&, &c_1c_2 \  
\end{diagram} \Bigg) 
$$\\
for any disjoint union 
$\{1<\dots< m+n\} \   \ =\  \  \{i_1 < \dots < i_n\} \coprod \{j_1 < \dots < j_m\}  $.

Observe that this multiplication is well-defined, commutative and associative.
\end{definition}
\begin{proposition}
Assume we are given a span of well-pointed   $J$-graded commutative monoid spaces $B \leftarrow A \rightarrow C$ and that the unit of $A$ is a cofibration of spaces. The square
\begin{diagram}
A &  \rInto & A\mywedgetwo{A}{\hobased} C \\
\dTo & \SEpbk& \dTo \\
B    & \rInto & B \mywedgetwo{A}{\hobased} C
\end{diagram}
is a pushout, and in fact a homotopy pushout in $\mathbf{CMon}_{ }^J$.
\end{proposition}
\begin{proof}
It is evident hat the square is indeed an ordinary pushout square.
In order to prove the second claim, we factor $(A\xrightarrow{g} B) = (A \xrightarrow{g'} B' \xrightarrow{w} B)$ as a cofibration followed by a trivial fibration in  $\mathbf{CMon}^J_{}$. The fact that $A$ is well-pointed and $A \xrightarrow{g'} B' $ is a cofibration implies that $B'$ is well-pointed. The space $ A\mywedgetwo{A}{\hobased} C$ also has this property.  
We obtain a diagram  
\begin{diagram}
A & \rTo & A\mywedgetwo{A}{\hobased} C \\
\dTo & \SEpbk_{\hobased}& \dTo \\
B' & \rTo & B' \mywedgetwo{A}{\hobased} C\\
\dTo & & \dTo \\
B & \rTo & B\mywedgetwo{A}{\hobased} C & .\\
\end{diagram}
The upper square is a homotopy pushout by Lemma \ref{pushout} since  $A \rightarrow B'$ is a cofibration. 

The (underlying) map of \textit{$A$-monoid modules} $B'\mywedgetwo{A}{\hobased} C   \rightarrow B\mywedgetwo{A}{\hobased}C$ 
can be obtained by smashing the weak equivalence  $B'\rightarrow B$ of well-pointed $A$-monoid modules with the cofibrant $J$-graded $A$-monoid module $(A\mywedgetwo{A}{\hobased} C)$.  
This implies by Lemma \ref{flat} that  $B'\mywedgetwo{A}{\hobased} C   \rightarrow B\mywedgetwo{A}{\hobased} C$  is a weak equivalence of $A$-monoid modules and hence also of graded commutative monoid spaces. 
\end{proof}
We now turn to \textit{augmented} commutative monoid spaces. As explained in  Lemma $1.3.$ of \cite{hirschhorn2015overcategories}, pushouts are computed in monoid spaces.

\begin{definition}
The suspension $\Sigma^\otimes A\in \mathbf{CMon}^{J,aug}$ of an augmented well-pointed   $J$-graded  commutative monoid space $A\rightarrow S^0$ is given by $\displaystyle S^0 \mywedgetwo{A}{\hobased} S^0$.
\end{definition}
\begin{remark} {This construction also appears in the work of 
{McCord} \cite{mccord1969classifying} and 
Kuhn \cite{kuhn2003mccord}}.
It is a model for the suspension in the pointed $\infty$-category of augmented   commutative \mbox{monoid spaces.}
\end{remark}

\begin{definition} A sequence $A\rightarrow B \rightarrow C$ in $\mathbf{CMon}^{J,aug}$ (or in $\mathbf{SCR}^{J,aug}$)  is a \textit{cofibre sequence} if the following square is a homotopy pushout: 
\begin{diagram}
A & \rTo & B\\
\dTo & \SEpbk_{\hobased}& \dTo\\
S^0 & \rTo & C .
\end{diagram} 
\end{definition} 
We can detect cofibre sequences of monoid spaces algebraically:
\begin{lemma} \label{integralcheck}
Let $ A \xrightarrow{} B \rightarrow C$
be a sequence of  monoid spaces in $\mathbf{CMon}^{J,aug}$
for which the induced sequence $ \ZZ\myotimes{}A \rightarrow \ZZ\myotimes{}B \rightarrow \ZZ\myotimes{}C $
is a cofibre sequence in $\mathbf{SCR}^{aug}$.\\
Then $ A \rightarrow B \rightarrow C$ is a cofibre sequence of $J$-graded augmented commutative monoid spaces.
\end{lemma}
\begin{proof}
We  take the cofibre of the first map, i.e. the homotopy pushout of $f$ against  $A\rightarrow S^0$:
\begin{diagram}
A & \rTo_f & B &\rTo & C \\
 &           &  & \rdTo & \uTo_k  \\
 & &    &   &   \cofib(f)
\end{diagram}
The original sequence is a homotopy cofibre sequence if and only if $k$ is a weak equivalence of $J$-graded spaces.  We apply the reduced integral chains functor $\ZZ\otimes (-)$ and obtain a diagram
\begin{diagram}
\ZZ \myotimes{ } A & \rTo_{\ZZ \myotimes{}f} & \ZZ \myotimes{ }B &\rTo & \ZZ \myotimes{ } C \\
 &           &  & \rdTo & \uTo_{ \ZZ \myotimes{ } k}  \\
 & &    &   &   \ZZ \myotimes{ }\cofib(f).
\end{diagram}
As  $\ZZ\myotimes{} (-)$ preserves homotopy colimits (cf.\ Proposition \ref{preserves}), we see that
\mbox{$ \ZZ \myotimes{}\cofib(f) \simeq \cofib(\ZZ \myotimes{} f) $.}
Our assumption implies that $\ZZ \myotimes{}k$ is a weak equivalence, and hence $k$ induces an isomorphism on integral homology.
Commutative monoid spaces are in particular $H$-spaces, and therefore simple, so Whitehead's theorem implies that $k$ is a weak equivalence.
\end{proof}

\subsection{Derived Pullbacks of Commutative Monoid Spaces}
The theory of pullbacks is substantially more straightforward since fibrations are transported from $J$-graded pointed spaces.

\begin{definition}\label{hpbk}
Given a cospan $B \xrightarrow{f} A \xleftarrow{g} C$ of $J$-graded   commutative monoid spaces, we write $A^I$ for the $J$-graded space whose $j^{th}$ component is $A_j^I := \Map_{\mathbf{Top}}([0,1], A_j)$.
We endow the $J$-graded space $B\mytimestwo{A}{\hobased} C $ defined by
$(B\mytimestwo{A}{\hobased} C)_j := \left\{(b, \upalpha, c) \in B_j\times A^I_j \times C_j \  \  |  \ \  \upalpha(0) = f(b), \upalpha(1) = g(c) \right\} $
with a commutative monoid structure given by 
$(b_1, \upalpha_1, c_1) \cdot (b_2, \upalpha_2, c_2) = (b_1 b_2, \upalpha_1 \upalpha_2, c_1c_2) $.
Here the path $\upalpha_1 \upalpha_2 $ is defined using pointwise multiplication in $A$.
\end{definition} 
\begin{proposition} The homotopy pullback of a 
cospan $B \xrightarrow{f} A \xleftarrow{g} C$ of $J$-graded   commutative monoid spaces is given by $B\mytimestwo{A}{\hobased} C $ (in the sense of Definition \ref{hpbk} above).

\end{proposition}
\begin{proof}
We define $D$ by a homotopy pullback in $J$-graded commutative monoid spaces:
\begin{diagram}
B\mytimestwo{A}{\hobased} C & &  & &  \\
& \rdTo(2,4)  \rdDashto \rdTo(4,2)&  &  &\\
&& D\SEpbk&\rTo  &  C\\
 & & \dTo &{}^h& \dTo \\
 &  & B& \rTo & A
\end{diagram} 
As the forgetful functor $U: \bfCMon^J_{} \rightarrow \mathbf{Top}^J_\ast$ is right Quillen, its right derived functor preserves homotopy pullbacks. The functor $U$ in fact computes its own right derived functor since all graded commutative monoid spaces are fibrant. Thus, $U$ preserves homotopy pullbacks. After applying $U$, the dashed arrow is a 
weak equivalence   since $B\mytimestwo{A}{\hobased} C$ computes the \mbox{homotopy pullback in $\mathbf{Top}^J_\ast$.}
\end{proof}

\begin{definition}
Given a $J$-graded commutative monoid space  $A \in \bfCMon^J_{}$, the 
homotopy pullback 
$S^0 \mytimestwo{A}{\hobased }S^0  $
is denoted by $\Omega^{\otimes} A$.
\end{definition}
\begin{remark}
The monoid $ \Omega^{\otimes} A$  models the loop object of $A$  in the pointed $\infty$-category of augmented $J$-graded commutative monoid spaces.
\end{remark}
\begin{remark}
If $A\rightarrow S^0$ has fibres $A^0$ and $A^1$ over $0$ and $1$, then 
the underlying  $J$-graded space of $\Omega^{\otimes} A$ is given by $\Omega A = \Omega A^0  \coprod \Omega A^1$.\end{remark} 
\subsection{The Suspension-Loops Adjunction}
We will now set up a Quillen adjunction 
$$\Sigma^{\otimes} : \bfCMon_{}^{J,aug} \leftrightarrows \bfCMon_{}^{J,aug} : \Omega^{\otimes}$$
by defining the unit $\eta_A$ and the counit  $\epsilon_A$ by
by 
$$\eta_A: A \rightarrow \Omega^{\otimes} \Sigma^{\otimes} A  \ \ \mbox{ by } \ \ a \ \ \ \mapsto  \  \ \ \upalpha_a := \Bigg((0\leq s \leq 1)\mapsto \Bigg(\begin{diagram}
\  0&\leq & s&\leq &1 \ \\
  \ & & a &  \  
\end{diagram} \Bigg)\Bigg)$$
$$\epsilon_A:  \Sigma^{\otimes} \Omega^{\otimes} A\rightarrow A\ \ \mbox{ by } \ \ \Bigg(\begin{diagram}
\  0&\leq & t_1&\leq& \dots& \leq  &  t_n  &\leq &1 \ \\
  \ & & \upalpha_1&, & \dots&, &\upalpha_n& &  \  
\end{diagram} \Bigg)  \ \  \ \ \mapsto  \ \ \  \ \upalpha_1(t_1) \cdot \ldots \cdot \upalpha_n(t_n)$$
It is readily verified that $ \epsilon_{ \Sigma^{\otimes}} \circ \Sigma^{\otimes}\eta $ and $\Omega^{\otimes} \epsilon \circ \eta_{\Omega^{\otimes}}$ are indeed given by the identity transformations and we have therefore defined an adjunction. 
\begin{lemma} 
The adjunction $ \Sigma^{\otimes} : \bfCMon_{}^{J,aug} \leftrightarrows \bfCMon_{}^{J,aug} : \Omega^{\otimes} $ 
is Quillen. 
\end{lemma}
\begin{proof}\
Let $f: X\rightarrow Y$ be a fibration or acyclic fibration in $\bfCMon_{}^{J,aug}$. Then $f$ is a fibration or acyclic fibration of underlying graded spaces. Write $f_0: X_0 \rightarrow Y_0$ and $f_1: X_1 \rightarrow Y_1$ for the maps induced on fibres over $0$ and $1$. The map $\Omega^\otimes (f)$ is given on graded spaces by $\Omega(f_0) \coprod \Omega(f_1)$ and hence a fibration or acyclic fibration of graded spaces. This implies that $\Omega^\otimes (f)$ is a fibration or acyclic fibration of commutative monoid spaces. 
\end{proof}
\subsection{Andr\'{e}-Quillen Homology for Commutative Monoid Spaces}
We will follow Quillen's general approach in the context of $J$-graded commutative monoid spaces. 
\begin{definition}
The \textit{indecomposables functor} $Q: \bfCMon_{}^{J,nu} \rightarrow \mathbf{Top}_\ast^J$ 
assigns to a $J$-graded nonunital commutative monoid space $A$ the quotient space
$ Q(A) = A /{ A\cdot A}$

Here  $A\cdot A \subset A$ is the $J$-graded space consisting of all elements which can be decomposed into a product of two elements in $A$. We can also form square-zero extensions:\ 
\end{definition}
 
\begin{definition}
Given a  $J$-graded space $X$, write $\ovA{X}$ for the nonunital $J$-graded commutative monoid space obtained by declaring all products $x\cdot y$ of points in $X$ to be equal to $0$.\ 
\end{definition} 
As expected, these functors form a Quillen adjunction 
$\begin{diagram}
Q: \bfCMon_{}^{J,nu}  & & \pile{  \rTo^{} \\  \lTo_{}}&  & \mathbf{Top}^J_\ast : \ovA{(-)} 
\end{diagram}$.
 
\begin{definition}
The \textit{Andr\'{e}-Quillen chains} of a {nonunital} $J$-graded commutative monoid space $A \in \bfCMon_{}^{J,nu}$ are the value of  the left derived functor $L(Q)$ at $A$, i.e.\ \mbox{$\AQ(A) = L(Q)(A)\in \Ho(\mathbf{Top}^J_\ast) $.}
The \textit{Andr\'{e}-Quillen chains} of an {augmented} $L$-graded commutative monoid space are given by $\AQ(I_A)$, where $I_A$ denotes the augmentation ideal of $A$ (i.e. the fibre over $0$).

The \textit{Andr\'{e}-Quillen homology} of an augmented monoid space 
is given by $\widetilde{\HH}_\ast^Q(A) := \pi_\ast(\AQ(A))$. 
\end{definition} 
We give a formula for the Andr\'{e}-Quillen chains of a graded commutative monoid space. Write $\mathbf{T}^{>0}$ for the monad building the free $J$-graded nonunital commutative monoid space. Then: 
\begin{proposition}\label{abc}
If $A\in \bfCMon^{J,nu}_{}$ is a nonunital commutative monoid space, then the Andr\'{e}-Quillen chains of $A$ are given by 
$ \AQ(A) \simeq \myhocolim{\Delta^{op}}( \Barr_\bullet (1,\mathbf{T}^{>0},A))$.
\end{proposition}
\begin{proof}
The $J$-graded  nonunital  simplicial commutative monoid space
$\Barr_\bullet (\mathbf{T}^{>0},\mathbf{T}^{>0},A)\rightarrow A  $
has a contracting homotopy. Hence
$\myhocolim{\Delta^{op}}( \Barr_\bullet (\mathbf{T}^{>0},\mathbf{T}^{>0},A))\cong A$.
Since the left derived functor $LQ$ preserves homotopy colimits, we compute
$$ \AQ(A) = L(Q)(A)\simeq \myhocolim{\Delta^{op}}\ L(Q) ( \Barr_\bullet (\mathbf{T}^{>0},\mathbf{T}^{>0},A) )
 \simeq  \myhocolim{\Delta^{op}}\  ( \Barr_\bullet (1,\mathbf{T}^{>0},A) )\vspace{-10pt}$$ \vspace{-10pt}
\end{proof}  
We will now see that Andr\'{e}-Quillen chains for $J$-graded commutative monoid spaces behave well under ``base-change" to ordinary rings. 
We fix an ordinary ring $R$ and write \mbox{$Q:\mathbf{SCR}^{J,nu}_R \rightarrow \mathbf{sMod}^J_R$} for the indecomposables functor on nonunital $J$-graded simplicial commutative $R$-algebras. Recall:
\begin{definition}The Andr\'{e}-Quillen chains of a nonunital $J$-graded simplicial commutative  $R$-algebra $A$ are defined as
$\AQ^R(A):=L(Q)(A) \in  \Ho(\mathbf{sMod}_R)$.  We write $\AQ^R_\ast(A)$ for the homology and $\AQ_R^\ast(A)$ for the cohomology groups of $\AQ^R(A)$. \vspace{3pt}
\end{definition}
Let  $\mathbf{Sym}_R^{>0}(X) = \bigoplus_{n>0} X^{\myotimes{R} n}_{\Sigma n}$ be the free nonunital $J$-graded simplicial commutative $R$-algebra monad on $J$-graded simplicial $R$-modules. 
We can again give an explicit formula for the Andr\'{e}-Quillen chains as
$ \AQ^R(A) =\myhocolim{\Delta^{op}} \Barr_\bullet(1, \mathbf{Sym}_R^{>0} , A) $.

We recall that \textit{reduced} $R$-valued chains  $\tilde{C}_\bullet(-,R)$ interact well with symmetric powers, i.e. $$\tilde{C}_\bullet(\mathbf{T}^{>0}(X),R) \cong \mathbf{Sym}^{>0}(\tilde{C}_\bullet(X,R))$$
for all $J$-graded spaces $X$. Andr\'{e}-Quillen chains therefore intertwine with extension of scalars:
\begin{lemma}\label{AQSquare}
There is a commutative square
\begin{diagram}
\Ho (\bfCMon_{}^{J,nu})  & \rTo^{\AQ} & \Ho( \mathbf{Top}^J_\ast )\\ 
\dTo^{R\myotimes{}(-)} & & \dTo_{\tilde{C}_\bullet(-,R)} \\
 \Ho(\bfCAlg_{R}^{J,nu} ) & \rTo^{\AQ^R} & \Ho(\mathbf{sMod}^J_R)
\end{diagram}
\end{lemma}
\begin{proof}
Given a nonunital $J$-graded commutative monoid space $A$, we compute
$$\tilde{C}_\bullet(\AQ(A),R)\simeq
(\tilde{C}_\bullet(- ,R)\circ L(Q))(A) \simeq
\myhocolim{\Delta^{op}}  (\tilde{C}_\bullet(\Barr_\bullet (1,\mathbf{T}^{>0},A) ,R) )$$
$$ \ \ \ \ \  \ \ \ \ \  \simeq \myhocolim{\Delta^{op}} \   (\Barr_\bullet (1,{\mathbf{Sym}^{>0}},\tilde{C}_\bullet(A ,R) )\simeq
 \AQ^R(R\myotimes{} A)$$
Here we used that $\tilde{C}_\bullet(-,R)$ preserves homotopy colimits. 
\end{proof} 
\subsection{An EHP Sequence in $\mathbf{CMon}$}
In this section, we will construct a new cofibre sequence for $J$-graded commutative monoid spaces. The ``Hopf map"  is difficult to define, and this is where our point-set approach becomes helpful. We recall the following notation:
\begin{definition}
Given a $J$-graded space $X$, the \textit{trivial square-zero extension} of $S^0$ on $X$, written $S^0\vee\ovA{X}\in \mathbf{CMon}^{aug}$, is given by endowing the space $S^0\vee X$ with the multiplication  $$ x\cdot y = \begin{cases}  x & \mbox{ if } y = 1  \\ y & \mbox{ if } x = 1 \\ 0  & \mbox{ else }  \end{cases}$$
\end{definition}
Algebraic and topological square-zero extensions interact well, i.e.
$R\otimes (S^0 \vee\ovA{X}) = R \oplus \tilde{C}_\bullet(X,R)$.

The monoid suspension $\Sigma^{\otimes}$ of a trivial square zero extension has an explicit description:
\begin{proposition}
Given a $J$-graded pointed space $X$, there is a splitting of \textit{$J$-graded  spaces}
$$ \Sigma^{\otimes} (S^0 \vee \ovA{X})  \cong \bigvee_{n\geq 0} S^n \wedge X^{\wedge n}$$
\end{proposition}
\begin{proof}
This is evident from Definition \ref{homotopypush} since points in $X$ multiply to zero in  $S^0 \vee \ovA{X}$.
\end{proof}
For  the {Hopf map},  fix a $J$-graded pointed space $X$ and define a map $\phi$ of $J$-graded \mbox{pointed spaces  as}
$$S^1 \wedge  X^{\wedge 2} \xrightarrow{}   \Omega (S^2 \wedge X^{\wedge 2}) \hookrightarrow  \Omega ( \bigvee_{n\geq 1} S^n \wedge X^{\wedge n}  ) \cong \Omega^{\otimes} \Sigma^{\otimes} (S^0 \vee \ovA{X}) $$
$$(0\leq t \leq 1, x,y)  \ \ \  \ \ \  \mapsto \ \ \  \ \ \ 
\Bigg( s \mapsto \Bigg(\begin{diagram}
\  0&\leq & ts&\leq  & s  &\leq &1 \ \\
  \ & & x & &y& &  \  
\end{diagram} \Bigg)\Bigg) $$ 
\begin{proposition}
Given two points  $a_1, a_2 \in S^1 \wedge X^{\wedge 2} $, we have $\phi(a_1) \cdot \phi(a_2) = 0$.
\end{proposition}
\begin{proof}
Write $ a_i = (0 \leq t_i \leq 1, x_i, y_i )$ for $i=1,2$. Then: 
$$ \phi(a_1) \cdot \phi(a_2) =\Bigg(s \mapsto 
\Bigg(\begin{diagram}
\  0&\leq & t_1s&\leq&  s  &\leq &1 \ \\
  \ & & x_1&,  &y_1& &  \  
\end{diagram}  \Bigg)  \cdot 
\Bigg(\begin{diagram}
\  0&\leq & t_2s&\leq&  s  &\leq &1 \ \\
  \ & & x_2&,  &y_2& &  \  
\end{diagram}  \Bigg) \Bigg) =  \bigg(s \mapsto 0 \bigg) $$ 
We used that  a coordinate is repeated and  products  of elements in $X$ \mbox{vanish in $S^0 \vee \ovA{X}$.}
\end{proof}
We  get a map $\phi: (S^0 \vee \ovA{\Sigma X^{\wedge 2}}) \rightarrow \Omega^{\otimes} \Sigma^{\otimes} (S^0 \vee \ovA{X})$ of $J$-graded augmented \mbox{commutative monoid spaces.}
\begin{definition}
The \textit{Hopf map} for $J$-graded augmented commutative monoid spaces is given by the adjoint $\HH: \Sigma^{\otimes} (S^0 \vee \ovA{ \Sigma X^{\wedge 2}}) \longrightarrow 
\Sigma^{\otimes} (S^0 \vee \ovA{X}) $ of $\phi$.
\end{definition}  
More explicitly, $\HH$ is given by the following map:
$$\Bigg(\begin{diagram}
\  0&\leq &s_1 &\leq & \ldots & \leq&  s_n & \leq & 1 \ \\
  \  &      & (t_1, x_1,y_1) &    &  \ldots &  &  (t_n, x_n,y_n)&  &   \  
\end{diagram}  \Bigg)\ \ \ \ \ \  \ \ \ \ \ \ \ \ \ \ \ \ \ \ \ \  \ \ \ \ \ \  \ \ \ \ \ \ \ \ \ \ \ \ \ \ \ \ \ \ \ \ \ \  \ \ \ \ \ \ \ \ \ \ \ \ \ \ \ \ \ \ \ \ \ \  \ \ \ \ \ \ \ \ \ \ \ \ \ \ \ \ $$ $$ \ \ \ \ \ \  \ \ \ \ \ \ \ \ \ \ \ \ \ \ \ \ \mapsto \ \ \ \ \ \ 
\Bigg(\begin{diagram}
\  0&\leq & t_1s_1&\leq&  s_1  &\leq &1 \ \\
  \ & & x_1&,  &y_1& &  \  
\end{diagram}  \Bigg)  \cdot \ldots \cdot \Bigg(\begin{diagram}
\  0&\leq & t_ns_n&\leq&  s_n  &\leq &1 \ \\
  \ & & x_n&,  &y_n& &  \  
\end{diagram}  \Bigg) 
$$ 
\begin{proposition}\label{E1}
There is a commutative diagram
\begin{diagram}
S^1\wedge (\Sigma X^{\wedge 2}) & \rTo^\sim & S^2 \wedge X^{\wedge 2} \\
\dInto & & \dInto \\
\Sigma^{\otimes} (S^0 \vee \ovA{\Sigma X^{\wedge 2}})  & \rTo^{\HH}  &
\Sigma^{\otimes} (S^0 \vee \ovA{X})
\end{diagram}
The top map is the weak equivalence
$$\Bigg(\begin{diagram}
\  0&\leq &s &\leq &1 \ \\
  \  && (t, x,y) &&  \  
\end{diagram}  \Bigg)  \ \ \ \ \ \mapsto \ \ \ \ \ \ 
\Bigg(\begin{diagram}
\  0&\leq & ts&\leq&  s  &\leq &1 \ \\
  \ & & x&,  &y& &  \  
\end{diagram}  \Bigg) 
$$
\end{proposition}
We define the second map $\EE$ in our EHP sequence:
\begin{definition}
The  map  $\EE:  \Sigma^{\otimes} (S^0 \vee \ovA{X}) \rightarrow (S^0 \vee \ovA{\Sigma X})$ of $J$-graded augmented commutative monoid spaces
given by collapsing all higher summands to the basepoint is called the \textit{Einh\"{a}ngung}.
\end{definition}  
\begin{proposition} \label{E2}
Note that the following diagram commutes:
\begin{diagram}
\Sigma X & & \\
 \dTo & \rdTo&  \\
\Sigma^{\otimes} (S^0 \vee \ovA{X}) & \rTo & S^0 \vee \ovA{ \Sigma X}.& &
\end{diagram}
\end{proposition}

\begin{definition}\label{ehpdef}
The EHP sequence based at a $J$-graded space $X$ is given by the following sequence of $J$-graded augmented commutative monoid spaces:
$$ \Sigma^{\otimes}( S^0 \vee \ovA{ \Sigma X^{\wedge 2}}) \xrightarrow{\HH}
\Sigma^{\otimes} (S^0 \vee \ovA{X})\xrightarrow{\EE} S^0 \vee \ovA{\Sigma X} \ \ .$$
\end{definition}
\begin{proposition}
The composite $\EE \circ \HH$ factors through the zero object $S^0$.
\end{proposition}
\begin{proof}This follows immediately from our explicit description of the map $\HH$.
\end{proof}
Let $J = \NN$ be the monoid of nonnegative integers with +.
We prove  the main theorem \mbox{of this section:}

\begin{theorem}\label{EHPMonoid}
For $X=S^{n}$ a sphere of even dimension, thought of as an $\NN$-graded space placed in degree $1$, the EHP sequence 
$$ \Sigma^{\otimes} (S^0 \vee \ovA{S^{2n+1}}) \xrightarrow{\HH}
\Sigma^{\otimes} (S^0 \vee \ovA{S^{n}})\xrightarrow{\EE}  S^0 \vee \ovA{S^{n+1}}$$
is a homotopy cofibre sequence of augmented $\NN$-graded strictly commutative monoid spaces.

Here $S^{n+1} = \Sigma S^n$ is placed in degree $1$ and $ S^{2n+1} = \Sigma (S^n)^{\wedge 2}$ is placed in degree $2$. 
\end{theorem}
\begin{proof}  By Lemma \ref{integralcheck}, it suffices to prove that this sequence is a cofibre sequence after extension of scalars to $\ZZ$.
Hence, we    need to check that the following sequence of simplicial commutative rings is a cofibre sequence:
\begin{diagram}
 \Sigma^{\otimes} (\ZZ \oplus \Sigma^{2n+1}\ZZ)&  \rTo^{\ \ZZ \myotimes{ }\HH\ }& \Sigma^{\otimes} (\ZZ \oplus \Sigma^n \ZZ)& \rTo^{\  \ZZ \myotimes{ }\EE\ } &\ZZ \oplus \Sigma^{n+1}\ZZ.\\
\end{diagram}
Here we  made use of the weak equivalence $\tilde{C}_\bullet (S^0 \vee S^m) \cong \ZZ \oplus \Sigma^m \ZZ$.

Let $D$ be the cofibre of the first map.
By the folklore computation of Tor groups of graded exterior algebras (cf.\ e.g.\ \cite[Chapter 7]{mccleary2001user}), we have:
$$\hspace{-5pt} \pi_\ast ( \Sigma^{\otimes} (\ZZ \oplus \Sigma^{2n+1} \ZZ) ) = \Gamma[y_{2n+2}] \ \ \ \ \  \ \ \  \  \pi_\ast (\Sigma^{\otimes} (\ZZ \oplus \Sigma^n \ZZ) ) =\Lambda[z_{n+1}] \otimes \Gamma[y_{2n+2}]  \ \ \ \ \  \ \ \  \ \pi_\ast( \ZZ \oplus \Sigma^{n+1}\ZZ ) =\Lambda[z_{n+1}] $$
 Propositions \ref{E1} and \ref{E2} show  that $\ZZ\otimes \HH$ and $\ZZ\otimes \EE$ map elements to elements with  \mbox{same name.}
There is a spectral sequence
$E^2_{p,q}= (\Tor_p^{ \Gamma[y_{2n+2}] } (\ZZ,\Lambda[z_{n+1}] \otimes \Gamma[y_{2n+2}]) )_q \Rightarrow \pi_{p+q} (D)$.
Since $\Lambda[z_{n+1}] \otimes \Gamma[y_{2n+2}]$ is free over $\Gamma[y_{2n+2}]$,
 the edge homomorphism 
$ E^2_{0,q } = \ZZ \myotimes{ \Gamma[y_{2n+2}]} (\Lambda[x_{n+1}] \otimes \Gamma[y_{2n+2}]) \rightarrow \pi_q (D)$
is an isomorphism.
The composite 
$ \ZZ \myotimes{ \Gamma[y_{2n+2}]}(\Lambda[z_{n+1}] \otimes \Gamma[y_{2n+2}])  \rightarrow \pi_\ast (D) \rightarrow \pi_\ast( \ZZ \oplus \ZZ[n+1] ) \cong \Lambda[z_{n+1}]$
is evidently an isomorphism as well. This implies that $D \rightarrow  \ZZ \oplus \ZZ[n+1] $ is a weak equivalence of simplicial commutative rings.
\end{proof}
\newpage

\section{Strict Quotients}
Given a Young subgroup $\Sigma_{n_1} \times \dots \times \Sigma_{n_k} \subset \Sigma_n$, we can ask the following natural question:
\begin{question}
What is the homology and homotopy type of the \textit{strict} quotient 
$\Sigma |\Pi_n|^\diamond/_{\Sigma_{n_1} \times \dots \times \Sigma_{n_k}}\ \  ? $
\end{question}
In this section, we will  relate these quotients to the Andr\'{e}-Quillen homology of commutative monoid spaces and simplicial commutative rings, describe the conditions under which they are wedges of spheres,   use our EHP-like sequence and the branching rule to decompose them into simpler ``atoms'',  and  finally compute their  homology with coefficients in $\QQ$ and $\FF_p$ for any \mbox{prime $p$.}

\subsection{Atomic Decomposition} \ Given a commutative  indexing monoid $J$ in sets with unit $0$, we recall the symmetric monoidal category $\mathbf{Top}^J_\ast$ of $J$-graded pointed spaces from Section \ref{indexingmonoid}.

We consider the functor $C_{\Lie}: \mathbf{Top}^J_\ast \rightarrow   \mathbf{Top}^J_\ast$ \mbox{defined as}
$C_{\Lie}(X) =  \bigvee_{n\geq 1} \Sigma  |\Pi_n|^\diamond \mywedge{\Sigma_n} X^{\wedge n}$, where $ \Sigma  |\Pi_n|^\diamond$ is placed in degree $0\in J$.
In order to answer the above question, it evidently suffices to analyse $C_{\Lie}(X)$ for $X=S^{a_1} \vee \ldots \vee S^{a_k}$ a wedge sum of $k$  spheres, thought of as a $\ZZ^k$-graded space, where the spheres are placed  in degrees \vspace{3pt}$e_1 = (1,0,\ldots,0)$, $e_2= (0,1,\ldots,0),$ $\ldots,$ $e_k= (0,0,\ldots,1)$. 
We will now decompose the spaces $C_{\Lie}(S^{a_1} \vee \ldots \vee S^{a_k})$  into atomic building blocks of the form $ C_{\Lie}(S^\ell) = \bigvee_{n\geq 1} \Sigma  |\Pi_n|^\diamond \mywedge{\Sigma_n} (S^{\ell})^{\wedge n} $
for $\ell$ an odd natural number.

\subsubsection*{From wedges of spheres to spheres:}
Given pointed spaces $X_1\vee \ldots \vee X_k$, graded over $\ZZ^k$ by placing $X_i$ in multi-degree $e_i$, our Theorem \ref{main} on Young restrictions gives an equivalence 
$$ C_{\Lie}(X_1\vee\dots\vee X_k)  \cong \bigvee_{\substack{\ell_1,\dots,\ell_k\\w\in B(\ell_1,\dots,\ell_k) }} C_{\Lie}(S^{\ell_1 + \dots + \ell_k -1} \wedge X_1^{\wedge \ell_1} \wedge \dots \wedge X_k^{\wedge \ell_k} )$$
The  $d^{th}$ piece in the summand $C_{\Lie}(S^{\ell_1 + \dots + \ell_k -1} \wedge X_1^{\wedge \ell_1} \wedge \dots \wedge X_k^{\wedge \ell_k} )$  has multi-degree $(d\ell_1,\dots,d \ell_k)$.

\begin{corollary}\label{spheres-hm}
 If $X_i = S^{a_i}$, then we obtain an equivalence of $\NN^k$-graded spaces
$$ C_{\Lie}(S^{a_1} \vee \ldots \vee S^{a_k} )  \cong \bigvee_{\substack{\ell_1,\dots,\ell_k\\w\in B(\ell_1,\dots,\ell_k) }} C_{\Lie}(S^{\wedge (a_1+1) \ell_1 + \ldots + (a_k+1)\ell_k-1} )$$
\end{corollary}
Hence, it suffices to study the value of $C_{\Lie}$ on spheres to describe its behaviour on \mbox{wedges of spheres.}

Restricting to multi-degree $(n_1,\dots,n_k)$ with $n= \sum n_i$, we obtain the double suspension of the following statement (which also follows immediately from Theorem \ref{main}):
\begin{proposition}\label{proposition: orbits again}
Suppose $n=n_1+\cdots+n_k$.  There is an equivalence $$
|\Pi_n|/_{\Sigma_{n_1}\times\cdots\times\Sigma_{n_k}} \xrightarrow{\ \ \simeq \ \ } \bigvee_{\substack{d|\gcd(n_1, \ldots, n_k) \\ w \in B(\frac{n_1}{d}, \ldots, \frac{n_k}{d})}}  \left(\Sigma^{-1}(S^{\frac{n}{d}-1})^{\wedge d}\mywedge{\Sigma_d} |\Pi_d|^\diamond\right).
$$
More generally, given pointed spaces $X_1,\ldots,X_k$, there is an equivalence
$$ |\Pi_n|^\diamond \mywedge{\Sigma_n}(X_1\vee \ldots \vee X_k)^{\wedge n} \ \ \  \xrightarrow{\ \ \simeq \ \ } \ \ \  \bigvee_{\substack{n=n_1+\ldots+n_k \\ d|\gcd(n_1, \ldots, n_k) \\ w \in B(\frac{n_1}{d}, \ldots, \frac{n_k}{d})}} \left( (S^{\frac{n}{d}-1}\wedge X_1^{\frac{n_1}{d}} \wedge \ldots \wedge X_k^{\frac{n_k}{d}})^{\wedge d}\mywedge{\Sigma_d} |\Pi_d|^\diamond\right). $$
\end{proposition}

\begin{corollary}\label{corollary: gcd one}
If $\gcd(n_1, \ldots, n_k)=1$ then $
|\Pi_{n}|/_{\Sigma_{n_1} \times \dots \times \Sigma_{n_k}}  \simeq \bigvee_{B({n_1}, \ldots, {n_k})} S^{n-3}.$
\end{corollary}

\subsubsection*{From spheres to odd spheres:} 
To reduce to the case of odd spheres (which is a key simplification for our  cohomological computations), we need a conceptual understanding of the \mbox{functor $C_{\Lie}$,}
which  is closely related to \textit{square zero extensions} in commutative monoids by the following observation:
\begin{lemma}\label{strictlie}
If $X$ is a well-pointed $J$-graded space, then 
$\AQ(S^0 \vee\ovA{X})\simeq C_{\Lie}(X) $.\vspace{-5pt}
\end{lemma}
\begin{proof}
Write $\mathbf{O}^{nu}_{\Comm}$ for the reduced operad in $\mathbf{Top}_\ast$ whose component at a finite set $S$ is given by  $S^0$ if $|S|\neq 0$ and by $\ast$ if $S=\emptyset$. All structure maps are the identity.
We can think of $\mathbf{O}^{nu}_{\Comm}$ as an algebra object in the category $ {\SSeq}(\mathbf{Top}_\ast)$ of symmetric sequences in pointed spaces. Here  the \textit{composition product} serves as the monoidal structure. 

There is a natural monoidal functor $F:  {\SSeq}(\mathbf{Top}_\ast)\rightarrow \End(\mathbf{Top}_\ast^J)$ sending a symmetric sequence $\mathbf{O}$ to the functor $F_{\mathbf{O}}(-) = \bigvee_{n} \mathbf{O}_n \mywedge{\Sigma_n} (-)^{\wedge n}$.
We have $\mathbf{T}^{>0} = F_{\mathbf{O}^{nu}_{\Comm}} = \bigvee_{n\geq 1} (-)^{\wedge n}/_{\Sigma_n}$. Using Proposition \ref{abc}, we see that $\AQ(S^0 \vee\ovA{X})$ is equivalent to
$$\myhocolim{\Delta^{op}}\  \Barr_\bullet(1, F_{\mathbf{O}^{nu}_{\Comm}},\ovA{X}) \simeq \myhocolim{\Delta^{op}} \   F_{\Barr_\bullet({\mathbf{O}^{nu}_{\Comm}})} (\ovA{X}) \simeq \bigvee_{n\geq 1} \Sigma |\Pi_n|^\diamond \mywedge{\Sigma_n} X^{\wedge n} = C_{\Lie}(X).\vspace{-3pt}$$
In the second equivalence, we have used the well-known  identification (cf.\ \cite{fresse346koszul} \cite{ching2005bar}) between the simplicial set $(\Barr_\bullet({\mathbf{O}^{nu}_{\Comm}}))_n$ and the simplicial model for $\Sigma |\Pi_n|^\diamond$ described in\vspace{-5pt} Section \ref{clock}.
\end{proof}

Combining Lemma \ref{strictlie}   with our EHP sequence for commutative monoid spaces in Theorem \ref{EHPMonoid}, \mbox{we can   reduce the study of   $C_{\Lie}$ evaluated on spheres to the study  of $C_{\Lie}$ evaluated on \textit{odd} spheres:}
\begin{theorem}\label{EHPLIE}
For $X$ a $J$-graded space, there is a natural sequence of pointed $J$-graded spaces\vspace{-2pt}
$$\Sigma  C_{\Lie}(\Sigma X^{\wedge 2}) \xrightarrow{\HH}  \Sigma C_{\Lie}(X) \xrightarrow{\EE} C_{\Lie} (\Sigma X ).$$
For $X=S^n$ an even sphere in degree $1\in J=\NN$, we get a cofibre sequence of\vspace{-2pt} pointed \mbox{$\NN$-graded spaces}
$$\Sigma  C_{\Lie}(S^{2n+1}) \xrightarrow{\HH}  \Sigma C_{\Lie}(S^{n}) \xrightarrow{\EE} C_{\Lie} (S^{n+1}).\vspace{-1pt}$$
More explicitly, for each $d \in \NN$, there is a cofibre sequence
$$ \Sigma^2 |\Pi_{\frac{d}{2}}|^\diamond \mywedge{\Sigma_{\frac{d}{2}}} (S^{2n+1})^{\wedge \frac{d}{2}} 
\rightarrow  \Sigma^2 |\Pi_d|^\diamond \mywedge{\Sigma_d} (S^{n})^{\wedge d} \rightarrow  \Sigma |\Pi_d|^\diamond \mywedge{\Sigma_d} (S^{n+1})^{\wedge d}.\vspace{-1pt}$$
Here we use the convention that the space on the left is equal to the point if $d$ is  odd.
\end{theorem}
\begin{example}
We unpack the EHP sequence at weight $d=2$. Here, symmetric squares can be expressed in terms of real projective spaces (cf.\ \cite{james1963symmetric}) and we in fact recover the cofibre sequence $
\Sigma^{n+3} S^{n-1} \xrightarrow{\HH} \Sigma^{n+3} \RR P^{n-1} \xrightarrow{\EE} \Sigma^{n+3} \RR P^{n}  
$.
\end{example}

\begin{notation} Extending the above cofibre sequence gives a map  $P:  C_{\Lie} (S^{n+1})\xrightarrow{ }  \Sigma^2  C_{\Lie}(S^{2n+1})$.  
\end{notation}
\begin{remark} We suspect that  the above sequence is induced by applying a conjectural strict variant of Goodwillie calculus to the 
 classical EHP sequence in topology.
\end{remark}

\begin{proof}[Proof of Theorem \ref{EHPLIE}] The first claim follows by applying $\AQ$ to the EHP sequence in Definition \ref{ehpdef} and using that $\AQ$ preserves homotopy colimits to commute suspension past the $\AQ$ functor.
 
For the second claim, we   again use that $\AQ$ preserves homotopy colimits to deduce that  the EHP-cofibre sequence 
$ \Sigma^{\otimes} (S^0 \vee \ovA{S^{2n+1}}) \xrightarrow{\HH}
\Sigma^{\otimes} (S^0 \vee \ovA{S^{n}})\xrightarrow{\EE}  S^0 \vee \ovA{S^{n+1}}$
from Theorem \ref{EHPMonoid} gives   a homotopy cofibre sequence of $\NN$-graded spaces
$ \Sigma  \AQ (S^0 \vee \ovA{S^{2n+1}}) \xrightarrow{\HH}
\Sigma \AQ (S^0 \vee \ovA{S^{n}})\xrightarrow{\EE}  \AQ (S^0 \vee \ovA{S^{n+1}})$.
The result follows by applying Lemma \ref{strictlie}, and  the final  expression follows by passing to  degrees.
\end{proof}

We have   assembled $C_{\Lie}(S^{\ell_1} \vee \ldots \vee S^{\ell_k})$ from the simpler spaces
$C_{\Lie}(S^\ell)$ for $\ell$  {odd}. This reduces the problem of analysing $\Sigma |\Pi_n|^\diamond/_{\Sigma_{n_1}\times\cdots\times\Sigma_{n_k}}$ to spaces of the form  \vspace{3pt}$\Sigma |\Pi_d|^\diamond \mywedge{\Sigma_d} (S^\ell)^{\wedge d}$ \mbox{for $\ell \ge 1$ odd.}

\subsection{The space $C_{\Lie}(S^1)$} We will now  compute   the Andr\'{e}-Quillen homology of the \mbox{monoid $S^0 \vee \ovA{S^1}$:}
\begin{lemma}\label{lemma: contractible}
There is an equivalence of spaces
$C_{\Lie}(S^1) =\bigvee_{d\geq 1} \Sigma  |\Pi_d|^\diamond \mywedge{\Sigma_d}(S^1)^{\wedge d} \simeq S^1 $.
In fact, the space $ \displaystyle  |\Pi_d|^\diamond \mywedge{\Sigma_d} (S^1)^{\wedge d}$ is contractible for all $d>1$.\vspace{-2pt}
\end{lemma}
In our computations, we may therefore often assume  $\ell\neq 1$ (which implies that \mbox{$ \Sigma^{-1} |\Pi_d|^\diamond\mywedge{\Sigma_d} (S^{\ell})^{\wedge d}$} is simply connected). Moreover, we can give a new proof of  Kozlov's theorem \cite[Corollary 4.3]{kozlov2000collapsibility}:
\begin{corollary}[Kozlov]\label{Kozlov} The quotient $|\Pi_d|/\Sigma_d$ is contractible for all $d\geq 2$.\vspace{-5pt}
\end{corollary}
\begin{proof}
Combine Lemma \ref{lemma: contractible} with our EHP sequence (for $n=0$) in Theorem \ref{EHPLIE}.\vspace{-5pt}
\end{proof}

To prove Lemma \ref{lemma: contractible}, we need an auxiliary result: 
\begin{lemma}\label{lemma: preliminary contractible}
Let $Y:=\Sigma_{d_1}\times\cdots\times \Sigma_{d_i}$ be a non-trivial Young subgroup of $\Sigma_d$, where $d=d_1+\cdots+d_i$. Let $N$ be the normaliser of $Y$ and let $W$ be any group satisfying $Y\subseteq W \subseteq N$. Equip $S^{d-1}$ with the standard action of $\Sigma_d$. Then the orbit space $S^{d-1}/_W$ is contractible.\vspace{-5pt}
\end{lemma}
\begin{proof}
We start by observing the well-known fact that the orbits space $S^{d-1}/_{\Sigma_d}$ is contractible for $d>1$.  One way to see this is to identify $S^{d-1}$ with $|\ovA{\Bcal_d}|^\diamond$, where $\ovA{\Bcal_d}$ again denotes the poset of proper, non-trivial subsets of the set ${\mathbf d}$ (Example~\ref{example: sphere}). Then $S^{d-1}/_{\Sigma_d} \cong|\ovA{\Bcal_d}|^\diamond/_{\Sigma_d}.$ The quotient of the simplicial nerve of ${\Bcal_d}$ by the action of $\Sigma_d$ is isomorphic to the nerve of the linear poset $1<2<\cdots< d-1$. In particular, its geometric realisation is contractible.

For the general case, we first observe the $Y$-equivariant homeomorphism $S^{d-1}\cong S^{d_1-1}*\cdots*S^{d_i-1}$. It follows that 
$S^{d-1}/_{Y}\cong  S^{d_1-1}/{_{\Sigma_{d_1}}}*\cdots*S^{d_i-1}/{_{\Sigma_{d_i}}}.
$
By our assumption, at least one of the $d_j$s is greater than one, and so the right hand side is a contractible space. Moreover, we claim that it is contractible as a $N/Y$-equivariant space. In other words, we claim that for every subgroup $H\subset N/Y$, the fixed point space 
$(  S^{d_1-1}/_{\Sigma_{d_1}}*\cdots* S^{d_i-1}/_{\Sigma_{d_i}})^H$ 
is contractible. To see this, observe that $N/Y$ is a Young subgroup of $\Sigma_i$ (recall that $i$ is the number of factors $\Sigma_{d_j}$ of $Y$). $N/Y$ acts on the space  
$S^{d-1}\cong S^{d_1-1}/_{\Sigma_{d_1}}*\cdots*S^{d_i-1}/_{\Sigma_{d_i}}$ by permuting join factors that happen to be homeomorphic. It follows that the fixed point space $(  S^{d_1-1}/_{\Sigma_{d_1}}*\cdots* S^{d_i-1}/_{\Sigma_{d_i}})^H$  is a join of factors of the form $S^{d_j-1}/{_{\Sigma_{d_j}}}$. Again, at least one of the $d_j$ is greater than $1$, so at least one of these factors is contractible, and therefore the whole space is contractible.

It follows that the orbit space
$\left(S^{d-1}/_Y\right)/_{H}$
is contractible for every $H\subset N/Y$. Finally, it follows that for every $Y\subset W \subset N$, the orbit space $S^{d-1}/_W$
is contractible.\vspace{-1pt}
\end{proof}

\begin{proof}[Proof of Lemma \ref{lemma: contractible}]
The space $\displaystyle S^{d-1} \mywedge{\Sigma_d} |\Pi_d|^\diamond$ is a pointed homotopy colimit of spaces of the form $ S^{d-1}/_G$, where $G$ is an isotropy group of $ |\Pi_d|^\diamond $. We claim that the space $  S^{d-1}/_G$ is contractible for every $G$ that occurs. From the claim it follows that the pointed homotopy colimit is contractible.

It remains to prove the claim. Let $G$ be an isotropy group of $|\Pi_d|^\diamond$. Then either $G=\Sigma_d$ or $G$ is the stabiliser group of a chain of proper non-trivial partitions of $\mathbf d$. 

In the first case, we have $S^{d-1}/_G=S^{d-1}/{_{\Sigma_d}}\simeq *.$ In the second case, suppose that $G$ is the stabiliser of the chain of partitions $[x_0<\ldots <x_r]$. Here $x_0$ is the finest partition in the chains. Suppose $x_0$ has $i$ equivalence classes of sizes $d_1, \ldots, d_i$. Write $Y\cong \Sigma_{d_1}\times\cdots\times \Sigma_{d_i}$ for  the group of permutations that leave the equivalence classes of $x_0$ invariant. Then $G$ contains $Y$ and is contained in the normaliser of $Y$, and $S^{d-1}/_G$ is contractible by Lemma~\ref{lemma: preliminary contractible}.\vspace{-5pt}
\end{proof}

\subsection{Homology of Strict Orbits}
We proceed to examine the homology of strict Young quotients of the partition complex, or more generally of spaces of the form $C_{\Lie}(X)$ for $X$ a wedge of spheres.
We begin with a  conceptual interpretation of the homology of  spaces $C_{\Lie}(X)$. Given a ring $R$ and a $J$-graded simplicial $R$-module $M$, we write $R\oplus  {M}$ for the trivial square zero \mbox{extension of $R$ by $M$.}

 Lemma \ref{strictlie} and Lemma \ref{AQSquare} together imply:
 \begin{theorem}\label{surprise}
If $X$ is a well-pointed $J$-graded space and $R$ is a  ring, then 
$$  \widetilde{\HH}_*\left(C_{\Lie} (X), R\right) = \widetilde{\HH}_* \bigg(\bigvee_{d\geq 1} \Sigma  |\Pi_d|^\diamond \mywedge{\Sigma_d} X^{\wedge d} , R\bigg)\cong \AQ^R_\ast \left(R\oplus \tilde{C}_\bullet (X,R)\right).$$
\end{theorem} 
\subsubsection*{An aside on Operations.} We shall briefly digress and explain why the $R$-valued singular cohomology classes  of the spaces $C_{\Lie}(S^{\ell_1}\vee \ldots \vee S^{\ell_k})$  give rise to natural   operations on the algebraic Andr\'{e}-Quillen homology groups of simplicial commutative $R$-algebras.

Write $\Sigma $ for the endofunctor on simplicial $R$-modules  corresponding to the left shift on chain complexes. We observe that  Lemma \ref{AQSquare} and the standard adjunction between algebraic Andr\'{e}-Quillen chains $\AQ^R$ and trivial square-zero extensions together imply an identification 
$$H^j(C_{\Lie}(S^{\ell_1}\vee \ldots \vee S^{\ell_k}),R) \cong  \Map_{\hsMod_R}(\AQ^R(R\oplus (\Sigma^{\ell_1} R\oplus \ldots \oplus \Sigma^{\ell_k} R) ),\Sigma^j R)$$ 
$$\ \ \ \ \ \ \ \ \ \ \ \ \ \ \ \ \ \ \ \ \ \ \ \ \ \ \ \ \ \ \ \ \ \ \  \cong  \Map_{\hSCR_R^{aug}}(R\oplus (\Sigma^{\ell_1} R\oplus \ldots \oplus \Sigma^{\ell_k} R),R\oplus \Sigma^j R) .$$

Moreover, for any simplicial $R$-algebra $A$, we have an identification
$$\AQ_R^{\ell_1} (A)\times \ldots \times \AQ_R^{\ell_k} (A)  \cong \Map_{\hsMod_R}(\AQ^R(A),\Sigma^{\ell_1} R\oplus \ldots \oplus \Sigma^{\ell_k} R)$$ $$\ \  \ \ \ \ \ \ \ \ \ \ \ \ \ \ \ \ \ \ \ \ \ \ \ \ \ \ \ \ \ \ \ \ \cong \Map_{\hSCR_R^{aug}}(A,R\oplus (\Sigma^{\ell_1} R\oplus \ldots \oplus \Sigma^{\ell_k} R))   $$
We therefore obtain maps
$H^j(C_{\Lie}(S^{\ell_1}\vee \ldots \vee S^{\ell_k}),R) \times  \left( \AQ_R^{\ell_1} (A)\times \ldots \times \AQ_R^{\ell_k} (A)   \right)\longrightarrow  \AQ_R^j(A).$\\

We return to   computing the homology of  $C_{\Lie}(X)$ for $X$ a wedge of spheres.
\mbox{Our combinatorial} work in Section \ref{c4} and \ref{c5} has algebraic consequences. We apply $\widetilde{\HH}_*(-,R)$ to Corollary \ref{spheres-hm} \mbox{and deduce:}
\begin{corollary}\label{banterbanter}
There is a splitting
\begin{diagram}
\widetilde{\HH}_*(C_{\Lie}(S^{\ell_1}\vee \ldots\vee S^{\ell_k}),R) & \rTo^{\cong}& \bigoplus_{\substack{n_1,\dots,n_k\\w\in B(n_1,\dots,n_k) }} \widetilde{\HH}_*(C_{\Lie}(S^{(1+\ell_1)n_1 + \ldots + (1+\ell_k) n_k-1}),R) \\
\\
\dTo^{\cong} & & \dTo^{\cong} \\
\AQ^R_\ast (R\oplus (\Sigma^{\ell_1} R \oplus \ldots \oplus \Sigma^{\ell_k}R)) & \rTo^{\cong} & \bigoplus_{\substack{n_1,\dots,n_k\\w\in B(n_1,\dots,n_k) }} \AQ^R_\ast(R\oplus \Sigma^{(1+\ell_1)n_1 + \ldots + (1+\ell_k) n_k-1}R )
\end{diagram}
\end{corollary}
It is noteworthy that this nontrivial theorem in algebra follows formally from our discrete Morse theoretic computations. For $R=\FF_2$, it has been proven by Goerss \cite{goerss1990andre} by \mbox{entirely different means.}

We start by describing the homology of $C_{\Lie}(S^\ell)$ for $\ell$  \textit{odd}, and begin with rational coefficients:
\begin{lemma}\label{lemma: rational}
Let $\ell \geq 1$ be an \textit{odd} integer. Then $\widetilde{\HH}_*(\Sigma |\Pi_d|^\diamond\mywedge{\Sigma_d} (S^{\ell})^{\wedge d}, \QQ) = 
\begin{cases}  \QQ & \mbox{ if } d=1, \ast = \ell \\ 0 & \mbox{else}\end{cases} $.
\end{lemma}
\begin{proof}
We give one of several possible proofs. It suffices to show that \mbox{$\widetilde{\HH}_*(\Sigma^{-1} |\Pi_d|^\diamond \mywedge{\Sigma_d} (S^{\ell})^{\wedge d} ,\QQ)$} vanishes for $d>1$. 
The  natural map
$\Sigma^{-1} |\Pi_d|^\diamond \mywedge{\hobased\Sigma_d} (S^{\ell})^{\wedge d} \longrightarrow \Sigma^{-1} |\Pi_d|^\diamond \mywedge{\Sigma_d} (S^{\ell})^{\wedge d}$   is well-known to induce an isomorphism on rational homology. Therefore it is enough to prove the lemma with the homotopy orbit space replacing the strict orbit space. This was done in~\cite{arone1999goodwillie, arone2006note}. For the sake of completeness we will sketch a proof. 
The space $\Sigma^{-1} |\Pi_d|^\diamond \mywedge{\hobased \Sigma_d} (S^{\ell})^{\wedge d} $ is a pointed homotopy colimit of spaces  $ S^{\ell  d-1}_{\hobased G}$, where $G$ is an isotropy group of $ |\Pi_d|^\diamond $. By an easy Serre spectral sequence argument, such a space is rationally trivial if $\ell $ is odd and $G$ contains a transposition. It is clear that every isotropy group of $|\Pi_d|^\diamond$ contains a transposition.
\end{proof}

For $R=\FF_p$ with $p$ a prime, the computation is significantly more difficult.
The first-named author and Mahowald computed the homology  of the  {homotopy} orbits $\DD(\Sigma |\Pi_d|^\diamond )\mywedge{\hobased\Sigma_d} (S^\ell)^{\wedge d}$ for $\ell$ odd by using a bar spectral sequence whose input is  the homology of extended powers as computed  by Cohen-Lada-May \cite{cohen1978homology}, Araki-Kudo \cite{kudo1956topology}, and Dyer-Lashof \cite{dyer1962homology}.

We can in fact run a related  computation, the input of which is the homology of symmetric powers as computed by Dold \cite{dold1958homology}, Nakaoka \cite{nakaoka1957cohomology} \cite{nakaoka1957cohomology2}, and Milgram \cite{milgram1969homology}. 
We obtain:\vspace{-3pt}

\begin{theoremn}[\ref{theorem: thehomology}] Let $\ell\geq 1$ be a positive integer, odd if the prime $p$ is odd.\\
If $n$ is not a power of $p$, then $\widetilde{\HH}_*\left(\Sigma |\Pi_n|^\diamond \wedge_{\Sigma_{n}}S^{\ell n} ,\FF_p\right)$ is trivial.\\ If $n=p^a$, then $\widetilde{\HH}_*\left(\Sigma |\Pi_{p^a}|^\diamond \wedge_{\Sigma_{p^a}} S^{\ell p^a},\FF_p\right)$ has a basis consisting of sequences $(i_1, \ldots, i_a)$, where $i_1, \ldots, i_a$ are positive integers satisfying:
\begin{enumerate}[leftmargin=26pt]
\item Each $i_j$  is congruent to $0$ or $1$ modulo $2(p-1)$. \label{congruence}
\item For all $1\le j<a$ we have $1<i_j< p i_{j+1}$.
\item We have $1<i_a\le (p-1)\ell$ (notice that if $p>2$, then~\eqref{congruence} means that the inequality is   strict).
\end{enumerate}
The homological degree of $(i_1, \ldots, i_a)$ is $i_1+\cdots+i_a+\ell+a$.
\end{theoremn}
This theorem is proven in our final Section \ref{theproof}. 
We will now use  the   EHP sequence from Theorem \ref{EHPLIE} to express
the homology of  spaces $C_{\Lie}(S^\ell)$ with $\ell$ \textit{even} in terms of the homology groups of    $C_{\Lie}(S^j)$  with $j$ odd as computed in Theorem \ref{theorem: thehomology}.
We begin on the level of chains:
\begin{corollary}
For $j\geq 0$ even, there is a cofibre sequence of $\NN$-graded simplicial $R$-modules
\begin{diagram} 
\Sigma \tilde{C}_\bullet(C_{\Lie}(S^{2j+1}),R) & & \rTo^{\HH} & & \Sigma \tilde{C}_\bullet(C_{\Lie}(S^{j}),R)& & \rTo^{\EE}& &\tilde{C}_\bullet(C_{\Lie}(S^{j+1}),R) \\ 
\dTo & & & & \dTo & & &  & \dTo \\
 \Sigma \AQ^R(R\oplus \Sigma^{2j+1}) & & \rTo^{\HH}& & \Sigma \AQ^R(R\oplus \Sigma^j R) & &\rTo^{\EE}& & \AQ^R(R\oplus \Sigma^{j+1} R)\end{diagram}
\end{corollary}

We  analyse the effect of the EHP sequence on (co)homology with coefficients in  $R=\QQ$ and $R=\FF_p$:
\begin{theorem} \label{EHPhomology} Let $n\geq 0$ be   even.
For $R=\QQ$, or $R=\FF_p$ with $p$  odd, or $R=\FF_2$ and $n=0$, the sequence $\Sigma  C_{\Lie}(S^{2n+1})    \xrightarrow{\HH}   \Sigma C_{\Lie}(S^{n}) \xrightarrow{\EE} C_{\Lie}(S^{n+1})$ induces a short exact sequence of $\NN$-graded  vector spaces on (reduced)  \mbox{homology and cohomology.}

For $R=\FF_2$ and $n>0$, the extended EHP--sequence $\Sigma C_{\Lie}(S^{n})  \xrightarrow{\EE} C_{\Lie}(S^{n+1})  \xrightarrow{\PP} \Sigma^2 C_{\Lie}(S^{2n+1})$   induces a short exact sequence of $\NN$-graded  vector spaces on (reduced)  \mbox{homology and cohomology.}
\end{theorem}
\begin{remark} The cohomology of these spaces are in fact shifted Lie algebras, and for $p=2$ and $n>0$, the map $P^\ast$ sends  one fundamental class to the self-bracket of another \mbox{fundamental class.}
\end{remark}
\begin{proof}[Proof of Theorem \ref{EHPhomology}]
Recall from Theorem \ref{EHPLIE} that the EHP--sequences decomposes into a wedge of cofibre sequences
$ \Sigma^2 |\Pi_{\frac{d}{2}}|^\diamond \mywedge{\Sigma_{\frac{d}{2}}} (S^{2n+1})^{\wedge \frac{d}{2}} 
\rightarrow  \Sigma^2 |\Pi_d|^\diamond \mywedge{\Sigma_d} (S^{n})^{\wedge d} \rightarrow  \Sigma |\Pi_d|^\diamond \mywedge{\Sigma_d} (S^{n+1})^{\wedge d}$.
It is enough to check the claim for each of them.
For $R=\QQ$, the claim follows from \mbox{Lemma \ref{lemma: rational}.}

Now take $R=\FF_p$ with $p$   odd. 
If $d=p^k$,   the left hand side is contractible and the claim follows. If $d=2p^k$, Theorem \ref{theorem: thehomology} shows that the right hand side has vanishing $\FF_p$ homology and cohomology; the claim follows.
If $d$ is not of this form, both sides have vanishing $\FF_p$-homology by Theorem \ref{theorem: thehomology}.

For  $R=\FF_2$, we can read of  the cohomological case for from Theorem \ref{surprise} and  Corollary B.$2.$ in Goerss' \cite{goerss1990andre}, and   the homological case follows by applying the universal coefficient theorem. \mbox{Alternatively,} we could also prove this claim by unravelling the effect of the map $E$ on homology in Theorem \ref{finalthecohomology}  below, and check that it induces an injection for $n>0$.
\end{proof}
 
\begin{theorem}\label{finalthecohomology} 
Fix integers $\ell_1,\dots , \ell_k\geq 0$ and consider the homology group $$\widetilde{\HH}_*\bigg(\bigvee_{d\geq 1} \Sigma |\Pi_d|^\diamond \mywedge{\Sigma_d} (S^{\ell_1}\vee \dots \vee S^{\ell_k})^{\wedge d},R\bigg).$$
This group is given by the algebraic Andr\'{e}-Quillen homology $\AQ_\ast^R(R\oplus (\Sigma^{\ell_1} R \oplus \ldots \oplus \Sigma^{\ell_k} R))$ of the trivial square zero extension of $R$ by generators $x_{ 1}, \dots, x_{ k}$ in simplicial  degrees $\ell_1,\dots, \ell_k$.\vspace{3pt}

For $R=\QQ$, the above homology group has a basis indexed by pairs $(e,w)$, where $w\in B_k(n_1,\ldots,n_k)$ is a Lyndon word and $e=0$ if $|w|:= \sum_i (1+\ell_i)n_i-1$ is odd and  $e\in \{0,1\}$ if $|w|$ is even. 
 The homological degree of $(e,w)$ is $(1+e)|w| + e$ and it lives in multi-weight $(n_1  (1+e) ,\ldots,n_k  (1+e))$.\vspace{3pt}

For $R=\FF_p$, the above homology group has a basis indexed by sequences $(i_1, \ldots, i_a, e,w)$, where \mbox{$w\in B_k(n_1,\ldots,n_k)$} is a Lyndon word and  $e$ lies in  $ \{0,\epsilon\}$. Here
 $\epsilon = 1$ if $p$ is odd and $|w|$ is even or if $|w|=0$. Otherwise, we have $\epsilon = 0$. 
 The sequence $i_1, \ldots, i_a$ consists of positive integers satisfying:
\begin{enumerate}[leftmargin=26pt]
\item Each $i_j$ is congruent to $0$ or $1$ modulo $2(p-1)$. \label{congruence}
\item For all $1\le j<a$, we have  $1<i_j< p i_{j+1}$.
\item We have $1<i_a\le (p-1)(1+e)|w|+\epsilon$.
\end{enumerate}
The homological degree of $(i_1, \ldots, i_a,e,w)$ is $i_1+\cdots+i_a+(1+e)|w| + e+a$ and it lives in multi-weight $(n_1 p^a (1+e) ,\ldots,n_k  p^a (1+e))$. Note that $a=0$ is allowed.
\end{theorem}
\begin{proof}

The identification with  Andr\'{e}-Quillen homology is a special case of Theorem \ref{surprise}.
 
We first examine the case $R=\QQ$. 
For $k=1$, and $\ell_1$ odd, this follows from Lemma \ref{lemma: rational}. For $\ell_1$ even, it is implied by Theorem \ref{EHPhomology}. For $k>1$, the rational statement follows from Corollary \ref{banterbanter}.\vspace{3pt}

Now let $R=\FF_p$  with $p$ odd. 
If $k=1$ and $\ell_1\geq 1$ is  {odd}, the statement follows from Theorem \ref{theorem: thehomology}.
If   $\ell_1$ is even (and thus $\epsilon = 1$),  the  short exact sequence asserted for $p$ odd in Theorem \ref{EHPhomology} gives a direct sum decomposition 
$\widetilde{\HH}_*(C_{\Lie}(S^{\ell}),\FF_p) \cong \widetilde{\HH}_{\ast+1}(C_{\Lie}(S^{\ell+1}),\FF_p) \oplus  \widetilde{\HH}_{\ast}(C_{\Lie}(S^{2\ell+1}),\FF_p)$.

By the ``odd case'', the first summand has a basis consisting of all sequences $(i_1,\ldots,i_a)$ satisfying $(1)$, $(2)$, and the condition $i_a \leq (p-1)(\ell_1+1)$. In light of $(1)$, the condition $i_a \leq (p-1)(\ell_1+1)$'' is equivalent to $i_a \leq (p-1)\ell_1 + 1$. We therefore obtain all sequences with $e=0$ appearing in the theorem. The homological degree of this sequence in \mbox{$(i_1+ \ldots + i_a +\ell_1+1+a)$} in $\widetilde{\HH}_{\ast}(C_{\Lie}(S^{\ell+1}),\FF_p)$ and hence goes to an element of dimension  $(i_1+ \ldots + i_a +\ell_1+a)$ in $\widetilde{\HH}_{\ast}(C_{\Lie}(S^{\ell+1}),\FF_p)$.

The second summand has a basis consisting of all sequences $(i_1,\ldots,i_a)$ satisfying $(1)$, $(2)$, and the condition $i_a \leq (p-1)(2\ell_1+1)$. Using $(1)$ again, this condition is equivalent to $i_a \leq (p-1)2\ell_1 + 1$, and we therefore exactly obtain all sequences with $e=1$. The homological degree of the  sequence $(i_1,\ldots,i_a)$ is indeed $i_1+\ldots+ i_a +2\ell_1+1+a$. If $k>1$, the\vspace{2pt} statement follows by \mbox{Corollary \ref{banterbanter}}.

For  $R=\FF_2$, the second statement for $k=1$ and $\ell_1\geq 1$ follows directly from Theorem \ref{theorem: thehomology}. For $\ell_1 = 0$, we simply use Corollary \ref{Kozlov}. We deduce the case $k>1$ from Corollary \ref{banterbanter}.\vspace{-1pt}
\end{proof}
For $p=2$, we obtain an independent proof of Goerss' computation (cf.\ \cite{goerss1990andre}) of the algebraic Andr\'{e}-Quillen homology of trivial square-zero extensions over $\FF_2$. Note that in this case, the proof of Theorem \ref{finalthecohomology}  does not make any use of the EHP sequence. \vspace{4pt}

We call a sequence $(i_1,\ldots,i_a,e,w)$ \textit{allowable} with respect to $\QQ$ or $\FF_p$  if it satisfies the conditions in the theorem above  (here we use the convention that $a=0$ in the rational case).\vspace{3pt}

This concludes the computation of the homology of strict Young orbits of the partition complex, and we can finally answer the question we started with by setting $\ell_1 = \ldots = \ell_k = 0$:
\begin{corollary}\label{strqu}
Let $n=n_1+\ldots+n_k$.

The vector space $\widetilde{\HH}_*(|\Pi_n|/_{\Sigma_{n_1} \times \ldots \times \Sigma_{n_k}},\QQ)$ has a basis consisting of all sequences $(e,w\in B(m_1,\ldots,m_k))$ which are  allowable with respect to $\QQ$ and satisfy $ m_i (1+e)= n_i$ for all $i$.  

The vector space $\widetilde{\HH}_*(|\Pi_n|/_{\Sigma_{n_1} \times \ldots \times \Sigma_{n_k}},\FF_p)$ has a basis consisting of all  $\FF_p$-allowable sequences    $(i_1,\ldots,i_a,e,w\in B(m_1,\ldots,m_k))$  satisfying 
$m_i p^a (1+e)  = n_i$ for all $i$.

The sequence $(i_1,\ldots,i_a,e,w)$ sits in homological degree  $i_1+\ldots +i_a +(1+e) |w| + e + a -2$.
\end{corollary} 
\begin{corollary}\label{lemma: torsion}
The integral homology groups $\widetilde{\HH}_*(|\Pi_n|/_{\Sigma_{n_1}\times\cdots\times\Sigma_{n_k}},\ZZ)$ have $p$-primary torsion only for primes that divide $\gcd(n_1, \ldots, n_k)$, and in particular 
satisfy $p\le \gcd(n_1, \ldots, n_k)$.
\end{corollary}\vspace{-3pt}\vspace{-3pt}
\begin{proof}
This follows  directly from Corollary \ref{strqu} and the universal coefficient theorem.
\end{proof}
\vspace{-8pt}
\subsection{Spherical Quotients} We will now provide an answer to the following elementary question:
\begin{question}\vspace{-3pt}
For which  $n=n_1+\ldots + n_k$ is the quotient  $|\Pi_n|/_{\Sigma_{n_1}\times\cdots\times\Sigma_{n_k}}$  a wedge sum of spheres?
\end{question}\vspace{-3pt}
For this, we will first carry out a detailed analysis  of the spaces 
\mbox{$\Sigma^{-1} |\Pi_p|^\diamond \mywedge{\Sigma_p} (S^{\ell})^{\wedge p} \cong  |\Pi_p|^\diamond \mywedge{\Sigma_p} S^{\ell p-1}$}  for $p$ a prime. We will show that for many values of $\ell$ and $p$, these do not form wedges of spheres. This will allow us to conclude that the above quotients do not form wedges of spheres for many values of $n_i$ by our wedge decomposition in Proposition \ref{proposition: orbits again}. For the remaining values of $n_i$,   we do indeed  find a wedge of spheres. 
We will begin by examining  the weight $p$ component $ |\Pi_p|^\diamond \mywedge{\Sigma_p} S^{\ell p-1}$:
\begin{proposition} \label{proposition: prime} Let $p$ be a prime and $\ell$ a positive integer.

For $p=2$, we have  $|\Pi_2|^\diamond \mywedge{\Sigma_2} S^{2\ell-1}\cong \Sigma^\ell\reals P^{\ell-1} $ (this is well-known).

For $p$ odd and $\ell$ odd, the quotient    $|\Pi_p|^\diamond  \mywedge{\Sigma_p} S^{\ell p-1}$ is equivalent to the $p$-localisation of  $S^{\ell p-1}/{_{\Sigma_p}}$.

For $p$ odd and $\ell$ even, the space $|\Pi_p|^\diamond  \mywedge{\Sigma_p} S^{\ell p-1}$ is equivalent to the $p$-localisation of the homotopy cofibre of the quotient map  $S^{\ell p-1}\rightarrow S^{\ell p-1}/_{\Sigma_p}$.\vspace{6pt}
\end{proposition}

For $p$ an odd prime, the statement of this Lemma is subtle. In order to prove it, we will first analyse the space $ |\Pi_p|^\diamond \mywedge{\Sigma_p} S^{\ell p-1}$ in terms of homotopy coinvariants. 

For this, recall that there is a $\Sigma_p$-equivariant homeomorphism $S^{\ell p-1}\cong S^{\ell-1} \Smash (S^{p-1})^{\wedge \ell}$, where $\Sigma_p$ acts on $S^{p-1}$ via the reduced standard representation, and acts trivially on $S^{\ell-1}$. The fixed-points space $(S^{\ell p-1})^{\Sigma_p}$ is homeomorphic to $S^{\ell-1}$. Let $S^{\ell-1}\hookrightarrow S^{\ell p-1}$ be the inclusion of fixed points.
The key to analysing $ |\Pi_p|^\diamond\mywedge{\Sigma_p} S^{\ell p-1}\displaystyle $ for a general prime $p$ is the following proposition:
\begin{proposition}\label{proposition: primal pushout}
The following is a homotopy pushout square 
\begin{diagram}\displaystyle
|\Pi_p|^\diamond  \mywedge{\hobased\Sigma_p} S^{\ell -1}& &  \rTo & & \displaystyle|\Pi_p|^\diamond  \mywedge{\Sigma_p}S^{\ell -1} \\
\dTo & &  & &  \dTo \\
 \displaystyle|\Pi_p|^\diamond \mywedge{\hobased\Sigma_p} S^{\ell p-1}& &  \rTo &  & |\Pi_p|^\diamond\displaystyle\mywedge{\Sigma_p} S^{\ell p-1} 
\end{diagram}
\end{proposition}
\begin{proof}
The homotopy cofibre of the map $S^{\ell-1}\rightarrow S^{\ell p-1}$ is $S^{\ell }\Smash (S^{\ell(p-1)-1})_+$, where $\Sigma_p$ acts through the reduced regular representation on $S^{p-1}$. Taking homotopy cofibres of the vertical maps gives
\[
|\Pi_p|^\diamond  \mywedge{\hobased\Sigma_p} (S^{\ell}\Smash(S^{\ell (p-1)-1})_+)  \rightarrow |\Pi_p|^\diamond
\mywedge{\Sigma_p} (S^{\ell }\mywedge{} (S^{\ell(p-1)-1})_+ )
\]
We want to prove that this map is a homotopy equivalence. For this, it is enough to prove that the following map is an equivalence $
 |\Pi_p|^\diamond \mywedge{\hobased\Sigma_p}   (S^{\ell (p-1)-1})_+   \rightarrow |\Pi_p|^\diamond   \mywedge{\Sigma_p} (S^{\ell(p-1)-1})_+ .$
 
The $\Sigma_p$-space $S^{\ell (p-1)-1}$ can be written as a homotopy colimit of sets of the form $\Sigma_p/G$, where $G$ is an isotropy group of $S^{\ell (p-1)-1}$. Therefore, it is enough to prove that for every isotropy group $G$ of $S^{\ell (p-1)-1}$, the map $|\Pi_p|^\diamond  \mywedge{\hobased\Sigma_p} (\Sigma_p/G)_+ \rightarrow    |\Pi_p|^\diamond  \mywedge{\Sigma_p} (\Sigma_p/G)_+$ is an equivalence. \vspace{2pt}

Every isotropy group of $S^{\ell (p-1)-1}$ is contained in a group of the form $\Sigma_{p_1}\times \cdots\times \Sigma_{p_k}$, where $p_1+\cdots+p_k=p$, $k>1$,  and  $p_i>0$. 
For $\ell=1$, $S^{p-2}$ is the boundary of the $(p-1)$-simplex, and it is easy to see that the isotropy groups have this form. For $\ell>1$, we can argue by induction on $\ell$. \vspace{2pt}

Since $p$ is a prime, we have $\gcd(p_1, \ldots, p_k)=1$ for every decomposition  $p=p_1+\ldots+p_k$ as above.  Theorem \ref{main} implies that the action of $\Sigma_{p_1}\times \cdots\times \Sigma_{p_k}$ on $\Pi_p$ is essentially pointed-free, in the sense that for every subgroup $H\subset \Sigma_{p_1}\times \cdots\times \Sigma_{p_k}$, the fixed points space $\Pi_p^H$ is contractible. Hence the action of $G$ on $\Pi_p$ is essentially pointed-free, which shows that the map $ |\Pi_p|^\diamond \displaystyle \mywedge{\hobased\Sigma_p} (\Sigma_p/G)_+  \rightarrow    |\Pi_p|^\diamond \mywedge{\Sigma_p} (\Sigma_p/G)_+ $, which is the same as   $  |\Pi_p|^\diamond/_{\hobased G} \rightarrow  |\Pi_p|^\diamond/_G$, is an equivalence.\vspace{-2pt}
\end{proof}
\begin{corollary}\label{corollary: odd cofibration}
If $p$ is an odd prime, then there is a homotopy cofibration sequence\vspace{-1pt}
\[
 |\Pi_p|^\diamond\mywedge{\hobased\Sigma_p} S^{\ell -1}\rightarrow|\Pi_p|^\diamond  \mywedge{\hobased \Sigma_p}  S^{\ell p-1} \rightarrow |\Pi_p|^\diamond  \mywedge{{\Sigma_p}} S^{\ell p-1}.
\]
\end{corollary}

\begin{proof}
Consider the upper right corner of the square in Proposition \ref{proposition: primal pushout}. Since $p>2$, we have $\displaystyle |\Pi_p|^\diamond  \mywedge{\Sigma_p} S^{\ell -1}\cong \left(|\Pi_p|^\diamond /_{\Sigma_p}\right)\Smash  S^{\ell -1}\simeq \ast$ by Corollary \ref{Kozlov}. \vspace{-5pt}
\end{proof}

\begin{corollary}\label{corollary: p-local}
If $p$ is an odd prime and $\ell \geq 1$, then the space ${|\Pi_p|^\diamond \mywedge{\Sigma_p}}S^{\ell p-1} $ is $p$-local.\vspace{-5pt}
\end{corollary}
\begin{proof}
By Corollary~\ref{corollary: odd cofibration} it is enough to prove that  $ |\Pi_p|^\diamond  \mywedge{\hobased \Sigma_p} S^{\ell -1}$ and \mbox{$|\Pi_p|^\diamond \mywedge{\hobased \Sigma_p} S^{\ell p-1} $} are  $p$-local. Since $\ell\ge 1$ and $p>2$ both of these spaces are simply connected. Hence, it is enough to show that the integral homology of these spaces   is $p$-torsion.  
We will spell out the proof that $\widetilde{\HH}_* ( |\Pi_p|^\diamond \mywedge{\hobased \Sigma_p} S^{\ell  p-1} , \ZZ  )$ is $p$-torsion. The proof for the other space is \mbox{similar. Consider the homomorphisms}
\begin{equation}\label{equation: transfer}
\widetilde{\HH}_*( |\Pi_p|^\diamond\mywedge{\hobased \Sigma_p}  S^{\ell p-1},\ZZ  )   \rightarrow \widetilde{\HH}_*(  |\Pi_p|^\diamond \mywedge{\hobased \Sigma_{p-1}} S^{\ell p-1},\ZZ   ) \rightarrow \widetilde{\HH}_* ( |\Pi_p|^\diamond \mywedge{\hobased \Sigma_p} S^{\ell p-1} ,\ZZ ).
\end{equation}
Here the first homomorphism is the transfer in homology associated with the inclusion \mbox{$\Sigma_{p-1}\hookrightarrow \Sigma_p$,} and the second homomorphism is the quotient map. It is well-known that the composed homomorphism is multiplication by $p=|\Sigma_p/\Sigma_{p-1}|$. By Theorem \ref{main}, $|\Pi_p|$ is $\Sigma_{p-1}$-equivariantly equivalent to $(\Sigma_{p-1})_+ \wedge S^{p-3}$. Therefore $|\Pi_p|^\diamond     \mywedge{\hobased \Sigma_{p-1}}   S^{\ell p-1}\simeq S^{\ell p+p-3}.$ The reduced homology of this space is $\integers$ in dimension $\ell p+p-3$ and zero in all other dimensions. On the other hand, by Lemma~\ref{lemma: rational} the homology of $ |\Pi_p|^\diamond  \mywedge{\hobased \Sigma_p} S^{\ell p-1}$ is all torsion. This implies  that the composed homomorphism~\eqref{equation: transfer} is zero in all dimensions. Multiplication by $p$ is therefore zero on the group $\widetilde{\HH}_* ( |\Pi_p|^\diamond  \mywedge{\hobased \Sigma_p} S^{\ell p-1},\ZZ )$. It follows that it is a $p$-torsion group.\vspace{-2pt}
\end{proof}
We can finally prove our description of the weight $p$ component of the spaces $C_{\Lie}(S^\ell)$:\vspace{-2pt}
\begin{proof}[Proof of \ref{proposition: prime}] The case $p=2$ is well-known  (cf.\ \cite{james1963symmetric}). \mbox{Hence assume that  $p$ is an odd prime.}
We   use the  cofibre sequence $
|\Pi_p|_+\mywedge{\Sigma_p} S^{\ell p-1}   \rightarrow  S^{\ell p-1}/{_{\Sigma_p}} \rightarrow |\Pi_p|^\diamond \mywedge{\Sigma_p} S^{\ell p-1}.$
It suffices to show that for $\ell$ odd, the $p$-localisation of $|\Pi_p|_+ \mywedge{\Sigma_p} S^{\ell p-1} $ is trivial, and for $\ell$ even,  any choice of point in $ |\Pi_p|$  induces a $p$-local equivalence $S^{\ell p-1}\rightarrow |\Pi_p|_+\mywedge{\Sigma_p}  S^{\ell p-1}$.

Consider the pointed $\Sigma_p$-space $|\Pi_p|_+\wedge S^{\ell p-1} $. It is easy to check that the isotropy group of every point in this space (aside from the basepoint) is a non-transitive subgroup of $\Sigma_p$. It follows that $p$ does not divide the order of any of the isotropy groups. Thus for every isotropy group $G$, the map $BG\rightarrow *$ induces an isomorphism in $\FF_p$-homology. Therefore, the  map $  |\Pi_p|_+ \mywedge{\hobased \Sigma_p}S^{\ell p-1} \rightarrow     |\Pi_p|_+ \mywedge{\Sigma_p}S^{\ell p-1} $ is a mod $p$-homology isomorphism and hence a $p$-local isomorphism. 

The space $  |\Pi_p|_+\mywedge{\hobased \Sigma_p} S^{\ell p-1}$ is a pointed homotopy colimit of spaces of the form $ (S^{\ell p-1})_{\hobased G}$, where $G$ is an isotropy group of $|\Pi_p|$. The relevant properties of $G$ are that (1) $p$ does not divide the order of $G$ and (2) $G$ contains a transposition. Suppose that $\ell$ is odd. By an easy Serre spectral sequence argument like in Lemma \ref{lemma: rational}, 
we see that $\widetilde{\HH}_*\left((S^{\ell p-1})_{\hobased G};\FF_p\right)\cong \{0\}$, and so $(S^{\ell p-1})_{\hobased G}$ is $p$-locally trivial. It follows that $  |\Pi_p|_+ \wedge (S^{\ell p-1})_{\hobased \Sigma_p} $ is $p$-locally trivial.

 Now suppose $\ell$ is even. The space $\left( |\Pi_p|_+ \wedge S^{\ell p-1}\right)_{\hobased \Sigma_p}$ is a Thom space of a bundle over $|\Pi_p|_{\hobased \Sigma_p}$. Since $\ell$ is even, this bundle is orientable, and therefore the Thom isomorphism holds.
 Any map $*\rightarrow |\Pi_p|_{\hobased \Sigma_p}$ is a \mbox{mod $p$} homology isomorphism, and by the Thom isomorphism theorem,  we obtain a mod $p$ homology isomorphism of Thom spaces $S^{\ell p-1}\rightarrow  \left(|\Pi_p|_+\wedge S^{\ell p-1}  \right)_{\hobased \Sigma_p}$. 
Here, we have used that since  all  stabiliser groups   of  $|\Pi_p|$ are proper, their classifying spaces are   $p$-locally contractible, and we have $|\Pi_p|_{\hobased \Sigma_p}\simeq |\Pi_p|_{  \Sigma_p}  \simeq \ast$ by Kozlov's Corollary \ref{Kozlov}\vspace{-3pt}.
\end{proof}

Let $p$ be a prime. For $\ell=1$,  $|\Pi_p|^\diamond   \mywedge{\Sigma_p}  S^{\ell p-1}$ is contractible by Lemma~\ref{lemma: contractible}.  
For $\ell>1$, we have: 

\begin{corollary}\label{corollary: also l=2}
For $\ell> 2$, the space $  |\Pi_p|^\diamond \mywedge{\Sigma_p}  S^{\ell p-1}$ is not equivalent to a wedge of spheres. \\
For $\ell=2$, it is equivalent to $S^{3}$ if $p=2$ and is contractible if $p$ is odd.\vspace{-5pt}
\end{corollary} 
\begin{proof}
For $p=2$, we have seen that $|\Pi_p|^\diamond \mywedge{\Sigma_p} S^{\ell p-1} \simeq \Sigma^\ell\reals P^{\ell -1}$. For $\ell=2$ this space is equivalent to $S^{3}$. For $\ell>2$ it is not equivalent to a wedge of spheres since its homology has $2$-torsion.

Suppose now that $p>2$. By Theorem \ref{theorem: thehomology}, the homology  with $\FF_p$ coefficients of $  |\Pi_p|^\diamond\mywedge{\Sigma_p}S^{\ell p-1} $ is generated by symbols $\{i\}$, where $i$ is an integer with $i\equiv 0\mbox{ or } 1 \,(\mbox{mod } 2(p-1))$ and $1<i<(p-1)\ell$.
The degree of $\{i\}$ is $\ell+i-1$. It is easy to see that for $\ell\le 2$ there are no integers $i$ satisfying these constraints. It follows that for $\ell\le 2$ the mod $p$ homology of $|\Pi_p|^\diamond \mywedge{\Sigma_p} S^{\ell p-1}$ is trivial. Since the space is $p$-local by Corollary \ref{corollary: p-local}, it follows that it is contractible. \vspace{1pt}

For $\ell>2$, there is more than one value of $i$ satisfying the constraints. It follows that the space is not contractible. Since it is $p$-local, it is not equivalent to a wedge of spheres.\vspace{-4pt}
\end{proof}

\begin{corollary}\label{corollary: wedge of spheres}
The quotient $|\Pi_n|/_{\Sigma_{n_1}\times\cdots\times\Sigma_{n_k}}$ is   equivalent to a wedge of spheres if and only if  \vspace{-3pt}
$$\gcd(n_1, \ldots, n_k)=1 \ \ \ \ \ \ \mbox{or} \ \ \ \ \ \ \ p \mbox{ is a prime, } n=2p \mbox{ or } 3p \mbox{, and } \gcd(n_1, \ldots, n_k)=p\vspace{-3pt}$$
\end{corollary}
\begin{proof}
The case $\gcd(n_1, \ldots, n_k)=1$ was dealt with in Corollary~\ref{corollary: gcd one}. In the second case, Proposition~\ref{proposition: orbits again} tells us that $|\Pi_n|/_{\Sigma_{n_1}\times\cdots\times\Sigma_{n_k}}$ is equivalent to a wedge sum of spaces of the form $S^{n-3}$ and $ |\Pi_p|^\diamond \mywedge{\Sigma_p} S^{n-p-1}$. Since $n=2p$ or $3p$, $n-p=\ell p$ where $\ell=1, 2$. By Lemma~\ref{lemma: contractible} and Corollary~\ref{corollary: also l=2} all these spaces are either equivalent to a sphere or are contractible.
 
In all other cases, let $p$ be the smallest prime that divides $\gcd(n_1, \dots, n_k)$. Then $p$ divides $n$ and $\frac{n}{p}>3$. It follows that $|\Pi_n|/_{\Sigma_{n_1}\times\cdots\times\Sigma_{n_k}}$ has a wedge summand equivalent to  $
\bigvee_{B(\frac{n_1}{p}, \ldots, \frac{n_k}{p})} |\Pi_p|^\diamond\mywedge{\Sigma_p}S^{\ell p-1}
$, 
where $\ell>2$. By Corollary~\ref{corollary: also l=2} this space is not equivalent to a wedge of spheres. \vspace{-5pt}
\end{proof}
\begin{example}
To illustrate our results, let us analyse the homotopy type and the homology groups of $ |\Pi_8|^\diamond/_{\Sigma_4\times \Sigma_4}$. By Proposition~\ref{proposition: orbits again} there is a homotopy equivalence\vspace{-1pt}
\[
|\Pi_8|^\diamond/_{\Sigma_4\times \Sigma_4}\simeq \bigvee_{B(4, 4)} S^5 \vee \bigvee_{B(2, 2)} S^5/_{\Sigma_2} \vee \bigvee_{B(1,1)} (|\Pi_4|^\diamond \Smash  S^3 )/_{\Sigma_4}.
\]
The last factor is contractible by Lemma~\ref{lemma: contractible}. A quick calculation shows that $|B(4, 4)|=8$ and $|B(2,2)|=1$. We already observed that  $ S^5/{_{\Sigma_2}}\cong \Sigma^3\reals P^2$. We conclude that there is an equivalence\vspace{-1pt}
\[
|\Pi_8|^\diamond_{\Sigma_4\times \Sigma_4}\simeq \Sigma^3\reals P^2\vee \bigvee_{8} S^5 .\vspace{-2pt}
\]
It follows that the reduced homology of $|\Pi_8|^\diamond/_{\Sigma_4\times \Sigma_4}$ is isomorphic to $\FF_2$ in dimension $4$, to $\integers^8$ in dimension $5$, and is zero otherwise. This is consistent with computer \mbox{calculations by Donau~\cite{donau1}.}
\end{example}

\newpage
\section{The $\FF_p$-homology of \texorpdfstring{$|\Pi_n|^\diamond \mywedge{\Sigma_n} (S^{\ell})^{\wedge n}$}{the $\Sigma_n$-orbits of $|\Pi_n|{\diamond}    \wedge (S^{\ell n})$}}  \label{theproof}
We fix a prime $p$.  We shall use  Bredon homology to prove the following result:
\begin{theorem}\label{theorem: thehomology} Let $\ell\geq 1$ be a positive integer, assumed to be odd if the prime $p$ is odd.\\
If $n$ is not a power of $p$, then $\widetilde{\HH}_*\left(\Sigma |\Pi_n|^\diamond \wedge_{\Sigma_{n}}S^{\ell n} ,\FF_p\right)$ is trivial.\\ If $n=p^k$, then $\widetilde{\HH}_*\left(\Sigma |\Pi_{p^k}|^\diamond \wedge_{\Sigma_{p^k}} S^{\ell p^k},\FF_p\right)$ has a basis consisting of sequences $(i_1, \ldots, i_k)$, where $i_1, \ldots, i_k$ are positive integers satisfying:
\begin{enumerate}[leftmargin=26pt]
\item Each $i_j$  is congruent to $0$ or $1$ modulo $2(p-1)$. \label{congruence}
\item For all $1\le j<k$ we have $1<i_j< p i_{j+1}$.
\item We have $1<i_k\le (p-1)\ell$ (notice that if $p>2$, then~\eqref{congruence} means that the inequality is strict).
\end{enumerate}
The homological degree of $(i_1, \ldots, i_k)$ is $i_1+\cdots+i_k+\ell+k$.\vspace{3pt}
\end{theorem}

Throughout this section, we will write $\tilde{\HH}_\ast(X) = \tilde{\HH}_\ast(X,\FF_p)$ for the reduced singular $\FF_p$-homology.

We begin by reviewing the definition and basic properties of Bredon homology.
Suppose that $\nu$ is a Mackey functor for a finite group $G$. Let $X$ be a simplicial $G$-set. The \textit{Bredon homology groups} $\widetilde{\HH}^{\br}_*(X; \nu)$ of $X$ with coefficients in $\nu$
are defined as the  homology groups of the simplicial abelian group obtained by applying $\nu$ to $X$ degree-wise. If $X$ is a pointed $G$-set, then we define the reduced Bredon homology of $X$ as the quotient $\widetilde{\HH}^{\br}_*(X; \nu):=\widetilde{\HH}^{\br}_*(X; \nu)/\widetilde{\HH}^{\br}_*(*; \nu)$ of unreduced Bredon homology of $X$ by the Bredon homology of  the basepoint. 

If $\nu_\ast$ is a graded Mackey functor, then the Bredon homology is bigraded. In this case, we define:
\begin{definition}\label{eulerchar}
 The \textit{Euler characteristic} $\chi_\ast$ of a pointed simplicial $G$-set $X$  with respect to a  graded Mackey functor $\nu_\ast$ is defined as the following sequence of integers: $$ \chi_n = \sum_{i} (-1)^i \rk\left( \widetilde{\HH}^{\br}_i(X; \nu_n) \right)$$
 Informally speaking, we take Euler characteristic ``in the Bredon direction''. We extend this notion in the evident way to  bigraded submodules of $\widetilde{\HH}^{\br}_*(X; \nu_\ast)$.\vspace{3pt}
 \end{definition}

We will now consider the graded Mackey functor
$\mu_*$ which is defined on $\Sigma_{n}$-sets  by the\vspace{-1pt} formula $$\mu_*(S)=\widetilde{\HH}_*(S^{\ell n}\wedge_{\Sigma_{n}} S_+).\vspace{-1pt}$$
We observe that if $H$ is a subgroup of $\Sigma_n$, then $\mu_*(\Sigma_{n}/H)\cong\widetilde{\HH}_*({S^{\ell n}}/_H)$.

 If $W$ is a general simplicial $\Sigma_n$-set, then there is a well-known spectral sequence of signature \vspace{-1pt}$$
\widetilde{\HH}_s^{\br}(W; \mu_t)\Rightarrow \widetilde{\HH}_{s+t}(S^{\ell n}\wedge_{\Sigma_{n}} W)
\vspace{-1pt}.$$

We will see below that for $W=|\Pi_n|^\diamond$, the spectral sequence collapses at $E_2$, and we may  therefore focus on calculating the Bredon homology groups $\widetilde{\HH}_*^{\br}(|\Pi_n|^\diamond; \mu_*)$. 

The next proposition says that the Mackey functor $\mu_*$ satisfies the hypotheses needed for the main theorem of~\cite{arone2016bredon} to apply. Let $C_G(H)$ denote the centraliser of the subgroup $H$ in $G$.
\begin{proposition}\label{prop: satisfy}
The Mackey functor
$\mu_*(S)= \widetilde{\HH}_*(S^{\ell n}\wedge_{\Sigma_{n}} S_+)$
has the following \vspace{-1pt}properties:
\begin{enumerate}[leftmargin=26pt]
\item \label{ptransfer} If $Z$ is a $\Sigma_{n}$-set whose cardinality is coprime to $p$ and $S$ is any $\Sigma_{n}$-set, then the composed homomorphism $\mu_*(S)\stackrel{tr}{\rightarrow}\mu_*(S\times Z)\rightarrow \mu_*(S)$ is an isomorphism.
\item \label{centraliser} For every elementary abelian subgroup $D\subset \Sigma_{n}$ that acts freely and non-transitively on the set $\{1, \ldots, n\}$, the kernel of   $C_{\Sigma_{n}}(D)\rightarrow \pi_0C_{\GL_{n}({\mathbb R})}(D)$ acts trivially on $\mu_*(\Sigma_{n}/D)$.
\item \label{involution} For $p>2$ odd and $D$ as in $(2)$,  any odd involution in $C_{\Sigma_{n}}(D)$ acts on $\mu_*(\Sigma_{n}/D)$ as \mbox{multiplication by $(-1)$.}
\end{enumerate}
\end{proposition}
\begin{proof}  
The proof is similar to~\cite[Proposition 11.4]{arone2016bredon}. Since homology is taken with $\FF_p$ coefficients, the Mackey functor $\mu_*$ takes values in $\FF_p$-vector spaces. It is well-known that the composite homomorphism $\mu_*(S)\stackrel{tr}{\rightarrow}\mu_*(S\times Z)\rightarrow \mu_*(S)$ is, in general, multiplication by the cardinality of $Z$. If the cardinality is coprime to $p$, then the homomorphism is an isomorphism. This proves~\eqref{ptransfer}.

For~\eqref{centraliser}, let $\sigma\in \ker(C_{\Sigma_{n}}(D)\rightarrow \pi_0C_{\GL_{n}({\mathbb R})}(D))$. Then $\sigma$ is an element of $\Sigma_{n}$, $\sigma$ centralises $D$, and the linear automorphism of $\reals^{n}$ induced by $\sigma$ can be connected to the identity by a path that goes through linear transformations that centralise the action of $D$. It follows that $\sigma$ acts on $S^{\ell n}$ by a map that is $D$-equivariantly homotopic to the identity. Therefore, $\sigma$ acts on $S^{\ell n}/_{D}$ by a map that is homotopic to the identity. In particular, it acts trivially on $\mu_*(\Sigma_{n}/D) = \widetilde{\HH}_*(S^{\ell n}/_{D})$. (See  the proof of~\cite[Proposition 11.4]{arone2016bredon} for a very similar argument spelled out in more detail).

Finally, for~\eqref{involution} suppose that $p$ is odd and $\tau\in\Sigma_{n}$ is an odd permutation of order $2$ centralising $D$. We need to show that $\tau$ acts on $\widetilde{\HH}_*(S^{\ell n}/_D)$ as multiplication by $(-1)$. As explained in the last part of the proof of~\cite[Proposition 11.4]{arone2016bredon}, an involution $\tau$ on a space $W$ acts by $(-1)$ on the homology of $W$ (with coefficients in a group where $2$ is invertible) if and only if $(\tau-1)$ is an isomorphism on homology. This in turn is equivalent to requiring that $(\tau - 1)$ induces a quasi-isomorphism on the chain complex $\tilde C_*(W)$  of reduced  singular chains with coefficients in $\FF_p$ on $W$. We apply this to the case $W=S^{\ell n}/_D$. There is a natural homotopy equivalence between $\tilde C_*(S^{\ell n}/_D)$ and a homotopy colimit of chain complexes of the form $\tilde C_*((S^{\ell n})^A)$, where $A$ ranges over subgroups of $D$. Using that $\ell$ is odd, it is not difficult to show that $\tau$ acts by $(-1)$ on the homology of each fixed point space $(S^{\ell n})^A$. It follows that $\tau-1$ induces a quasi-isomorphism on each chain complex $\tilde C_*((S^{\ell n})^A)$ and therefore it induces a quasi-isomorphism on $\tilde C_*(S^{\ell n}/_D)$. This in turn \mbox{implies that $\tau$} acts by $(-1)$ on $\widetilde{\HH}_*(S^{\ell n}/_D)$.
\end{proof}
By~\cite[Lemma 3.8 and Theorem 1.1]{arone2016bredon} it follows that $\widetilde{\HH}_s^{\br}(|\Pi_n|^\diamond; \mu_t)=0$ unless $n=p^k$ for some positive integer $k$ and $s=k-1$. Therefore $\widetilde{\HH}_*(S^{\ell n}\wedge_{\Sigma_{n}} |\Pi_n|^\diamond)$ vanishes unless $n$ is a power of $p$. In the case $n=p^k$ the spectral sequence collapses, and there is an isomorphism
\[
\widetilde{\HH}_{*+k-1}(S^{\ell p^k}\wedge_{\Sigma_{p^k}} |\Pi_{p^k}|^\diamond)\cong \widetilde{\HH}_{k-1}^{\br}(|\Pi_{p^k}|^\diamond; \mu_{*}).
\]
We will therefore focus on calculating the Bredon homology group $\widetilde{\HH}_{k-1}^{\br}(|\Pi_{p^k}|^\diamond; \mu_{*})$. 
Abstractly, this  group can be described in terms of the Steinberg module $\St_k \cong \widetilde{\HH}_{k-1}(|\BT(\FF_p^k)|^\diamond, \ZZ)$ as 
$\mu_\ast(\Sigma_{p^k}/\FF_p^k) \otimes_{\ZZ[\GL_k(\FF_p)]} \St_k$, cf.\ \cite[Corollary 1.2]{arone2016bredon}.
To get a concrete description,  however, we will implement  a different approach and use an explicit chain complex, \vspace{3pt} which we will now describe.

We define a reduced version $\tilde\mu_*$ of $\mu_*$ on  pointed sets by setting $\tilde\mu_*(S)=\widetilde{\HH}_\ast(S^{\ell n}\wedge_{\Sigma_n} S)$ for $S$  a {\em pointed} $\Sigma_{n}$-set. To calculate the reduced Bredon homology of  $|\Pi_{p^k}|^\diamond$, write it as the geometric realisation of a pointed simplicial set with an action of $\Sigma_{p^k}$ (we will describe an explicit model below). Applying $\tilde\mu_*$ levelwise to $|\Pi_{p^k}|^\diamond$ yields a graded simplicial  abelian group. We refer to the associated normalised chain complex as the Bredon chains on $|\Pi_{p^k}|^\diamond$. The bigraded homology groups of this chain complex are precisely the Bredon homology groups  $\widetilde{\HH}_{*}^{\br}(|\Pi_{p^k}|^\diamond; \mu_{*})$.

Since the Bredon homology is concentrated in the single degree $(k-1)$, it  is essentially determined by the Euler characteristic in the Bredon direction \mbox{(cf.\ Definition \ref{eulerchar}).}

To determine the Euler characteristic, we will analyse $\tilde\mu_*(\Pi_{p^k,i}^\diamond)$, where ${\Pi_{p^k,i}^\diamond}$ is the quotient of the set of $i$-simplices of ${|\Pi_{p^k}|^\diamond}$ by the subset of degenerate simplices. We will show that $\tilde\mu_*({\Pi_{p^k,i}^\diamond})$ splits as a direct sum, and that most of the summands can be arranged in isomorphic pairs in adjacent dimensions, so as to cancel out and contribute nothing to the Euler characteristic. Finally, we will calculate the Euler characteristic of the complex obtained from summands that are not cancelled out. This approach is essentially the same as was taken in~\cite{arone1999goodwillie} towards calculating the homology of the {\it homotopy} orbit space $S^{\ell p^k}\wedge_{\hobased\Sigma_{p^k}} |\Pi_{p^k}|^\diamond$. However the situation in the present paper is simpler, as we are only calculating the Euler characteristic and therefore  do not need to determine the actual boundary homomorphism between the various summands in our decomposition. Our simplified approach can in fact  also be implemented in the case of homotopy orbits, but we will not pursue this here.

The set ${\Pi_{p^k,i}^\diamond}$ is the wedge sum of sets of the form $\Sigma_{p^k}/H_+$, where $H$ ranges over a set of representatives of stabilisers of nondegenerate $i$-simplices of ${|\Pi_{p^k}|^\diamond}$. It follows that $\tilde\mu_*({\Pi_{p^k,i}^\diamond})$ is isomorphic to a  corresponding direct sum of groups $\widetilde{\HH}_*(S^{\ell p^k}/_H)$. Isotropy groups of ${\Pi_{p^k,i}^\diamond}$ are products of wreath products of symmetric groups (see Section \ref{homorbit} below for a more precise statement). Therefore, before describing $\widetilde{\HH}_*(S^{\ell p^k}/_H)$ for a general isotropy group $H$, we review some results about the homology of $X^{\wedge n}/_{\Sigma_n}$ for   general $X$.

\subsection{Homology of Symmetric Smash Products}
Before describing the homology, let us recall the very useful general fact, due to Dold~\cite{dold1958homology}, that the homology of a symmetric (smash) product of a pointed space $X$ depends only on homology of $X$. Dold proved the result for integral homology. We need a version of it for homology with mod $p$ coefficients. A proof can be found in Bousfield's unpublished manuscript~\cite{bousfield1967operations}, where an explicit description of the mod $p$ homology of  symmetric products is given. For the reader's convenience, we shall include a proof along the lines of Dold's argument.
\begin{lemma}\label{lemma: Dold}
The $\FF_p$-homology groups of $X^{\wedge n}/_{\Sigma_n}$ only depend on the homology of the pointed space $X$. More precisely, there exists an endofunctor $G$ of the category of non-negatively graded $\FF_p$-vector spaces such that $\widetilde{\HH}_*(X^{\wedge n}/_{\Sigma_n},\FF_p)$ is isomorphic to $G(\widetilde{\HH}_*(X))$, naturally in $X$.
\end{lemma}
\begin{proof}
For a set $S$, let $F_{\hspace{1pt} \FF_p}(S)$ denote the $\FF_p$-vector space with basis $S$. If $S$ is pointed, let $\tilde{F}_{\hspace{1pt}\FF_p}(S)$ be the quotient of $F_{\hspace{1pt}\FF_p}(S)$ by the subspace generated by the basepoint. Note that there is a natural isomorphism of vector spaces ${\tilde{F}_{\hspace{1pt}\FF_p}}(S^{\wedge n}/_{\Sigma_n})\cong ({\tilde{F}_{\hspace{1pt}\FF_p}}(S)) ^{\otimes n}_{\Sigma_n}$.

Let $\Sing_\bullet(X)$ be the pointed simplicial set of singular simplices of $X$. Let $C_\bullet(X) =F_{\hspace{1pt}\FF_p}(\Sing_\bullet(X))$ be the simplicial $\FF_p$-vector space generated by $X_\bullet$ and write $\tilde{C}_\bullet(X) =\tilde{F}_{\hspace{1pt}\FF_p}(\Sing_\bullet(X))$ for the reduced \mbox{$\FF_p$-chains on $X$.} The reduced homology of $X^{\wedge n}/_{\Sigma_n}$ is the homology of the simplicial vector space $\tilde{C}_\bullet(X^{\wedge n}/_{\Sigma_n})$, and this simplicial vector space is isomorphic to   $ \tilde{C}_\bullet(X)^{\otimes n}_{\Sigma_n}.$

Let $N$ be the normalised chain complex functor from the category of simplicial $\FF_p$-vector spaces to the category of non-negative chain complexes over $\FF_p$. Let $\Gamma$ be the inverse of $N$ provided by the Dold-Kan correspondence. Then $\widetilde{\HH}_*(X^{\wedge n}/_{\Sigma_n})$ is naturally isomorphic to
$ {\HH}_*\left( N\left((\Gamma N \tilde{C}_\bullet(X))^{\otimes n}_{\Sigma_n}\right)\right).$
It follows that the functor $X\mapsto \widetilde{\HH}_*(X^{\wedge n}/_{\Sigma_n})$ from pointed spaces to graded vector spaces factors as the composition of the reduced chain complex functor $X\mapsto N\tilde{C}_\bullet(X) = \tilde{C}_\ast(X)$, and the   functor $
C\mapsto  {\HH}_*\left(N\left((\Gamma C)^{\otimes n}_{\Sigma_n}\right)\right)
$ from chain complexes to graded vector spaces.

Clearly, this is a homotopy functor. Therefore, it factors through the homotopy category of non-negative chain complexes over $\FF_p$. It is well known that this homotopy category is equivalent to the category of graded vector spaces, and $ {\HH}_*$ induces the equivalence of categories. Therefore, $\widetilde{\HH}_*(X^{\wedge n}/_{\Sigma_n})$ is naturally isomorphic to 
${\HH}_*\left( N\left((\Gamma \widetilde{\HH}_*(X))^{\otimes n}_{\Sigma_n}\right)\right)$,
where we consider $\widetilde{\HH}_*(X)$ as a chain complex with zero differential.  Setting $G(-)= {\HH}_*\left(N\left(\left( \Gamma -\right) ^{\otimes n}_{\Sigma_n}\right)\right)$ proves the result.
\end{proof}
It follows that in order to describe the homology of $X^{\wedge n}/_{\Sigma_n}$ for   general $X$, it is enough to complete this computation  for $X$ a wedge of spheres. We will describe the direct sum $\bigoplus_{n\geq 0}  \widetilde{\HH}_*(X^{\wedge n}/_{\Sigma_n})$ as a bigraded vector space.

\begin{definition}\label{definition: Fk}
For every integer $k\geq 0$, we define an endofunctor ${\Fcal}_k$ of the category of graded $\FF_p$-vector spaces as follows:

For $k=0$, the functor $\Fcal_0$ is the identity. 

For $k>0$, the graded vector space $\Fcal_k(V)$ is  generated by symbols of the form $(i_1, \ldots, i_k; v)$
where $v$ is an element of $V$, and $i_1, \ldots, i_k$ are positive integers satisfying the following conditions:  
\begin{enumerate}[leftmargin=26pt]
\item Each $i_j$ is congruent to $0$ or $1$ mod $2(p-1)$, 
\item $i_j\ge pi_{j+1}$ for all $1\le j<k$,  \label{admissible}
\item If $p$ is odd, then $pi_1<(p-1)(|v|+i_1+\cdots+i_k)$, and if $p=2$, then $pi_1\le(p-1)(|v|+i_1+\cdots+i_k)$. (This convention at $p=2$ is slightly non-standard, we will explain the reason for it below),  
\item We have $i_j\ne 1$ for all $1\leq j \leq k$.
\end{enumerate}
These symbols are linear in $v$, i.e. satisfy $(i_1, \ldots, i_k; u)+ (i_1, \ldots, i_k; v)=(i_1, \ldots, i_k; u+v)$

If $v$ is homogeneous of degree $|v|$, then $(i_1, \ldots, i_k; v)$ is homogeneous of degree $i_1+\cdots+i_k+|v|$.  
\end{definition}
Note that $\Fcal_k$ is a well-defined endofunctor of graded vector spaces  preserving direct sums.

Suppose $V$ is a graded vector space. Then $\bigoplus_{k\geq 0}  \Fcal_k(V)$ is a graded vector space. We endow it with a second grading, which we call ``weight'': elements of $\Fcal_k(V)$ are given weight $p^k$.

For a graded vector space $V$, define $S(V)$ to be the free graded symmetric algebra on $V$ if $p$ is odd (with Koszul sign rule), and the free exterior algebra on $V$ if $p=2$. Any grading of $V$ extends to a grading of $S(V)$ in the usual way: the degree of a product is the sum of degrees. Let $S_n(V)$ be the direct summand of $S(V)$ consisting of products of exactly $n$ elements of $V$.

\begin{proposition}\label{prop: full homology}
There is an isomorphism of bigraded vector spaces
\[
\bigoplus_{n\geq 0}  \widetilde{\HH}_*(X^{\wedge n}/_{\Sigma_n})\cong S\left(\bigoplus_{k\geq 0}\Fcal_k\widetilde{\HH}_*(X)\right)
\]
where $\widetilde{\HH}_*(X^{\wedge n}/_{\Sigma_n})$ corresponds to elements of weight $n$.
\end{proposition}
\begin{remark}
This is not an isomorphism of algebras. Information about the algebra structure on $\bigoplus_{n\geq 0}\widetilde{\HH}_*(X^{\wedge n}/_{\Sigma_n})$ can be found, for example, in the paper of Milgram~\cite{milgram1969homology}.
\end{remark}
\begin{remark}
For odd primes, our conventions follow closely those of Nakaoka~\cite{nakaoka1957cohomology2}. When \mbox{$p=2$}, our conventions differ in that we include in $\Fcal_k$ elements $(i_1, \ldots, i_k; v)$ satisfying the equality  \mbox{$i_1=i_2+\cdots + i_k + |v|$.} In this case, the element $(i_1, \ldots, i_k; v)$ stands in for the square of the element $(i_2, \ldots, i_k; v)$. We are using the evident isomorphism of graded vector spaces between the polynomial algebra $P(x)$ and the exterior algebra $\Lambda(x_1, x_2, x_4, \ldots, x_{2^i}, \ldots, )$. This is also why we defined $S(V)$ to be the exterior rather than polynomial algebra when $p=2$. The purpose of this convention is the following: when $V$ is a vector space with one generator, assumed to \mbox{be odd if $p\neq 2$} 
(for us, $V$ will be the reduced homology of a sphere), we want homogeneous summands $S_n(V)$ to be null for $n>1$. All this is for the sake of Lemma~\ref{lemma: vanish} below.
\end{remark}
\begin{proof}[Proof of Proposition \ref{prop: full homology}]
The proposition is well-known. The case when $X$ is a sphere  can be read from Nakaoka's~\cite{nakaoka1957cohomology2}, with the aforementioned change of conventions at the prime $2$. (Nakaoka's paper is written in terms of cohomology, but as far as vector space dimension goes, there is no difference). It is easy to see that both sides of the claimed equation take wedge sums in the variable $X$ to tensor products. It follows that the formula is valid for $X$ a wedge of spheres. By Lemma~\ref{lemma: Dold}, the result holds for a general space $X$. One can also read off the result from \cite{bousfield1967operations} or \cite{milgram1969homology}.
\end{proof}
Now we can write an explicit formula for $\widetilde{\HH}_*(X^{\wedge n}/_{\Sigma_n})$. For this, we first introduce some notation: \begin{definition}
A $p$-partition of an integer $n$ is a partition into components of sizes that are powers of the prime $p$. A $p$-partition  is encoded by a sequence $(a_0, a_1, \ldots )$ of non-negative integers (almost all of them zero) with $n=\sum_{k\geq 0} a_k p^k$. We write $P(n)$ for the set of $p$-partitions of $n$.\vspace{3pt}
\end{definition}
The following statement  is an immediate consequence of Proposition~\ref{prop: full homology}.
\begin{proposition}\label{prop: single homology}
There is an isomorphism, where the direct sum is indexed on \mbox{$p$-partitions of $n$}
\[
\widetilde{\HH}_*(X^{\wedge n}/_{\Sigma_n} )\cong \bigoplus_{(a_0, a_1, \ldots )\in P(n)}  \  \bigotimes_{k\geq 0} S_{a_k}(\Fcal_k \widetilde{\HH}_*(X )).
\]
\end{proposition}
\subsection{The Homology of Orbit Spaces of Isotropy Groups}\label{homorbit}
Our next task is to describe  $\widetilde{\HH}_*(X^{\wedge n}/_H)$ for $H$   an isotropy group of the suspended partition complex $|\Pi_{n}|^\diamond$ and $X$ \vspace{1pt}  \mbox{a  pointed space.}

We begin by reviewing the structure of isotropy groups of simplices in  $|\Pi_n|^\diamond$ for $n>1$. Write $\mathcal{P}_n$ for the poset of all partitions of $\{1,\ldots,n\}$. 
Similarly as in Section \ref{clock}, the space $|\Pi_n|^\diamond$ arises as the realisation of the following simplicial set:
\begin{definition} \label{simplicialmodel} The \textit{simplicial model of $|\Pi_n|^\diamond$} is given by the quotient $N_\bullet(\mathcal{P}_n-\{\hat{0}\})/_{N_\bullet(\mathcal{P}_n-\{\hat{0},\hat{1}\})}$ of the nerve of the poset  of partitions which are not initial   by the nerve  of the subposet of partitions which are neither initial nor final. 
\end{definition}
This is a pointed simplicial $\Sigma_n$-set. The nondegenerate non-basepoint $i$-simplices of $|\Pi_n|^\diamond$ are in bijective correspondence with strictly increasing chains  of partitions
$$
\sigma = [\hat{0}< x_1<\ldots<x_i < \hat 1]
$$ 
In particular, this means that there is a single non-basepoint zero-simplex.  Notice that for every equivalence class of $x_i$, the chain $\sigma$ induces a chain of partitions of this class of length $(i-1)$. We call these subchains the {\it restrictions} of $\sigma$ to the equivalence classes of $x_i$.

The symmetric group $\Sigma_n$ permutes such chains of partitions $\sigma$. We say that two chains $\sigma, \sigma'$ are of the same type if they are in the same orbit under this action. More generally, we say that two chains of partitions of two possibly different sets are of the same type if there is a bijection between the sets that takes one of the chains to the other.  

Let $\sigma=[\hat{0}< x_1<\ldots< x_i <\hat 1]$ be an increasing chain of partitions of the set $\{1,\ldots,n\}$  as above. Write $K_{\sigma}\subset \Sigma_n$ for the isotropy group of this chain. Then $K_\sigma$ permutes the equivalence classes of $x_i$ (the coarsest partition in the chain). Two equivalence classes of $x_i$ are in the same orbit of $K_\sigma$ if and only if the restrictions of $\sigma$ to these two equivalence classes are of the same type. Let us list the different types of chains that occur among the restrictions of $\sigma$ to the  classes of $x_i$ as type 1, type 2, etc. Let $K_j$ denote the isotropy group of a chain of type $j$. Suppose that among the various  restrictions of $\sigma$ to the  classes of $x_i$, there are exactly $i_j$ many classes of type $j$. 

Then there is an isomorphism
$$
K_\sigma \cong \prod_j K_j\wr \Sigma_{i_j}.
$$
By induction, this identifies $K_\sigma$ with a product of iterated wreath products of symmetric groups. Furthermore, suppose that each $x_i$-class of type $j$ has cardinality $n_j$. Then we have an isomorphism\vspace{-3pt}   $$
X^{\wedge n}/_{K_\sigma}\cong \bigwedge_j (X^{\wedge n_j}/_{K_j})^{\wedge i_j}/_{\Sigma_{i_j}}\vspace{-3pt}$$
for any pointed\vspace{3pt} space $X$.

Applying the K{\"u}nneth formula and Proposition~\ref{prop: single homology}, we see that the homology groups of the space $X^{\wedge n}/_{K_\sigma}$ split as a direct sum indexed by simultaneous choices of  $p$-partitions of $i_j$ for each \mbox{type $j$.} Equivalently, the summands are indexed by isomorphism classes of chains\vspace{-3pt}
$$
[ \hat{0}< x_1<\ldots< x_i \le e_{i+1}\le \hat 1]\vspace{-3pt}
$$
where $e_{i+1}$ is a coarsening of $x_i$ with the following two properties: each $e_{i+1}$-class   contains a power of $p$ many   $x_i$-classes, and the restrictions of $\sigma$ to any two $x_i$-classes that lie in the same $e_{i+1}$-class are of the same\vspace{3pt} type.
The desire to iterate this procedure motivates the following definition:
\begin{definition}\label{definition: enhancement}
Let $\sigma=[\hat{0}< x_1<\ldots< x_i <  \hat 1]$ be a strictly increasing chain of partitions of $\{1, \ldots, n\}$. Set $x_0 = \hat{0}$,  $x_{i+1} = \hat{1}$.\vspace{-4pt}  A $p$-enhancement of $\sigma$ consists of  a refining chain  
$$ [\hat{0} \leq e_1\le x_1\le\ldots\le e_i\le x_i \le\ e_{i+1}\le  \hat 1]$$
for which the following two conditions hold true for $k=0,\ldots,i$:
\begin{enumerate}[leftmargin=26pt]  
\item Each equivalence class of $e_{k+1}$ contains a power of $p$ many $x_{k}$-classes.
\item The restriction of the chain $[\hat 0  \le e_1\le x_1\le\ldots\le e_i \le x_i \le e_{i+1}\le \hat 1]$ to any two $x_{k}$-classes lying in the same $e_{k+1}$-class are isomorphic. \vspace{-2pt}
\end{enumerate}
\end{definition}
We say that two $p$-enhancements of $\sigma$ are isomorphic if they lie in the same orbit under  the action of $K_\sigma$. We will primarily be concerned with isomorphism classes of enhancements, which we can use to define endofunctors of graded vector spaces:

\begin{definition}
Let \mbox{$\Theta=[\hat{0} \leq e_1\le x_1\le\ldots\le e_i\le x_i \le\ e_{i+1}\le  \hat 1]$} be a $p$-enhancement of a chain of partitions $\sigma=[\hat{0}<x_1<\ldots<x_i<\hat{1}]$ of the set $\{1, \ldots, n\}$. Write $[\Theta]$ for its isomorphism class.   

Given a graded vector space $V$, we define a  graded vector space $[\Theta](V)$ by the following recursion:
\begin{itemize}[leftmargin=26pt]  \label{definition: enhancement2}
\item If $i=0$, then $\Theta=[\hat 0 \le e_1\le \hat 1]$, so the isomorphism type of $\Theta$ is determined by the isomorphism type of $e_1$, i.e a $p$-partition of $n$. Let the $p$-partition be given by $(a_0, a_1, \dots )$. Then we define $[\Theta](V):=\bigotimes_j S_{a_j}(\Fcal_j(V)).$ 
\item If $i>0$, we first list the isomorphism types of the chains obtained by  restricting $\Theta$ to the   classes of $e_{i+1}$ as type 1, type 2, \ldots. Let   $a_t$ be the number of classes of $e_{i+1}$  having type $t$.\\
Assume that each $e_{i+1}$-class  of type $t$ contains exactly $p^{b_t}$ many   $x_i$-classes.
The restrictions of $\Theta$ to  all  $x_i$-classes contained in   $e_{i+1}$-classes  of same type $t$ must be pairwise isomorphic. Each   of these isomorphic restrictions corresponding to type $t$ can be thought of as a $p$-enhancement of \mbox{a chain of length $(i-1)$}. Write $[\Theta_t]$ for the resulting class of   $p$-enhancements. \\
We then define  $[\Theta](V):=\bigotimes_t S_{a_t}(\mathcal{F}_{b_t}([\Theta_t](V))).$
\end{itemize}
\end{definition}

The discussion above together with induction lead   to the following result:
\begin{proposition}\label{prop: splitting}
Assume that  $\sigma=[\hat0 < x_1<\ldots< x_i < \hat 1]$ is a strictly increasing chain of partitions of $\{1, \ldots, n\}$. Let $K_\sigma$ be the isotropy group of $\sigma$ and write $E[\sigma]$ for the set of isomorphism classes of $p$-enhancements of $\sigma$.
Given any pointed space $X$, there is an isomorphism
\[
\widetilde{\HH}_*(X^{\wedge n}/_{K_\sigma})\cong\bigoplus_{[\Theta]\in  E[\sigma]} [\Theta](\widetilde{\HH}_*(X)).
\] 
\end{proposition}
\begin{proof}
We list the different types of chains that occur among restrictions of $\sigma$ to classes of $x_i$ as type 1, type 2, \ldots. For each type $j$, we  assume that there are $i_j$ many $x_i$-classes of type $j$, and that each of them has cardinality $n_j$. We again write $K_j$ for the isotropy \mbox{group of any chain of type $j$.} 

The discussion preceding Definition~\ref{definition: enhancement} and Proposition \ref{prop: single homology} together imply $$\widetilde{\HH}_*(X^{\wedge n}/_{K_\sigma} ) \cong \bigotimes_j \widetilde{\HH}_*((X^{\wedge n_j}/_{K_j})^{\wedge i_j}/_{\Sigma_{i_j}}) \cong \bigoplus_{\{(a_0^j, a_1^j,\ldots)\in P(i_j)\}_j}    \left(\bigotimes_{j,k} S_{a_k^j} (\mathcal{F}_k  \widetilde{\HH}_*(    X^{\wedge n_j}/_{K_j}))\right) $$

For each $j$, we pick a chain $\sigma_j$ of type $j$ which is obtained by restricting the chain  $\sigma$ to an $x_i$-class. By induction and additivity of each functor $\mathcal{F}_k$,  we obtain \mbox{an isomorphism}
$$\widetilde{\HH}_*(X^{\wedge n}/_{K_\sigma} )  \cong \bigoplus_{\{(a_0^j, a_1^j,\ldots)\in P(i_j)\}_j }   \left(\bigotimes_{j,k} S_{a_k^j} \left( \bigoplus_{[\Theta_j]\in  E[\sigma_j]} \mathcal{F}_k  \left( [\Theta_j]\left(\widetilde{\HH}_*(X,\FF_p) \right) \right) \right) \right) $$

We can expand  the functors $S_{a_k^j}$  above and obtain an isomorphism to the following vector space:

$$     \bigoplus_{\substack{\{(a_0^j, a_1^j,\ldots)\in P(i_j)\}_j\\\{b^j_{k,1} + \ldots + b^j_{k,s^j_k} = a_k^j \ | \ b^j_{k,s}>0\ \}_{j,k} \\ \{[\Theta_{k,1}^j], \ldots , [ \Theta^j_{k,s^j_k}] \in E[\sigma_j] \mbox{\small{ distinct} }\}_{j,k} }} \hspace{-10pt}  \Bigg(\bigotimes_{j,k,s}  \ \ \ \ S_{b_{k,s}^j} \left(  \mathcal{F}_k  \left( [\Theta^j_{k,s}]\left(\widetilde{\HH}_*(X,\FF_p) \right) \right) \right)\Bigg) $$

The indexing set of this sum is in fact just $E[\sigma]$. 

Indeed, specifying a $p$-enhancement $\Theta$ of $\sigma$ up to isomorphism is equivalent to the following \vspace{-2pt}data:
\begin{enumerate}[leftmargin=26pt]
\item A $p$-partition $\sum a_k^j p^k = i_j$  for each type $j$ of chains obtained by restricting $\sigma$ to $x_i$-classes. This corresponds to   partitions $e_{i+1}\geq x_i$ such that   there are exactly $a_k^j$ many $e_{i+1}$-classes that contain $p^k$ many $x_i$-classes of type $j$, up to isomorphism.
\item For each  $j,k$ a collection of $a^j_k$ many $p$-enhancements $[\Theta_{k,1}^j], \ldots ,  [\Theta^j_{k,s^j_k}] $  of $\sigma_j$ \mbox{up to isomorphism.}  Write $b^j_{k,s} $ for the multiplicity of $[\Theta_{k,s}^j]$. Each of these $a_k^j$ many $p$-enhancements induces a simultaneous $p$-enhancement of some $p^k$ many $x_i$-classes   of type $j$ lying in the same  $e_{i+1}$-class. 
\end{enumerate}
We restrict attention to the summand  above corresponding to the isomorphism class of a given  $p$-enhancement $\Theta=[\hat{0} \leq e_1\le x_1\le\ldots\le e_i\le x_i \le\ e_{i+1}\le  \hat 1]$ of $\sigma$.

The triples $j,k,s$ indexing the inner tensor product above then corresponds   exactly to the   possible types $t$ of  restrictions of $\Theta$   to  $e_{i+1}$-classes $S$.
Here $j$ encodes the type of the restriction of $\sigma$ to any $x_i$-class in $S$, the number $k$ is chosen so that $S$ contains $p^k$ many $x_i$-classes, and the index $s\in \{1,\ldots, s_k^j\}$ specifies the type of the restriction of $\Theta$ to any of the  $x_i$-classes in $S$.  There are  $b_{j}^{k,s}$ many $e_{i+1}$-classes of this type $t$, and the   claim follows from the \vspace{-4pt} second clause of \mbox{Definition \ref{definition: enhancement2}. }
\end{proof}

Before coming back to our computation of Bredon homology, we need two rather easy observations. 

For the first, assume we are given a chain of partitions $\sigma=[\hat0  < x_1<\ldots< x_i <  \hat 1]$ and a $p$-enhancement $\Theta=[\hat 0  \le e_1\le x_1\le\ldots\le e_i\le x_i \le e_{i+1}\le   \hat 1]$. Suppose that for some $1\le j\le i$, we have $e_j=x_j=e_{j+1}$. Omitting $x_j$ from the chain, we get the chain $\sigma_j=[\hat0  < x_1<\ldots <x_{j-1} < x_{j+1}< \ldots < \hat 1]$ and a $p$-enhancement \mbox{$\Theta_j=[\hat0  \le e_1\le x_1\le \ldots \le x_{j-1} \le e_j \le x_{j+1}\le \ldots \le \hat 1]$.}
\begin{lemma}\label{lemma: match}
There is an isomorphism
$
[\Theta](\widetilde{\HH}_*(X,\FF_p))\cong [\Theta_j](\widetilde{\HH}_*(X,\FF_p))
$. \end{lemma}
The second observation covers the case $\hat 0 = e_1$:
\begin{lemma}\label{lemma: vanish}
Let $\Theta=[\hat 0  \le e_1\le x_1\le\ldots\le e_i\le x_i \le e_{i+1}\le \hat 1]$ be as before. 
Let $\ell$ be an integer, assumed to be odd if $p \neq 2$. 
If $ e_1=\hat{0}$, then $[\Theta](\widetilde{\HH}_*(S^{\ell}),\FF_p)$ is the zero vector space.
\end{lemma}
\subsection{Calculation of the Bredon Homology}
We are now ready to calculate $\widetilde{\HH}^{\br}_*(|\Pi_n|^\diamond; \mu_*)$. 
Recall that this group is computed as the homology of the normalised chain complex of the simplicial abelian group obtained by applying $\tilde\mu_*$ levelwise to the  simplicial model of $|\Pi_n|^\diamond$ \mbox{from Definition \ref{simplicialmodel}.}

In simplicial degree $i$, this chain complex is isomorphic to the direct sum of $\widetilde{\HH}_*(S^{\ell n}/_{K_\sigma})$, where $K_\sigma$ ranges over a set of representatives of isotropy groups of nondegenerate, non-basepoint \mbox{$i$-simplices} of $|\Pi_n|^\diamond$, i.e. strictly increasing chains of partitions of $\{1, \ldots, n\}$. 

Let $K_\sigma$ be the isotropy group of  $\sigma=[\hat0  < x_1<\ldots< x_i < \hat 1]$. Proposition~\ref{prop: splitting} decomposes $\widetilde{\HH}_*(S^{\ell n}/_{K_\sigma})$ even further and 
identifies it with a direct sum indexed by isomorphism types of $p$-enhancements of $\sigma$.

Our next task is to arrange most of the summands in isomorphic pairs, so that they ``cancel out''. For this, we shall distinguish between four kinds of (isomorphism classes of) $p$-enhancements.
\begin{definition}\label{definition: classification}
Let $\sigma=[\hat0  < x_1<\ldots<x_i < \hat 1]$ be a chain of partitions. By convention, we set $x_0 = \hat{0}$ and $x_{i+1} = \hat{1}$.
A $p$-enhancement  $\Theta=[\hat 0 \le e_1\le x_1\le\ldots\le e_i\le x_i \le e_{i+1}\le \hat 1]$ is said to be
\begin{enumerate}[leftmargin=26pt]
\item \textit{negligible}, if $e_1=\hat 0$.
\item \textit{matching up} if there exists a $1\le j\le i+1$ such that $e_j < x_j$, and for the smallest such $j$, there is a strict inequality $x_{j-1}<e_j$. 
\item \textit{matching down} if there exists a $1< j\le i+1$ such that $e_j < x_j$, and for the smallest such $j$, there is an equality $x_{j-1}=e_j$. 
\item \textit{pure} if $e_j = x_j$ for all $1\le j\le i+1$. 
\end{enumerate}
\end{definition}
Clearly, these notions are  invariant  under the action of $\Sigma_n$. 
The following observation is key:
\begin{proposition}\label{prop: matching}
For all $0\le i \le n-3$, there is a bijective correspondence between isomorphism types of $p$-enhancements of chains of length $i$ that are matching up and  isomorphism types of $p$-enhancements of chains of length $i+1$ that are matching down. This bijection induces an isomorphism of corresponding summands inside the groups $\widetilde{\HH}_*(S^{\ell n}/_{K_\sigma})$ attached to chains $\sigma$.
\end{proposition}
\begin{proof}
Let $0\le i \le n-3$. Let $\sigma=[\hat0  < x_1<\ldots< x_i < \hat 1]$ be a chain of partitions of length $i$, and let $\Theta=[\hat 0 \le e_1\le x_1\le\ldots\le e_i\le x_i \le e_{i+1}\le \hat 1]$ be a $p$-enhancement of $\sigma$ that is \mbox{matching up.} By definition, this means that there is a $1\le j\le i+1$ such that $x_{j-1}< e_j < x_j$, and such that for all $0\le j'< j$, we have $e_{j'}=x_{j'}$. 

We define a chain $\sigma^+$ of length  $i+1$ as $\sigma^+=[\hat0  <x_1< \ldots<x_{j-1}<e_j<x_{j}<\ldots < x_i < \hat 1]$. Let $\Theta^+$ be the $p$-enhancement of $\sigma^+$ given by $$\Theta^+=[\hat 0  \le e_1\le x_1\le\ldots \le x_{j-1}\le e_j\le e_j\le e_j\le x_{j}<\ldots\le e_i\le x_i \le e_{i+1}\le \hat 1]$$ Clearly, $\Theta^+$ is a matching down enhancement of $\sigma^+$, and this procedure induces the desired bijection. The corresponding homology summands are isomorphic by Lemma~\ref{lemma: match}.
\end{proof}
\begin{proposition}\label{prop: cancel}
The Euler characteristic of the Bredon homology $\widetilde{\HH}^{\br}_*(|\Pi_n|^\diamond;\tilde\mu_*)$ is the same as the Euler characteristic of the bigraded submodule spanned by all summands corresponding to pure $p$-enhancements.
\end{proposition}
\begin{proof}
The negligible summands vanish  by Lemma~\ref{lemma: vanish}. The matching up and matching down summands cancel out by Proposition~\ref{prop: matching}. The only summands left are the pure ones.
\end{proof}

In order to examine the  pure summands in $\widetilde{\HH}_*(S^{\ell n}/_{K_\sigma})$, we introduce the following notation:
\begin{definition}
A partition is {\it regular} if all its equivalence classes have the same size. More generally, a chain of partitions is regular if each partition in the chain is regular.
\end{definition}
\begin{lemma}\label{lemma: puregular}
A chain of partitions has a unique pure $p$-enhancement if and only if is a regular chain of partitions of a power of $p$. Otherwise it has no pure $p$-enhancements.
\end{lemma}
\begin{proof}
Pure $p$-enhancements are chains of the form  $\Theta=[\hat 0 < e_1=x_1 <\ldots <e_i= x_i < e_{i+1}=  \hat 1]$ where $e_j=x_j$ for all $1\le j\le i+1$. Therefore, the $e_j$'s are determined by the $x_j$'s and a chain can have at most one pure enhancement. 

Suppose $\Theta$ is pure. In particular, $e_{i+1}= \hat 1$. Thus, $e_{i+1}$ is the indiscrete partition with a single component of size $n$. By definition of $p$-enhancements, this implies that the number of equivalence classes of $x_i$ is a power of $p$, and that the restrictions of $\Theta$ to the equivalence classes of $x_i$ are all pairwise isomorphic. Moreover, the assumption $e_j=x_j$ for all $1\le j\le i$ implies that the restriction of $\Theta$ to any equivalence class of $x_i$ is again pure. By induction on the length of the chain, the size of each equivalence class of $x_i$ is a power of $p$, so $n$ must be a power of $p$ as well. Induction  shows that the chain $\sigma = [\hat{0} <x_1 <\ldots < x_i< \hat{1}]$ is regular as  all $x_i$ are pairwise isomorphic.

On the other hand, it is easy to check that if $\sigma = [\hat{0}<x_1<\ldots<x_i<\hat{1}]$ is a regular chain of partitions of a set of size $p^k$, then setting $e_j=x_j$ for all $j$  defines a pure $p$-enhancement of $\sigma$.
\end{proof}
In particular, if $n$ is not a power of $p$, then there are no pure summands in $\widetilde{\HH}_*(S^{\ell n}/_{K_\sigma})$, where $K_\sigma$ is any isotropy group of $|\Pi_n|^\diamond$. We obtain the following (superfluous) corollary.
\begin{corollary}
If $n$ is not a power of $p$, then the Euler characteristic of $\widetilde{\HH}^{\br}_*(|\Pi_n|^\diamond;\tilde\mu_*)$ is zero.
\end{corollary}
\vspace{5pt}
Of course, a much stronger statement is true: if $n$ is not a power of $p$, then all the Bredon homology groups vanish by~\cite{arone2016bredon} and Proposition~\ref{prop: satisfy} (cf.\ the remark following  \vspace{3pt} Proposition~\ref{prop: satisfy}).

 It remains to calculate the Euler characteristic when $n=p^k$. 
\begin{proof}[Proof of Theorem~\ref{theorem: thehomology}]  Let $n=p^k$ be a power of $p$. By Proposition~\ref{prop: cancel}, it suffices to calculate the alternating sum (with respect to the Bredon grading) of the dimensions of the pure summands and check that it matches the dimensions asserted in  the statement of the theorem. 

By Lemma~\ref{lemma: puregular} there is a unique pure summand for each isomorphism type of regular chains of partitions of $\{1,\ldots,p^k\}$. Such chains are in bijective correspondence with ordered partitions of $k$, or equivalently with subsets of $\{1, \ldots, k-1\}$. 

We will now make the pure summands more explicit. Given an ordered partition \mbox{$(k_1,\ldots,k_r)$ of $k$,} we pick a regular chain of partitions $\sigma_{(k_1,\ldots,k_r)}$ representing the corresponding isomorphism class. There is a unique pure $p$-enhancement $\Theta$ of $\sigma$, and the corresponding pure summand $[\Theta] (\widetilde{\HH}_*(S^{\ell } ) )$ inside  $\widetilde{\HH}_*(S^{\ell n}/_{K_\sigma} )$ (sitting in Bredon degree $r-1$) is readily seen to be given by 
$ \mathcal{F}_{k_1} \ldots \mathcal{F}_{k_r} (\widetilde{\HH}_*(S^{\ell } ))$. The Bredon complex can therefore be depicted as: 

\begin{diagram}
\mbox{dim} & &   & 0 & & &   1 & & &  \ldots & & &  k \\ 
 & & & \Fcal_k\widetilde{\HH}_*(S^\ell ) && &  \bigoplus_{k_1+k_2=k} \Fcal_{k_1}\Fcal_{k_2}\widetilde{\HH}_*(S^\ell ) &&  & \ldots & &&   {\Fcal}_1\ldots\Fcal_1\widetilde{\HH}_*(S^\ell ) 
\end{diagram}

\bigskip
 
Consider a typical pure summand, of the form $\Fcal_{k_1}\Fcal_{k_2}\ldots\Fcal_{k_r}\widetilde{\HH}_*(S^{\ell} )$ where $k_1+\cdots+ k_r=k$. 

An easy combinatorial check reveals that this summand has a basis consisting of sequences $(i_1, \ldots, i_k)$ satisfying the following conditions:
\begin{enumerate}[leftmargin=26pt]
\item Each $i_j$ is congruent to $0$ or $1$ mod $2(p-1)$, 
\item $i_j\ge pi_{j+1}$ for all $1\le j<k$ with $j \neq k_1, k_1+k_2,\ldots .$  \label{admissible}
\item If $p$ is odd, \ then $pi_{k_1+\ldots+k_t+1}<(p-1)(\ell+i_{k_1+\ldots+k_t+1}+\cdots+i_{k_1+\ldots+k_r})$ for $t=0,\ldots,r-1$.\\
If $p=2$, we have $pi_{k_1+\ldots+k_t+1}\le(p-1)(\ell+i_{k_1+\ldots+k_t+1}+\cdots+i_{k_1+\ldots+k_r})$ for  $t=0,\ldots,r-1$. 
\item $i_j\ne 1$ for all $1\leq j \leq k$.
\end{enumerate}
Here we suppressed the canonical generator of $\widetilde{\HH}_*(S^\ell)$ from our notation.

We now fix a sequence $(i_1, \ldots, i_k)$ which satisfies   the four conditions above and    moreover satisfies $$i_j< pi_{j+1}\mbox{  for }j=k_1, k_1+k_2, \ldots. $$
We observe that $(i_1,\ldots,i_k)$ also defines a unique  basis element in  $\mathcal{F}_{m_1} \ldots \mathcal{F}_{m_s}$ for any ordered partition $(m_1,\ldots,m_s)$ of $k$ which \textit{refines} the ordered partition $(k_1,\ldots,k_r)$.
Here, we have used  that  condition $(2)$ can be rephrased as \mbox{$i_j - (p-1){i_{j+1}} \geq i_{j+1}$.} Note that any basis element of a pure summand $\mathcal{F}_{m_1} \ldots \mathcal{F}_{m_s}$ arises uniquely in this way.

The collection of   partitions $(m_1,\ldots,m_s)$ of length $s$ refining $(k_1,\ldots,k_r)$ lies in bijection with the collection of subsets $S$ of $\{1,\ldots,k-r\}$ of size $s-r$.
The contribution from the various basis elements corresponding to the given sequence $(i_1, \ldots, i_k)$ to the  Euler characteristic in Bredon direction (in the relevant homological degree $i_1+\ldots+i_k+\ell$) is therefore given by 
$$\sum_{s = r}^{k}(-1)^{s-1}{k-r \choose s-r}=(-1)^{r-1} \sum_{s=0}^{k-r}(-1)^{s} {k-r \choose s}$$
If $k-r>0$, then this alternating sum of binomial coefficients is well-known to vanish. 

For $k=r$, this sum is equal to $(-1)^{k-1}$. This corresponds to the case where the sequence $(i_1,\ldots,i_k)$ lies in $\mathcal{F}_{1} \ldots \mathcal{F}_{1}$ and
violates condition  $(2)$ for \textit{every possible} value of  $j$. More
precisely, condition $(2)$ becomes empty and $i_j < pi_{j+1}$ for all $j$.

Hence the Euler characteristic exactly counts the number of sequences $(i_1,\ldots,i_k)$ satisfying:
\begin{enumerate}[leftmargin=26pt]
\item Each $i_j$ is congruent to $0$ or $1$ mod $2(p-1)$, 
\item $i_j<pi_{j+1}$ for all $1\le j<k$  \label{admissible}
\item If $p$ is odd, \ then $pi_{j}<(p-1)(\ell+i_{j}+\cdots+i_{k})$ for  $j=1,\ldots,k$.\\
If $p=2$, we have $pi_{j}\le(p-1)(\ell+i_{j}+\cdots+i_{k})$ for $j=1,\ldots,k$. 
\item  $i_j\ne 1$ for all $1\leq j \leq k$.
\end{enumerate}
It is now evident that these conditions are equivalent to the conditions appearing in the statement of the theorem.
\end{proof}

\newpage

\bibliographystyle{alpha} 
\bibliography{There}

\end{document}